\documentclass{amsart}
\usepackage[utf8]{inputenc}
\usepackage[T1]{fontenc}
\usepackage{lmodern}
\usepackage{amsmath}
\usepackage{amssymb}
\usepackage{mathtools}
\usepackage{latexsym}
\usepackage[lite]{amsrefs}
\usepackage{nicefrac}
\usepackage{microtype}
\usepackage{color} 
\usepackage{tikz-cd}
\usepackage{MnSymbol}
\usepackage{enumitem} 
\setlist[enumerate,1]{label=\textup{(\arabic*)}}

\usepackage[all]{xy}
\newdir{ >}{{}*!/-5pt/@{>}}
\usepackage[pdftitle={Bornological K-theory},
pdfauthor={Devarshi Mukherjee},
pdfsubject={Mathematics}
]{hyperref}

\BibSpec{book}{%
  +{}  {\PrintPrimary}                {transition}
  +{,} { \textit}                     {title}
  +{.} { }                            {part}
  +{:} { \textit}                     {subtitle}
  +{,} { \PrintEdition}               {edition}
  +{}  { \PrintEditorsB}              {editor}
  +{,} { \PrintTranslatorsC}          {translator}
  +{,} { \PrintContributions}         {contribution}
  +{,} { }                            {series}
  +{,} { \voltext}                    {volume}
  +{,} { }                            {publisher}
  +{,} { }                            {organization}
  +{,} { }                            {address}
  +{,} { \PrintDateB}                 {date}
  +{,} { }                            {status}
  +{}  { \parenthesize}               {language}
  +{}  { \PrintTranslation}           {translation}
  +{;} { \PrintReprint}               {reprint}
  +{.} { }                            {note}
  +{.} {}                             {transition}
  +{} { \PrintDOI}                   {doi}
  +{} { available at \url}            {eprint}
  +{}  {\SentenceSpace \PrintReviews} {review}
}

\renewcommand*{\PrintDOI}[1]{\href{http://dx.doi.org/\detokenize{#1}}{doi: \detokenize{#1}}}

\newcommand{\comment}[1]{}  
\newcommand{\edit}[1]{\marginpar{\footnotesize{#1}}}
\newcommand{\adj}[4]{#1\negmedspace: #2\rightleftarrows #3:\negmedspace #4}

\theoremstyle{plain}
\newtheorem{theorem}{Theorem}[section]
\newtheorem*{theorem*}{Theorem}
\newtheorem*{definition*}{Definition}
\newtheorem{lemma}[theorem]{Lemma}
\newtheorem{corollary}[theorem]{Corollary}
\newtheorem{proposition}[theorem]{Proposition}
\theoremstyle{remark}
\newtheorem{remark}[theorem]{Remark}
\theoremstyle{question}

\theoremstyle{definition}
\newtheorem{definition}[theorem]{Definition}
\newtheorem{example}[theorem]{Example}
\numberwithin{theorem}{section}

\newcommand\C{\mathbb C}
\newcommand{\fC}{\mathsf{C}}

\newcommand\N{\mathbb N}
\newcommand\Q{\mathbb Q}
\newcommand\R{\mathbb R}
\newcommand\Z{\mathbb Z}

\newcommand{\heart}{\heartsuit}

\newcommand{\bA}{\mathbf{A}}

\newcommand{\bD}{\mathbf{D}}
\newcommand{\bC}{\mathbf{C}}

\newcommand{\bdd}{\mathcal B}
\newcommand{\coma}{\widehat}

\newcommand{\defeq}{\mathrel{:=}} 

\newcommand*{\onto}{\twoheadrightarrow}

\newcommand{\colim}{\mathrm{colim}}

\DeclarePairedDelimiter{\abs}{\lvert}{\rvert}
\DeclarePairedDelimiter{\norm}{\lVert}{\rVert}
\DeclarePairedDelimiter{\gen}{\langle}{\rangle}
\DeclarePairedDelimiter{\floor}{\lfloor}{\rfloor}
\DeclarePairedDelimiter{\ceil}{\lceil}{\rceil}
\DeclarePairedDelimiterX{\setgiven}[2]{\{}{\}}{#1\,{:}\,\mathopen{}#2}

\newcommand{\rig}{\mathrm{rig}}
\newcommand\haotimes{\mathbin{\coma{\otimes}}}

\DeclareMathOperator{\Spec}{Spec}

\newcommand*{\an}{\mathrm{an}}
\newcommand*{\cont}{\mathrm{cont}}

\newcommand{\cU}{\mathcal{U}}

\newcommand{\op}{\mathrm{op}}
\newcommand{\ev}{\mathrm{ev}}
\newcommand{\dvgen}{\pi}

\newcommand{\tf}{\mathrm{tf}}

\DeclareMathAlphabet{\mathpzc}{OT1}{pzc}{m}{it}

\DeclareMathOperator{\Hom}{Hom}

\begin{document}
\title{Localising invariants in derived bornological geometry}

\author{Jack Kelly}
\email{jack.kelly@lincoln.ox.ac.uk}

\author{Devarshi Mukherjee}
\email{Devarshi.Mukherjee@maths.ox.ac.uk}

\thanks{We thank Maxime Ramzi and Thomas Nikolaus for several discussions surrounding this article. We also thank Federico Bambozzi and Ralf Meyer for pointing out useful references in earlier stages of this project, and Zhouhang Mao for some useful comments. The second named author was funded by a DFG Eigenestelle (project number 534946574) and the Deutsche Forschungsgemeinschaft (DFG, German Research Foundation) under Germany's Excellence Strategy EXC 2044390685587, Mathematics Münster: Dynamics Geometry Structure and a UK Research and Innovation Horizon Europe Guarantee MSCA Postdoctoral Fellowship award.}

\begin{abstract}
We study several categories of analytic stacks relative to the category of bornological modules over a Banach ring. When the underlying Banach ring is a non-Archimedean valued field, this category contains derived rigid analytic spaces as a full subcategory. When the underlying field is the complex numbers, it contains the category of derived complex analytic spaces. In the second part of the paper, we consider localising invariants of rigid categories associated to bornological algebras. The main results in this part include Nisnevich descent for derived analytic spaces and a version of the Grothendieck-Riemann-Roch Theorem for derived dagger analytic spaces over an arbitrary Banach ring. 
\end{abstract}

\maketitle

\tableofcontents

\section{Introduction}\label{sec:intro}

In this article, we build on the general machinery developed in \cite{ben2024perspective} to further study categories of analytic stacks and their invariants from the point of view of bornological spaces. In particular, in both the Archimedean and non-Archimedean settings, we define the (bornological) continuous $K$-theory of a scheme, and prove descent for $K$-theory in the Nisnevich topology.

Very briefly, the idea as propounded by To\"en-Vezzosi (\cite{toen2004homotopical}) is to start with a symmetric monoidal presentable stable \(\infty\)-category \(\bC\), equip a suitable  category \(\bA\subset(\mathbf{DAlg}(\mathbf{C}_{\ge0}))^{op}\) of connective derived commutative algebra objects in \(\bC\) with a topology \(\tau\) arising from the monoidal structure, and then take sheaves relative to the topology. Thereafter, one considers covers of sheaves by representables in a manner that morphisms in the cover lie in a distinguished class \(\mathbf{P}\) depending on the type of geometry that is desired. We call such a tuple \((\bC, \bA, \tau, \mathbf{P})\) a \emph{relative pre-geometry tuple} and denote the geometric stacks relative to it by \(\mathbf{Stk}_{geom}(\bA,\tau, \mathbf{P})\). When one starts with the category \(\bC = \bD(\Z)\) and the category \(\bA\) that is opposite to animated commutative rings, equipped with the \'etale topology, and smooth maps as the distinguished class of morphisms, one arrives at algebraic stacks. In earlier work (\cites{MR3626003,BaBK, bambozzi2016dagger, ben2017non, kelly2022analytic}, \cite{ben2024perspective}) of the authors, Bambozzi, Ben-Bassat and Kremnizer, we have shown that the To\"en-Vezzosi machinery may be applied towards a theory of derived analytic geometry, if one starts with the derived category \(\bD(R)\) of complete bornological modules over a base Banach ring \(R\) as the initial input. This entails proving that the category \(\bD(R)\) is a derived and homotopical algebraic context.

In the complex analytic setting, in \cite{ben2024perspective}, the first author together with Ben-Bassat and Kremnizer showed that Porta and Yue Yu's model of derived complex analytic geometry \cites{porta2015derived,porta2016higher} (with some additional mild finiteness assumptions) embeds fully faithfully in geometric stacks relative to a certain relative pre-geometry tuple based on bornological $\mathbb{C}$-vector spaces. In this relative pre-geometry tuple, we regarded algebras of functions on Stein spaces as affines. To see the usual covers from complex analytic geometry, we are required to allow countable covers. However, this countability assumption means that descent statements for quasi-coherent sheaves are difficult to formulate, and somewhat ad hoc. In the present work we take the perspective that the affines should be (opposite to) quotients of complex dagger affinoid algebras (algebras of overconvergent functions on a closed polydisc, see e.g., \cite{bambozzi2016dagger}). We topologise the category \(\bA_\C^\dagger\) of such affines using rational localisations. In this setup, the affines are in some sense compact Steins, and so we are justified in restricting to finite covers. To summarise, we get a relative pre-geometry tuple \((\bD_{\geq 0}(\C), \bA_\C^\dagger, \tau_{\mathbb{C}}^{rat}, \mathbf{P}^{o-sm})\), where $ \mathbf{P}^{o-sm}$ is a suitably defined class of smooth maps. Moreover, we show that complex analytic spaces do still embed fully faithfully as schemes on this site. This program can also be carried out for non-trivially valued non-Archimedean fields, and in this case one gets an embedding of partially proper rigid analytic spaces.

In the non-Archimedean setting, in \cite{soor2024derived}, following \cite{ben2024perspective}, Soor defined the category of derived rigid analytic spaces over a non-trivially valued non-Archimedean Banach field $K$ as a full subcategory of the category stacks relative to bornological $K$-vector spaces. The category of derived rigid analytic spaces contains rigid analytic spaces as a full subcategory. In this article, we explain why Soor's derived analytic stacks may in fact be regarded as schemes (in particular, as geometric stacks) relative to the relative pre-geometry tuple \((\bD_{\geq 0}(K), \bA_K, \tau_K^{rat}, \mathbf{P}^{o-sm})\), where $\bA_{K}$ is the category of derived affinoids, $\tau_K^{rat}$ is the rational localisation topology, and $\mathbf{P}^{o-sm}$ is a certain class of smooth maps. This proof is not difficult - it is essentially a matter of untangling the two definitions.

To summarise, we get the following.

\begin{theorem*}
We have the following inclusions.
    \begin{enumerate}
        \item The category of derived rigid analytic spaces embeds fully faithfully inside the category of geometric stacks \(\mathbf{Stk}_{geom}(\bA_K,\tau_{K}^{rat}, \mathbf{P}^{o-sm})\) relative to \((\bD_{\geq 0}(K), \bA_K, \tau_K^{rat}, \mathbf{P}^{o-sm})\);
        \item The category of partially proper rigid analytic spaces embeds fully faithfully inside the category of geometric stacks \(\mathbf{Stk}_{geom}(\bA_K^\dagger, \tau_K^{rat}, \mathbf{P}^{sm})\) relative to \((\bD_{\geq 0}(K), \bA_K^\dagger, \tau_K^{rat}, \mathbf{P}^{o-sm})\);
        \item The category of derived Stein spaces embeds fully faithfully inside geometric stacks \(\mathbf{Stk}_{geom}(\bA_\C^\dagger, \tau_\C^{rat}, \mathbf{P}^{o-sm})\) relative to \((\bD_{\geq 0}(\C), \bA_\C^\dagger, \tau_\C^{rat}, \mathbf{P}^{o-sm})\).
    \end{enumerate}
\end{theorem*}

\comment{
The focus of this article is to obtain rigid analytic spaces and complex manifolds as full subcategories of geometric stacks relative to bornological modules. To formalise what we mean, we set up some notation.  Let \(K\) be a non-trivially valued non-Archimedean Banach field of characteristic zero and \(\mathcal{O}_K\) its ring of integers with pseudo-uniformiser \(\dvgen\). Consider the categories \(\bA_K\) and \(\bA_{\mathcal{O}_K}^{form}\) of (connective derived) affinoid \(K\)-algebras and (connective derived) admissible formal \(\mathcal{O}_K\)-algebras, viewed as full subcategories of \(\mathbf{DAlg}(\bD_{\geq 0}(K))\) and \(\mathbf{DAlg}(\bD_{\geq 0}(\mathcal{O}_K))\). Then equipped with the topologies \(\tau_K^{rat}\) and \(\tau_{\mathcal{O}_K}^{et}\) of (derived) rational localisations and adic \'etale maps, and \(\mathbf{P}^{sm}\) the class of smooth maps, we obtain relative pre-geometry tuples \((\bD(K), \bA_K, \tau_K^{rat}, \mathbf{P}^{sm})\) and \((\bD(\mathcal{O}_K), \bA_{\mathcal{O}_K}, \tau_{\mathcal{O}_K}^{et}, \mathbf{P}^{sm})\). Over the complex numbers, Denote the corresponding relative pre-geometry tuple by \((\bD_{\geq 0}(\C), \bA_\C^\dagger, \tau_{\mathbb{C}}^{rat}, \mathbf{P}^{sm})\). Our first main result demonstrating that these are indeed appropriate categories for the purposes of analytic geometry is following:

}


We then consider descent results for quasi-coherent sheaves. Given a complete bornological algebra \(A \in \mathbf{DAlg}(\bD_{\geq 0}(R))\), we denote by \(\bD(A) \defeq \mathbf{Mod}_A(\bD(R))\) the derived category of \(A\)-modules relative to \(\bD(R)\). In what follows, denote by \(\mathbf{Pr}_{st}^{L}\) the category of presentably stable \(\infty\)-category with left adjoint functors as morphisms. We wish to show that $\mathbf{Mod}(-):(\mathbf{DAlg}(\bD_{\geq 0}(R)))^{op}\rightarrow\mathbf{Pr}_{st}^{L}$ is a sheaf for the various topologies we consider. For the finite homotopy monomorphism topology, which includes all rational topologies, this is not hard. In fact it has appeared in various guises in \cite{ben2024perspective} and \cite{soor2024derived} (and hark back to similar results in \cites{MR3626003,BaBK, bambozzi2016dagger, ben2017non}, and \cite{ben2024perspective}). Our more significant contribution happens in the non-Archimedean world. In particular, we prove that quasi-coherent sheaves on derived rigid analytic spaces satisfy descent for the \'{e}tale topology rather than just the rational localisation topology. More precisely, we define the \'{e}tale topology for general derived algebraic contexts (and the associated datum of a Lawvere theory) in such a way that it automatically satisfies descent, and show that this topology restricts to the usual \'{e}tale topology on affinoids. Significantly, this requires a major detour through the world of (affine) formal geometry. To formalise what we mean, we set up some notation.  Let \(K\) be a non-trivially valued non-Archimedean Banach field of characteristic zero and \(\mathcal{O}_K\) its ring of integers with pseudo-uniformiser \(\dvgen\). Consider the categories \(\bA_K\) and \(\bA_{\mathcal{O}_K}^{form}\) of (connective derived) affinoid \(K\)-algebras and (connective derived) admissible formal \(\mathcal{O}_K\)-algebras, viewed as full subcategories of \(\mathbf{DAlg}(\bD_{\geq 0}(K))\) and \(\mathbf{DAlg}(\bD_{\geq 0}(\mathcal{O}_K))\) respectively. First, we define a topology $\tau_{ad}^{rig-\'{e}t}$ on $(\mathbf{DAlg}(\bD_{\geq 0}(\mathcal{O}_K)))^{op}$ which restricts to the usual rig-\'{e}tale topology on discrete admissible $\mathcal{O}_{K}$-algebras. Then we show that this topology is descendable - this uses \textit{adic descent} techniques, based on similar results of \cite{mann2022p} in the condensed setting. Finally, we show that any \'{e}tale map of (derived) affinoids has a $\tau_{ad}^{rig-\'{e}t}$-formal model. Since the base-change of a descendable morphism is descendable, this proves descendability of \'{e}tale maps of (derived) affinoids. To summarise, we get the following.

\begin{theorem*}
     Let \(\bA_R\) be one of \(\bA_K\), or \(\bA_\C^\dagger\).
    \begin{enumerate}
        \item     The assignment \[\bA_R \to \mathbf{DAlg}(\mathbf{Pr}_{st}^{L}), \quad A \mapsto \bD(A)\] is a sheaf of \(\infty\)-categories for the topologies \(\tau_K^{\'{e}t}\) and \(\tau_{\C}^{rat}\).
        \item    For \(A \in \bA_R\), the representable functor \(\bA_R \to \mathbf{An}\), \(B \mapsto \mathbf{Hom}(A,B)\) is a hypersheaf.
    \end{enumerate}    
\end{theorem*}

The importance of incorporating the \'{e}tale topology is that it allows us to both define and prove descent for the Nisnevich topology. Our definition of the Nisnevich topology requires yet another detour. Precisely, we need to consider properties of the spaces of prime and maximal filters associated to the locale of subspaces of derived schemes in our various settings. In the derived rigid analytic setting, this was previously considered in \cite{soor2024derived}/ \cite{soor2025privatecommunicationhesis}. This involves a general discussion of what we call \textit{Berkovich locales}.


We then consider localising invariants of analytic spaces. Classically, a \emph{localising invariant} was defined as a functor \(E \colon \mathbf{Cat}^{\mathrm{perf}} \to \bD\) from the category of small stable \(\infty\)-categories and exact functors to a stable \(\infty\)-category, preserving Verdier localisation sequences. For a quasi-compact, quasi-separated scheme \(X\), the derived category \(\bD(X)\) of quasi-coherent sheaves of \(\mathcal{O}_X\)-modules is a compactly generated presentable stable \(\infty\)-category. One then applies \(E\) to the category of compact objects \(\mathbf{Perf}(X)\), which is small. Here a major difference arises from the purely algebraic setting: in the derived category of complete bornological modules, we have split exact sequences \[l^1(\N,R) \to l^1(\N,R) \to R,\] so that the functor \(A \mapsto l^1(A)\) on Banach spaces (which are the compact objects in \(\mathsf{Ind}(\mathsf{Ban}_R)\)) satisfies \(F(A) = A \oplus F(A)\), so that we would have Eilenberg swindle if we defined localising invariants using perfect complexes. 

From a functional analytic point of view, a more appropriate finiteness condition is nuclearity. It turns out that there is a functorial way to associate to a complete bornological \(R\)-module a \emph{nuclear module} - defined as a colimit of trace-class maps between Banach \(R\)-modules, and this functor is right adjoint to the inclusion of the full subcategory of nuclear \(R\)-modules \(\mathsf{Nuc}(R) \subset \mathsf{CBorn}_R\) inside bornological \(R\)-modules (\cites{hogbe2011nuclear, Meyer-Mukherjee:HL}). In the non-Archimedean setting, one first forces nuclearity on a Banach \(\mathcal{O}_K\)-algebra, and then applies the analytic cyclic homology functor to define local cyclic homology(\cites{Cortinas-Meyer-Mukherjee:NAHA, Meyer-Mukherjee:HL}). This process ensures that local cyclic homology only depends on the reduction mod \(p\) of a Banach algebra. Remarkably, one can do this completely categorically by associating to a symmetric monoidal category \(\bC\) a presentable category \(\bC^{\rig}\) that is built out of sequential colimits of trace-class maps when the tensor unit is compact(\cite{nkp}).  When applied to our setting \(\bC = \bD(R)\), we get the following:

\begin{theorem*}
   \begin{enumerate}
\item Let \(\bD(\mathsf{Nuc}^\infty(\C))\) be the full subcategory of \(\bD(\C)\) generated under colimits by sequential colimits of Banach spaces with order \(0\) trace-class transition maps. Then we have a fully faithful inclusion \[\bD(\mathsf{Nuc}^\infty(\C))\hookrightarrow \bD(\C)^\rig\] of presentable stable \(\infty\)-categories.  
\item When \(R\) is a non-Archimedean Banach ring, we have an equivalence \[\bD(R)^\rig \simeq \bD(\mathbf{Nuc}(R)).\] 
\end{enumerate}
\end{theorem*}

While the theorem above shows that the assignment \(\mathbf{C} \mapsto \mathbf{C}^\rig\) in some sense recovers the internal nuclearisation \((-)^\rig \colon \mathsf{CBorn}_R \to \mathsf{Nuc}^{\infty}(R)\) of a bornological module, we do not yet know whether the category of nuclear modules is a legitimate input category for \(K\)-theory. Indeed, the rigidification \(\bD(R)^\rig\) is not compactly generated but only \emph{dualisable}, and therefore in particular, \(\omega_1\)-compactly generated. Fortunately, by Efimov's theorem, any localising invariant extends uniquely to the category \(\mathbf{Cat}^{\mathrm{dual}}\) of dualisable categories. This leads to the following:

\begin{definition*}
    Let \(E\) be a localising invariant \(\mathbf{Cat}^{\mathrm{perf}} \to \bD\) into a stable \(\infty\)-category. Its associated \emph{bornological invariant} is defined as the functor \[E^\an \colon \bA_R \to \bD, \quad \mathsf{Spec}(A) \mapsto  E^\cont(\mathbf{Mod}_{A^\rig}(\bD(R)^\rig)),\] where \(E^\cont\) is the continuous extension of \(E\) to dualisable categories.  
\end{definition*}

The second part of the paper focusses on descent properties of these invariants. We first observe that the definition of a bornological  invariant globalises as the assignment \(A \mapsto \mathbf{Mod}_{A}(\mathbf{C}^{rig})\) satisfies \(\tau_R^{rat}\) and \(\tau_R^{et}\)-descent (in the category of dualisable categories) when \(A \in \bD(R)^\rig\). The latter condition holds for derived dagger affinoid \(R\)-algebras, and derived admissible \(\mathcal{O}_K\)-algebras.


Our next result is the following generalisation of Thomason-Trobaugh's result to the analytic setting:

\begin{theorem*}
    For any localising invariant \(E \colon \mathbf{Cat}^{\mathrm{dual}} \to \bD\), its associated bornological invariant satisfies Nisnevich descent. More specifically, consider a commuting diagram 
    \begin{equation*}
\begin{tikzcd}
U \times_X V \arrow{r}{} \arrow{d}{} & V \arrow{d}{p}\\
U \arrow{r}{e} & X,
\end{tikzcd}
\end{equation*} of derived analytic spaces such that \(e\) is an open immersion, \(p\) is \'etale and spatial on the complement of \(U\),  and \(|V|_{Ber} \times_|X|_{Ber} (|X|_{Ber} \setminus |U|_{Ber}) \to |X|_{Ber} \setminus |U|_{Ber}\) is a homeomorphism on Berkovich points (see \ref{eq:preNis} for precise definition). Then we have the following pullback square 
 \begin{equation*}
\begin{tikzcd}
E^\an(X) \arrow{r}{} \arrow{d}{} & E^\an(U) \arrow{d}{p}\\
E^\an(V) \arrow{r}{e} & E^\an(U \times_X V),
\end{tikzcd}
\end{equation*} in \(\bD\).

\end{theorem*}

Finally, we come to the Grothendieck-Riemann-Roch Theorem for derived analytic spaces. Let \(f \colon X \to Y\) be a quasi-separated morphism of quasi-compact, quasi-separated derived dagger \(R\)-analytic spaces. 

\begin{theorem*}
We have the following commuting diagram of spectra
\[
\begin{tikzcd}
K^\an(X) \arrow{r}{\mathrm{ch}} \arrow{d}{f_{!}} & TC^{- \text{ }\an}(X) \arrow{d}{f_*} \\
K^\an(Y) \arrow{r}{\mathrm{ch}} & TC^{- \text{ }\an}(Y),
\end{tikzcd}
\] where \(\mathrm{ch}\) is the Chern character and \(TC^{- \text{ }\an}\) denotes the continuous extension of negative topological cyclic homology. 
\end{theorem*}

\subsection*{Relation to condensed mathematics}

Several results in this article have analogues in the condensed world. In this section, we discuss how our approach is similar and different to that of Clausen-Scholze (\cite{clausenscholze3}). To begin with, both approaches lead to categories of analytic stacks for which quasi-coherent sheaves satisfy descent. They also develop \(K\)-theory as continuous \(K\)-theory of suitable rigid categories associated to \(\mathbf{QCoh}\). The main point of difference lies in the following: we use a uniform tensor product throughout our theory. In the archimedean setting, our category of nuclear modules \(\mathbf{Nuc}(\C)\) seems closely related to nuclear objects in the liquid setting (in fact we believe that these categories are, on the nose, equivalent) In any case, in both settings these categories of nuclear objects are not rigid. If instead, in the condensed setting, one works with gaseous $\mathbb{C}$-vector spaces and the gaseous tensor product, the category of nuclear modules \textit{is} rigid. However, in the bornological setting we still need to further rigidify our category of nuclear modules (this also coincides the rigidification of the original category). It is possible that the rigidifcation we describe is related to nuclear gaseous spaces, but we have not thought about this in detail. Furthermore, the solid framework is used in \cite{andreychev2023k} to prove Nisnevich descent and Grothendieck-Riemann-Roch Theorem for analytic adic spaces. In the present article, we use the bornological framework to prove versions of these results for derived dagger analytic spaces over an arbitrary base Banach ring. The general definition of Nisnevich covers at first glance seems difficult to check (and over general Banach rings it probably is), but we show that it includes many examples from complex geometry and rigid analytic geometry.

\comment{\subsection*{What this article does not do}

While we do not compare our theories, some comparison results seem to automatically emerge. For instance, the \(\omega_1\)-compact objects in the complex (resp. rigid) setting coincide in both cases to strongly dual nuclear Frechet (resp. sequential colimits of Banach) spaces. In the formal scheme setting, our nuclear category is equivalent to \(\widetilde{\mathbf{Nuc}}(R)\), which in turn is \(\mathbf{NcMot}\)-equivalent to the nuclear solid category of Clausen-Scholze. We will take up comparisons between bornological modules and various categories of condensed modules in a successor article. Finally, in this article we do not treat bornological ring spectra; this will be done in a future article.  }

\section{Derived geometry relative to a derived algebraic context - generalities}\label{sec:der-context}

In this section, we lay out a very general framework for algebraic geometry relative to a \emph{derived algebraic context} introduced by Raksit in \cite{raksit2020hochschild}, following the work of \cite{brantnermathew}. In the first three subsections, we largely recall material from \cite{ben2024perspective} and \cite{kelly2022analytic}.

\begin{definition}\label{def:DAC}
A \emph{derived algebraic context} consists of the following data
$$(\mathbf{C},\mathbf{C}_{\ge0},\mathbf{C}_{\le0},\mathbf{C}^{0})$$
where 
\begin{enumerate}
    \item $\mathbf{C}$ is a presentably symmetric monoidal stable $\infty$-category.
    \item $(\mathbf{C}_{\ge0},\mathbf{C}_{\le0})$ is a right complete $t$-structure on $\mathbf{C}$.
    \item $\mathbf{C}^{0}\subseteq\mathbf{C}^{\heartsuit}$ is a set of compact projective generators of $\mathbf{C}_{\ge0}$ - that is for each $P\in\mathbf{C}^{0}$ the functor
    $$\mathbf{Map}(P,-):\mathbf{C}_{\ge0}\rightarrow\mathbf{An}$$
    commutes with sifted colimits, and every object of $\mathbf{C}_{\ge0}$ can be written as a sifted colimit of objects of $\mathbf{C}^{0}$.
\end{enumerate}

such that

\begin{enumerate}
\item
the monoidal unit $\mathbb{I}$ is in $\mathbf{C}^{0}$
    \item 
    $\mathbf{C}_{\le0}$ is closed under filtered colimits in $\mathbf{C}$
    \item 
    if $X,Y\in\mathbf{C}_{\ge0}$ then so is $X\otimes Y$,
    \item 
    if $P,Q\in\mathbf{C}^{0}$ then so is $P\otimes Q$
    \item 
    if $P\in\mathbf{C}^{0}$ then each $1$-categorical symmetric power (i.e., computed in the heart $\mathbf{C}^{\heartsuit})$), $\mathrm{Sym}^{n}(P)$, is in $\mathbf{C}^{0}$. 
\end{enumerate}

\end{definition}

\begin{example}\label{ex:initial}
Let \(R\) be a (discrete) commutative unital ring. Then the derived \(\infty\)-category \(\bD(R)\) with the projective \(t\)-structure yields a derived algebraic context \(\mathbf{C}_R \defeq (\bD(R), \bD_{\geq 0}(R), \bD_{\leq 0}(R), \mathbf{Mod}_R^{ffg})\), where \(\mathbf{Mod}_R^{ffg}\) is the full subcategory of \(\bD(R)\) generated by free \(R\)-modules of finite rank.   
\end{example}

Using techniques developed by \cite{brantnermathew} the functor $\mathbf{C}^{0}\rightarrow\mathbf{C}^{\heartsuit}, P\mapsto\mathrm{Sym}(P)$ extends to a sifted-colimit preserving functor 
$$\mathbf{C}_{\ge0}\rightarrow\mathbf{C}_{\ge0}$$
and then to a functor
$$\mathbf{C}\rightarrow\mathbf{C}$$
This functor has a natural structure as a monad, which we denote by $\mathbb{L}\mathrm{Sym}$. The category of algebras over this monad will be denoted $\mathbf{DAlg}(\mathbf{C})$, and the category of algebras whose underlying object lies in $\mathbf{C}_{\ge0}$ will be denoted $\mathbf{DAlg}^{cn}(\mathbf{C})$. There is also the usual monad $\mathbf{Comm}$ associated to the commutative operad. There is a natural map of monads $\mathbf{Comm}\rightarrow\mathbb{L}\mathrm{Sym}$ which is an equivalence whenever $\mathbf{C}$ is enriched over $\mathbb{Q}$. The map of monads $\mathbf{Comm}\rightarrow\mathbb{L}\mathrm{Sym}$ induces a forgetful functor $\theta:\mathbf{DAlg}(\mathbf{C})\rightarrow\mathbf{Comm}(\mathbf{C})$. For $A\in\mathbf{DAlg}(\mathbf{C})$ we write $\mathbf{D}(A)=\mathbf{Mod}_A(\mathbf{C})=\mathbf{Mod}_{\theta(A)}(\mathbf{C})$.

Let \(\mathbf{Pr}_{st}^L\) denote the \(\infty\)-category of  presentable, stable \(\infty\)-categories with left adjoint functors as morphisms. Note that this is itself a presentably symmetric monoidal stable $\infty$-category. The assignment \[\mathbf{DAlg}(\bC) \to \mathbf{CAlg}(\mathsf{Pr}_{st}^L), \quad A \mapsto \mathbf{D}(A)\] can be upgraded to a symmetric monoidal functor. 

To `globalise' this definition, in this section we introduce several Grothendieck topologies on \emph{affine stacks relative to \(\bC\)} denoted by $\mathbf{Aff}_{\mathbf{C}}=(\mathbf{DAlg}^{cn}(\mathbf{C}))^{op}$. For $A\in\mathbf{DAlg}^{cn}(\mathbf{C})$ we denote by \(\mathsf{Spec}(A)\) the same object viewed in the opposite category $\mathbf{Aff}^{cn}_{\mathbf{C}}$. We say that a \emph{topology \(\tau\) satisfies descent for quasi-coherent sheaves} if whenever $\{\mathrm{Spec}(B_{i})\rightarrow\mathrm{Spec}(A)\}_{i\in\mathcal{I}}$ is a cover, the natural functor
$$\mathbf{QCoh}(\mathrm{Spec}(A))\rightarrow\lim_{(i_{1},\ldots,i_{n})\in\mathcal{I}^{n}}\mathbf{QCoh}(\mathrm{Spec}(B_{i_{1}}\otimes_{A}B_{i_{2}}\otimes_{A}\ldots\otimes_{A}B_{i_{n}})$$
is an equivalence.

\subsubsection{Descent for covers of stacks}

Let $(\mathbf{Aff}_\bC,\tau,\mathbf{P},\mathbf{A})$ be a relative $(\infty,1)$-geometry tuple. 

\begin{definition}
    The \textit{effective epimorphism topology on} $\mathbf{Stk}(\mathbf{Aff}_\bC,\tau)$ consists of collections of maps $\{f:\mathcal{X}\rightarrow\mathcal{Y}\}$ which are effective epimorphisms. We denote this topology by $\tau_{eff}$.
\end{definition}

Suppose that $\mathbf{QCoh}$ satisfies descent. The following is immediate.

\begin{theorem}\label{thm:adicdescendable}
    Let $f:\mathcal{X}\rightarrow\mathcal{Y}$ be an effective epimorphism of stacks. Then the natural map
    $$\mathbf{QCoh}(\mathcal{Y})\rightarrow\mathbf{lim}_{n}\mathbf{QCoh}(\mathcal{X}^{\times_{\mathcal{Y}}n})$$
    is an equivalence. In particular $\mathbf{QCoh}$ satisfies descent for $\tau_{eff}$.
\end{theorem}

Effective epimorphisms can be tested by pulling back to an atlas. More generally we have the following.

\begin{lemma}
    Let $f:\mathcal{X}\rightarrow\mathcal{Y}$ be a map of stacks. Suppose there exists an effective epimorphism $V\rightarrow\mathcal{X}$ and a morphism $U\rightarrow\mathcal{Y}$ such that the morphism $U\times_{\mathcal{X}}V\rightarrow V$ is an effective epimorphism. Then $f$ is an effective epimorphism.
\end{lemma}

\begin{proof}
    Let $W\rightarrow\mathcal{Y}$ with $W$ affine. There exists a cover $\{W_{k}\rightarrow W\}$ such that each map $W_{k}\rightarrow\mathcal{X}$ factors through $V\rightarrow\mathcal{X}$. But then there for each $k$ there is a cover $\{W_{ik}\rightarrow W_{i}\}$ such that each $W_{ik}\rightarrow V$ factors through $U\times_{\mathcal{X}}V$. In particular each $W_{ik}$ factors through $f:\mathcal{X}\rightarrow\mathcal{Y}$. Now $\{W_{ik}\rightarrow W\}$ is a cover, and this proves that $f$ is an epimorphism of sheaves.
\end{proof}

\subsection{Classes of maps}

We now introduce the various classes of maps we will primarily be interested in to define topologies on \(\mathbf{Aff}_\bC\). 

\subsubsection{Derived strong maps}

\begin{definition}[\cite{ben2024perspective}*{Definition 2.3.71}]
    A map $f:A\rightarrow B$ in $\mathbf{DAlg}(\mathbf{C})$ is said to be \textit{derived strong} if the natural map of graded objects
    $$\pi_{0}(B)\hat{\otimes}^{\mathbb{L}}_{\pi_{0}(A)}\pi_{*}(A)\rightarrow\pi_{*}(A)$$
    is an equivalence.
\end{definition}

\subsubsection{Formally \'{e}tale maps}
Recall that a map $A\rightarrow B$ in $\mathbf{DAlg}(\mathbf{C})$ has a cotangent complex $\mathbb{L}_{B\big\slash A}$ satisfying the usual properties (see \cite{ben2024perspective} for details).

\begin{definition}
    We call a map $A\rightarrow B$ in $\mathbf{DAlg}(\mathbf{C})$ \textit{formally \'{e}tale} if $\mathbb{L}_{B\big\slash A}\cong 0$.
\end{definition}

Let $f:A\rightarrow B$ be a map in $\mathbf{CAlg}(\mathbf{C}^{\heart})$, We write 
$$\Omega^{1}_{B\big\slash A}=\pi_{0}(\mathbb{L}_{B\big\slash A}).$$


\begin{lemma}
Let $f:A\rightarrow B$ be a map in $\mathbf{CAlg}(\mathbf{C}^{\heartsuit})$ such that 
\begin{enumerate}
    \item $\Omega^{1}_{B\big\slash A}\cong0;$
    \item 
    the map $B\otimes^{\mathbb{L}}_{A}B\rightarrow B\otimes_{A}B$
    is an equivalence;
    \item 
    the kernel $I$ of the map $B\otimes_{A}B\rightarrow B$ is finitely generated.
\end{enumerate}
Then $f$ is formally \'{e}tale.
\end{lemma}

\begin{proof}
    We have $\Omega^{1}_{B\big\slash A}\cong0$. In particular the map is discrete formally unramified. By \cite{ben2024perspective}*{Lemma 2.6.154}, this implies that $B\otimes_{A}B\rightarrow B$ is a Zariski localisation, and hence homotopy epimorphism. Thus it is in particular formally \'{e}tale. Moreover by assumption $B\otimes^{\mathbb{L}}_{A}B\rightarrow B\otimes_{A}B$ is an equivalence, so $B\otimes^{\mathbb{L}}_{A}B\rightarrow B$ is formally \'{e}tale. Hence $A\rightarrow B$ is formally \'{e}tale.
\end{proof}

Most of our formally \'{e}tale maps will arise in the following fashion. 

\begin{definition}\label{def:generatingclass}
    Let $\underline{\Lambda}$ be a filtered category. A $\underline{\Lambda}$-\textit{generating class of algebras} is a functor $\mathrm{T}:\underline{\Lambda}^{op}\rightarrow\mathbf{CAlg}(\mathbf{C}^{\heart})$, together with maps
    $$p_{\lambda} \colon \mathrm{Sym}(\mathbb{I})\rightarrow\mathrm{T}(\lambda)$$
    such that  
    \begin{enumerate}
        \item 
         for all maps $\lambda\rightarrow\gamma$ in $\underline{\Lambda}$ the diagram below commutes
         \begin{displaymath}
             \xymatrix{
             \mathrm{Sym}(\mathbb{I})\ar[r]^{p_{\gamma}}\ar[dr]^{p_{\lambda}} &\mathrm{T}(\gamma)\ar[d]\\
             & \mathrm{T}(\lambda);
             }
             \end{displaymath}
             \item
each $\mathrm{T}(\lambda)$ is flat as an object of $\mathbf{C}^{\heart}$.
    \end{enumerate}
    The class is said to be \textit{of \'{e}tale polynomial type} if each $p_{\lambda}$ is formally \'{e}tale. 
\end{definition}

When we specialise to bornological algebras in the next section, $\underline{\Lambda}$ will either be the one element category or $\mathbb{R}_{>0}$. The latter case will arise when we want to consider algebras on discs of different radii.  In \cite{ben2024perspective} these generating classes were framed in terms of Lawvere theories, but we will not need the full formalism in this article. 

Let $\mathrm{T}$ be a $\underline{\Lambda}$-generating class of algebras of \'{e}tale polynomial type. We define
$$\mathrm{T}(\lambda_{1},\ldots,\lambda_{n})=\mathrm{T}(\lambda_{1})\otimes\ldots\otimes\mathrm{T}(\lambda_{n}).$$

We recall from \cite{ben2024perspective} what the Jacobian of a collection of maps is, and what it means for a map to be $\mathrm{T}$-\'{e}tale.

Let $A\in\mathbf{DAlg}^{cn}(\mathbf{C})$, and let $f_{1},\ldots,f_{n}:\mathbb{I}\rightarrow A\otimes^{\mathbb{L}}\sf T(\lambda_{1},\ldots,\lambda_{n})$ be a collection of maps. This gives a map of algebras
$$\mathbb{L}\sf{Sym}_{A}(A^{\oplus n})\rightarrow A\otimes^{\mathbb{L}}\sf T(\lambda_{1},\ldots,\lambda_{n}).$$
Passing to cotangent complexes, this gives a map 
\begin{align*}
A\otimes^{\mathbb{L}}\sf T(\lambda_{1},\ldots,\lambda_{n})\otimes_{\mathbb{L}\sf{Sym}_{A}(A^{\oplus n})}\mathbb{L}_{\mathbb{L}\sf{Sym}_{A}(A^{\oplus n})\big\slash A} & \cong (A\otimes^{\mathbb{L}}\sf T(\lambda_{1},\ldots,\lambda_{n}))^{\oplus n}\\
&\rightarrow \mathbb{L}_{A\otimes^{\mathbb{L}}\sf T(\lambda_{1},\ldots,\lambda_{n})\big\slash A}\\
&\cong  (A\otimes^{\mathbb{L}}\sf T(\lambda_{1},\ldots,\lambda_{n}))^{\oplus n},
\end{align*}
where in the last equivalence we have used that $\mathbb{L}\mathrm{Sym}_{A}(A^{\oplus n})\rightarrow A\otimes^{\mathbb{L}}T(\lambda_{1},\ldots,\lambda_{n})$ is formally \'{e}tale.
This corresponds to an element of 
$$\sf{Hom}(\mathbb{I},\pi_{0}(A\otimes^{\mathbb{L}}\sf T(\lambda_{1},\ldots,\lambda_{n})))^{n^{2}}.$$
We call this element the \textit{Jacobian} of $(f_{1},\ldots,f_{n})$, and denote it by $J(f_{1},\ldots,f_{n})$. We recall the following from \cite{ben2024perspective}.

\begin{proposition}\label{prop:Tstandard}
Let $\mathrm{T}$ be a $\underline{\Lambda}$-generating class of algebras of \'{e}tale polynomial type.  Let  $f_{1},\ldots,f_{n}:\mathbb{I}\rightarrow\pi_{0}(A)\otimes^{\mathbb{L}} T(\lambda_{1},\ldots,\lambda_{n}))$ be such that the image of $J(f_{1},\ldots,f_{n})$ in $\mathrm{Hom}(\mathbb{I},\pi_{0}(B))$ is a unit . Then
   $$A\rightarrow A\otimes^{\mathbb{L}}\sf{T}(\lambda_{1},\ldots,\lambda_{n}))\big\slash\big\slash(f_{1},\ldots,f_{n})$$
   is formally \'{e}tale.
\end{proposition}

\begin{proof}
     We have a homotopy pushout diagram
\begin{displaymath}
\xymatrix{
\mathrm{Sym}_{A}(A^{\oplus n})\ar[d]\ar[r] & A\otimes\mathrm{T}(\lambda_{1},\ldots,\lambda_{n})\ar[d]\\
A\ar[r] & B.
}
\end{displaymath}
Thus 
$$\mathbb{L}_{B\slash A}\cong B\otimes^{\mathbb{L}}_{A\otimes\mathrm{T}(\lambda_{1},\ldots,\lambda_{n})}\mathbb{L}_{A\otimes\mathrm{T}(\lambda_{1},\ldots,\lambda_{n})\slash\mathrm{Sym}_{A}(A^{\oplus n})}.$$
Now we have 
$$\mathbb{L}_{\mathrm{Sym}_{A}(A^{\oplus n})\slash A}\cong \mathrm{Sym}_{A}(A^{\oplus n})^{\oplus n},$$
and 
$$\mathbb{L}_{A\otimes\mathrm{T}(\lambda_{1},\ldots,\lambda_{n})\slash A}\cong A\otimes\mathrm{T}(\lambda_{1},\ldots,\lambda_{n})^{\oplus n}.$$
We get the standard fibre-cofibre sequence 
$$A\otimes\mathrm{T}(\lambda_{1},\ldots,\lambda_{n})^{\oplus n}\rightarrow A\otimes\mathrm{T}(\lambda_{1},\ldots,\lambda_{n})^{\oplus n}\rightarrow\mathbb{L}_{A\otimes\mathrm{T}(\lambda_{1},\ldots,\lambda_{n})\slash\mathrm{Sym}_{A}(A^{\oplus n})
},$$
where the first morphism is $J$.
Thus $B\otimes^{\mathbb{L}}\mathbb{L}_{A\otimes\mathrm{T}(\lambda_{1},\ldots,\lambda_{n})\slash\mathrm{Sym}_{A}(A^{\oplus n})}$ is the cofibre of the map
$$B^{\oplus n}\rightarrow B^{\oplus n}$$
induced by the Jacobian. The Jacobian is a unit in 
$$\mathrm{Hom}(\mathbb{I}^{\oplus n},\pi_{0}(B)^{\oplus n})\cong\pi_{0}\mathbf{Map}(\mathbb{I}^{\oplus n},B^{\oplus n})\cong\pi_{0}\mathbf{Map}_{B}(B^{\oplus n},B^{\oplus n}).$$
In particular the map is an equivalence, so 
$$\mathbb{L}_{B\slash A}\cong B\otimes^{\mathbb{L}}_{A\otimes\mathrm{T}(\lambda_{1},\ldots,\lambda_{n})}\mathbb{L}_{A\otimes\mathrm{T}(\lambda_{1},\ldots,\lambda_{n})\slash\mathrm{Sym}_{A}(A^{\oplus n})}\cong0$$ as required. 
\end{proof}

\begin{definition}[\cite{ben2024perspective}*{Definition 4.5.84}]\label{def:T-standard-et}
    We call a map $$A\rightarrow A\otimes\mathrm{T}(\lambda_{1},\ldots,\lambda_{n})\big\slash\big\slash(f_{1},\ldots,f_{n})$$ as in Proposition \ref{prop:Tstandard}, a $\mathrm{T}$-\textit{standard \'{e}tale map}.
\end{definition}

When we come to bornological algebras in the next section, we will give several examples of interest in analytic geometry. However one example that works in full generality is the case that $\underline{\Lambda}=\{*\}$ is the one-element category. We then set $\mathrm{T}(*)=\mathbb{L}\mathrm{Sym}(\mathbb{I})$. The map $p_{*}:\mathbb{L}\mathrm{Sym}(\mathbb{I})\rightarrow\mathrm{T}(*)$ is the identity. For the derived algebraic context $\mathbf{C}_{R}$ of complexes of $R$-modules over a ring $R$, a standard $\mathrm{T}$-\'{e}tale map is nothing but a standard \'{e}tale map over $R$ in the usual sense. 

\subsubsection{Homotopy epimorphisms}

An important subclass of the \'{e}tale morphisms are the homotopy monomorphisms.

\begin{lemma}\label{lem:hepi-equivalent}
    Let $f:\mathrm{Spec}(B)\rightarrow\mathrm{Spec}(A)$ be a map in $\mathbf{Aff}_{\mathbf{C}}$. The following are equivalent:
    \begin{enumerate}
\item The counit of the adjunction \(f^* f_* \to \mathrm{id}\) is a natural isomorphism;
\item the functor \(f_* \colon \mathbf{D}(B) \to \mathbf{D}(A)\) is fully faithful;
\item the natural map \(B \otimes_A B \to B\) is an equivalence in \(\mathbf{D}(A)\).  
\end{enumerate} 
\end{lemma} 

\begin{proof}
See \cite{ben2017non}*{Section 4.3}, for instance. 
\end{proof}

\begin{definition}\label{def:hepi}
We call a morphism \(f \colon A \to B\) in \(\mathbf{DAlg}^{cn}(\mathbf{C})\)  (or dually, \(\Spec(B) \to \Spec(A)\) in \(\mathbf{Aff}_{\mathbf{C}}\)) satisfying any of the equivalent conditions of Lemma \ref{lem:hepi-equivalent} a \emph{homotopy epimorphism} (or dually, \emph{homotopy monomorphism}).
\end{definition}

\begin{proposition}[\cite{ben2024perspective}*{Lemma 2.1.42}]
    Homotopy epimorphisms are formally \'{e}tale.
\end{proposition}

\begin{proof}
    This follows immediately from the fact that for a homotopy epimorphism $A\rightarrow B$, $B\otimes^{\mathbb{L}}_{A}B\rightarrow B$ is in fact an equivalence. 
\end{proof}

\begin{definition}
Let $\underline{\Lambda}$ be a filtered category. A $\underline{\Lambda}$-generating class of algebras $\mathrm{T}$ is said to be \textit{of homotopy polynomial type} if each map $p_{\lambda}:\mathrm{Sym}(\mathbb{I})\rightarrow\mathrm{T}(\lambda)$ is a homotopy epimorphism.
\end{definition}

An important class of homotopy epimorphisms arise from so-called $\mathrm{T}$-\textit{rational localisations}. Recall that a collection of maps $(f_{1},\ldots,f_{n}):\mathbb{I}\rightarrow A$ is said to \textit{generate the unit ideal} if the induced map of $A$-modules
$$A^{\oplus n}\rightarrow A$$
is an epimorphism.

\begin{definition}[\cite{ben2022analytification}*{Definition 6.9}]\label{defn:Trat}
Let $A\in\mathbf{DAlg}^{cn}(\mathbf{C})$, let $f_{0},f_{1},\ldots,f_{n}:\mathbb{I}\rightarrow A$ be maps generating the unit ideal in $\pi_{0}(A)$, and let $(\lambda_{1},\ldots,\lambda_{n})$ be a sequence of elements of $\Lambda$. The $\mathrm{T}$-\textit{rational localisation of} $A$ \textit{at } $(f_{0},f_{1},\ldots,f_{n},(\lambda_{1},\ldots,\lambda_{n}))$ is the map
$$A\rightarrow A\otimes^{\mathbb{L}}T(\lambda_{1},\ldots,\lambda_{n}))\big\slash\big\slash(f_{0}y_{\lambda_{1}}-f_{1},\ldots,f_{0}y_{\lambda_{n}}-f_{n}).$$
The class of all $\mathrm{T}$-rational localisations is denoted $\mathbf{rat}^{\mathrm{T}}$. 
\end{definition}

\begin{lemma}\label{lem:rational-loc-etale}
    A $\mathrm{T}$-rational localisation is a homotopy epimorphism, and a standard $\mathrm{T}$-\'{e}tale morphism.
\end{lemma}

\begin{proof}
Write $B\cong A\otimes^{\mathbb{L}}T(\lambda_{1},\ldots,\lambda_{n}))\big\slash\big\slash(f_{0}y_{\lambda_{1}}-f_{1},\ldots,f_{0}y_{\lambda_{n}}-f_{n})$.
    The Jacobian is $\mathrm{diag}(g,\ldots,g)$. Since $g$ is a unit in the quotient, this matrix is invertible in the quotient. The fact it is a homotopy epimorphism is \cite{ben2022analytification}*{Lemma 6.10}.
\end{proof}

\begin{lemma}
   For each map $\lambda\rightarrow\gamma$ in $\underline{\Lambda}$ the map
   $$\mathrm{T}(\gamma)\rightarrow\mathrm{T}(\lambda)$$
   is a $\mathrm{T}$-rational localisation.
\end{lemma}

\begin{proof}
    Consider the exact sequence 
    \begin{displaymath}
        \xymatrix{
\mathrm{Sym}(\mathbb{I})\otimes\mathrm{Sym}(\mathbb{I})\ar[r]^{\times (x-y)}&\mathrm{Sym}(\mathbb{I})\otimes\mathrm{Sym}(\mathbb{I})\ar[r] & \mathrm{Sym}(\mathbb{I}).
        }
    \end{displaymath}
    This is a sequence of flat $\mathrm{Sym}(\mathbb{I})$-modules so tensoring on the right with $\mathrm{T}(\lambda)$ gives an exact sequence
    \begin{displaymath}
        \xymatrix{
\mathrm{Sym}(\mathbb{I})\otimes\mathrm{T}(\lambda)\ar[r]^{\times (x-y)}&\mathrm{Sym}(\mathbb{I})\otimes\mathrm{T}(\lambda)\ar[r] & \mathrm{T}(\lambda).
        }
    \end{displaymath}
    This is a resolution of $\mathrm{T}(\lambda)$ by flat $\mathrm{Sym}(\mathbb{I})$-modules. Then tensoring on the left with $\mathrm{T}(\gamma)$ gives that $\mathrm{T}(\gamma)\otimes^{\mathbb{L}}_{\mathrm{Sym}(\mathbb{I})}\mathrm{T}(\lambda)$ is computed by the complex
       \begin{displaymath}
        \xymatrix{
\mathrm{T}(\gamma)\otimes\mathrm{T}(\lambda)\ar[r]^{\times (x-y)}&\mathrm{T}(\gamma)\otimes\mathrm{T}(\lambda)\ar[r] & \mathrm{T}(\lambda).
        }
    \end{displaymath}
    But $\mathrm{Sym}(\mathbb{I})\rightarrow\mathrm{T}(\gamma)$ is a homotopy epimorphism so 
    $$\mathrm{T}(\gamma)\otimes^{\mathbb{L}}_{\mathrm{Sym}(\mathbb{I})}\mathrm{T}(\lambda)\cong\mathrm{T}(\gamma)\otimes^{\mathbb{L}}_{\mathrm{T}(\gamma)}\mathrm{T}(\lambda)\cong\mathrm{T}(\lambda).$$
    Thus 
    $$\mathrm{T}(\lambda)\cong\mathrm{T}(\gamma)\otimes^{\mathbb{L}}\mathrm{T}(\lambda)\big\slash\big\slash(x-y)$$ as required.
\end{proof}

\subsection{The Descendable Topology}

A topology satisfying descent which will play a central role is the \textit{descendable topology}.  In \cite{mathew2016galois} Mathew introduced the notion of a \textit{descendable} map $f:A\rightarrow B$ in $\mathbf{CAlg}(\mathbf{C})$. Consider the cosimplicial object in $\mathbf{CAlg}(\mathbf{C})$ given by the bar complex 
 $$\mathrm{CB}_{n}(f)\defeq B^{\otimes_{A}n+1},$$
 and its augmented version $\mathrm{CB}^{aug}_{\bullet}(f)$ with $\mathrm{CB}^{aug}_{-1}(f)=A$. Consider the pro-cosimplicial object given the tower of partial totalisations
 $$\{\mathrm{Tot}_{n}(\mathrm{CB}_{\bullet}(f))\}_{n\ge0}.$$
 
  Recall that an object of $\mathbf{Pro(C)}$ is constant (\cite{mathew2016galois} Definition 3.9) if it is equivalent to an object in the image of $\mathbf{C}\rightarrow\mathbf{Pro(C)}$. 
 
 \begin{definition}[\cite{mathew2016galois} Definition 3.18/ Proposition 3.20]
A morphism $A\rightarrow B$ is said to be \textit{descendable} if it is a constant pro-object which converges to $A$, i.e. 
 $$A\rightarrow\varprojlim_{[n] \in \Delta}\mathrm{Tot}_{n}(\mathrm{CB}_{\bullet}(f))$$
 is an equivalence.
 \end{definition}

 Descendability can also be phrased in terms of the augmented complex. For each $n$, write
 $$T_{aug}^{n}(f)\defeq\mathrm{cofib}(A\rightarrow\mathrm{Tot}_{n}(\mathrm{CB}_{\bullet}(f)).$$

 Then $A\rightarrow B$ is descendable precisely if $\{T_{aug}^{n}(f)\}_{n\ge0}$ is pro-constant and converges to the zero object.
 \comment{
  \begin{definition}
 A tower $\{X_{n}\}_{n\ge0}$ is said to be \textit{nilpotent of degree} $N$ if there exists an $N$ such that for all $n\ge0$ the map $X_{n+N}\rightarrow X_{n}$ is null-homotopic. 
 \end{definition}
 
 As explained in \cite{mathew2016galois}, nilpotent objects are pro-constant and converge to $0$, but in general nilpotence is in fact much stronger than this. 
 
 We also recall that an object $X=``\lim_{n}''X_{n}\in\mathrm{Tot}(\mathbf{C})$ is \textit{nilpotent of degree }$k$ if for each $n$, $X_{n+k}\rightarrow X_{n}$ is nilpotent, and \textit{strongly constant of degree }$k$ if the augmented object 
 $$\overline{X}=``\lim_{n}''\mathrm{cofib}(\lim_{m}X_{m}\rightarrow X_{n})$$
 is nilpotent of degree $k$.

 \begin{lemma}
     Let $X_{\bullet,\bullet}\in\mathrm{Fun}(\mathbb{Z}^{op}_{\ge0}\times,\mathbb{Z}^{op}_{\ge0},\mathbf{C})$. Suppose that there is $k$ such that for each $n$, the towers $X_{n,\bullet}$ and $X_{\bullet,m}$ are strongly constant of degree $k$. Then the diagonal tower $\mathrm{diag}(X)_{n}=X_{n,n}$ is strongly constant of degree $k$. 
 \end{lemma}

 \begin{proof}
Write $H_{n}=\lim_{m}X_{n,m}$. Define $\overline{X}_{n,m}\defeq\mathrm{cofib}(H_{n}\rightarrow X_{n,m})$. Then $\overline{X}_{n,\bullet}$ is, by definition, nilpotent of degree $n$. 
\end{proof}
}
We have the following useful criterion, which we will use to prove \'{e}tale descendability later.

\begin{proposition}\cite{mathew2016galois}*{Proposition 3.22}\label{prop:descendability-descent}
If $f:A\rightarrow B$ is descendable then the adjunction
$$B\otimes^{\mathbb{L}}_{A}(-) \colon \mathbf{D}(A) \to \bD(B) \colon {|-|}$$
is comonadic. Thus the functor
$$\mathbf{Mod}_A(\mathbf{C})\rightarrow\varprojlim_{n\in \Delta}\mathbf{Mod}_{B^{\otimes_{A}n+1}}(\bC)$$
is an equivalence. 
\end{proposition}
 
 \begin{remark}
 There are maps of algebras which satisfy descent, in that the functor
$$\mathbf{Mod}_A(\mathbf{C})\rightarrow\varprojlim_{n\in \Delta}\mathbf{Mod}_{B^{\otimes_{A}n+1}}(\bC)$$
is an equivalence, but which are not descendable in the sense of Mathew. A recent counterexample to this was provided in \cite{aoki2024cohomology}. 
 \end{remark}
 
 The class of descendable morphisms of commutative monoid objects forms a topology. This follows immediately from the following two results.
 
 \begin{proposition}\cite{mathew2016galois}*{Proposition 3.23}
 Let $A\rightarrow B$ and $B\rightarrow C$ be maps in $\mathbf{CAlg(C)}$. Then
 \begin{enumerate}
 \item
 If $A\rightarrow B$ and $B\rightarrow C$ are descendable, then so is $A\rightarrow C$.
 \item
 If $A\rightarrow C$ is descendable then so is $A\rightarrow B$.
 \end{enumerate}
 \end{proposition}
 
 \begin{proposition}\cite{mathew2016galois}*{Corollary 3.21}
 Let $F:\mathbf{C}\rightarrow\mathbf{C'}$ be a symmetric monoidal functor between symmetric monoidal, stable $\infty$-categories. Let $f:A\rightarrow B$ be a descendable map in $\mathbf{CAlg(C)}$. Then $F(f):F(A)\rightarrow F(B)$ is a descendable map in $\mathbf{CAlg(C')}$.
 \end{proposition}

\begin{corollary}
Let $\mathbf{C}$ be a derived algebraic context. The collection of maps $\mathrm{Spec}(B)\rightarrow\mathrm{Spec}(A)$ such that $\Theta(A)\rightarrow\Theta(B)$ is descendable defines a pre-topology on $\mathbf{Aff}^{cn}_{\mathbf{C}}$. 
\end{corollary}

\begin{proposition}\label{prop:descfull}
    Let $\mathbf{D}$ be a presentable, stable, closed symmetric monoidal $\infty$-category. Let $\mathbf{C}\subset\mathbf{D}$ be a full subcategory which is closed under finite colimits and tensor products. Let $f:A\rightarrow B$ be a map in $\mathbf{Comm}(\mathbf{C})$ which is descendable as a map in $\mathbf{Comm}(\mathbf{D})$. Then it is descendable in $\mathbf{Comm}(\mathbf{C}).$
\end{proposition}

\begin{proof}
    
    All terms in the augmented complex $T^{n}_{aug}(f)$ are in $\mathbf{C}$. Since this pro-object is pro-zero in $\mathbf{D}$, it is evidently pro-zero in $\mathbf{C}$.
\end{proof}

\subsubsection{The homotopy monomorphism topology}

An important example of a refinement of the descendable topology is the \textit{homotopy monomorphism topology}. This is a specific example of the topology defined in \cite{ben2024perspective}*{Section 7.2.3}. The characterisation of analytic topologies in terms of the homotopy monomorphism topology was developed in \cites{MR3626003,BaBK, bambozzi2016dagger}. Fix a derived algebraic context $\mathbf{C}$. The following is \cite{ben2024perspective}*{Definition 7.2.31} for $\kappa=\aleph_{0}$, $\mathbf{P}$ is the class of homotopy monomorphisms, and $\mathbf{N}(-)=\mathbf{QCoh}(-).$

\begin{definition}
A collection of maps $\{f_{i}:\mathrm{Spec}(B_{i})\rightarrow\mathrm{Spec}(A)\}_{i\in\mathcal{I}}$ in $\mathbf{DAlg}(\mathbf{C})^{op}$ is said to be a \textit{finite homotopy monomorphism cover} if the following two conditions are satisfied
\begin{enumerate}
    \item 
    each map $A\rightarrow B_{i}$ is a homotopy epimorphism;
    \item 
    there is a finite subset $\mathcal{J}\subset\mathcal{I}$ such that if $M\in\mathbf{D}(A)$ satisfies $B_{j}\otimes^{\mathbb{L}}_{A}M\cong 0$ for all $j\in\mathcal{J}$, then $M\cong 0$.
\end{enumerate}
Note in particular for such a $\mathcal{J}$,  $\{f_{j}:\mathrm{Spec}(B_{j})\rightarrow\mathrm{Spec}(A)\}_{j\in\mathcal{J}}$ is a homotopy monomorphism cover.
\end{definition}

The proof that this topology satisfies descent is \cite{ben2024perspective}*{Proposition 7.1.14}. \textit{Descendability}, which is strictly stronger, was mentioned in \cite{ben2024perspective}*{Example 7.2.22} but not proven. In \cite{soor2023quasicoherent}*{Proposition 2.19}, descendability of this topology when restricted to discrete algebras in so-called homotopy transversal pairs (\cite{soor2023quasicoherent}*{Definition 2.13}) was established. Here we give the proof for derived algebras. 

\begin{lemma}\label{lem:hepi-descendable}
Let $\{f_{i}:\mathrm{Spec}(B_{i})\rightarrow\mathrm{Spec}(A)\}_{i\in\mathcal{I}}$ be a finite collection of homotopy monomorphisms. Then it is a homotopy monomorphism cover if and only if the map 
$$A\rightarrow\prod_{i\in\mathcal{I}}B_{i}$$
is descendable.
\end{lemma}

\begin{proof}
Suppose the map is descendable. Then the map $ A\rightarrow\prod_{i\in\mathcal{I}}B_{i}$ is evidently conservative. But for any $A$-module $M$ we have 
$$(\prod_{i\in\mathcal{I}}B_{i})\otimes_{A}^{\mathbb{L}}M\cong\prod_{i\in\mathcal{I}}(B_{i}\otimes_{A}^{\mathbb{L}}M),$$
so clearly if $B_{i}\haotimes^{\mathbb{L}}_{A}M\cong 0$ for all $i\in\mathcal{I}$ then $M\cong 0$.

Conversely suppose $\{f_{i}:\mathrm{Spec}(B_{i})\rightarrow\mathrm{Spec}(A)\}_{i\in\mathcal{I}}$ is a cover in the homotopy monomorphism topology. Consider the limit
$$\varprojlim_{n\in \Delta}(\prod_{i\in\mathcal{I}}B_{i}^{\otimes^{\mathbb{L}}_{A}n}).$$
Note that this is in fact a corealisation of an essentially finite cosimplicial object, as the tensor powers eventually stabilise. Thus the proof amounts to showing that the map 
$$A\rightarrow\varprojlim_{n\in \Delta}(\prod_{i\in\mathcal{I}}B_{i}^{\otimes^{\mathbb{L}}_{A}n})$$
is an equivalence. For this it suffices to prove that for each $j$, the map
$$B_{j}\rightarrow B_{j}\otimes_{A}^{\mathbb{L}}\varprojlim_{n\in \Delta}(\prod_{i\in\mathcal{I}}B_{i})^{\otimes^{\mathbb{L}}_{A}n})$$
is an equivalence. Now since the limit is essentially finite, we have 
$$B_{j}\otimes_{A}^{\mathbb{L}}\varprojlim_{n\in \Delta}(\prod_{i\in\mathcal{I}}B_{i})^{\otimes^{\mathbb{L}}_{A}n})\cong\varprojlim_{n\in \Delta}B_{j}\otimes_{A}^{\mathbb{L}}(\prod_{i\in\mathcal{I}}B_{i})^{\otimes^{\mathbb{L}}_{A}n})\cong\varprojlim_{n\in \Delta}(\prod_{i\in\mathcal{I}}B_{j}\otimes_{A}^{\mathbb{L}}B_{i})^{\otimes^{\mathbb{L}}_{B_{j}}n}),$$
where we have used that $A\rightarrow B_{i}$ is a homotopy epimorphism. Now the limit is the corealisation of the \v{C}ech co-nerve of the map $\{B_{j}\rightarrow\prod_{i\in\mathcal{I}}B_{j}\otimes_{A}^{\mathbb{L}}B_{i}\}$. But this is split, and hence descendable. 
\end{proof}

\begin{corollary}\label{cor:QCoh-descent}
The functor \(\mathbf{QCoh} \colon (\mathbf{Aff}_\bC)^\op \to \mathbf{CAlg}(\mathbf{Pr}_{st}^L)\) satisfies descent with respect to the finite homotopy monomorphism topology. More generally, if \(\bA\) is a full subcategory of \(\mathbf{Aff}_\bC\) closed under pushouts, then \(\mathbf{QCoh} \colon \bA^\op \to  \mathbf{CAlg}(\mathbf{Pr}_{st}^L)\) satisfies descent with respect to the finite homotopy monomorphism topology restricted to \(\bA\). 
\end{corollary}
\begin{lemma}\label{lem:derstrongdescendable}
Let $A\rightarrow B$ be a map of connective algebras which is derived strong. If $A\rightarrow B$ is descendable, then $\pi_{0}(A)\rightarrow\pi_{0}(B)$ is descendable. Conversely, if $\pi_{0}(A)\rightarrow\pi_{0}(B)$ is descendable and each $\pi_{n}(B)$ is transverse to $\pi_{0}(B)$ over $\pi_{0}(A)$, then the map
$$A\rightarrow\varprojlim_{n\in \Delta} B^{\otimes^{\mathbb{L}}_{A}n}$$ is an equivalence. 
\end{lemma}

\begin{proof}
Suppose $A\rightarrow B$ is descendable and derived strong. By \cite{ben2024perspective}*{Proposition 2.3.87} $\pi_{0}(A)\rightarrow\pi_{0}(B)$ is the base change of $A\rightarrow B$ along $A\rightarrow\pi_{0}(A)$, and is therefore descendable.

For the partial converse, by descendability on $\pi_{0}$ we have that
  $$\pi_{*}(A)\rightarrow \varprojlim_{n\in \Delta}\pi_{*}(A)\otimes^{\mathbb{L}}_{\pi_{0}(A)}\pi_{0}(B)^{\otimes^{\mathbb{L}}_{\pi_{0}(A)}n}$$
 is an equivalence. By transversality and strength we get
$$\pi_{*}(A)\rightarrow\varprojlim_{n\in \Delta}\pi_{*}(B^{\otimes_{A}^{\mathbb{L}}n})$$
  is an equivalence. An easy spectral sequence argument now proves the claim.\qedhere
   
   

\end{proof}

Derived strength is local for strong descendable morphisms. 

\begin{proposition}\label{prop:stronglocstrong}
    Let $f:A\rightarrow B$ be a map of connective algebras. Suppose that there is a descendable map $g:B\rightarrow C$ such that the map $g\circ f:A\rightarrow C$ is derived strong.
 Then $A\rightarrow B$ is derived strong.
\end{proposition}

\begin{proof}
   Note that by base change, $\pi_{0}(B)\rightarrow\pi_{0}(C)$ is descendable. Consider the natural map 
   $$\pi_{0}(B)\otimes^{\mathbb{L}}_{\pi_{0}(A)}\pi_{n}(A)\rightarrow\pi_{n}(B).$$
   We need to show that this is an equivalence. By descent, it suffices to show that 
   $$\pi_{0}(C)\otimes^{\mathbb{L}}_{\pi_{0}(B)}\pi_{0}(B)\otimes^{\mathbb{L}}_{\pi_{0}(A)}\pi_{n}(A)\cong\pi_{0}(C)\otimes^{\mathbb{L}}_{\pi_{0}(A)}\pi_{n}(A)\rightarrow\pi_{0}(C)\otimes^{\mathbb{L}}_{\pi_{0}(B)}\pi_{n}(B)\cong\pi_{n}(C)$$
   is an equivalence. But this follows from the fact that $A\rightarrow C$ is derived strong.
\end{proof}

\begin{corollary}\label{lem:derstrongconver}
    Let $\{\mathrm{Spec}(A_{i})\rightarrow\mathrm{Spec}(A)\}$ be a finite collection of homotopy epimorphisms of connective algebras such that
    \begin{enumerate}
        \item 
        each $A\rightarrow A_{i}$ is derived strong;
        \item 
        the map $\pi_{0}(A)\rightarrow\prod_{i}\pi_{0}(A_{i})$ is conservative.
    \end{enumerate}

    Then the map $A\rightarrow\prod_{i}A_{i}$ is descendable.
\end{corollary}

\begin{proof}
    The descendability tower is in fact eventually constant. We just need to check it converges to $A$. But this follows from Lemma \ref{lem:derstrongdescendable}.
\end{proof}

\subsubsection{The \'{e}tale and $G$-topologies}

We now define topologies that will be relevant for analytic geometry. 

\begin{definition}[\cite{ben2024perspective} Definition 8.2.23]
    Let $\mathrm{T}$ be a $\underline{\Lambda}$-generating class of algebras of homotopy polynomial type. The finite $G-T$-\textit{topology} (or $\mathrm{T}$-\textit{rational topology}) is the topology consisting of covers $\{\mathrm{Spec}(B_{i})\rightarrow\mathrm{Spec}(A)\}$ in the homotopy monomorphism topology such that each map $A\rightarrow B_{i}$ is a $\mathrm{T}$-rational localisation.  We denote this topology by \(\tau^{rat}\). 
\end{definition}

\begin{definition}
     Let $\mathrm{T}$ be a $\underline{\Lambda}$-generating class of algebras of homotopy polynomial type. A map $f:A\rightarrow B$ in $\mathbf{DAlg}^{cn}(\mathbf{C})$ is said to be $\mathrm{T}$-\textit{\'{e}tale} if there is a finite $G-T$ cover $\{\mathsf{Spec}(B_{i})\rightarrow \mathsf{Spec}(B)\}$ such that each composite map $A\rightarrow B\rightarrow B_{i}$ is $\mathrm{T}$-standard \'{e}tale. 
\end{definition}

\begin{definition}[\cite{ben2024perspective}*{Definition 8.2.25}]\label{def:\'etale-abstract}
A collection $\{\mathrm{Spec}(B_i)\rightarrow\mathrm{Spec}(A)\}_{i=0}^n$ in \(\mathbf{Aff}_\bC\) is said to be a $\mathrm{T}$-\textit{descendable \'{e}tale pre-cover} if each map $A\rightarrow B_i$ is
\begin{enumerate}
    \item 
    $\mathrm{T}$-\'{e}tale, 
    \item 
    and descendable.
\end{enumerate}
\end{definition}

Denote the collection of morphisms in a \(\mathrm{T}\)-\'etale cover above by \(\tau^{pre}_{desc-et}\). 

\begin{proposition}\label{prop:etale-top}
The collection \(\tau^{pre}_{desc-et}\) defines a Grothendieck pre-topology on \(\mathbf{Aff}_\bC^{cn}\).
\end{proposition}

\begin{proof}
That $\mathrm{T}$-\'{e}tale maps are closed under isomorphisms, compositions, and pushouts is \cite{ben2024perspective}*{Proposition 6.1.4}.

\comment{
We first show that \(T\)-\'etale maps are closed under isomorphisms, compositions and pushouts. It is easy to see that isomorphisms are \(T\)-\'etale. To see closure under compositions, let \(f \colon A \to B\) and \(g \colon B \to C\) be \(T\)-\'etale maps. Then there are finite \(G\)-\(T\)-\'etale covers \(\{B \to B_i\}\) and \(\{C \to C_j\}\) such that each composites \(A \to B \to B_i\) and \(B \to C \to C_j\) are \(T\)-standard \'etale. Since \(T\)-rational localisations are closed under pushouts, we get a finite \(G\)-\(T\)-\'etale cover \(\{C \to C_j \otimes_B B_i\}\). And since the maps \(B \to C_j\) and \(A \to B_i\) are \(T\)-standard \'etale, each composition \(A \to B \to B_i \otimes_B C_j\) is \(T\)-standard \'etale. For pushouts, consider \(T\)-\'etale maps \(A \to C\) and \(A \to B\). Then there is a finite \(G\)-\(T\)-cover \(\{B \to B_i\}\) such that each \(B \to B_i\) is a \(T\)-rational localisation and \(A \to B \to B_i\) is \(T\)-standard \'etale. Since \(T\)-rational localisations are closed under pushouts, the map \(C \otimes_A B \to C \otimes_A B_i\) is a \(T\)-rational localisation. A simple computation then yields that \(C \to C \otimes_A \to B \to C \otimes_A B_i\) is a \(T\)-standard \'etale map. Symmetrically, the map \(B \to B \otimes_A C\) is a \(T\)-standard \'etale map, as required. }By \cite{mathew2016galois}*{Proposition 3.24, Corollary 3.21}, descendable maps are closed under isomorphisms, compositions and pushouts. The result now follows from the observation that \(\theta\) commutes with small colimits.   
\end{proof}

We shall from here on refer to the Grothendieck topology generated by descendable \'{e}tale covers by \(\tau^{desc-et}\) as the \emph{\'{e}tale topology}. Note that since descendability is not a local property, covers in $\tau^{desc-et}$ may not be descendable. However they will still satisfy descent for $\mathbf{QCoh}$, as we shall later see. \comment{Again, as a consequence of Proposition \ref{prop:descendability-descent}, we have the following:

\begin{corollary}\label{cor:QCoh-etale-descent}
The functor \(\mathbf{QCoh} \colon (\mathbf{Aff}_\bC^{cn})^\op \to \mathbf{CAlg}(\mathbf{Pr}_{st}^L)\) satisfies descent with respect to the \'etale topology. 
\end{corollary}

\begin{proof}
It suffices to show that \(T\)-standard \'etale descendable covers satisfy descent for quasi-coherent sheaves.  Consider a \(T\)-standard \'etale cover \(\{\mathsf{Spec}(B_i) \to \mathsf{Spec}(A)\}_{i=0}^n\), where each map \(A \to B_i\) is descendable. Then the hypotheses ensure that the total map \(A \to \prod_{i=0}^n B_i = B\) is descendable. Now use Proposition \ref{prop:descendability-descent}.
\end{proof}}

Finally we define \textit{smooth maps}. 

\begin{definition}
    A morphism $g:\mathrm{Spec}(B)\rightarrow \mathrm{Spec}(A)$ in $\mathbf{Aff}^{cn}_{\mathbf{C}}$ is said to be \textit{standard} $\mathrm{T}$-\textit{smooth} (resp. \textit{standard open} $\mathrm{T}$-\textit{smooth}) if the corresponding morphism $f:A\rightarrow B$ factors as 
    $$A\rightarrow A\otimes^{\mathbb{L}}T(\lambda_{1},\ldots,\lambda_{k})\rightarrow B,$$
    where $A\otimes^{\mathbb{L}}T(\lambda_{1},\ldots,\lambda_{k})\rightarrow B$ is $\mathrm{T}$-standard \'{e}tale (resp. $\mathrm{T}$-rational embedding). $g$ is said to be $\mathrm{T}$-\textit{smooth} (resp. $\mathrm{T}$-\textit{open smooth}) if there is an \'{e}tale (resp. rational) cover $\{\mathrm{Spec}(B_{i})\rightarrow\mathrm{Spec}(B)\}$ such that each composite $\mathrm{Spec}(B_{i})\rightarrow\mathrm{Spec}(B)$ is standard $\mathrm{T}$-smooth (resp. standard open $\mathrm{T}$-smooth).
\end{definition}

We denote the class of $\mathrm{T}$-smooth maps by $\mathbf{P}^{sm}$, and open $\mathrm{T}$-smooth maps by $\mathbf{P}^{o-sm}$.

\comment{\subsection{$\mathcal{F}$-descent topologies}

Let $\mathbf{D},\mathbf{E}$ be presentable $(\infty,1)$-categories with $\mathbf{D}$ being small, $\tau$ be a topology on $\mathbf{D}$, and let $\mathcal{F}\in\mathbf{PreShv}(\mathbf{D};\mathbf{E})$ be a $\mathbf{E}$-valued presheaf. When $\mathbf{E}=\mathbf{An}$ we just write $\mathbf{PreShv}(\mathbf{D})$. Note that $\mathcal{F}:\mathbf{D}^{op}\rightarrow\mathbf{E}$ extends by colimits to a functor
$$\mathbf{PreShv}(\mathbf{D})\rightarrow\mathbf{PreShv}(\mathbf{D};\mathbf{E})$$

Let $f:\mathcal{X}\rightarrow\mathcal{Z}$ be a map in $\mathbf{PreShv}(\mathbf{D})$. We say that $f$ \textit{satisfies} $\mathcal{F}$-\textit{descent} if 
the map
$$\mathcal{F}(\mathcal{Z})\rightarrow\lim_{n}\mathcal{F}(\mathcal{X}^{\times_{\mathcal{Z}n}})$$
is an isomorphism.

We say that $f$ \textit{satisfies universal} $\mathcal{F}$-\textit{descent} if for any affine $U\rightarrow\mathcal{Z}$, the map
$$f':U\times_{\mathcal{Z}}\mathcal{X}\rightarrow U$$
satisfies $\mathcal{F}$-descent.

\begin{proposition}
    If $f:\mathcal{X}\rightarrow\mathcal{Z}$ satisfies universal $\mathcal{F}$-descent then for any map $\mathcal{Y}\rightarrow\mathcal{Z}$ the map 
    $$\mathcal{X}\times_{\mathcal{Z}}\mathcal{Y}\rightarrow\mathcal{Y}$$
    satisfies $\mathcal{F}$-descent. In particular $f$ satisfies $\mathcal{F}$-descent.
\end{proposition}

\begin{proof}
    Write $$\mathcal{Y}\cong\mathrm{colim}_{\mathbf{D}_{\big\slash\mathcal{Y}}}U$$
    By universality of colimits,
    $$\mathcal{X}\times_{\mathcal{Z}}\mathcal{Y}\cong\mathrm{colim}_{\mathbf{D}_{\big\slash\mathcal{Y}}}U\times_{\mathcal{Z}}\mathcal{X}$$
    Now we have
    \begin{align*}
        \mathcal{F}(U)&\cong\lim _{n}\mathcal{F}((U\times_{\mathcal{{Z}}}\mathcal{X})^{\times_{U}n})\\
        &\cong\lim _{n}\mathcal{F}(U\times_{\mathcal{{Z}}}\mathcal{X}^{\times_{Z}n})
    \end{align*}
    So
    \begin{align*}
        \mathcal{F}(\mathcal{Y})&\cong\lim_{\mathbf{D}_{\big\slash\mathcal{Y}}}\mathcal{F}(U)\\
&\cong\lim_{\mathbf{D}_{\big\slash\mathcal{Y}}}\lim _{n}\mathcal{F}(U\times_{\mathcal{{Z}}}\mathcal{X}^{\times_{Z}n})\\
        &\cong\lim _{n}\lim_{\mathbf{D}_{\big\slash\mathcal{Y}}}\mathcal{F}(U\times_{\mathcal{{Z}}}\mathcal{X}^{\times_{Z}n})\\
&\cong\lim_{n}\mathcal{F}(\mathcal{Y}\times_{\mathcal{Z}}(\mathcal{X}^{\times_{\mathcal{Z}n}}))\\
&\cong\lim_{n}\mathcal{F}((\mathcal{Y}\times_{\mathcal{Z}}\mathcal{X})^{\times_{\mathcal{Y}n}})
    \end{align*}
    as required.
\end{proof}

We say that a square 
\begin{displaymath}
\xymatrix{
\mathcal{X}\times_{\mathcal{Z}}\mathcal{Y}\ar[r]\ar[d]  & \mathcal{Y}\ar[d]^{p}\\
\mathcal{X}\ar[r]^{e} & \mathcal{Z}
}
\end{displaymath}

is a \textit{(universal)} $\mathcal{F}$-\textit{descent square} if the map
$$\mathcal{X}\coprod\mathcal{Y}\rightarrow\mathcal{Z}$$
is a (universal) $\mathcal{F}$-descent morphism.

Denote by $\chi^{\mathcal{F}-desc}$ the class of $\mathcal{F}$-descent squares. Then $R(\chi^{\mathcal{F}-desc})$ consists of universal $\mathcal{F}$-descent squares such that $e$ is a monomorphism. 

Let $\tau^{\mathcal{F}-cd}$ denote the topology generated by the cd-structure $R(\chi^{\mathcal{F}-desc})$.

\begin{lemma}
Let $\mathbf{D}$ be equipped with a topology $\tau$.
    Let $\mathbf{E}=\mathbf{Cat}^{L}$ be the category of $(\infty,1)$-categories with maps being left adjoint functors, and let $\mathbf{Q}\in\mathbf{PreShv}(\mathbf{D};\mathbf{Cat}^{L})$ satisfy descent for $\tau$. Let $f:\mathcal{X}\rightarrow\mathcal{Y}$ and $g:\mathcal{Y}\rightarrow\mathcal{Z}$ be maps in $\mathbf{Shv}(\mathbf{D})$ with $f$ being an effective epimorphism of stacks. Then $g\circ f$ satisfies $\mathbf{Q}$-descent if and only if $g$ satisfies $\mathbf{Q}$-descent.
\end{lemma}

\begin{proof}
    Since $f:\mathcal{X}\rightarrow\mathcal{Y}$ is an effective epimorphism in the topos of sheaves, we have 
    $$\overline{\mathrm{im}}(g\circ f)\cong\overline{\mathrm{im}}(g)$$
    for any $g$, where $\overline{\mathrm{im}}(h)$ denotes the image of $h$ in the topos of sheaves. Since $\mathbf{Q}$ satisfies descent for $\tau$ we have
    $$\mathbf{Q}(\overline{\mathrm{im}}(h))\cong\mathbf{Q}(\mathrm{im}(h))$$
    Thus to verify that a map of sheaves 
    $$h:\mathcal{V}\rightarrow\mathcal{W}$$ 
    satisfies descent it suffices to prove that the map
    $$\mathbf{Q}(\mathcal{W})\rightarrow\mathbf{Q}(\overline{\mathrm{im}}(h))$$
    is an equivalence. Now by the above we have 
      $$\mathbf{Q}(\overline{\mathrm{im}}(g\circ f))\cong\mathbf{Q}(\overline{\mathrm{im}}(g))$$
      so clearly 
       $$\mathbf{Q}(\mathcal{Z})\rightarrow\mathbf{Q}(\overline{\mathrm{im}}(g))$$
       is an equivalence if and only if 
          $$\mathbf{Q}(\mathcal{Z})\rightarrow\mathbf{Q}(\overline{\mathrm{im}}(g\circ f))$$
          is.
\end{proof}

The main purpose of this result is the following situation. We know that some $\mathbf{Q}$ satisfies descent for class of covers of affines by affines, and we wish to `globalise' this. That is we consider representable maps of stacks $\mathcal{X}\rightarrow\mathcal{Y}$ such that for any map $U\rightarrow\mathcal{Y}$ with $U$ affine, $\mathcal{X}\times_{\mathcal{Y}}U$ admits an atlas $\{U_{i}\rightarrow\mathcal{X}\times_{\mathcal{Y}}U\}$ such that the composite cover $\{U_{i}\rightarrow U\}$ is in a class for which we know $\mathbf{Q}$ satisfies descent. Now $\coprod_{i}U_{i}\rightarrow\mathcal{X}\times_{\mathcal{Y}}U$ is an effective epimorphism of stacks. It then follows that $\mathbf{Q}$ satisfies descent for $\mathcal{X}\times_{\mathcal{Y}}U\rightarrow U$. Thus $\mathcal{X}\rightarrow\mathcal{Y}$ satisfies universal $\mathbf{Q}$-descent. 

\textcolor{red}{Complete?}

\subsubsection{Descent and rigidification}

Let $\mathbf{D}$ be equipped with a topology $\tau$.
    Let $\mathbf{E}=\mathbf{Cat}^{L}$ be the category of $(\infty,1)$-categories with maps being left adjoint functors, and let $\mathbf{Q}\in\mathbf{PreShv}(\mathbf{D};\mathbf{Cat}^{L})$ satisfy descent for $\tau$. Suppose further that each $\mathbf{Q}(U)$ is a dualisable category. Finally, we let $\mathbf{E}'$ denote the full subcategory $(\mathbf{Pr}^{L})^{dlb}$ of dualisable categories.

\textcolor{red}{Complete?}}

\subsection{Globalisation and geometric stacks}

Let \(\bC\) be a derived algebraic context and \(\mathbf{Aff}_\bC\) the opposite category of connective derived commutative algebra objects in \(\bC\), equipped with a Grothendieck topology \(\tau\). Denote by \(\mathbf{PSh}(\bC)\) the functor \(\infty\)-category \(\mathbf{Fun}(\mathbf{Aff}_\bC^\op, \mathbf{An})\) of \emph{prestacks on \(\bC\)} with values in anima. We say that a prestack \(\mathcal{F} \in \mathbf{PSh}(\bC)\) is a \emph{stack relative to \((\bC,\tau)\)} if it satisfies descent for covers in \(\tau\). For a derived commutative algebra object \(A \in \mathbf{DAlg}^{cn}(\bC)\), the prestack \[\mathbf{QCoh} \colon A \mapsto \mathbf{Mod}_A(\bC)\] satisfies Cech descent (that is, descent for Cech hypercovers) for the descendable topology and its refinements. 

\comment{In general, this does not imply hyperdescent for quasi-coherent sheaves. We do however have hyperdescent under conditions which we establish below. The credit for the following result belongs to Lucas Mann (\cite{mann2022p}*{Theorem 3.1.27}), who proves it in the case of solid quasi-coherent sheaves on perfectoid spaces, with the topology being the $v$-topology. However the proof works identically in the more general context we have in mind.

\begin{theorem}[\cite{mann2022p}*{Theorem 3.1.27}]\label{thm:hyperdescent}
Let \(\tau\) be a Grothendieck topology on \(\mathbf{Aff}_{\bC} = \mathbf{DAlg}^{cn}(\bC)^\op\) that is generated from a pre-topology $\tau'$. Suppose there is an $n\ge0$ such that all morphisms in any covering family in $\tau'$ have finite Tor-dimension $j$.  Then \(\mathbf{QCoh} \colon \mathbf{Aff}_\bC^\op \to \mathbf{CAlg}(\mathbf{Pr}_{st}^L)\) satisfies hyperdescent if it satisfies Cech descent. 
\end{theorem}

\begin{proof}
 By \cite{mann2022p}*{A.3.21} it suffices to show that if \(\mathbf{QCoh}\) satisfies descent, then under the hypotheses of the statement of the Theorem, for a hypercover \(U_\bullet \to X\), the canonical map \(\bD(X) \to \varprojlim_{[n] \in \Delta} \bD(U_n)\) is fully faithful.

Let $\mathrm{Spec}(B^{\bullet})=Y_{\bullet}\rightarrow X=\mathrm{Spec}(A)$ be a hypercover. Since each map $A\rightarrow B^{n}$ is of finite Tor-dimension, the functor $B^{n}\otimes^{\mathbb{L}}_{A}(-)$ commutes with Postnikov limits. Totalisation commutes with all limits. Hence in showing that map
$$M\rightarrow\varprojlim_{[n] \in \Delta} |B^{n}\otimes^{\mathbb{L}}_{A}M|$$
is an equivalence, we may assume that $M$ is concentrated in degrees $\le 0$. 

Note moreover that if $M$ is concentrated in degrees $\le0$, then each $B^{n}\otimes^{\mathbb{L}}_{A}M$ is concentrated in degrees $\le n$. Thus
$$\pi_{k}\mathrm{Tot}(M^{\bullet})\cong\pi_{k}\mathrm{Tot}_{m}(M^{\bullet})$$
for $k>-m+n$.
where $\mathrm{Tot}_{m}$ denotes the $m$th piece of the totalisation filtration.

Fix $n_{0}\ge0$ and let $\tilde{Y}_{n}=cosk_{n_{0}}Y_{0}=\mathrm{Spec}(\tilde{B}^{\bullet})$. Write $\tilde{M}^{\bullet}=\tilde{B}^{\bullet}\otimes^{\mathbb{L}}_{A}M$. Note that for $m\le n_{0}$ we have $\tilde{Y}_{m}=Y_{m}$ and hence $\tilde{M}^{m}=M^{m}$. Now by \v{C}ech descent and \cite{lurie2009higher}*{Lemma 6.5.3.9} we have that 
$$M^{\bullet}\rightarrow\mathrm{Tot}(\tilde{M}^{\bullet})$$
is an equivalence. Now again for $k>-n_{0}+n$ we have
$$\pi_{k}\mathrm{Tot}(M^{\bullet})=\pi_{k}\mathrm{Tot}_{n_{0}}(M^{\bullet})\cong\pi_{k}\mathrm{Tot}_{n_{0}}(\tilde{M}^{\bullet})\cong\pi_{k}\mathrm{Tot}(\tilde{M}^{\bullet})\cong\pi_{k}M,$$
and this completes the proof.
\end{proof}

}

The \(\infty\)-category \(\mathbf{Stk}(\bC, \tau)\) of stacks relative to \((\bC,\tau)\) is a reflective subcategory of \(\mathbf{PSh}(\bC)\), where the left adjoint is the \emph{stackification} functor. 

For our applications to derived analytic geometry, we keep track of two additional data: a full subcategory \(\bA \subset \mathbf{Aff}_{\bC}\) of affines; a distinguished collection \(\mathbf{P}\) of morphisms formally playing the role of smooth maps. This is packaged into the following:

\begin{definition}\label{def:relative-geometry-tuple}
A \emph{relative pre-geometry tuple} is a tuple \((\mathbf{Aff}_\bC, \tau, \mathbf{P}, \bA)\) consisting of a Grothendieck topology on derived commutative algebra objects in a derived algebraic context; a full subcategory \(\bA \subset \mathbf{Aff}_\bC\) and a class of morphisms \(\mathbf{P}\) satisfying:

\begin{enumerate}
\item if \(\{\mathsf{Spec}(B_i) \to \mathsf{Spec}(A)\}\) is a cover in \(\tau\), then each \(\mathsf{Spec}(B_i) \to \mathsf{Spec}(A) \in \mathbf{P}\);
\item if \(f \colon \mathsf{Spec}(B) \to \mathsf{Spec}(A)\) is a map in \(\mathbf{Aff}_\bC\) and \(\{\mathsf{Spec}(B_i) \to \mathsf{Spec}(B)\}\) is a \(\tau\)-covering family such that each composition \(\mathsf{Spec}(B_i) \to \mathsf{Spec}(A) \in \mathbf{P}\), then \(f \in \mathbf{P}\);
\item if \(f \colon \mathsf{Spec}(C) \to \mathsf{Spec}(A) \in \mathbf{P} \cap \bA\) and \(\mathsf{Spec}(B) \to \mathsf{Spec}(A)\) is any map with \(\mathsf{Spec}(B) \in \bA\), then the pullback \(\mathsf{Spec}(B) \times_{\mathsf{Spec}(A)} \mathsf{Spec}(C) \in \bA\).    
\end{enumerate}

A relative pre-geometry tuple is said to be \emph{strong} if whenever \(\{\mathsf{Spec}(A_i) \to \mathsf{Spec}(A)\}\) is a cover in \(\tau\) (not necessarily in \(\bA\)), and \(\mathsf{Spec}(B) \to \mathsf{Spec}(A)\) is a morphism in \(\mathsf{Spec}(B) \in \bA\), then \(\{\mathsf{Spec}(B \otimes_A A_i) \to \mathsf{Spec}(A)\}\) is a cover in \(\tau_{\vert \bA}\).  
\end{definition}

\begin{remark}\label{rem:forcestrong}
If a relative pre-geometry tuple is not strong, we may force this by defining \(\mathbf{P}_\bA\) to be the collection of all \(f \colon \mathsf{Spec}(B) \to \mathsf{Spec}(A) \in \mathbf{P}\) such that if \(\mathsf{Spec}(C) \to \mathsf{Spec}(A)\) is a map with \(\mathsf{Spec}(C) \in \bA\), then \(\mathsf{Spec}(C \otimes_A B) \in \bA\), and \(\tau_\bA\) the class of all covers in \(\tau\) such that each morphism belongs to \(\mathbf{P}_\bA\). Then \((\mathbf{Aff}_\bC, \tau_\bA, \mathbf{P}_\bA, \bA)\) is a strong relative pre-geometry tuple. 
\end{remark}

\begin{definition}\label{def:admissible}
Let \((\mathbf{Aff}_\bC, \tau, \mathbf{P}, \bA)\) be a relative pre-geometry tuple. We call an object \(X =\mathsf{Spec}(A) \in \mathbf{Aff}_\bC\) \emph{\(\bA\)-admissible} if the representable functor \[\mathbf{Map}_{\mathbf{Aff}_\bC}(-,X) \colon \bA^\op \to \mathbf{An}\] is a stack for the restriction of the topology \(\tau\) to \(\bA\). A relative pre-geometry tuple is called a \emph{relative geometry tuple} if every object of \(\bA\) is \(\bA\)-admissible.  
\end{definition}

\comment{The restriction to a full subcategory \(\bA\) of objects in \(\mathbf{Aff}_\bC\) necessitates that we isolate a convenient collection of full subcategories of \(\mathbf{QCoh}_{\vert \bA}\) - this is required to globalise cotangent complexes to stacks and prove descent for subcategories we will use to define \(K\)-theory.

\begin{definition}\label{def:good-system}
A \emph{good system of modules on \(\bA\)} is an assignment \(\mathbf{M}\) to each \(\mathsf{Spec}(A) \in \bA\), a full subcategory \(\mathbf{M}_A \subseteq \bD(A)\) satisfying the following conditions:
\begin{itemize}
\item \(A \in \mathbf{M}_A\) for each \(\mathsf{Spec}(A) \in \bA\);
\item the contangent complex \(\mathbb{L}_A \in \mathbf{M}_A\) for each \(\mathsf{Spec}(A) \in \bA\);
\item \(\mathbf{M}_A\) is closed under isomorphisms, finite colimits and retracts;
\item the functor \(\mathbf{M} \colon A \mapsto \mathbf{M}_A\) satisfies weak Cech descent;
\item for any \(f \colon \mathsf{Spec}(C) \to \mathsf{Spec}(B)\), we have \(C \otimes_B M \in \mathbf{M}_C\) for all \(M \in \mathbf{M}_B\).  
\end{itemize}

\end{definition}}

Now consider a relative geometry tuple \((\mathbf{Aff}_\bC, \tau, \mathbf{P}, \bA)\). Denote by \(\mathbf{PSh}(\bA)\) and \(\mathbf{Stk}(\bA,\tau_{\vert \bA})\) the \(\infty\)-categories of prestacks and stacks on \(\bA^\op\). The inclusion functor \(i \colon \bA^\op  \to \mathbf{Aff}_\bC^\op\) induces a restriction-left Kan extension adjoint pair 

\[i^* \colon \mathbf{PSh}(\bC) \rightleftarrows \mathbf{PSh}(\bA) \colon i_{!}.\] The restriction functor \(\mathbf{Stk}(\bC,\tau) \rightleftarrows \mathbf{Stk}(\bA, \tau_{\vert \bA})\) also has a left adjoint given by the composition of the stackification functor with \(i\). 

\comment{

The following is \cite{savage2024representability}*{Proposition 2.3.11}. Our statement is a minor variation, as the proof of the statement in fact uses less than what is required by the corresponding statement in loc. cit.

\begin{proposition}\label{prop:locallystack}
Suppose that whenever 
    $$\{\mathrm{Spec}(A_{i})\rightarrow\mathrm{Spec}(A)\}_{i\in\mathcal{I}}$$
    is a finite cover in $\tau$ with each $\mathrm{Spec}(A_{i_{1}}\otimes^{\mathbb{L}}_{A}\ldots\otimes^{\mathbb{L}}_{A}A_{i_{n}})\in\mathbf{A}$, then $\mathrm{Spec}(A)\in\mathbf{A}$.
 Let $\mathcal{X}\rightarrow\mathrm{Spec}(C)$ be a morphism with $\mathrm{Spec}(C)\in\mathbf{A}$. Suppose there exists a cover $\{\mathrm{Spec}(B_{i})\rightarrow\mathrm{Spec}(C)\}$  in $\tau$ such that each $\mathcal{X}\times_{\mathrm{Spec}(C)}\mathrm{Spec}(B_{i})$ is in $\mathbf{A}$. Then $\mathcal{X}\in\mathbf{A}$.
\end{proposition}

\begin{definition}
    Let $\mathbf{G}=(\mathbf{Aff}_{\mathbf{C}},\tau,\mathbf{P},\mathbf{A})$ be a strong relative pre-geometry tuple. We say that $\mathbf{G}$ is \textit{tempered} 
\end{definition}

}

\begin{lemma}\cite{kelly2022analytic}\label{lem:subgeom}
Let \((\mathbf{Aff}_\bC, \tau, \mathbf{P}, \bA)\) be a strong relative pre-geometry tuple. Then the functor \(i_{!} \colon \mathbf{Stk}(\bA, \tau_{\vert A}) \to \mathbf{Stk}(\bC, \tau)\) is fully faithful.  
\end{lemma}

We will use Lemma \ref{lem:subgeom} in later sections to embed known categories of analytic stacks inside bornological stacks. 

\subsubsection{Geometric Stacks}

In this subsection we recall the definitions of geometric stacks and schemes. These are stacks which are glued together from affines in a controllable way. 

\begin{definition}[\cite{toen2008homotopical}*{Definition 1.3.3.1}]
Let \(\bC\) be a derived algebraic context and  $(\mathbf{Aff}_\bC,\tau,\mathbf{P},\mathbf{A})$ a relative $(\infty,1)$-pre-geometry tuple.
 
\begin{enumerate}
\item
A prestack $\mathcal{X}$ in $\mathbf{Stk}(\mathbf{A},\tau|_{\mathbf{A}})$ is $(-1)$-\textit{geometric} if it is of the form $\mathcal{X}\cong\mathbf{Map}(-,M)$ for some $\mathbf{A}$-admissible $M\in\mathbf{Aff}_{\bC}$.
\item
A morphism $f:\mathcal{X}\rightarrow\mathcal{Y}$ in $\mathbf{Stk}(\mathbf{A},\tau|_{\mathbf{A}})$ is $(-1)$-\textit{representable} if for any map $U\rightarrow\mathcal{Y}$ with $U\in\mathbf{A}$ the pullback $\mathcal{X}\times_{\mathcal{Y}}U$ is $(-1)$-geometric.
\item
A morphism $f:\mathcal{X}\rightarrow\mathcal{Y}$ in  $\mathbf{Stk}(\mathbf{A},\tau|_{\mathbf{A}})$ is in $(-1)-\mathbf{P}$ if it is $(-1)$-representable and for any map $U\rightarrow\mathcal{Y}$ with $U\in\mathbf{A}$, the map $\mathcal{X}\times_{\mathcal{Y}}U\rightarrow U$ is in $\mathbf{P}$.
\end{enumerate}
\end{definition}

\begin{definition}[\cite{toen2008homotopical}*{Definition 1.3.3.1}]\label{defn:ngeom}
Let $n\ge0$ and let $\mathcal{X}$ be in  $\mathbf{Stk}(\mathbf{A},\tau|_{\mathbf{A}})$
\begin{enumerate}
\item
An $n$\textit{-atlas} for $\mathcal{X}$ is a set of morphisms $\{U_{i}\rightarrow \mathcal{X}\}_{i\in\mathcal{I}}$ such that 
\begin{enumerate}
\item
Each $U_{i}$ is $(-1)$-geometric.
\item
Each map $U_{i}\rightarrow\mathcal{X}$ is in $(n-1)$-$\mathbf{P}$.
\item
$\coprod_{i\in\mathcal{I}}U_{i}\rightarrow\mathcal{X}$ is an epimorphism of stacks.
\end{enumerate}
\item
A stack $\mathcal{X}$ is $n$-\textit{geometric}  if
\begin{enumerate}
\item
The map $\mathcal{X}\rightarrow\mathcal{X}\times\mathcal{X}$ is $(n-1)$-representable.
\item
$\mathcal{X}$ admits an $n$-atlas.
\end{enumerate}
\item
A morphism of stacks $f:\mathcal{X}\rightarrow\mathcal{Y}$ is $n$-\textit{representable} if for any map  $U\rightarrow \mathcal{Y}$ with $U\in\mathbf{A}$ the pullback $\mathcal{X}\times_{\mathcal{Y}}U$ is $n$-geometric.
\item
A morphism of stacks $\mathcal{X}\rightarrow\mathcal{Y}$ is in $n$-$\mathbf{P}$ if it is $n$-representable, and for any map $U\rightarrow\mathcal{Y}$ with $U\in\mathbf{A}$, there is an $n$-atlas $\{U_{i}\rightarrow\mathcal{X}\times_{\mathcal{Y}}U\}_{i\in\mathcal{I}}$ such that each map $U_{i}\rightarrow U$ is in $\mathbf{P}$.
\end{enumerate}
\end{definition}

The full subcategory of $\mathbf{Stk}(\mathbf{A},\tau|_{\mathbf{A}})$ consisting of $n$-geometric stacks is denoted $\mathbf{Stk}_{n}(\mathbf{Aff}_\bC,\tau,\mathbf{P},\mathbf{A})$. As in \cite{toen2008homotopical}*{Proposition 1.3.3.6}, we have  $\mathbf{Stk}_{m}(\mathbf{Aff}_\bC,\tau,\mathbf{P},\mathbf{A})\subseteq \mathbf{Stk}_{n}(\mathbf{Aff}_\bC,\tau,\mathbf{P},\mathbf{A})$ for $-1\le m\le n$. The full subcategory of geometric stacks is $\mathbf{Stk}_{geom}(\mathbf{Aff}_\bC,\tau,\mathbf{P},\mathbf{A})\defeq\bigcup_{n=-1}^{\infty}\mathbf{Stk}_{n}(\mathbf{Aff}_\bC,\tau,\mathbf{P},\mathbf{A})$. As in \cite{toen2008homotopical}*{Corollary 1.3.3.5}, $\mathbf{Stk}_{n}(\mathbf{Aff}_\bC,\tau,\mathbf{P},\mathbf{A})$ is closed under pullbacks.

If $(\mathbf{Aff}_\bC,\tau,\mathbf{P})$ is an $(\infty,1)$-geometry triple then $(\mathbf{Aff}_\bC,\tau,\mathbf{P},\mathbf{Aff}_\bC)$ is a relative $(\infty,1)$-pre-geometry tuple. In this case we write
$$\mathbf{Stk}_{m}(\mathbf{Aff}_\bC,\tau,\mathbf{P})\defeq\mathbf{Stk}_{m}(\mathbf{Aff}_\bC,\tau,\mathbf{P},\mathbf{Aff}_\bC),$$
and we say that map $f:\mathcal{X}\rightarrow\mathcal{Y}$ is $n$-representable, or in $n-\mathbf{P}$ if it is so in the sense of Definition \ref{defn:ngeom} for the relative $(\infty,1)$-pre-geometry tuple $(\mathbf{M},\tau,\mathbf{P},\mathbf{M})$. The following is immediate:

\begin{proposition}
Let $(\mathbf{Aff}_\bC,\tau,\mathbf{P},\mathbf{A})$ be a strong relative geometry tuple. Then for each $m\ge -1$ we have $\mathbf{Stk}_{m}(\mathbf{A},\tau|_{\mathbf{A}},\mathbf{P}|_{\mathbf{A}})\subseteq \mathbf{Stk}_{m}(\mathbf{Aff}_\bC,\tau,\mathbf{P},\mathbf{A})$.
\end{proposition}

If $(\mathbf{Aff}_\bC,\tau,\mathbf{P},\mathbf{A})$ is a relative $(\infty,1)$-geometry tuple such that $(\mathbf{Aff}_\bC,\tau,\mathbf{P})$ is an $(\infty,1)$-geometry triple (i.e. all objects of $\mathbf{Aff}_\bC$ are admissible) then we have a permanence property for geometric stacks.
Consider the functor $i_!:\mathbf{Stk}(\mathbf{A},\tau|_{\mathbf{A}})\rightarrow\mathbf{Stk}(\mathbf{Aff}_\bC,\tau)$.

\comment{\begin{lemma}\cite{kelly2022analytic}*{Lemma 6.12}
Let $(\mathbf{Aff}_\bC,\tau,\mathbf{P},\mathbf{A})$ be a relative $(\infty,1)$-geometry tuple. Let $f:\mathcal{X}\rightarrow\mathcal{Y}$ be a map in $\mathbf{Stk}_{n}(\mathbf{A},\tau|_{\mathbf{A}},\mathbf{P}|_{\mathbf{A}})$ which is in $n-\mathbf{P}|_{\mathbf{A}}$. If $\mathbf{A}$ has all finite limits which are preserved by the inclusion $\mathbf{A}\rightarrow \mathbf{Aff}_\bC$ (or more generally if the left Kan extension commutes with finite limits) then $i_!(f)$ is in $n-\mathbf{P}$.
\end{lemma}
\begin{proof}
The proof is by induction on $n$. For $n=-1$ the claim is clear. Suppose the claim has been proven for some $(n-1)$ with $n\ge1$, and consider $n$. Let $\mathbf{Map}(-,X)\rightarrow i_!(\mathcal{Y})$ be a map. There is a cover $\{\mathbf{Map}(-,X_{i})\rightarrow\mathbf{Map}(-,X)\}_{i\in\mathcal{I}}$ such that each $\mathbf{Map}(-,X_{i})\rightarrow i_!(\mathcal{Y})$ factors through $i(\mathcal{Y})\rightarrow i_!(\mathcal{Y})$. But then each map $\mathbf{Map}(-,X_{i})\rightarrow i(\mathcal{Y})$ factors through some map $\mathbf{Map}(-,A_{i})\rightarrow i(\mathcal{Y})$ with $A_{i}\in\mathbf{A}$. Consider the fibre product $\mathbf{Map}(-,A_{i})\times_{\mathcal{Y}}\mathcal{X}$. There is an an $n$-atlas $\mathbf{Map}(-,A_{ij})\rightarrow\mathbf{Map}(-,A_{i})\times_{\mathcal{Y}}\mathcal{X}$ such that each composition $\mathbf{Map}(-,A_{ij})\rightarrow\mathbf{Map}(-,A_{i})$ is in $\mathbf{P}$.  Now both the left Kan extension and stackification commute with finite limits. Thus $\mathbf{Map}(-,A_{ij})\rightarrow\mathbf{Map}(-,A_{i})\times_{i^{\#}(\mathcal{Y})}i_!(\mathcal{X})$ is an $n$-atlas. Write $X_{ij}\defeq X_{i}\times_{A_{i}}A_{ij}$. Then $\{\mathbf{Map}(-,X_{ij})\rightarrow\mathbf{Map}(-,X_{i})\times_{i_!(\mathcal{Y})}i_!(\mathcal{X})\}$ is an $n$-atlas, and the composition $\mathbf{Map}(-,X_{ij})\rightarrow\mathbf{Map}(-,X_{i})$ is in $\mathbf{P}$. The result now follows from \cite{toen2008homotopical}*{Proposition 1.3.3.4}.
\end{proof}
}

\begin{lemma}\cite{kelly2022analytic}*{Corollary 6.13}
Let $(\mathbf{Aff}_\bC,\tau,\mathbf{P},\mathbf{A})$ be a relative $(\infty,1)$-geometry tuple. If $\mathbf{A}$ has all finite limits which are preserved by the inclusion $\mathbf{A}\rightarrow\mathbf{Aff}_\bC$ (or more generally if the left Kan extension commutes with finite limits)  then the functor  $i_!:\mathbf{Stk}(\mathbf{A},\tau|_{\mathbf{A}})\rightarrow\mathbf{Stk}(\mathbf{Aff}_\bC,\tau)$ induces a functor 
 $$i^{\#}:\mathbf{Stk}_{n}(\mathbf{A},\tau|_{\mathbf{A}},\mathbf{P}|_{\mathbf{A}})\rightarrow\mathbf{Stk}_{n}(\mathbf{Aff}_\bC,\tau,\mathbf{P})$$
 for each $n$, which is fully faithful if the relative $(\infty,1)$-geometry tuple is strong.
\end{lemma}

Finally, we come to the definition of schemes.

\begin{definition}
Let $(\mathbf{Aff}_\bC,\tau,\mathbf{P},\mathbf{A})$ be a relative $(\infty,1)$-geometry tuple. An $n$-geometric stack $\mathcal{X}$ in \(\mathbf{Stk}(\bA, \tau_{\vert \bA})\) is said to be an $n$-\textit{weak scheme} if 
\begin{enumerate}
\item
it is a $(-1)$-geometric stack for $n=-1$.
\item
for $n\ge 0$ it has an  $n$-atlas $\{U_{i}\rightarrow\mathcal{X}\}_{i\in\mathcal{I}}$ such that 
\begin{enumerate}
\item
each $U_{i}\rightarrow\mathcal{X}$ is a monomorphism in $\mathbf{PreStk}(\mathbf{A},\tau|_{\mathbf{A}})$, that is, \(U_i \to U_i \times_{\mathcal{X}} U_i\) is an equivalence;
\item
for each $m\in\mathbb{N}$ and each $(i_{1},\ldots,i_{m})\in\mathcal{I}^{m}$, each $U_{i_{1}}\times_{\mathcal{X}}U_{i_{2}}\times_{\mathcal{X}}\ldots\times_{\mathcal{X}}U_{i_{m}}$ is a weak $(n-1)$-scheme.
\end{enumerate}
\end{enumerate}
\end{definition}

\begin{definition}
Let $(\mathbf{Aff}_\bC,\tau,\mathbf{P},\mathbf{A})$ be a strong relative $(\infty,1)$-geometry tuple. An $n$-geometric stack $\mathcal{X}$ in \(\mathbf{Stk}(\bA, \tau_{\vert \bA})\) is said to be an $n$-\textit{scheme} if 
\begin{enumerate}
\item
it is a $(-1)$-geometric stack for $n=-1$.
\item
for $n\ge 0$ it has an  $n$-atlas $\{U_{i}\rightarrow\mathcal{X}\}_{i\in\mathcal{I}}$ such that 
\begin{enumerate}
\item
each $U_{i}\rightarrow\mathcal{X}$ is a monomorphism in $\mathbf{PreStk}(\mathbf{A},\tau|_{\mathbf{A}})$, that is, \(U_i \to U_i \times_{\mathcal{X}} U_i\) is an equivalence;
\item
the diagonal map $\mathcal{X}\rightarrow\mathcal{X}\times\mathcal{X}$ is representable by an $(n-1)$-scheme. 
\end{enumerate}
\end{enumerate}
\end{definition}

\begin{remark}
    Weak schemes are what we called schemes in \cite{kelly2022analytic}*{Section 6.2}. This definition was precisely what we required to prove the global HKR theorem. However, the definition of scheme above is better behaved. One way in which it is better behaved is the next result.
\end{remark}

The following can be proven identically to \cite{toen2008homotopical}*{Corollary 1.3.3.5}.

\begin{proposition}\label{prop:fibrescheme}
    Let $f:X\rightarrow Z$, $g:Y\rightarrow Z$ be morphisms of $n$-schemes. Then $X\times_{Z}Y$ is an $n$-scheme.
\end{proposition}

\begin{corollary}
Let $\mathcal{X}$ be a stack. The map $\mathcal{X}\rightarrow\mathcal{X}\times\mathcal{X}$ is representable by a scheme if and only if for all maps $\mathrm{Spec}(A)\rightarrow\mathcal{X},\mathrm{Spec}(B)\rightarrow\mathcal{X}$, we have that $\mathrm{Spec}(A)\times_{\mathcal{X}}\mathrm{Spec}(B)$ is a scheme. 
\end{corollary}

\begin{proof}
    Suppose that $\mathcal{X}\rightarrow\mathcal{X}\times\mathcal{X}$ is representable by a scheme. We have the following fibre product diagram
    \begin{displaymath}
        \xymatrix{
        \mathrm{Spec}(A)\times_{\mathcal{X}}\mathrm{Spec}(B)\ar[d]\ar[r] & \mathcal{X}\ar[d]\\
        \mathrm{Spec}(A\otimes^{\mathbb{L}}B)\ar[r] & \mathcal{X}\times\mathcal{X}.
        }
    \end{displaymath}
    Since the diagonal is representable by a scheme, $\mathrm{Spec}(A)\times_{\mathcal{X}}\mathrm{Spec}(B)$ is a scheme.

    Conversely suppose that for all maps $\mathrm{Spec}(A)\rightarrow\mathcal{X},\mathrm{Spec}(B)\rightarrow\mathcal{X}$, we have that $\mathrm{Spec}(A)\times_{\mathcal{X}}\mathrm{Spec}(B)$ is a scheme. Consider the following diagram in which all squares are pullbacks
    \begin{displaymath}
        \xymatrix{
        \mathrm{Spec}(B)\times_{\mathcal{X}\times\mathcal{X}}\mathcal{X}\ar[d]\ar[r] & \mathrm{Spec}(B)\times_{\mathcal{X}}\mathrm{Spec}(B)\ar[d]\ar[r] & \mathcal{X}\ar[d]\\
        \mathrm{Spec}(B)\ar[r] & \mathrm{Spec}(B\otimes^{\mathbb{L}}B)\ar[r] & \mathcal{X}\times\mathcal{X}.
        }
    \end{displaymath}
    $\mathrm{Spec}(B)\times_{\mathcal{X}}\mathrm{Spec}(B)$ is a scheme by assumption. Since schemes are closed under fibre products by Proposition \ref{prop:fibrescheme}, $\mathrm{Spec}(B)\times_{\mathcal{X}\times\mathcal{X}}\mathcal{X}$ is a scheme. 
\end{proof}

\begin{definition}
    Let $X$ be an $n$-scheme. A \textit{subscheme of }$X$ is a map $Y\rightarrow X$ in $\mathbf{Stk}(\mathbf{A},\tau|_{\mathbf{A}})$
 which is a homotopy monomorphism, is in $\mathbf{P}$, and for any $U\rightarrow X$ with $U$ affine, $U\times_{X}Y$ is a scheme. 
 \end{definition}

\begin{proposition}\label{prop:subschemen}
Let $X$ be an $n$-scheme for $n\ge0$. A map $f:Y\rightarrow X$ is a subscheme if and only if $Y$ is an $n$-scheme and the map $f$ is a homotopy monomorphism in $\mathbf{P}$.
\end{proposition}

\begin{proof}
Suppose that $f:Y\rightarrow X$ is a subscheme. Pick an atlas $\{U_{i}\rightarrow X\}_{i\in\mathcal{I}}$. Each map $U_{i}\times_{X}Y\rightarrow Y$ is a homotopy monomorphism and in $\mathbf{P}$.  Pick schematic atlases $\{V_{j_{i}}\rightarrow U_{i}\times_{X}Y\}_{j_{i}\in\mathcal{J}_{i}}$ for each $i\in\mathcal{I}$. Note that each composite $V_{j_{i}}\rightarrow Y$ is a homotopy monomorphism and in $\mathbf{P}$. Then $\coprod_{j_{i}\in\mathcal{J}_{i},i\in\mathcal{I}}V_{j_{i}}\rightarrow Y$ is an effective epimorphism, and $\{V_{j_{i}}\rightarrow Y\}_{j_{i}\in\mathcal{J}_{i},i\in\mathcal{I}}$ is an atlas. Finally, for any affine $U\rightarrow X$ and any affine $V\rightarrow X$, $U\times_{Y}V\cong U\times_{X}V$ is an $(n-1)$-scheme. 

Conversely, suppose that $Y$ is an $n$-scheme, and $f:Y\rightarrow X$ is a homotopy monomorphism in $\mathbf{P}$. By Proposition \ref{prop:fibrescheme}, for any affine $U\rightarrow X$, $U\times_{X}Y$ is a scheme. Hence $Y\rightarrow X$ is a subscheme. 
\end{proof}

\comment{
Using Proposition \ref{prop:locallystack}, the following can be proven exactly as in \cite{toen2008homotopical}*{Proposition 1.3.3.4}.

\begin{proposition}\label{prop:descentforschemes}
Suppose that $\mathbf{G}$ is tempered.
    Let $f:\mathcal{X}\rightarrow\mathcal{Y}$ be a morphism in $\mathbf{Stk}(\mathbf{A},\tau|_{\mathbf{A}})$ such that $G$ is an $n$-scheme. Suppose there exists an $n$-atlas $\{U_{i}\rightarrow\mathcal{Y}\}_{i\in\mathcal{I}}$ of $\mathcal{Y}$ such that each $\mathcal{X}\times_{\mathcal{Y}}U_{i}$ is an $n$-scheme. Then $\mathcal{X}$ is an $n$-scheme in $\mathbf{Stk}(\mathbf{A},\tau|_{\mathbf{A}})$.

    Moreover, if each projection $\mathcal{X}\times_{\mathcal{Y}}U_{i}\rightarrow U_{i}$ is in $n-\mathbf{P}$ and representable by a scheme, then $f$ is in $n-\mathbf{P}$ and representable by a scheme.
\end{proposition}
}

The following definition of a `derived space' is the obvious generalisation of \cite{soor2023quasicoherent}*{Definition 2.17}.

\begin{definition}\label{Defn:soorspace}
    \begin{enumerate}
        \item An \textit{affine derived} $\mathbf{G}$-\textit{space} is an object of $\mathbf{Stk}(\mathbf{A},\tau|_{\mathbf{A}})$ which is isomorphic to $\mathrm{Spec}(A)$ for some $\mathrm{Spec}(A)\in\mathbf{A}$.
        \item Let $X=\mathrm{Spec}(A)$ be an affine derived $\mathbf{G}$-space. A $\mathbf{G}$-\textit{subspace} is a homotopy monomorphism $U\rightarrow X$ such that there exists a small collection $\{\mathrm{Spec}(A_{i})\rightarrow U\}_{i\in\mathcal{I}}$ of derived $\mathbf{G}$-spaces and an effective epimorphism $\coprod_{i\in\mathcal{I}}\mathrm{Spec}(A_{i})\rightarrow U$ such that each $\mathrm{Spec}(A_{i})\rightarrow U\rightarrow X$ is a homotopy monomorphism in $\mathbf{P}$.
        \item Let $X\in\mathbf{Stk}(\mathbf{A},\tau|_{\mathbf{A}})$. A $\mathbf{G}$-\textit{subspace} $Y\rightarrow X$ is a homotopy monomorphism such that for every map $\mathrm{Spec}(A)\rightarrow X$ with $\mathrm{Spec}(A)\in\mathbf{A}$, $Y\times_{X}\mathrm{Spec}(A)\rightarrow\mathrm{Spec}(A)$ is a derived $\mathbf{G}$-subspace in the sense of (2).
        \item A \textit{derived} $\mathbf{G}$-\textit{space} is an object $X\in\mathbf{Stk}(\mathbf{A},\tau|_{\mathbf{A}})$ such that there exists a small collection $\{\mathrm{Spec}(A_{i})\rightarrow X\}_{i\in\mathcal{I}}$ of $\mathbf{G}$-subspaces which are affine, and such that $\coprod_{i\in\mathcal{I}}\mathrm{Spec}(A_{i})\rightarrow X$ is an effective epimorphism.
    \end{enumerate}
\end{definition}

The reader may worry that we are going somewhat overboard with various notions of spaces. However this definition of a derived space is purely for comparison with \cite{soor2024derived}*{Definition 2.17}. 

\begin{proposition}
    A scheme is a derived $\mathbf{G}$-space.
\end{proposition}

\begin{proof}
    Let $X$ be a $2$-scheme. Pick an atlas $\{\mathrm{Spec}(A_{i})\rightarrow X\}_{i\in\mathcal{I}}$. By definition, $\coprod_{i\in\mathcal{I}}\mathrm{Spec}(A_{i})\rightarrow X$ is an effective epimorphism, and each $\mathrm{Spec}(A_{i})\rightarrow X$ is a homotopy monomorphism in $\mathbf{P}$. Let $\mathrm{Spec}(A)\rightarrow X$ be arbitrary, and consider $\mathrm{Spec}(A)\times_{X}\mathrm{Spec}(A_{i})$. We need to show this is a $\mathbf{G}$-subspace of $\mathrm{Spec}(A)$. Now the map $\mathrm{Spec}(A)\times_{X}\mathrm{Spec}(A_{i})\rightarrow\mathrm{Spec}(A)$ is a homotopy monomorphism in $\mathbf{P}$, since the class of maps having this property is stable by pullback. By Proposition \ref{prop:fibrescheme} $\mathrm{Spec}(A)\times_{X}\mathrm{Spec}(A_{i})$ is a scheme. Thus it has an atlas $\mathrm{Spec}(B_{j_{i}})\rightarrow\mathrm{Spec}(A)\times_{X}\mathrm{Spec}(A_{i})$. These are homotopy monomorphisms in $\mathbf{P}$. Hence each composite $\mathrm{Spec}(B_{j_{i}})\rightarrow\mathrm{Spec}(A)$ is a homotopy monomorphism in $\mathbf{P}$. Hence $\{\mathrm{Spec}(B_{j_{i}})\rightarrow \mathrm{Spec}(A)\times_{X}\mathrm{Spec}(A_{i})\}_{j_{i}\in\mathcal{J}_{i}}$ gives the set of maps required by Definition \ref{Defn:soorspace} (2).
\end{proof}

\begin{proposition}\label{prop:derg0sub}
      A derived $\mathbf{G}$-subspace $U\rightarrow\mathrm{Spec}(C)$ of an affine scheme is a $0$-scheme. Moreover, the map $\rightarrow\mathrm{Spec}(C)$ is a homotopy monomorphism in $\mathbf{P}$.
\end{proposition}

\begin{proof}
 For any maps $\mathrm{Spec}(A)\rightarrow U,\mathrm{Spec}(B)\rightarrow U$, we have $\mathrm{Spec}(A)\times_{U}\mathrm{Spec}(B)\cong\mathrm{Spec}(A)\times_{\mathrm{Spec}(C)}\mathrm{Spec}(B)$ is affine. From this one deduces immediately that the diagonal map $U\rightarrow U\times U$ is representable by an (affine) scheme, and that the collection $\{\mathrm{Spec}(A_{i})\rightarrow U\}$ furnished by Definition \ref{Defn:soorspace} (2) is an atlas. Hence, $U$ is a scheme. The fact that $U\rightarrow\mathrm{Spec}(C)$ it is a homotopy monomorphism in $\mathbf{P}$ follows immediately from the definitions.
\end{proof}

\begin{proposition}\label{prop:derg0}
  Suppose the topology $\tau$ consists of homotopy monomorphisms in $\mathbf{P}$. Let $X$ be a derived $\mathbf{G}$-space admitting a cover $\{\mathrm{Spec}(A_{i})\rightarrow X\}$ as in Definition \ref{Defn:soorspace} (2). Suppose further that for any affine $\mathrm{Spec}(B)\rightarrow X$ we have that $\mathrm{Spec}(B)\times_{X}\mathrm{Spec}(A_{i})$ is affine. Then $X$ is a $1$-scheme.
\end{proposition}

\begin{proof}
   By assumption, for any affine $\mathrm{Spec}(B)\rightarrow X$, $\mathrm{Spec}(B)\times_{X}\mathrm{Spec}(A_{i})$ is an affine scheme, and hence a scheme. Moreover, the projection to $\mathrm{Spec}(A_{i})$ is (a homotopy monomorphism) in $\mathbf{P}$. In particular, each map $\mathrm{Spec}(A_{i})\rightarrow X$ is representable by an affine scheme, and is a map in $\mathbf{P}$.
    
    It remains to prove that $X\rightarrow X\times X$ is representable by a $0$-scheme.  Let $\mathrm{Spec}(A)\rightarrow X$ be any map. By assumption $\mathrm{Spec}(A)\times_{X}\mathrm{Spec}(A_{i})$ is a $\mathbf{G}$-subspace of $\mathrm{Spec}(A)$, and hence a $0$-scheme. We need to show that $\mathrm{Spec}(A)\times_{X}\mathrm{Spec}(B)$ is a scheme for any $\mathrm{Spec}(B)\rightarrow X$. There is a cover $\{\mathrm{Spec}(B_{j})\rightarrow \mathrm{Spec}(B)\}_{j\in\mathcal{J}}$ in $\tau$ such that each $\mathrm{Spec}(B_{j})\rightarrow X$ factors through some $\mathrm{Spec}(A_{j(i)})$. Then $\{\mathrm{Spec}(A)\times_{X}\mathrm{Spec}(B_{i})\rightarrow\mathrm{Spec}(A)\times_{X}\mathrm{Spec}(B)\}$ is a collection of homotopy monomorphisms in $\mathbf{P}$. We have 
    $$\mathrm{Spec}(A)\times_{X}\mathrm{Spec}(B_{i})\cong(\mathrm{Spec}(A)\times_{X}\mathrm{Spec}(A_{i(j)}))\times_{\mathrm{Spec}(A_{i(j)})}\mathrm{Spec}(B_{i}).$$ 
By assumption $\mathrm{Spec}(A)\times_{X}\mathrm{Spec}(A_{i(j)})$ is an affine scheme. Thus $\mathrm{Spec}(A)\times_{X}\mathrm{Spec}(B_{i})$ is an affine scheme. Finally, for any $\mathrm{Spec}(C)\rightarrow\mathrm{Spec}(A)\times_{X}\mathrm{Spec}(B)$, we have $\mathrm{Spec}(C)\times_{\mathrm{Spec}(A)\times_{X}\mathrm{Spec}(B)}(\mathrm{Spec}(A)\times_{X}\mathrm{Spec}(B_{i}))\cong\mathrm{Spec}(C)\times_{\mathrm{Spec}(B)}\mathrm{Spec}(B_{i})$ which is affine. Hence this gives a $0$-schematic atlas. 
    

\end{proof}

\begin{proposition}
  Suppose the topology $\tau$ consists of homotopy monomorphisms in $\mathbf{P}$. Let $X$ be a derived $\mathbf{G}$-space. Then $X$ is a $2$-scheme.
\end{proposition}

\begin{proof}
Let $\{\mathrm{Spec}(A_{i})\rightarrow X\}_{i\in\mathcal{I}}$ be a cover as in Definition \ref{Defn:soorspace} (2). Let $\mathrm{Spec}(A)\rightarrow X$ be any map with $\mathrm{Spec}(A)\in\mathbf{A}$. Then $\mathrm{Spec}(A_{i})\times_{X}\mathrm{Spec}(A)\rightarrow\mathrm{Spec}(A)$ is a derived $\mathbf{G}$-subspace. By Proposition \ref{prop:derg0sub}, $\mathrm{Spec}(A_{i})\times_{X}\mathrm{Spec}(A)$ is a $0$-scheme, and $\mathrm{Spec}(A_{i})\times_{X}\mathrm{Spec}(A)\rightarrow\mathrm{Spec}(A)$ is a homotopy monomorphism in $\mathbf{P}$.

It remains to prove that the diagonal $X\rightarrow X\times X$ is representable by a $1$-scheme. 

By assumption $\mathrm{Spec}(A)\times_{X}\mathrm{Spec}(A_{i})$ is a $\mathbf{G}$-subspace of $\mathrm{Spec}(A)$, and hence a $0$-scheme. We need to show that $\mathrm{Spec}(A)\times_{X}\mathrm{Spec}(B)$ is a $1$-scheme for any $\mathrm{Spec}(B)\rightarrow X$. By Proposition \ref{prop:derg0} it suffices to find an atlas consisting of homotopy monomorphisms in $\mathbf{P}$ which are representable by affines. There is a cover $\{\mathrm{Spec}(B_{j})\rightarrow \mathrm{Spec}(B)\}_{j\in\mathcal{J}}$ in $\tau$ such that each $\mathrm{Spec}(B_{j})\rightarrow X$ factors through some $\mathrm{Spec}(A_{i(j)})$. Then $\{\mathrm{Spec}(A)\times_{X}\mathrm{Spec}(B_{j})\rightarrow\mathrm{Spec}(A)\times_{X}\mathrm{Spec}(B)\}_{j\in\mathcal{J}}$ is a collection of homotopy monomorphisms in $\mathbf{P}$. For any $\mathrm{Spec}(C)\rightarrow\mathrm{Spec}(A)\times_{X}\mathrm{Spec}(B)$, we have $\mathrm{Spec}(C)\times_{\mathrm{Spec}(A)\times_{X}\mathrm{Spec}(B)}(\mathrm{Spec}(A)\times_{X}\mathrm{Spec}(B_{j}))\cong\mathrm{Spec}(C)\times_{\mathrm{Spec}(B)}\mathrm{Spec}(B_{j})$ which is affine.
\end{proof}

In the specific situation considered in \cite{soor2023quasicoherent}, the topology does indeed consist of homotopy monomorphisms. Thus, in this case, spaces coincide with $2$-schemes. Since we have already shown that all schemes are derived $\mathbf{G}$-spaces, we have the following.

\begin{corollary}
   Suppose the topology $\tau$ consists of homotopy monomorphisms in $\mathbf{P}$. Then any $n$-scheme is a $2$-scheme.
\end{corollary}

\begin{proposition}\label{prop:imagegeom}
    Let $f:X\rightarrow Y$ be a morphism in $\mathbf{P}$ between geometric stacks. Then $\mathrm{im}(f)\rightarrow Y$, computed in $\mathbf{Stk}(\mathbf{Aff}_{\mathbf{C}},\tau)$, is a geometric stack. Moreover, the map $i_{f}:\mathrm{im}(f)\rightarrow Y$ is a homotopy monomorphism in $\mathbf{P}$.
\end{proposition}

\begin{proof}
    The stack $\mathrm{im}(f)$ is computed as the nerve of $f$. This is the quotient of the Segal groupoid $n\mapsto X^{\times_{Y}n}$. The map $X\times_{Y}X\rightarrow X$ is in $\mathbf{P}$. This implies that $X\rightarrow\mathrm{im}(f)$ is a geometric stack by \cite{toen2008homotopical}*{Proposition 1.3.4.2}. The image is a sub-object of the target in any topos. 
\end{proof}

\begin{proposition}\label{prop:schunion}
    Let $X$ be a scheme and $\{U_{i}\rightarrow X\}_{i\in\mathcal{I}}$ a collection of sub-schemes - that is, each $U_{i}\rightarrow X$ is a homotopy monomorphism in $\mathbf{P}$. Then $\bigcup_{i\in\mathcal{I}}U_{i}\defeq\mathrm{im}(\coprod_{i\in\mathcal{I}}U_{i}\rightarrow X)$ is a subscheme.
\end{proposition}

\begin{proof}
   Write $U=\coprod_{i\in\mathcal{I}}U_{i}$. The stack $\mathrm{im}(\coprod_{i\in\mathcal{I}}U_{i}\rightarrow X)$ is the geometric realisation of the Segal groupoid $U^{\times_{X}n}$. This is a geometric stack by Proposition \ref{prop:imagegeom}. Moreover, by construction, $\{U_{j}\rightarrow\bigcup_{i\in\mathcal{I}}U_{i}\}_{j\in\mathcal{J}}$ is an atlas. The maps $U_{j}\rightarrow\bigcup_{i\in\mathcal{I}}U_{i}\rightarrow X$ are all homotopy monomorphisms, so $U_{j}\rightarrow \bigcup_{i\in\mathcal{I}}U_{i}$ is as well. Let $V\rightarrow \bigcup_{i\in\mathcal{I}}U_{i}$, $W\rightarrow \bigcup_{i\in\mathcal{I}}U_{i}$ be any morphism with $V$ affine. Then $U_{j}\times_{\bigcup_{i\in\mathcal{I}}U_{i}}V\cong U_{j}\times_{X}V$ is a scheme. 
\end{proof}

\noindent The full subcategory of $\mathbf{Stk}(\mathbf{A},\tau|_{\mathbf{A}})$ consisting of schemes is denoted $\mathbf{Sch}(\mathbf{A}, \tau_{\vert \bA},\mathbf{P}_{|\mathbf{A}})$.

\comment{
\noindent The full subcategory of $\mathbf{Stk}(\mathbf{A},\tau|_{\mathbf{A}})$ consisting of $n$-schemes is denoted $\mathbf{Sch}_{n}(\mathbf{A}, \tau_{\vert \bA})$. We also write $\mathbf{Sch}(\mathbf{A}, \tau_{\vert \bA}) \defeq\bigcup_{n=-1}^{\infty}\mathbf{Sch}_{n}(\mathbf{A}, \tau_{\vert \bA})$.

\begin{proposition}
    Let $X$ be an $n$-scheme for $n\ge 1$, and let $\{V_{i}\rightarrow X\}_{i=1}^{m}$ be homotopy monomorphisms with each $V_{i}$ affine. Then $V_{1}\times_{X}\ldots_{X}V_{m}$ is an $(n-1)$-scheme.
\end{proposition}

\begin{proof}
    The proof is by induction on $n$. Let $\{U_{i}\rightarrow X\}_{i\in\mathcal{I}}$ be an $n$-atlas. Write 
\end{proof}

}
\subsubsection{Qcqs morphisms and schemes}

Following \cite{soor2024derived}*{Definition 2.19}, we define quasi-compact and quasi-separated morphisms.

\begin{definition}
\begin{enumerate}
    \item 
    A scheme $X$ is said to be \textit{quasi-compact} if it admits an atlas $\coprod_{i=1}^{n}\mathrm{Spec}(U_{i})\rightarrow X$ with each $U_{i}$ affine.
    \item 
    A morphism $f:\mathcal{X}\rightarrow\mathcal{Y}$ of stacks is said to be \textit{quasi-compact} if for every map $Z\rightarrow Y$ with $Z$ affine, $Z\times_{\mathcal{Y}}\mathcal{X}$ is a quasi-compact scheme.
    \item A morphism $f:\mathcal{X}\rightarrow\mathcal{Y}$ of stacks is said to be \textit{quasi-separated} if its diagonal $\Delta_{f}:\mathcal{X}\rightarrow\mathcal{X}\times_{\mathcal{Y}}\mathcal{X}$ is quasi-compact.
    \end{enumerate}
\end{definition}

A quasi-compact and quasi-separated morphism will be called a \textit{qcqs morphism}. 

\begin{definition}
    A scheme $X$ is said to be \textit{quasi-separated} if the map $X\rightarrow\mathrm{Spec}(R)$ is quasi-separated.
\end{definition}

The following is clear.

\begin{proposition}
    Let $X$ be quasi-separated, and let $U\rightarrow X, V\rightarrow X$ be open immersions with $U$ and $V$ affine. Then $U\times_{X}V$ is quasi-affine.
\end{proposition}

\begin{proof}
    The quasi-separated assumption means that the intersection is quasi-compact. We evidently have an open immersion $U\times_{X}V\rightarrow U$.
\end{proof}

\begin{proposition}\label{prop:qcqscompact}
    Let $X$ be a quasi-compact scheme, and let $\{U_{i}\rightarrow X\}_{i\in\mathcal{I}}$ be sub-schemes such that $\bigcup_{i\in\mathcal{I}}U_{i}=X$. Then there is a finite subset $\mathcal{J}\subset\mathcal{I}$ such that $X=\bigcup_{j\in\mathcal{J}}U_{j}$.
\end{proposition}

\begin{proof}
    Since $X$ is quasi-compact, there exists a finite atlas $\{V_{k}\rightarrow X\}_{k=1}^{n}$. The map $\coprod_{i\in\mathcal{I}}U_{i}\rightarrow X$ is an effective epimorphism. This means that for each $k$ there is a finite cover $\{W_{l_{k}}\rightarrow V_{k}\}_{l_{k}\in\mathcal{L}_{k}}$ such that each $W_{l_{k}}\rightarrow V_{k}\rightarrow X$ factors through some $U_{i_{l_{k}}}$. Then $\coprod_{1\le k\le n,l_{k}\in\mathcal{L}_{k}} W_{l_{k}}\rightarrow X$ is an affective epimorphism, so $\coprod_{1\le k\le n,l_{k}\in\mathcal{L}_{k}}U_{i_{l_{k}}}\rightarrow X$ is an effective epimorphism. This implies $X=\bigcup_{1\le k\le n,l_{k}\in\mathcal{L}_{k}}U_{i_{l_{k}}}$.
\end{proof}

\begin{proposition}
    Let $f:\mathcal{X}\rightarrow\mathcal{Y}$ be a quasi-compact morphism, and $Z\rightarrow\mathcal{Y}$ a morphism with $Z$ a quasi-compact scheme. Then $Z\times_{\mathcal{Y}}\mathcal{X}$ is a quasi-compact scheme.
\end{proposition}

\begin{proof}
    If $Z$ is affine then this is by definition of quasi-compactness. In general, choose an atlas $\{\mathrm{Spec}(A_{i})\rightarrow Z\}_{i=1}^{n}$. For each $i$, $\mathrm{Spec}(A_{i})\times_{\mathcal{Y}}\mathcal{X}$ is a quasi-compact scheme. Thus it admits an atlas $\{\mathrm{Spec}(B_{j_{i}})\rightarrow\mathrm{Spec}(A_{i})\times_{\mathcal{Y}}\mathcal{X}\}_{j_{i}\in\mathcal{J}_{i}}$ with $\mathcal{J}_{i}$ a finite set. Then $\{\mathrm{Spec}(B_{j_{i}})\rightarrow\mathrm{Spec}(A_{i})\times_{\mathcal{Y}}\mathcal{X}\rightarrow Z\times_{\mathcal{Y}}\mathcal{X}\}_{j_{i}\in\mathcal{J}_{i},1\le i\le n}$ is a finite atlas. 
\end{proof}

\subsection{Derived spaces and the underlying topological space}

Here we generalise the discussion from \cite{soor2024derived}*{Section 2.3} and \cite{soor2025privatecommunicationhesis} (which in turn followed \cite{clausenscholze2}) defining the underlying space, and the Berkovich space, of certain schemes.
 Consider a strong geometry tuple of the form
$$\mathbf{G}=(\mathbf{Aff}_{\mathbf{C}},\tau,\mathbf{P},\mathbf{A})$$
with $\tau$ a quasi-compact topology. Write $\mathbf{dSch_{G}}=\mathbf{Sch}(\mathbf{Aff}_{\mathbf{C}},\tau,\mathbf{P})$ for the category of schemes. Fix $X\in\mathbf{dSch_{G}}$, and write $\mathrm{Sub}_{\mathbf{G}}(X)$ for the category of sub-schemes of $X$. By Proposition \ref{prop:schunion}, this is a sub-locale of the locale of subobjects of $X$ in the topos $\mathbf{Stk}(\mathbf{Aff_{C}},\tau)$.

\begin{remark}
    Within the setup of derived $\mathbf{G}$-spaces (in the specific instance of derived rigid analytic geometry), Soor has previously given a different proof that subspaces form a sub-locale of the locale of subobjects, in \cite{soor2025privatecommunicationhesis}.
\end{remark}

We let 
    $$|X|=\mathrm{pt}(\mathrm{Sub}_{\mathbf{G}}(X))$$
    denote the topological space of points of $X$, i.e., the set of completely prime filters on the locale.

\subsubsection{Interlude: Locales, and Berkovich locales}

Before discussing the spaces $|X|$ in detail, let us here develop some generalities concerning what we call \textit{Berkovich locales}. Our main reference concerning spaces associated to lattices is \cite{MR0698074}.

In this part, by a \textit{pre-locale}, we mean a distributive lattice which has all finite meets and joins. In particular, the lattice is bounded. Let $L$ be a pre-locale. We define its associated locale $\bigvee L$ as follows. Let $\tilde{L}$ be the distributive lattice obtained from $L$ by freely adding all joins. To get to $\bigvee L$, we quotient $\tilde{L}$ by all relations needed so that the map $L\rightarrow\tilde{L}\rightarrow\bigvee L$ is a map of distributive lattices that commutes with all joins.

    \begin{definition}
        Let $L$ be a pre-locale. 
        \begin{enumerate}
            \item 
                An object $V\in L$ is said to be \textit{quasi-compact} if, whenever $V\subseteq\bigcup_{i\in\mathcal{I}} U_{i}$, then there is a finite subset $\mathcal{J}\subset\mathcal{I}$ such that $V\subseteq\bigcup_{j\in\mathcal{J}}U_{j}$.
                \item 
                    An object $V\in L$ is said to be \textit{co-compact} if, whenever $\bigcap_{i\in\mathcal{I}} U_{i}\subseteq V$, then there is a finite subset $\mathcal{J}\subset\mathcal{I}$ such that $\bigcap_{j\in\mathcal{J}}U_{j}\subseteq V$.
        \end{enumerate}
    
    \end{definition}

\begin{proposition}
    Let $L$ be a locale, and $B\subset L$ a distributive sub-lattice closed under finite meets and finite joins, and such that for every $B_{1},B_{2}\in B$, there is $C\in B$ such that $C\le B_{1}\cap B_{2}$. Suppose further that elements of $B$ are quasi-compact, and that $B$ is a basis for $L$. Then there are bijections between
    \begin{enumerate}
        \item 
        prime filters of $B$ and completely prime filters of $L$;
        \item 
        maximally prime filters of $B$ and maximally completely prime filters of $L$;
        \item 
        minimally prime filters of $B$, and minimally completely prime filters of $L$.
    \end{enumerate}
\end{proposition}

\begin{proof}
\begin{enumerate}
    \item 
        Let $\mathcal{F}$ be a prime filter of $B$. Let $\mathcal{F}'=\{U\in L:\exists V\subset U, V\in \mathcal{F}\}$. Clearly this is a proper filter. We claim that it is completely prime. Suppose that $\bigcup_{i\in\mathcal{I}}U_{i}\in\mathcal{F}'$. There is some $b\in \mathcal{F}$, $b\subset\bigcup_{i\in\mathcal{I}}U_{i}$. Write $U_{i}=\bigcup_{i\in\mathcal{I}}\bigcup_{j_{i}\in\mathcal{J}_{i}}b_{j_{i}}$, where each $b_{j_{i}}\in B$. By quasi-compactness, $b$ is a subset of a finite union $\bigcup_{k=1}^{n}b_{k}$ of the $b_{j_{i}}$. Thus $\bigcup_{k=1}^{n}b_{k}\in\mathcal{F}$, and so some $b_{k'}\in\mathcal{F}$. But $b_{k'}$ is a subset of one of the $U_{i}$, so this $U_{i}\in\mathcal{F}'$.

        Going the other way, let $\mathcal{G}$ be a completely prime filter on $L$. Let $\tilde{\mathcal{G}}=\{b\in B:b\in L\}$. This is a prime filter. Indeed, it is non-empty, since for $U\in\mathcal{G}$, we may write $U=\bigcup_{i\in\mathcal{I}}b_{i}$ for $b_{i}\in B$. By complete primality, some $b_{i}\in\mathcal{G}$. It is easily seen that $\tilde{\mathcal{G}}$ is a prime filter on $B$, and that the constructions
        $$\mathcal{F}\mapsto\mathcal{F}'$$
        $$\mathcal{G}\mapsto\tilde{\mathcal{G}}$$
        are in bijection with each other.
        \item 
        This follows immediately from the first claim, as the bijection is order preserving.
\item 
        This follows immediately from the first claim, as the bijection is order preserving.

\end{enumerate}
\end{proof}

For $L$ a distributive lattice, we let $|L|$ denote the space consisting of prime filters on $L$. A basis of open subsets of $|L|$ is given by sets of the form $|L_{\big\slash M}|=\{\mathcal{F}\in |L|:M\in\mathcal{F}\}$, as $M$ ranges over all elements of $L$. For $L$ a locale, we let $\mathrm{pt}(L)$ be the space of completely prime filters of $L$, with open subsets analogous to the ones above. If $L$ is a pre-locale consisting of quasi-compact objects, then we have $|L|\cong \mathrm{pt}(\bigvee L)$. 

Consider the basis of open sets given by $B_{M}=\{|L_{\big\slash M}|:M\in L\}$. This basis $B_{L}$ may itself be regarded as a distributive lattice. It is a pre-locale if $L$ is. Moreover, there is a morphism of lattices
$$L\rightarrow B_{L}.$$

\begin{definition}
    A distributive lattice $L$ is said to be \textit{spatial} if the morphism $L\rightarrow B_{L}$ is an isomorphism. 
\end{definition}

For $L$ a distributive lattice, we also let $|L|_{Ber}$ denote the subset of $|L|$ consisting of maximal filters. If $L$ is a locale, we let $\mathrm{pt}(L)_{Ber}$ denote the set of maximally completely prime filters. Again if $L$ is a pre-locale consisting of quasi-compact objects, then we have $|L|_{Ber}=\mathrm{pt}(\bigvee L)_{Ber}$.

Let $L$ be a lattice and $M\in L$. We get a map of lattices $i_{M}:L\rightarrow L_{\big\slash M}$, and hence a map of spaces
$$|i_{M}|:|L_{\big\slash M}|\rightarrow |L|.$$
From now on we will write $|M|\defeq |L_{\big\slash M}|$

\begin{proposition}\label{prop:imageberksub}
    Let $L$ be a distributive lattice, let $M\in L$ and consider the lattice $L_{\big\slash M}$.  The map of spaces $|i|:|M|\rightarrow |L|$ restricts to a map
    $$|i|_{Ber}:|M|_{Ber}\rightarrow |L|_{Ber},$$
    identifying maximal filters in $L_{\big\slash M}$ with maximal filters in $L$ containing $M$.  
\end{proposition}

\begin{proof}
    Let $\mathcal{F}\in|M|$ be a maximal filter, and let $i(\mathcal{F})\subset\mathcal{G}$, where $\mathcal{G}$ is a proper filter. We first need to show that $i(\mathcal{F})=\mathcal{G}$. Consider $\tilde{\mathcal{G}}=\{M\cap V:V\in\mathcal{G}\}$. If this is proper, then it is a filter containing $\mathcal{F}$. Suppose $\tilde{\mathcal{G}}$ were not proper. Then it would contain $0$. But then for some $V\in\mathcal{G}$, $V\cap M=0$. Since $M\in i(\mathcal{F})\subset\mathcal{G}$, this would give $0\in\mathcal{G}$, a contradiction. Hence $\tilde{\mathcal{G}}=i(\mathcal{F})$. Since for all $G\in\mathcal{G},G\cap M\in\tilde{\mathcal{G}}=\mathcal{F}$. Thus $G\in i(\mathcal{F})$. 

Clearly now if $\mathcal{F}$ is a maximal filter in $L_{\big\slash M}$ then $i(\mathcal{F})$ is a maximal filter in $L$ containing $M$. On the other hand, let $\mathcal{G}$ be a maximal filter in $L$ containing $M$. Again consider $\tilde{\mathcal{G}}=\{M\cap V:V\in\mathcal{G}\}$. By the same proof as above, $\tilde{\mathcal{G}}$ is a maximal filter in $L_{\big\slash M}$, and $i(\tilde{\mathcal{G}})=\mathcal{G}$.
\end{proof}

Let $L$ be a lattice. We topologise $|L|_{Ber}$ be declaring the set $\{|L|_{Ber}\setminus |M|_{Ber}:M\in L\}$ to be a basis of open sets. 

We also let $\mathrm{Spec}(L)$ denote the set of prime ideals, topologised as in \cite{MR0698074}*{Section II.3.4}. Finally, we let $\mathrm{MaxSpec}(L)$ be the subspace of $\mathrm{Spec}(L)$ consisting of maximal ideals, equipped with the subspace topology.

Note that the order reversing identity map
$$L\rightarrow L^{op}$$
interchanges (prime) filters with (prime) ideals. It also interchanges maximal filters with maximal ideals. In particular, we get a bijection
$$c:|L|\rightarrow \mathrm{Spec_{Zar}}(L^{op})$$
which restricts to a bijection
$$c:|L|_{Ber}\rightarrow\mathrm{MaxSpec_{Zar}}(L^{op}).$$
We equip $\mathrm{MaxSpec_{Zar}}(L^{op})$ with the subspace topology. $ \mathrm{Spec_{Zar}}(L^{op})$ has a basis of open subsets defined as follows. Let $M\in L^{op}$, and consider the ideal in $L^{op}$, $I_{M}=\{V\in L^{op}:V\le^{op} M\}=\{V\in L:V\ge M\}$. Then $D(I_{M})=\{\mathcal{F}:I_{M}\not\subset\mathcal{F}\}=\{\mathcal{F}:M\notin\mathcal{F}\}=c(|X|_{Ber}\setminus |M|_{Ber})$. We therefore have the following.

\begin{corollary}
    The map 
    $$c:|L|_{Ber}\rightarrow\mathrm{MaxSpec_{Zar}}(L^{op})$$
    is a homeomorphism.
\end{corollary}

\begin{definition}
    A distributive lattice $L$ is said to be \textit{co-normal} if its opposite lattice $L^{op}$ is normal (\cite{MR0698074}*{Definition 3.6}).
\end{definition}

If $L$ is co-normal then $L^{op}$ is normal.
The following is now a direct consequence of \cite{MR0698074}*{Section II.3.6}.

\begin{corollary}
    Let $L$ be a distributive lattice. Then $|L|_{Ber}$ is compact. If $L$ is co-normal, then it is Hausdorff.
\end{corollary}

The following is proven exactly as in \cite{MR1329448}*{Remark 2.1}.

\begin{proposition}[\cite{MR1329448}*{Remark 2.1}]
    Let $L$ be a distributive lattice, and let $S$ be any family of elements of $L$, closed with respect to finite joins. Assume there is a filter $\mathcal{F}$ on $L$ such that $\mathcal{F}\cap S=\emptyset$. Then there is a filter $\mathcal{G}$ on $L$ containing $\mathcal{F}$, which is maximal with respect to $\mathcal{G}\cap S=\emptyset$. Moreover $\mathcal{G}$ is prime.
\end{proposition}

Using this, the following can then be proven identically to \cite{MR1329448}*{Lemma 3.1 (i)}.

\begin{lemma}
    Let $L$ be a distributive lattice. Then $|L|$ is compact.
\end{lemma}

For a conormal lattice, the inclusion $|L|_{Ber}\rightarrow |L|$ has a section $r_{L}$, sending a prime filter $\mathcal{F}$ to the unique maximal filter containing it.

\begin{proposition}
    The map $r_{L}$ is continuous. In particular, it is a quotient map.
\end{proposition}

\begin{proof}
    We claim that $r_{L}^{-1}(|L|_{Ber}\setminus |L_{\big\slash M}|_{Ber})=Cl(|M|)$ is the closure of $|M|$ in $|L|$. Indeed, let $\mathcal{F}\in r_{L}^{-1}(|M|_{Ber})$, and let $|N|$ be a basic open containing $\mathcal{F}$ with $N\in\mathcal{F}$. Let $\mathcal{F}_{x}$ be the unique maximal filter containing $\mathcal{F}$. Then $M\in\mathcal{F}_{x}$ by assumption. On the other hand, $N\in\mathcal{F}\subset\mathcal{F}_{x}$, so $\mathcal{F}_{x}\in |M|\cap|N|$.

Conversely, let $\mathcal{F}\in Cl(|M|)$. Then for any $N\in\mathcal{F}$, $N\cap M\neq0$. In particular $\mathcal{F}_{M}=\{V|\exists N\in\mathcal{F}:(N\cap M)\subset V\}$ does not contain $0$. This is then a proper filter. It is contained in some maximal filter $\mathcal{F}_{x}$. Now $M\in\mathcal{F}_{x}$. Moreover, clearly $\mathcal{F}\subset\mathcal{F}_{M}\subset\mathcal{F}_{x}$. Thus $\mathcal{F}_{x}$ is the unique maximal filter containing $\mathcal{F}$, and $r(\mathcal{F})\in |M|_{Ber}$. 

    Since $|L|$ is compact, and $|L|_{Ber}$ is compact Hausdorff, $r_{L}$ is a quotient map.
\end{proof}

Let $f:K\rightarrow L$ be a morphism of conormal lattices. We define $|f|_{Ber}:|K|_{Ber}\rightarrow |L|_{Ber}$ by $r_{L}\circ |f|\circ i_{K}$, where $r_{L}:|L|\rightarrow |L|_{Ber}$ is the retraction, and $i_{K}:|K|_{Ber}\rightarrow |K|$ is the inclusion.

\begin{proposition}
Let $f:K\rightarrow L$ be a morphism in $\mathrm{Lattice}^{op}$, with corresponding map of lattices $\phi$. Then $|f|_{Ber}$ is continuous.
\end{proposition}

\begin{proof}
It suffices to prove that $|f|_{Ber}^{-1}(|M|_{Ber})=|\phi(M)|_{Ber}$ for any $M\in L$. But we have 
$$|f|^{-1}(r_{L}^{-1}(|M|_{Ber}))=|f|^{-1}(Cl(|M|))\subset Cl(|f|^{-1}|M|)=Cl(|\phi(M)|)=r_{K}^{-1}(|\phi(M)|_{Ber}).$$
Thus if $\mathcal{F}\in |f|_{Ber}^{-1}(|M|_{Ber})$, then $i_{K}(\mathcal{F})\in r_{K}^{-1}(|\phi(M)|_{Ber})$. Hence, $\mathcal{F}=r_{K}(i_{K}(\mathcal{F}))\in|\phi(M)|_{Ber}$. 

Conversely, if $\mathcal{F}\in|\phi(M)|_{Ber}$ then, by definition, $|f|(\mathcal{F})$ contains $M$. Thus $r_{L}(|f|(\mathcal{F}))\in|M|_{Ber}$. 
\end{proof}

\begin{proposition}
    Let $f:K\rightarrow L$ and $g:L\rightarrow M$ be morphisms of co-normal distributive lattices. Then $|g|_{Ber}\circ |f|_{Ber}=|g\circ f|_{Ber}$. 
\end{proposition}

\begin{proof}
  
  Let $\mathcal{F}\in|K|_{Ber}$. Then $|f|_{Ber}(\mathcal{F})$ is the unique maximal filter containing $f(\mathcal{F})$. Hence $f(\mathcal{F})\subset |f|_{Ber}(\mathcal{F})$, and $g(f(\mathcal{F}))\subset g(|f|_{Ber}(\mathcal{F}))$. Therefore the unique maximal filter containing $g(|f|_{Ber}(\mathcal{F}))$ coincides with the unique maximal filter containing $g(f(\mathcal{F}))$. These are $|g|_{Ber}(|f|_{Ber}(\mathcal{F}))$ and $|g\circ f|_{Ber}(\mathcal{F})$ respectively.
\end{proof}

\begin{definition}
    Let $\mathrm{CHaus}_{\mathrm{Ber}}$ denote the category defined as follows. An object is a pair $(X,\mathcal{P}(X))$ where

\begin{enumerate}
    \item 
    $X$ is a compact Hausdorff space;
    \item 
    $\mathcal{P}(X)$ is a distributive lattice of closed subsets of $X$, closed under finite joins and meets, such that $\{X\setminus V:V\in\mathcal{P}(X)\}$ forms a base for the topology on $X$.
\end{enumerate}

A morphism $f:(X,\mathcal{P}(X))\rightarrow (Y,\mathcal{P}(Y))$ in $\mathrm{CHaus_{Ber}}$ is a continuous map of compact Hausdorff spaces $f:X\rightarrow Y$ such that $f^{-1}(V)\in\mathcal{P}(X)$ for every $V\in\mathcal{P}(Y)$. 
\end{definition}

We have a functor
$$|-|_{Ber}:\mathrm{preLoc}_{\mathrm{con}}\rightarrow\mathrm{CHaus_{Ber}}$$
sending $L$ to $(|L|_{Ber},\mathcal{P}(L)=\{|V|_{Ber}:V\in L\})$.

There is also a functor
$$\Omega_{Ber}:\mathrm{CHaus_{Ber}}\rightarrow\mathrm{preLoc}_{\mathrm{con}}$$
sending $(X,\mathcal{P}(X))$ to the lattice $\mathcal{P}(X)$. Note that there is a natural morphism of lattices
$$L\rightarrow\mathcal{P}(L)$$
sending $M$ to $|M|_{Ber}$. This gives a natural morphism
$$\Omega_{Ber}\circ |-|_{Ber}\rightarrow\mathrm{Id}$$
of functors 
$$\mathrm{preLoc}_{\mathrm{con}}\rightarrow \mathrm{preLoc}_{\mathrm{con}}.$$
This is the counit of an adjunction
$$\adj{\Omega_{Ber}}{\mathrm{CHaus_{Ber}}}{\mathrm{preLoc}_{\mathrm{con}}}{|-|_{Ber}}.$$

\begin{proposition}
    The unit
    $$(X,\mathcal{P}(X))\rightarrow|-|_{Ber}\circ\Omega_{Ber}$$
    is an isomorphism. In particular, $\Omega_{Ber}$ is fully faithful.
\end{proposition}

\begin{proof}
    We prove that $|\mathcal{P}(X)|_{Ber}\cong X$. First, let us identify the points of $|\mathcal{P}(X)|_{Ber}$. An element is a maximal filter $\mathcal{F}$. Consider $C=\bigcap_{F\in\mathcal{F}}F$. This is a closed, non-empty set by the finite intersection property. Let $x,y\in C$. There are disjoint open sets $U=X\setminus M$, $V=X\setminus N$ with $x\in U,y\in V$, and $M,N\in\mathcal{P}(X)$. Now, $M\cup N=X\in\mathcal{F}$, so either $M\in\mathcal{F}$ or $N\in\mathcal{F}$. But then either $x\notin C$ or $y\notin C$, a contradiction. Hence $C$ consists of a single point $x$. Moreover $\mathcal{F}_{x}=\{M\in\mathcal{P}(X):x\in M\}$ is a maximal filter containing $\mathcal{F}$, so that $\mathcal{F}=\mathcal{F}_{x}$. This gives a bijection between $X$ and $|\mathcal{P}(X)|_{Ber}$. One checks easily that this bijection sends $X\setminus M$ to $|X|_{Ber}\setminus |M|_{Ber}$ for $M\in\mathcal{P}(L)$. In particular, it is a homeomorphism. 
\end{proof}

\begin{definition}
    A pre-locale $L$ is said to be \textit{Berkovich} if 
    \begin{enumerate}
        \item 
        the lattice $L$ is a co-normal distributive lattice with finite meets and joins;
        \item 
        whenever $M,N\in L$ are such that $|M|_{Ber}\le |N|_{Ber}$, then $M\le N$. 
    \end{enumerate}
    A locale $L$ is said to be \textit{Berkovich} if it has a Berkovich basis consisting of quasi-compact objects.
\end{definition}

Being Berkovich is a local condition.

\begin{proposition}\label{prop:localberk}
    Let $L$ be a pre-locale with terminal object $1$, and let $\{U_{i}\}_{i=1}^{n}$ be a finite collection of objects of $L$ such that $1=\bigcup_{i=1}^{n}U_{i}$. 
    \begin{enumerate}
        \item 
           Suppose that each $L_{\big\slash U_{i}}$ is conormal. Then $L$ is conormal.
           \item 
           Suppose that each $L_{\big\slash U_{i}}$ consists of quasi-compact objects. Then $L$ consists of quasi-compact objects.
           \item 
           Suppose that each $L_{\big\slash U_{i}}$ is Berkovich. Then $L$ is Berkovich.
    \end{enumerate}
 
\end{proposition}

\begin{proof}
\begin{enumerate}
    \item 
       Let $b_{1},b_{2}\in L$ be such that $b_{1}\cap b_{2}=0$. Then for each $i$, $(b_{1}\cap U_{i})\cap (b_{2}\cap U_{i})=0$. Thus there are $c_{1i},c_{2i}\le U_{i}$ with $c_{1i}\cup c_{2i}=1$, $c_{1i}\cap b_{2}\cap U_{i}=c_{1i}\cap b_{2}=0$, and $b_{1}\cap c_{2i}\cap U_{i}=b_{1}\cap c_{2i}=0$. Put $c_{1}=\bigcup_{i=1}^{n}c_{1i}$ and $c_{2}=\bigcup_{i=1}^{n}c_{2i}$. Then clearly $c_{1}\cup c_{2}=\bigcup_{i=1}^{n}U_{i}=1$, $c_{1}\cap b_{2}=0$, and $c_{2}\cap b_{1}=0$.
        \item 
        Let $M=\bigcup_{j\in\mathcal{J}}V_{i}\in L$. Then for each $i$, $M\cap U_{i}=\bigcup_{j\in\mathcal{J}}(V_{j}\cap U_{i}).$ There is thus a finite subset $\mathcal{J}_{i}\subset\mathcal{J}$ such that $M\cap U_{i}=\bigcup_{j_{i}\in\mathcal{J}_{i}}(V_{j_{i}}\cap U_{i})$. Then $M=\bigcup_{i=1}^{n}\bigvee_{j_{i}\in\mathcal{J}_{i}}V_{j_{i}}$ and this is a finite union. 
        \item
        We have already seen that $L$ is conormal. Now let $M,N\in L$ be such that $|M|_{Ber}\le |N|_{Ber}$. 
        Let $M,N\in L$ with $|M|_{Ber}\le|N|_{Ber}$. By Proposition \ref{prop:imageberksub} we have $|M|_{Ber}\cap|U_{i}|_{Ber}=|M|_{Ber}\cap|U_{i}\cap M|_{Ber}$ and similarly for $N$. By the Berkovich condition for each $U_{i}$, we get $U_{i}\cap M\le U_{i}\cap N$ for each $1\le i\le n$. This gives $M\le N$. 
\end{enumerate}

\end{proof}

\begin{proposition}
    A Berkovich pre-locale is spatial.
\end{proposition}

\begin{proof}
    Let $M,N\in L$ be such that $|M|\le |N|$. Then $|M|_{Ber}\le |N|_{Ber}$, so that $M\le N$.
\end{proof}

\begin{proposition}
    A pre-locale $L$ is Berkovich if and only if $L$ is co-normal, and the unit
    $$L\rightarrow|-|_{Ber}\circ\Omega_{Ber}(L)$$
    is an isomorphism.
\end{proposition}

\begin{proof}
    Suppose that $L\rightarrow|-|_{Ber}\circ\Omega_{Ber}(L)$ is an isomorphism. To prove that $L$ is Berkovich, it suffices to prove that whenever $(X,\mathcal{P}(X))\in\mathrm{CHaus_{Ber}},$ then $\mathcal{P}(X)$ is Berkovich.  The opposite lattice to $\mathcal{P}(L)$ identifies with $B=\{X\setminus V:V\in\mathcal{P}(X)\}$. This is normal, since $B$ is a basis of a Hausdorff space. By definition, objects of $\mathcal{P}(X)$ are compact, and $X\in\mathcal{P}(X)$.  Now suppose $M,N\in\mathcal{P}(X)$. Write $\mathcal{P}(M)=\mathcal{P}(X)_{\big\slash M}$, and similarly for $\mathcal{P}(N)$.  Since $|\mathcal{P}(M)|_{Ber}=M, |\mathcal{P}(N)|_{Ber}=N$, we have $|\mathcal{P}(M)|_{Ber}\subset|\mathcal{P}(N)|_{Ber}$ implies $M\le N$.  

    Conversely, suppose that $L$ is Berkovich.  The space $X_{L}=|L|_{Ber}$ is compact Hausdorff, with base of open sets $B=\{|L|_{Ber}\setminus |M|_{Ber}:M\in\mathcal{P}(L)\}$. Let $\widetilde{\mathcal{P}(L)}=\{|M|_{Ber}:M\in\mathcal{P}(L)\}$. This is a distributive lattice of closed sets in $X_{L}$, closed under finite meets and joins. Moreover it is a co-normal lattice. Indeed, its opposite lattice can be identified with $B$. The map
$$\mathcal{P}(L)\rightarrow\widetilde{\mathcal{P}(L)}$$
$$M\mapsto|M|_{Ber}$$
is a map of lattices. It is evidently surjective. The second Berkovich condition implies that it is injective. 
\end{proof}


\begin{definition}
    Let $f:L\rightarrow K$ be a morphism of lattices. We say that $f$ is an \textit{open immersion} if there is an element $U$ of $L$ such that $f$ factors through the projection $L\rightarrow L_{\big\slash U}$ followed by an isomorphism $L_{\big\slash U}\rightarrow K$.
\end{definition}

\begin{proposition}\label{prop:spatlocaleop}
     Let $f:K\rightarrow L$ be a morphism of pre-locales. If $f$ is an open immersion then $|f|:|K|\rightarrow |L|$ is an open immersion. If $K$ and $L$ are spatial locales then the converse is true. 
\end{proposition}

\begin{proof}
    Suppose that $f$ is an open immersion. By definition, for $U\in L$, $|U|$ is an open subset of $|L|$. Then $|K|\cong |U|$. Clearly, $|K|\rightarrow |L|$ is then an open immersion.

  Conversely, suppose that $K$ and $L$ are spatial locales, and that $|f|$ is an open immersion.  By assumption, $K$ is isomorphic to the lattice of open subsets of $|K|$. Also by assumption, $|K|\cong U$ for some open subset $U$ of $|L|$. Such an open subset is of the form $U|$. Note that $L_{\big\slash U}$ is also spatial, so that the lattice of open subsets of $|U|$ is isomorphic to $L_{\big\slash U}$.
\end{proof}

\begin{definition}
    Let $(X,\mathcal{P}(X))$, $(Y,\mathcal{P}(Y))$ be in $\mathrm{CHaus_{Ber}}$. We say that a morphism $f:(X,\mathcal{P}(X))\rightarrow (Y,\mathcal{P}(Y))$ is a \textit{closed immersion} if it restricts to a homeomorphism $f:X\cong V$ for some $V\in\mathcal{P}(Y)$.
\end{definition}

\begin{proposition}\label{prop:Berkclosed}
Let $f:K\rightarrow L$ be a morphism of pre-locales, with $|K|$ quasi-compact, and all elements of $K$ and $L$ quasi-compact. If $f$ is an open immersion then $|f|_{Ber}$ is a closed immersion. 
\end{proposition}

\begin{proof}
    First, suppose that $K\rightarrow L$ is an open immersion. Then
    $$\mathrm{pt}(\bigcup K)\cong |K|\rightarrow |L|\cong\mathrm{pt}(\bigvee L)$$
    is an open immersion. Therefore $|K|\cong\bigcup_{i\in\mathcal{I}}|U_{i}|$ for some collection $\{U_{i}\}$ of elements of $L$. Since $|K|$ is quasi-compact, $|K|=\bigcup_{i=1}^{n}|U_{i}|\cong|bigcup_{i=1}^{n}U_{i}|$
    for some finite collection $\{U_{1},\ldots, U_{n}\}$. But then by spatiality, $K\cong L_{\big\slash \bigcup_{i=1}^{n}U_{i}}$, and $|K|_{Ber}\cong |\bigcup_{i=1}^{n}U_{i}|_{Ber}$ which by definition is a closed subset of $|L|_{Ber}$. \qedhere
    
\end{proof}

\comment{
\begin{example}
    For $R$ a unital commutative ring let $\mathrm{Spec}
    _{Zar}(R)$ denote the topological space of prime ideals of $R$ with the Zariski topology. Let $\{\mathrm{Spec}(f):\mathrm{Spec}_{Zar}(S)\rightarrow\mathrm{Spec}_{Zar}(R))\}$ be a flat morphism of schemes which restricts to a surjective map $\mathrm{MaxSpec}_{Zar}(S)\rightarrow\mathrm{MaxSpec}_{Zar}(R)$. Then $\{\mathrm{Spec}_{Zar}(S)\rightarrow\mathrm{Spec}_{Zar}(R))\}$ is a faithfully flat cover, i.e., $\mathrm{Spec}_{Zar}(S)\rightarrow\mathrm{Spec}_{Zar}(R)$ is surjective. It suffices to prove that the flat morphism $R\rightarrow S$ is faithfully flat. Let $\mathfrak{m}$ be a maximal ideal of $R$.  Thus there is a $\mathfrak{n}\in\mathrm{MaxSpec}_{Zar}(S)_{Ber}$ such that $\mathrm{Spec}(f)(\mathfrak{n})=\mathfrak{m}$. Then $S\otimes_{R}R\big\slash\mathfrak{n}\cong S\big\slash\mathfrak{n}\neq0$.
\end{example}
}

\comment{

}

\subsubsection{The locale of subschemes}

Now we return to locales defined by schemes. 

   Let $X$ be a qcqs scheme. Note that, in this case a base for the locale $\mathrm{Sub}_{\mathbf{G}}(X)$ is given by finite unions of affine $\mathbf{G}$-subspaces. We call such finite unions \textit{special sub-spaces} The set $\mathrm{Special}_{\mathbf{G}}(X)$ of these special subspaces forms a distributive sub-lattice of $\mathrm{Sub}_{\mathbf{G}}(X)$. This sub-lattice consists of compact objects by .Proposition \ref{prop:qcqscompact}. In particular, $\mathrm{Sub}_{\mathbf{G}}(X)$ is a coherent, and hence spatial, locale.

\comment{
\begin{definition}
\begin{enumerate}
    \item 
        An object $X$ of $\mathbf{Stk}(\mathbf{Aff}_{\mathbf{C}},\tau)$ is said to be a \textit{derived} $\mathbf{G}$-\textit{scheme} if it is a $0$-scheme.
    \item 
        An object $X$ of $\mathbf{Stk}(\mathbf{Aff}_{\mathbf{C}},\tau)$ is said to be a \textit{derived} $\mathbf{G}$-\textit{space} if it is a $1$-scheme.
        \item 
        Let $f:U\rightarrow X$ be a morphism in $\mathbf{Stk}(\mathbf{Aff}_{\mathbf{C}},\tau)$. The map $f$ is said to be a \textit{derived} $\mathbf{G}$-\textit{subspace} if
        \begin{enumerate}
            \item 
            it is a homotopy monomorphism;
            \item
            for any $Z\rightarrow X$ with $Z$ in $\mathbf{A}$, the fibre product $Z\times_{X}U$ is a derived $\mathbf{G}$-space, and the morphism $f':Z\times_{X}U\rightarrow Z$ is in $\mathbf{P}$.
        \end{enumerate}
\end{enumerate}
\end{definition}
}

    As in \cite{soor2024derived}*{Remark 2.28} we have the following.

\begin{remark}
    \begin{enumerate}
        \item 
        $\mathrm{pt}_{\mathbf{G}}(X)$ is sober.
        \item 
        $\mathrm{pt}(\mathrm{Spec}(A))_{\mathbf{G}}$ is spectral. 
    \end{enumerate}
\end{remark}

\begin{definition}
    A tuple $\mathbf{G}$ is said to be \textit{Berkovich} if $\mathrm{Special}_{\mathbf{G}}(\mathrm{Spec}(A))$ is a Berkovich pre-locale for any affine $\mathrm{Spec}(A)$.
\end{definition}

The following is an immediate consequence of Proposition \ref{prop:localberk}.

\begin{proposition}
    Let $\mathbf{G}$ be Berkovich, and let $X$ be a qcqs scheme. Then $\mathrm{Sub}_{\mathbf{G}}(X)$ is a Berkovich pre-locale. 
\end{proposition}

The next result directly follows from Proposition \ref{prop:spatlocaleop} and Proposition \ref{prop:Berkclosed}.

\begin{proposition}
    Let $h:V\rightarrow W$ be an open immersion of $\mathbf{G}$-spaces. 
    \begin{enumerate}
        \item 
        The map $|h|:|V|\rightarrow |W|$ is an open immersion.
        \item 
        If $\mathbf{G}$ is Berkovich and the spaces are qcqs, then $|h|_{Ber}$ is a closed immersion.
    \end{enumerate}
\end{proposition}

\begin{proposition}
    Let $g:X\rightarrow Y$ be a morphism of qcqs $\mathbf{G}$-spaces, and $f:U\rightarrow Y$ an open immersion of qcqs $\mathbf{G}$-spaces. 
    
    \begin{enumerate}
        \item
        The natural morphism
    $$\pi:|U\times_{Y}X|\rightarrow |U|\times_{|Y|}|X|$$
    is a homeomorphism.
    \item 
  The map
       $$\pi:|U\times_{Y}X|_{Ber}\rightarrow |U|_{Ber}\times_{|Y|_{Ber}}|X|_{Ber}$$
       is a closed immersion. 
    \end{enumerate}
\end{proposition}

\begin{proof}
\begin{enumerate}
    \item 
    Denote by $f':U\times_{Y}X\rightarrow X$ and $g':U\times_{Y}X\rightarrow U$ the pullback maps.
First observe that $|U\times_{Y}X|\rightarrow |X|$ is an open immersion. This factors as $|U\times_{Y}X|\rightarrow |U|\times_{|Y|}|X|\rightarrow |X|$. The map $|U|\times_{|Y|}|X|\rightarrow |X|$ is also an open immersion. It follows easily that $|U\times_{Y}X|\rightarrow |U|\times_{|Y|}|X|$ is an open immersion. It remains to prove that $\pi$ is surjective.

Let $\mathcal{F}_{U}$ and $\mathcal{F}_{X}$ be completely prime filters such that $f(\mathcal{F}_{U})=g(\mathcal{F}_{X})$.  Consider $\mathcal{F}=\{V\rightarrow U\times_{Y}X:V\rightarrow X\in\mathcal{F}_{X}\}$. Since $U\times_{Y}X\rightarrow X$ is an open immersion, this $\mathcal{F}$ well-defined completely prime filter as long as it is non-empty. Clearly $U\rightarrow Y\in f(\mathcal{F}_{U})$. Therefore $U\times_{Y}X\rightarrow X\in\mathcal{F}_{X}$. But $U\times_{Y}X\rightarrow X$ factors through $U\rightarrow X$, and this suffices to prove the claim. Note that also $\mathcal{F}_{X}$ contains $U\rightarrow X$.

We claim that $f'(\mathcal{F})=\mathcal{F}_{X}$ and $g'(\mathcal{F})=\mathcal{F}_{U}$.

If $S\rightarrow Y$ is an open then 
$$S\times_{Y}X\rightarrow X\in\mathcal{F}_{X}\Leftrightarrow U\times_{Y}S\rightarrow U\in\mathcal{F}_{U}.$$
Now for any open immersion $S\rightarrow U$ we have $S\times_{X}U\cong S$. It follows that $\mathcal{F}_{U}=\{S\rightarrow U:S\times_{Y}X\rightarrow X\in\mathcal{F}_{X}\}$. 

Now 
\begin{align*}
    g'(\mathcal{F})&=\{S\rightarrow U: U\times_{Y}X\times_{U}S\cong S\times_{Y}X\rightarrow U\times_{Y}X\in\mathcal{F}\}\\
    &=\{S\rightarrow U: S\times_{Y}X\rightarrow X\in\mathcal{F}_{X}\}\\
    &=\mathcal{F}_{U}.
\end{align*}

We also have
\begin{align*}
    f'(\mathcal{F})&=\{S\rightarrow X:S\times_{X}(U\times_{Y}X)\rightarrow U\times_{Y}X\in\mathcal{F}\}\\
    &=\{S\rightarrow X:S\times_{X}(U\times_{Y}X)\rightarrow X\in\mathcal{F}_{X}\}.
\end{align*}

Now if $V\rightarrow U\times_{Y}X\in\mathcal{F}$ then $V\rightarrow X\in\mathcal{F}_{X}$, and since $f':U\times_{Y}X\rightarrow X$ is an open immersion, $V\times_{X}(U\times_{Y}X)\cong V$. Thus $V\rightarrow X\in\mathcal{F}_{X}$, so $f'(\mathcal{F})\subseteq\mathcal{F}_{X}$. Let $W\rightarrow X\in\mathcal{F}_{X}$. Since $U\rightarrow X\in\mathcal{F}_{X}$, $W\times_{X}U\rightarrow X\in\mathcal{F}_{X}$. Thus $W\times_{X}U\rightarrow X\times_{Y}U\in\mathcal{F}$. This completes the proof.
\item 
First observe that $|U\times_{Y}X|_{Ber}\rightarrow |X|_{Ber}$ is a closed immersion. This factors as $|U\times_{Y}X|_{Ber}\rightarrow |U|_{Ber}\times_{|Y|_{Ber}}|X|_{Ber}\rightarrow |X|_{Ber}$. The map $|U|_{Ber}\times_{|Y|_{Ber}}|X|_{Ber}\rightarrow |X|_{Ber}$ is also a closed immersion. It follows easily that $|U\times_{Y}X|_{Ber}\rightarrow |U|_{Ber}\times_{|Y|_{Ber}}|X|_{Ber}$ is a closed immersion. 
\end{enumerate}

\end{proof}

\begin{definition}
    The tuple $\mathbf{G}$ is said to be \textit{stably Berkovich}  if for $g:X\rightarrow Y$ any morphism of qcqs $\mathbf{G}$-spaces, and $f:U\rightarrow Y$ any open immersion $f:U\rightarrow Y$ of qcqs $\mathbf{G}$-spaces, the morphism
     $$\pi:|U\times_{Y}X|_{Ber}\rightarrow |U|_{Ber}\times_{|Y|_{Ber}}|X|_{Ber}$$
     is a homeomorphism. 
\end{definition}

\subsubsection{Immersions}
Here we define some class of immersions, and study their properties.

\begin{definition}
    A morphism $f:X\rightarrow Y$ is said to be a \textit{local open immersion} if we can write $X=\bigcup_{i\in\mathcal{I}}U_{i}$, with $U_{i}\rightarrow X$ an open immersion, and 
    $$f|_{U_{i}}:U_{i}\rightarrow Y$$
    is an open immersion.
\end{definition}

\begin{proposition}\label{prop:locopimm}
    Let $f:X\rightarrow Y$ be a locally open immersion. It is an open immersion if and only if $|f|:|X|\rightarrow |Y|$ is injective.
\end{proposition}

\begin{proof}
    Clearly if $f$ is an open immersion then $|f|$ is injective. Conversely, suppose that $|f|$ is injective. We need to prove that the map $X\rightarrow X\times_{Y}X$ is an equivalence. Write $X=\bigcup U_{i}$ where each $U_{i}\rightarrow Y$ is an open immersion. It is enough to prove that $U_{i}\times_{Y}X\rightarrow U_{i}$ is an equivalence.  Note that $U_{i}\times_{Y}X\rightarrow X$ and $U_{i}\rightarrow X$ are open immersions. Thus it is enough to prove that $|U_{i}|\cong |U_{i}\times_{X}Y|$ as sets. But we have $ |U_{i}\times_{X}Y|\cong |U_{i}|\times_{|Y|}|X|$. The claim now follows immediately from the fact that $|X|\rightarrow |Y|$ is injective.
\end{proof}

\begin{definition}[c.f. \cite{soor2024derived}*{Lemma 2.34}]
    A morphism $f:X\rightarrow Y$ in $\mathbf{dSch}_{\mathbf{G}}$ is said to be an \textit{algebraically closed immersion} (resp. a \textit{finitely presented algebraically closed immersion}) if there exists a covering $\{U_{i}\rightarrow Y\}_{i\in\mathcal{I}}$ of $Y$ by subspaces $U_{i}\cong\mathrm{Spec}(A_{i})\in\mathbf{A}$, such that 
            for each $i\in\mathcal{I}$, the pullback $Y\times_{X}U_{i}$ is represented by some $\mathrm{Spec}(B_{i})\in\mathbf{A}$, and the induced morphism $A_{i}\rightarrow B_{i}$ is a surjection on $\pi_{0}$ (resp. a surjection on $\pi_{0}$ with finitely generated kernel).
           It is said to be a \textit{locally closed immersion} (resp. a \textit{finitely presented locally closed immersion}) if in addition $|f|:|X|\rightarrow |Y|$ is a locally closed immersion of spaces.
 \end{definition}

 \begin{definition}
     A map $e:X\rightarrow Y$ in $\mathbf{dSch}_{\mathbf{G}}$ is said to be a \textit{Zariski open immersion} if there is a closed immersion $j:Z\rightarrow Y$ such that $|e|$ determines a homeomorphism $|X|\cong |Y|\setminus |Z|$.
 \end{definition}

\subsubsection{Underlying spaces and truncations}

Now, let us impose the following conditions on 
$$\mathbf{G}=(\mathbf{Aff}_{\mathbf{C}},\tau,\mathbf{P},\mathbf{A}).$$
Let $\mathbf{A}^{\heart}$ be a class of discrete objects in $\mathbf{Aff}_{\mathbf{C}}$. That is, each $\mathrm{Spec}(A)\in\mathbf{A}^{\heart}$ is such that $A\in\mathrm{Comm}(\mathbf{C}^{\heart})$.  For each $\mathrm{Spec}(A)\in\mathbf{A}^{\heart}$ fix a full subcategory $\mathbf{F}_{A}\subset\mathbf{Mod}(A)^{\heart}$. Finally, let $\mathbf{A}\subset\mathbf{Aff}_{\mathbf{C}}$ be the class of objects $\mathrm{Spec}(C)$ such that
\begin{enumerate}
    \item 
    $\mathrm{Spec}(\pi_{0}(C))\in\mathbf{A}$;
    \item 
    each $\pi_{n}(C)\in\mathbf{F}_{\pi_{0}(C)}$.
\end{enumerate}

We assume that:

\begin{enumerate}
    \item 
     if $f:B\rightarrow A$ is a morphism in $\mathbf{DAlg}^{cn}(\mathbf{C})$ with $\mathrm{Spec}(A)\in\mathbf{A}^{\heart}$, and $\mathrm{Spec}(f)\in\mathbf{P}$, then $\mathrm{Spec}(B)\in\mathbf{A}^{\heart}$;
     \item 
     if $f:B\rightarrow A$ is a morphism in $\mathbf{DAlg}^{cn}(\mathbf{C})$ with $\mathrm{Spec}(A)\in\mathbf{A}^{\heart}$, $\mathrm{Spec}(f)\in\mathbf{P}$, and $M\in\mathbf{F}_{A}$, then $B\otimes_{A}^{\mathbb{L}}M\in\mathbf{F}_{B}$ (in particular we are assuming that maps in $\mathbf{P}$ between objects in $\mathbf{A}^{\heart}$ are transverse to objects in $\mathbf{F}_{A}$);
     \item 
     a map $\mathrm{Spec}(f):\mathrm{Spec}(A)\rightarrow\mathrm{Spec}(B)$ in between objects in $\mathbf{A}$ is in $\mathbf{P}$ if and only if
     \begin{enumerate}
         \item 
         the map $\mathrm{spec}(\pi_{0}(f))$ is in $\mathbf{P}$,
         \item 
         the map $f$ is derived strong.
     \end{enumerate}
\end{enumerate}

Note that in this case we have two geometry contexts

$$\mathbf{G}=(\mathbf{Aff}_{\mathbf{C}},\tau,\mathbf{P},\mathbf{A})$$
and
$$\mathbf{G}^{\heart}=(\mathbf{Aff}_{\mathbf{C}},\tau,\mathbf{P},\mathbf{A}^{\heart}).$$

A $\mathbf{G}^{\heart}$-scheme/ space will be called a \textit{classical} $\mathbf{G}$-\textit{scheme/space}. This is a full subcategory of $\mathbf{dSch}_{\mathbf{G}}$, which we denote by $\mathbf{Sp}_{\mathbf{G}}$.
Exactly as in \cite{soor2024derived}*{Lemma 2.27} we have the following.

\begin{lemma}
    The inclusion $\mathbf{Sp}_{\mathbf{G}}\rightarrow\mathbf{dSch}_{\mathbf{G}}$ admits a right adjoint $X\mapsto X_{0}$ extending the functor $\mathrm{Spec}(A)\mapsto\mathrm{Spec}(\pi_{0}(A))$, for $\mathrm{Spec}(A)\in\mathbf{G}$ Moreover, $Y\rightarrow X$ is a subspace if and only if $Y_{0}\rightarrow X_{0}$ is a subspace.
\end{lemma}

 \begin{lemma}
     Let \(X \in \mathbf{dSch}_{\mathbf{G}}\), and let \(V \subset \abs{X}\) be an open subset. Then there exists a subspace \(U \subset X\) such that \(\abs{U} = V\) 
 \end{lemma}
 \begin{proof}
     The proofs of \cite{soor2024derived}*{Corollary 2.31, Theorem 2.32} go through mutatis-mutandis. 
 \end{proof}

\comment{
Let $g:X\rightarrow Y$ be a morphism of derived $\mathbf{G}$-spaces. Any completely prime filter $\mathcal{F}$ on $(\mathbf{dSch}_{\mathbf{G}})_{\big\slash Y}$ can be described as $\mathcal{F}_{y}=\{U\rightarrow Y:y\in |U|\}$ for some point $y\in |Y|$. Similarly any completely prime filter on $(\mathbf{dSch}_{\mathbf{G}})_{\big\slash X}$ is of the form $\mathcal{F}_{x}$ for some point $x\in |X|$. Moreover we have $|G|(\mathcal{F}_{x})=\mathcal{F}_{|g|(x)}$.
}

\subsubsection{Spatial morphisms}

Let us introduce a class of maps generalising open immersions.

\begin{definition}
Let $X,Y$ qcqs schemes.
    A morphism $f:X\rightarrow Y$  is said to be
    \begin{enumerate}
    \item
    \textit{spatial} if for any other morphism $Z\rightarrow Y$, the canonical map 
    $$|X\times_{Y}Z|\rightarrow |X|\times_{|Y|}|Z|$$
    is a homeomorphism;
        \item 
            \textit{Berkovich spatial} if for any other morphism $Z\rightarrow Y$, the canonical map 
    $$|X\times_{Y}Z|_{Ber}\rightarrow |X|_{Ber}\times_{|Y|_{Ber}}|Z|_{Ber}$$
    is a homeomorphism.
    \end{enumerate}

\end{definition}

\begin{proposition}
    Let $L$ be a lattice, and $\{U_{i}:i\in\mathcal{I}\}$ a cover of $L$. Then $|L|=\bigcup_{i\in\mathcal{I}}|U_{i}|$. 
\end{proposition}

\begin{proof}
    This is somewhat obvious. Certainly $\bigcup_{i\in\mathcal{I}}|U_{i}|\subset |L|$ is a subspace. Let $\mathcal{F}\in |L|$. Let $U\in\mathcal{F}$. Then $U\subset\bigcup_{i\in\mathcal{I}}U_{i}$. Thus $\bigcup_{i\in\mathcal{I}}U_{i}\in\mathcal{F}$. By complete primality, some $U_{i}\in\mathcal{F}$. Thus $\mathcal{F}\in U_{i}$.
\end{proof}

The next result is essentially identical.

\begin{proposition}
    Let $L$ be a Berkovich pre-locale, and let $\{U_{i}:i\in\mathcal{I}\}$ be a finite cover of $L$ by quasi-compact elements. Then $|L|_{Ber}=\bigcup_{i\in\mathcal{I}}|U_{i}|_{Ber}$.
\end{proposition}

The following is an immediate consequence of Proposition \ref{prop:localberk}.

\begin{proposition}
    Suppose that affine derived $\mathbf{G}$-spaces $\mathrm{Spec}(A)$ are such that $\mathrm{Special}_{\mathbf{G}}(\mathrm{Spec}(A))$ is a Berkovich pre-locale. Then any qcqs derived $\mathbf{G}$-space $X$ is such that $\mathrm{Special}_{\mathbf{G}}(X)$ is a Berkovich pre-locale.
\end{proposition}

\begin{corollary}
    Spatiality and Berkovich spatiality of morphisms are local conditions. Precisely:
    \begin{enumerate}
        \item 
        Let $f:X\rightarrow Y$ be a morphism of schemes. If $\{U_{i}\rightarrow X\}_{i\in\mathcal{I}}$ is a cover of $X$ by subschemes, and each $U_{i}\rightarrow X\rightarrow Y$ is spatial, then $f$ is spatial.
        \item 
        Let $f:X\rightarrow Y$ be a morphism of qcqs schemes. If $\mathbf{G}$ is stably Berkovich and $\{U_{i}\rightarrow X\}_{i\in\mathcal{I}}$ is a finite cover of $X$ by quasi-compact subschemes, and each $U_{i}\rightarrow X\rightarrow Y$ is Berkovich spatial, then $f$ is Berkovich spatial.
    \end{enumerate}
\end{corollary}


\begin{definition}
    Let $f:X\rightarrow Y$ be a morphism and $U\subset Y$ an open immersion. Then $f$ is said to be 
    
    \begin{enumerate}
    \item
      \textit{spatial on the complement of }$U$ if for every morphism $Z\rightarrow Y$, the natural map
    $$|Z\times_{Y}X|\setminus |Z\times_{Y}X\times_{Y}U|\rightarrow |Z|\times_{|Y|}|X|$$
  is a homeomorphism to

   $$  |Z|\times_{|Y|}|X|\setminus (|Z|\times_{|Y|}|X|\times_{|Y|}|U|).$$
    \item 
    \textit{Berkovich spatial on the complement of }$U$ if for every morphism $Z\rightarrow Y$, the natural map
    $$|Z\times_{Y}X|_{Ber}\setminus |Z\times_{Y}X\times_{Y}U|_{Ber}\rightarrow |Z|_{Ber}\times_{|Y|_{Ber}}|X|_{Ber}$$
  is a homeomorphism to

   $$  |Z|_{Ber}\times_{|Y|_{Ber}}|X|_{Ber}\setminus (|Z|_{Ber}\times_{|Y|_{Ber}}|X|_{Ber}\times_{|Y|_{Ber}}|U|_{Ber}).$$
   \end{enumerate}
\end{definition}

\begin{remark}\label{rem:absolute}
The case when $Z=Y$ (and, in the Berkovich case when $\mathbf{G}$ is stably Berkovich) is automatic. In fact this says that $|X|\setminus|X\times_{Y}U|\cong|X|\setminus |X|\times_{|Y|}|U|.$ In particular this amounts to $|X\times_{Y}U|\cong |X|\times_{|Y|}|U|,$ which simply follows from the fact that $U\rightarrow Y$ is an open immersion. 
\end{remark}

\begin{remark}
    If $f:X\rightarrow Y$ is (Berkovich) spatial on the complement of $U$, then for any $Z\rightarrow Y$, $f':X\times_{Y}Z\rightarrow Z$ is (Berkovich) spatial on the complement of $U\times_{Y}Z$.
\end{remark}

\begin{lemma}
    If $f:X\rightarrow Y$ is spatial and $U\rightarrow Y$ any open immersion, then $f$ is spatial on the complement of $U$. If $\mathbf{G}$ is stably Berkovich, $f:X\rightarrow Y$ is Berkovich spatial, and $U\rightarrow Y$ is an open immersion, then $f$ is Berkovich spatial on the complement of $U$. 
\end{lemma}

\begin{proof}
We prove the spatial claim, the Berkovich spatial claim being identical.
    Let $Z\rightarrow Y$ be a morphism. We have $|X\times_{Y}Z|\cong |X|\times_{|Y|}\times|Z|$ and $|Z\times_{Y}X\times_{Y}U|\cong |Z|\times_{|Y|}|X|\times_{|Y|}|U|$. Now the claim is obvious.
\end{proof}

\begin{lemma}\label{lem:andrspat}
 Let $f:X\rightarrow Y$ be a morphism such that 
    $$X= (X\times_{Y}U)\cup V$$
    where 
    \begin{enumerate}
        \item 
        $U\rightarrow Y$ is an open immersion;
        \item 
       each $V\rightarrow X\rightarrow Y$ is an open immersion. 
    \end{enumerate}
     Then $f$ is spatial on the complement of $U$. If $\mathbf{G}$ is stably Berkovich, then $f$ is Berkovich spatial on the complement of $U$.
\end{lemma}

\begin{proof}

Again we prove the spatiality version.
    Let $g:Z\rightarrow Y$ be any morphism, and $f':Z\times_{Y}X\rightarrow Z$ the morphism induced by pullback. We have 
    $$Z\times_{Y}X\cong (Z\times_{Y}X\times_{Y}U)\cup Z\times_{Y}V,$$
    where $Z\times_{Y}V\rightarrow Z$ is an open immersion, and 
    $$(Z\times_{Y}X\times_{Y}U)\cong (Z\times_{Y}X)\times_{Z}(Z\times_{Y}U),$$
    with $Z\times_{Y}U\rightarrow Z$ being an open immersion.
We then have $|Z\times_{Y}X|\cong |Z\times_{Y}X\times_{Y}U|\cup|Z\times_{Y}V|$.
   Then we have

\begin{align*}
       |Z\times_{Y}X|\setminus |Z\times_{Y}X\times_{Y}U| &=|Z\times_{Y}V|\setminus (|Z\times_{Y}X\times_{Y}U|\times_{|Z\times_{Y}X|}|Z\times_{Y}V|)\\
     &  \cong |Z\times_{Y}V|\setminus |(Z\times_{Y}X\times_{Y}U)\times_{Z\times_{Y}X}(Z\times_{Y}V)|\\
     &\cong |Z\times_{Y}V|\setminus |Z\times_{Y}V\times_{Y}U|\\
&\cong(|Z|\times_{|Y|}|V|)\setminus(|Z|\times_{|Y|}|V|\times_{|Y|}|U|),
   \end{align*}
where we have used that $V\rightarrow X$, $V\rightarrow Y$, and $U\rightarrow Y$ are open immersions.
\comment{
    Then we have 
    \begin{multline*}
    |Z\times_{Y}X|\setminus |Z\times_{Y}X\times_{Y}U|=|\tilde{Z}\times_{Y}X|\setminus (|\tilde{Z}|\times_{|Z|} |Z\times_{Y}X\times_{Y}U|) \\
    \cong |\tilde{Z}\times_{Y}X|\setminus |\tilde{Z}\times_{Z}Z\times_{Y}X\times_{Y}U|\cong|\tilde{Z}\times_{Y}X|\setminus |\tilde{Z}\times_{Y}X\times_{Y}U|\\
    \cong|\tilde{Z}|\times_{|Y|}|X|\setminus |\tilde{Z}|\times_{|Y|}|X|\times_{|Y}|U|,
    \end{multline*}
    where we have used that $\tilde{Z}\rightarrow Z$, $\tilde{Z}\rightarrow Y$, and $U\rightarrow Y$ are open immersions. }
On the other hand, a similar (even easier) argument gives
$$ |Z|\times_{|Y|}|X|\setminus |Z|\times_{|Y|}|X|\times_{|Y|}|U|
 \cong|Z|\times_{|Y|}|V|\setminus |Z|\times_{|Y|}|V|\times_{|Y|}|U|.$$

   This completes the proof.
   \end{proof}

Before proceeding, let us introduce one more definition.

\begin{definition}[c.f. \cite{soor2024derived}*{Lemma 2.34}]
    A morphism $f:X\rightarrow Y$ in $\mathbf{dSp}_{\mathbf{G}}$ is said to be an \textit{algebraically closed immersion} if there exists a covering $\{U_{i}\rightarrow Y\}_{i\in\mathcal{I}}$ of $Y$ by subspaces $U_{i}\cong\mathrm{Spec}(A_{i})\in\mathbf{A}$, such that 
            for each $i\in\mathcal{I}$, the pullback $Y\times_{X}U_{i}$ is represented by some $\mathrm{Spec}(B_{i})\in\mathbf{A}$, and the induced morphism $A_{i}\rightarrow B_{i}$ is a surjection on $\pi_{0}$.
           It is said to be a \textit{closed immersion} if in addition $|f|:|X|\rightarrow |Y|$ is closed immersion of spaces.
 \end{definition}

Below we shall show that in many instances, closed immersions are spatial.

\subsubsection{Derived $\mathrm{T}$-analytic subspaces}\label{subsubsec:Tansub}

Let us here specialise again. Let $\mathrm{T}$ be a $\underline{\Lambda}$-generating class of algebras of homotopy polynomial type. Denote the monoidal unit in $\underline{\mathbf{C}}$ by $\mathbb{I}$. Let $p_{\lambda_{1},\ldots,\lambda_{n}}:\mathrm{Sym}(\mathbb{I}^{\oplus n})\rightarrow\mathrm{T}(\lambda_{1},\ldots,\lambda_{n})$ denote the morphisms realising $\mathrm{T}$ as a class of homotopy polynomial type. For an object $A\in\mathrm{Comm}(\mathbf{C}^{\heart})$ we write $\underline{A}=\mathrm{Hom}(\mathbb{I},A)$ for the \textit{underlying algebra of} $A$.

We fix a category $\mathbf{A}_{\mathrm{T}}$ of so-called $\mathrm{T}$-\textit{analytic affinoids}, and let $\mathbf{A}_{\mathrm{T}}^{\heart}$ be the full subcategory of $\mathbf{A}_{\mathrm{T}}$ consisting of those objects $\mathrm{Spec}(A)$ with $A$ in the heart of the $t$-structure on $\mathbf{C}$. For $\tau$ we fix the $\mathrm{T}$-rational topology. 

\begin{definition}
    $\mathbf{A}_{\mathrm{T}}$ is said to be \textit{analytic} if the following conditions hold.
    \begin{enumerate}
    \item 
    There is a consistent system of modules $\mathbf{F}_{A}\subset\mathrm{Mod}(A)$ for $\mathrm{Spec}(A)\in\mathbf{A}_{\mathrm{T}}^{\heart}$ (i.e., whenever $\mathrm{Spec}(B)\rightarrow\mathrm{Spec}(A)$ is a map in $\mathbf{P}\cap\mathbf{A}_{\mathrm{T}}^{\heart}$, and $M\in\mathbf{F}_{A}$, then $B\otimes_{A}^{\mathbb{L}}M\in\mathbf{F}_{B}$), such that a finite collection of rational localisations $\{\mathrm{Spec}(A_{i})\rightarrow\mathrm{Spec}(A)\}_{i\in\mathcal{I}}$ between objects in $\mathbf{A}^{\heart}_{\mathrm{T}}$ is a cover if and only if the restriction of the functors $A_{i}\otimes_{A}(-):\mathrm{Mod}(A)\rightarrow\mathrm{Mod}(A_{i})$ to $\mathbf{F}_{A}$ is a jointly conservative collection.
        \item 
        For any object $\mathrm{Spec}(A)$ of $\mathbf{A}_{\mathrm{T}}^{\heart}$ and any $M$ of $\mathbf{F}_{A}$, the natural morphism
     $$\underline{M}[x_{1},\ldots,x_{k}]\rightarrow \underline{M}[[x_{1},\ldots,x_{n}]]$$
     factors as 
    $$\underline{M}[x_{1},\ldots,x_{k}]\rightarrow\underline{M\otimes\mathrm{T}(\lambda_{1},\ldots,\lambda_{k})}\rightarrow \underline{M}[[x_{1},\ldots,x_{n}]],$$
    where the first map is $\mathrm{Id}_{M}\otimes p_{\lambda_{1},\ldots,\lambda_{k}}$, and the second map is a monomorphism.
       \item 
   Consider any $\mathrm{T}(\lambda)$. The \textit{multiplication by} $z$ morphism
   $$z\times:\mathrm{T}(\lambda)\rightarrow\mathrm{T}(\lambda)$$
   has cokernel isomorphic to $\mathbb{I}$.
    \end{enumerate}
    If objects of objects of $\mathbf{F}_{A}$ are transverse to $\mathrm{T}$-\'{e}tale localisations, then we say that $\mathbf{A}_{\mathrm{T}}$ is \'{e}tale analytic.
\end{definition}

The second condition in particular says that we may regard $C\otimes\mathrm{T}(\lambda_{1},\ldots,\lambda_{k})$ as a power series algebra. The third condition gives a canonical augmentation.

\begin{example}
    Let $R$ be a unital commutative ring, and let $\mathrm{T}$ be given by $\mathrm{T}(n)=R[x_{1},\ldots,x_{n}]$. Let $\mathbf{A}_{\mathrm{T}}$ be the category of all derived $R$-algebras. For $A$ a discrete algebra, let $\mathbf{F}_{A}$ denote the category of all $A$-modules. The first two conditions clearly hold. Now, a rational localisation of underived $R$-algebras is in particular \'{e}tale of finite presentation. In particular it is flat. Hence a rational localisation is a flat homotopy epimorphism, and can therefore be refined by Zariski localisations. It follows that the derived rational localisation topology is just the usual derived Zariski topology. Now a finite collection of Zariski open immersions $\{\mathrm{Spec}(A_{i})\rightarrow\mathrm{Spec}(A)\}$ of discrete $R$-algebras is a cover precisely if it is a conservative collection. 
\end{example}

\begin{example}
    Let $K$ be a non-trivially valued, non-Archimedean Banach field. Let $\mathbf{C}$ be the derived context of complete bornological $K$-modules (Section \ref{sec:bornologies}), and let $\mathrm{T}$ be given by $\mathrm{T}(\lambda_{1},\ldots,\lambda_{n})=K<\frac{x_{1}}{\lambda_{1}},\ldots,\frac{x_{n}}{\lambda_{n}}>$, the Tate algebra of radius $(\lambda_{1},\ldots,\lambda_{n})$. Here each $\lambda_{i}\in |K|\cap\mathbb{R}_{>0}$, where $|K|$ is the image of $|-|:K\rightarrow\mathbb{R}_{\ge0}$. Condition $(3)$ is clearly satisfied. Let $\mathbf{A}_{\mathrm{T}}$ denote the class of objects $\mathrm{Spec}(A)$ such that $\pi_{0}(A)$ is of the form $K<\frac{x_{1}}{\lambda_{1}},\ldots,\frac{x_{n}}{\lambda_{n}}>\big\slash I$ for a (necessarily closed and finitely generated) ideal $I$, and $\pi_{n}(A)$ is a finitely generated $\pi_{0}(A)$-module. As we shall see, $\mathrm{T}$-\'{e}tale maps of discrete objects of $\mathbf{A}^{\heart}_{\mathrm{T}}$ are just \'{e}tale maps of rigid affinoids in the usual sense. Let $\tau$ be the $\mathrm{T}$-rational topology. For $\mathrm{Spec}(A)\in\mathbf{A}_{\mathrm{T}}^{\heart}$, let $\mathbf{F}_{\mathrm{A}}$ denote the class of all complete bornological $A$-modules which are transverse to $\mathrm{T}$-\'{e}tale morphisms. This class of modules clearly satisfies condition $(4)$. Condition $(2)$ is satisfied because the object $K<\frac{x_{1}}{\lambda_{1}},\ldots,\frac{x_{n}}{\lambda_{n}}>$ is strongly flat (i.e., tensoring commutes with kernels). Condition $(1)$ is a consequence of the proof of Theorem \ref{thm:refetalecoverequivalent}, and the fact that maps $A\rightarrow k$ with $k$ a non-trivial complete valued are transverse to $\mathrm{T}$-\'{e}tale morphisms, which follows from Lemma \ref{lem:discetaletrans}. If we were to restrict to localisations instead of worrying about \'{e}tale morphisms, we could instead have used \cite{MR3626003}*{Theorem 5.39}.
\end{example}

\begin{example}
    Let $\mathbf{C}$ be the derived context of complete bornological $\mathbb{C}$-modules (Section \ref{sec:bornologies}). For $(\underline{r})\in\mathbb{R}_{>0}$ let $\mathrm{T}(\underline{r})=W_n(\underline{r})^\dagger$ be the ring of overconvergent power series (Subsection \ref{subsec:overconv}). Let $\mathbf{A}_{\mathrm{T}}$ consist of objects $\mathrm{Spec}(A)$ with $\pi_{0}(A)$ of the form $\mathrm{T}(\underline{r})=W_n(\underline{r})^\dagger\big\slash I$, where $I$ is a (necessarily closed and finitely generated) ideal, and $\pi_{n}(A)$ is finitely presented over $\pi_{0}(A)$.
    Property $(3)$ is clear. For $\mathrm{Spec}(A)\in\mathbf{A}_{\mathrm{T}}$, let $\mathbf{F}_{A}$ denote the class of complete bornological $A$-modules which are transverse to rational localisations, so that Property $(4)$ is satisfied. Property $(2)$ is satisfied because the object $W_n(\underline{r})^\dagger$ is strongly flat (i.e., tensoring commutes with kernels). The joint conservativity condition follows immediately from \cite{bambozzi2016dagger}*{Theorem 5.15}. 
\end{example}

\begin{lemma}[c.f. \cite{soor2024derived}*{Lemma 2.30}]
    Let $X$ be a derived $\mathrm{T}$-analytic space. Let $V\subset t_{\le0}X_{0}$ be open immersion. Then there exists an open immersion $U\subset X$ (resp. an \'{e}tale morphism $U\rightarrow X$) such that $t_{\le0}U_{0}=V$. Moreover $U$ is unique up to equivalence. If $\mathbf{A}_{\mathrm{T}}$ is \'{e}tale analytic this is also true for \'{e}tale morphisms.
\end{lemma}

\begin{proof}
We prove the open immersion case, the \'{e}tale case being similar. Since the questions is local, we may assume that $X=\mathrm{Spec}(A)$, $V=\mathrm{Spec}(\tilde{B})$, and $\pi_{0}(A)\rightarrow\tilde{B}$ is a rational localisation. 
  We may write \begin{multline*}
    \tilde{B}\cong\pi_{0}(A)\otimes^{\mathbb{L}}\mathrm{T}(\lambda_{1},\ldots,\lambda_{n})\big\slash\big\slash(gx_{1}-f_{1},\ldots,gx_{n}-f_{n}) \\
    \cong\pi_{0}(A)\otimes\mathrm{T}(\lambda_{1},\ldots,\lambda_{n})\big\slash(gx_{1}-f_{1},\ldots,gx_{n}-f_{n})\end{multline*} where $(g,f_{1},\ldots,f_{n})$ generate the unit ideal. Put $B\defeq A\otimes^{\mathbb{L}}\mathrm{T}(\lambda_{1},\ldots,\lambda_{n})\big\slash\big\slash(gx_{1}-f_{1},\ldots,gx_{n}-f_{n})$. This is a derived rational localisation such that $\pi_{0}(B)\cong\tilde{B}$. Let $A\rightarrow B'$ be another rational localisation such that $\pi_{0}(A\rightarrow B')\cong\pi_{0}(A)\rightarrow\tilde{B}$. Now both $B$ and $B'$ are formally \'{e}tale over $A$. Thus the isomorphism $\pi_{0}(B)\cong\tilde{B}\cong\pi_{0}(B')$ lifts, uniquely up to a contractible choice, to an equivalence $B\cong B'$.
\end{proof}

\begin{corollary}
    The morphism of spaces $|X_{0}|\rightarrow |X|$ is a homeomorphism.
\end{corollary}

\begin{lemma}\label{lem:closedimlift}
Let $f:X\rightarrow Y$ be a closed immersion. For any open immersion $U\rightarrow X$, there exists an open immersion $V\rightarrow Y$ such that $U\cong V\times_{X}Y$. 
\end{lemma}

\begin{proof}
    We may assume that $X\cong\mathrm{Spec}(B)$ $Y\cong\mathrm{Spec}(A)$, and that they are both discrete. Then $B\cong A\big\slash(f_{1},\ldots,f_{n})$. Moreover we may assume that $U\cong\mathrm{Spec}(C)$ where $B\rightarrow C$ is a $\mathrm{T}$-rational localisation. Write $C\cong B\otimes T(\lambda_{1},\ldots,\lambda_{m})\big\slash(g_{0}x_{1}-g_{1},\ldots,g_{0}x_{m}-g_{m})$ with $(g_{0},\ldots,g_{m})$ generating the unit ideal in $B$. For each $0\le i\le m$ pick some preimage $\tilde{g}_{i}$ of $g_{i}$ in $A$. Then $(\tilde{g}_{0},\ldots,\tilde{g}_{m},f_{1},\ldots,f_{n})$ generate the unit ideal of $A$. Write $$\tilde{C}=A\otimes\mathrm{T}(\lambda_{1},\ldots,\lambda_{m},\gamma_{1},\ldots,\gamma_{n})\big\slash(\tilde{g}_{0}x_{1}-\tilde{g}_{1},\ldots,\tilde{g}_{0}x_{m}-\tilde{g}_{m},\tilde{g}_{0}z_{1}-f_{1},\ldots,\tilde{g}_{0}z_{n}-f_{n}).$$
    This is a rational localisation. Base-changing to $B$ gives the rational localisation
    $$B\otimes\mathrm{T}(\lambda_{1},\ldots,\lambda_{m},\gamma_{1},\ldots,\gamma_{n})\big\slash(g_{0}x_{1}-g_{1},\ldots,g_{0}x_{m}-g_{m},g_{0}z_{1},\ldots,g_{0}z_{n}).$$
    This is isomorphic to 
      $$C\otimes T(\gamma_{1},\ldots,\gamma_{n})\big\slash(g_{0}z_{1},\ldots,g_{0}z_{n}).$$
      We claim that $(g_{0}z_{1},\ldots,g_{0}z_{n})=(z_{1},\ldots,z_{n})$. Let $I$ denote the ideal on the left-hand side of the equation, and $I'$ the ideal on the right-hand side. Clearly $I\subset I'$. We have $g_{0}z_{j}\in I$ for all $j$. Then $[x_{i}]g_{0}z_{j}=g_{i}z_{j}$ for each $i$. Since we can write $1$ as a linear combination of the $g_{i}$s, we get $z_{j}\in I$ for all $j$. Thus $I'\subset I$, as required. 
\end{proof}

Let $f:X\rightarrow Y$ be a morphism of $\mathrm{T}$-analytic spaces. The morphism of frames
$$\pi_{f}:\mathrm{An_{T}}(Y)\rightarrow\mathrm{An_{T}}(X)$$
has a right adjoint given by

$$\pi_{f}^{*}(W)=\bigcup_{\{V\rightarrow Y:V\times_{Y}X\rightarrow W\}}V.$$
We get a nucleus $\pi_{f}^{*}\circ\pi_{f}$. The fixed points of this nuclear consist precisely of those $V\rightarrow Y$ such that whenever $U\rightarrow Y$ is such that $U\times_{Y}X\rightarrow V\times_{Y}X$, then $U\rightarrow V$.

\begin{theorem}
Let $X=\mathrm{Spec}(A\big\slash I)$ and $Y=\mathrm{Spec}(A)$ with $I$ a finitely generated ideal, and $A$ discrete. Consider the map $g:X\rightarrow Y$ corresponding to the quotient map $A\rightarrow A\big\slash I$. Then
    $\mathrm{An}_{\mathrm{T}}(X)$ is a closed sub-locale of $\mathrm{An}_{\mathrm{T}}(Y)$. In particular $f$ is a closed immersion.
\end{theorem}

\begin{proof}
We may assume that $I=(f)$ is principle.
    The morphism of frames $\pi_{X}:\mathrm{An}_{\mathrm{T}}(Y)\rightarrow\mathrm{An}_{\mathrm{T}}(X)$ is a surjection by Lemma \ref{lem:closedimlift}. This proves that $\mathrm{An}_{\mathrm{T}}(X)$ is a sub-locale of $\mathrm{An}_{\mathrm{T}}(Y)$.

Now $\pi_{X}$ has a right adjoint, $\pi_{X}^{*}$, given by 
$$\pi_{X}^{*}(W)=\bigcup\{V\rightarrow Y:V\times_{Y}X\subset W\}.$$
Consider an open immersion $\mathrm{Spec}(C)\rightarrow\mathrm{Spec}(A)$. The fixed points of the nucleus $\pi_{g}^{*}\pi_{g}$ are those $V$ satisfying the property that whenever $\mathrm{Spec}(D)\rightarrow X$ is such that $\mathrm{Spec}(D\big\slash(f))\subset X\times_{Y}V$, then $\mathrm{Spec}(D)\subset V$.

On the other hand consider $U=\mathrm{Spec}(A_{f}^{\mathrm{T}})\subset\mathrm{Spec}(A)$. The nucleus defining the closed complement of $U$ is 
$$j_{CU}:V\mapsto V\cup U.$$
The fixed points are clearly those $V$ such that $U\subset V$.
We need to show that the fixed points of these nuclei coincide. Without loss of generality we may assume that the fixed points in both cases are of the form $V=\mathrm{Spec}(C)$ for some $C$.

Let $\mathrm{Spec}(D)\rightarrow Y$ be such that $\mathrm{Spec}(D\big\slash f)\rightarrow X$ factors through $\mathrm{Spec}(C\big\slash f)$. We claim that $\mathrm{Spec}(D\otimes^{\mathbb{L}}_{A}C)\sqcup\mathrm{Spec}(D^{\mathrm{T}}_{f})\rightarrow\mathrm{Spec}(D)$ is a cover of $\mathrm{Spec}(D)$. They are both clearly open immersions. Let $M$ be a finitely presented $D$-module. Suppose that $C\otimes^{\mathbb{L}}_{A}M\cong 0$ and $D_{f}^{\mathrm{T}}\otimes^{\mathbb{L}}_{D}M\cong D_{f}^{\mathrm{T}}\otimes_{D}M\cong 0$. The second condition means that $(1-xf)$ is a surjective endomorphism of $M\otimes\mathrm{T}(\lambda)$. Let us pass to the underlying module $\underline{M\otimes\mathrm{T}(\lambda)}\cong\underline{M}[[x_{1},\ldots,x_{n}]]$. The surjectivity condition means that for any element $m$ of $M$ there is an element $g_{m}(x)$ of $M\otimes\mathrm{T}(\lambda)$ such that $(1-x f)(g_{m}(x))=m(x)$. Comparing coefficients gives that $g_{m}(x)=m$ and $fm=0$. Thus $f$ acts by zero on $M$. Hence $M\cong (M\big\slash (f)M)$. Then 
\begin{multline*}
0\cong C\otimes^{\mathbb{L}}_{A}M\cong C\otimes^{\mathbb{L}}_{A}(M\big\slash (f)M) \\
\cong C\otimes_{A}(M\big\slash (f)M)\cong C\big\slash(f)\otimes_{A} M\cong D\big\slash(f)\otimes_{A}M\cong M\big\slash(f).\end{multline*}
This proves that 
 $$\mathrm{Spec}(D\otimes_{A}C)\sqcup\mathrm{Spec}(D^{\mathrm{T}}_{f})\rightarrow\mathrm{Spec}(D)$$
 is a cover.

 Thus if $\mathrm{Spec}(A_{f}^{\mathrm{T}})\subset\mathrm{Spec}(C)$, then $\mathrm{Spec}(D)\subset\mathrm{Spec}(C)$. Hence if $\mathrm{Spec}(C)$ is a fixed point of $j_{CU}$, it is also a fixed point of $\pi_{X}^{*}\pi_{X}$.

Conversely suppose that $\mathrm{Spec}(C)$ is a fixed point of $\pi_{X}^{*}\pi_{X}$. We have $X\times_{Y}\mathrm{Spec}(A_{\mathrm{T}^{f}})=\emptyset\subset X\times_{Y}\mathrm{Spec}(C)$. Thus $\mathrm{Spec}(A_{f}^{\mathrm{T}})\subset\mathrm{Spec}(C)$, as required.\qedhere

\comment{
For $U\rightarrow Y$ an open immersion, consider the nucleus $j_{CU}:\mathrm{An}_{\mathrm{T}}(Y)\rightarrow \mathrm{An}_{\mathrm{T}}(Y), V\mapsto V\cup U$. The fixed points of this nucleus are those open immersions $V\rightarrow Y$ such that $U\rightarrow Y$ factors through $V$. Let $f\in A$ and consider the open immersion

$$i:\mathrm{Spec}(A^{\mathrm{T}}_{f})\rightarrow\mathrm{Spec}(A).$$

Write $Y=\mathrm{Spec}(A)$, $X=\mathrm{Spec}(A\big\slash(f))$, and $U=\mathrm{Spec}(A_{f}^{\mathrm{T}})$. First observe that $g$ restricts to a surjection on the fixed points of $j_{CU}$. Indeed if $W\rightarrow X$ is an open immersion then there is an open immersion $\tilde{W}\rightarrow Y$ such that $\tilde{W}\times_{Y}X\cong W$. Now $U\times_{Y}X\cong\emptyset$. Thus $(\tilde{W}\cup U)\times_{Y}X\cong W$.
}

\comment{
Now let $V\rightarrow Y$, $W\rightarrow Y$ be open immersions, both containing $U$, such that $V\times_{Y}X\cong W\times_{Y}X$. By passing to intersections, we may without loss of generality assume that $V\subset W$. }
\end{proof}

\comment{
\begin{lemma}
    Let $\mathrm{T}$ be analytic and $j:X\rightarrow Y$ a closed immersion. Then the image of $|j|:|X|\rightarrow |Y|$ is closed in $|Y|$.
\end{lemma}

\begin{proof}
    Without loss of generality we may assume that $X=\mathrm{Spec}(B)$ and $Y=\mathrm{Spec}(A)$ are both affine. We may further assume that they are both discrete. Write $B\cong A\big\slash(f_{1},\ldots,f_{n})$. Consider the open immersions $$A\rightarrow A^{\mathrm{T}}_{f_{i}}\defeq A\otimes^{\mathbb{L}}\mathrm{T}(\lambda)\big\slash\big\slash (1-f_{i}x)\cong A\otimes\mathrm{T}(\lambda)\big\slash (1-f_{i}x).$$
    Now let $\mathrm{Spec}(\phi):\mathrm{Spec}(C)\rightarrow\mathrm{Spec}(A)$ be a morphism. Suppose that 
    $$C\otimes^{\mathbb{L}}A^{\mathrm{T}}_{f_{i}}\cong C^{\mathrm{T}}_{\phi(f_{i})}\cong 0.$$
    Thus, multiplication by $(1-x f_{i})$ is a surjective endomorphism of $C\otimes\mathrm{T}(\lambda)$. Passing to underlying algebras, this means there is an element $g(x)$ of $C\otimes\mathrm{T}(\lambda)$ such that $(1-x \phi(f_{i}))g(x)=1$. We may regard $g(x)$ as a power series. Then comparing coefficients gives that $g(x)=1$ and $\phi(f_{i})=0$. 
     Hence the morphism $A\rightarrow C$ factors through $A\rightarrow A\big\slash (f_{i})$. Consequently, if $C\otimes^{\mathbb{L}}A^{\mathrm{T}}_{f_{i}}\cong 0$ for all $1\le i\le n$, then $A\rightarrow C$ factors as $A\rightarrow A\big\slash(f_{1},\ldots,f_{n})\rightarrow C$. 

    Now let $\mathcal{F}\in |Y|$. Write 
    $$\mathcal{F}_{X}=\{V\times_{Y}X\rightarrow X:V\rightarrow Y\in\mathcal{F}\}.$$
    We claim that this is a well-defined completely prime filter whenever it does not contain $\emptyset\rightarrow X$. It is clearly non-empty, as it contains $X\rightarrow X$. It also contains finite meets. Let $V\times_{Y}X\rightarrow X\in\mathcal{F}_{X}$, and let $W\rightarrow X$ contain $V\times_{Y}X\rightarrow X$. There is a $\tilde{W}\rightarrow Y$ such that $W=\tilde{W}\times_{Y}X$. By taking the union with $V$, we may assume that $\tilde{W}$ contains $V$. The $\tilde{W}\rightarrow Y$ is in $\mathcal{F}$. Hence $W\rightarrow X\in\mathcal{F}$. Now let $\{V_{i}\times_{Y}X\rightarrow X\}_{i\in\mathcal{I}}$ be a collection of subspaces whose union $(\bigcup_{i\in\mathcal{I}}V_{i})\times_{Y}X\rightarrow X\in\mathcal{F}_{X}$. Then $\bigcup_{i\in\mathcal{I}}V_{i}\rightarrow Y\in\mathcal{F}$. Hence some $V_{i}\rightarrow Y\in\mathcal{F}$, so $V_{i}\times_{Y}X\rightarrow X\in\mathcal{F}_{X}$.


   For the remainder we will assume without loss of generality that $n=1$, and $f_{1}=f$.


    Let $|j|(x)\in\mathrm{im}(|j|)$. Then $|j|(x)\notin |U_{i}|=|\mathrm{Spec}(A_{i})|$, since this would mean $U_{i}\rightarrow Y\in \mathcal{F}_{|j|(x)}$, so $\emptyset=\mathrm{Spec}(0)= U_{i}\times_{Y}X\rightarrow X\in\mathcal{F}_{|j|(x)}$, a contradiction. 

Conversely, suppose $\mathcal{F}_{y}\notin\mathrm{im}(f)$. 

    
\end{proof}

}
\begin{lemma}
    Let $X\rightarrow Y$ be a locally finitely presented closed immersion. Then $X\rightarrow Y$ is spatial.
\end{lemma}

\begin{proof}
     We may reduce to the case $Y\cong\mathrm{Spec}(A)$ and $X\cong\mathrm{Spec}(A\big\slash I)$. Let $U=\underset{f \in I}\bigcup\mathrm{Spec}(A_{f})$ be its open complement. Let $\phi:A\rightarrow B$ be a morphism. Put $V=\mathrm{Spec}(B)$. Then $V\times_{Y}X\cong\mathrm{Spec}(B\big\slash\big\slash IB)$. Hence $|V\times_{Y}X|$ is the closed complement of $|\bigcup_{f\in I}\mathrm{Spec}(B_{\phi(f)}^{\mathrm{T}})|$. Now $|\bigcup_{f\in I}\mathrm{Spec}(B_{\phi(f)}^{\mathrm{T}})|\cong|\mathrm{Spec}(B)\times_{\mathrm{Spec}(A)}U|\cong|\mathrm{Spec}(B)|\times_{|\mathrm{Spec}(A)|}|U|$. Therefore its closed complement is $|X|\setminus |X|\times_{|Y|}|U|$ which is precisely $|X|\times_{|Y|}(|Y|\setminus |U|)\cong |X|\setminus (|X|\times_{|Y|}|U|)$.
\end{proof}

\comment{
\subsubsection{Local rings}

Let $\mathbf{A}_{\mathrm{T}}$ be analytic, and let $X$ be a derived $\mathrm{T}$-analytic space. 

Let $\mathrm{Spec}(A)\in\mathbf{A}_{\mathrm{T}}$, and $a\in |\mathrm{Spec}(A)|$. We write
$$A_{a}=\colim_{\mathrm{Spec}(B)\in a}B.$$

\begin{lemma}
    The underlying ring $|A_{a}|$ is a local ring. 
\end{lemma}

\begin{proof}
    Let $f\in |A_{a}|$ be the restriction of some section $\tilde{f}\in |B|$, where $\mathrm{Spec}(B)\in a$. Let $\lambda$ be such that $\tilde{f}\mathrm{Sym}(\mathbb{I})\rightarrow B$ factors through a morphism $\mathrm{T}(\lambda)\rightarrow B$. The pair of maps $\{\mathrm{Spec}(\mathrm{T}(\lambda)_{\tilde{f}}^{\mathrm{T}})\rightarrow\mathrm{Spec}(\mathrm{T}(\lambda)),\mathrm{Spec}(\mathrm{T}(\lambda)_{\tilde{f}}^{\mathrm{T}})\rightarrow\mathrm{Spec}(\mathrm{T}(\lambda))_{1-\tilde{f}}^{\mathrm{T}}\}$ is a strict $\mathrm{T}$-rational cover. In particular there is a rational cover $\{\mathrm{Spec}(B_{i})\rightarrow\mathrm{Spec}(B)\}$ such that each composite $\mathrm{Spec}(B_{i})\rightarrow\mathrm{Spec}(\mathrm{T}(\lambda))$ factors through either $\mathrm{Spec}(\mathrm{T}(\lambda)_{\tilde{f}}^{\mathrm{T}})$ or $\mathrm{Spec}(\mathrm{T}(\lambda)_{1-\tilde{f}}^{\mathrm{T}})$. Let $j$ be such that $\mathrm{Spec}(B_{j})\in a$. Then either $\tilde{f}$ or $1-\tilde{f}$ is invertible in $|B_{j}|$. Hence either $f$ or $1-f$ is invertible in $|A_{a}|$, as required. 
\end{proof}
}

\section{The derived algebraic context of bornological modules}\label{sec:bornologies}

Let \(R\) be a Banach ring, that is, a ring with a norm\footnote{For us a norm is a function \(\abs{-} \colon R \to \mathbb{R}_{\geq 0}\) that is nondegenerate, satisfies the triangle inequality \(\abs{a+b} \leq \abs{a} + \abs{b}\), and \(\abs{ab} \leq C\abs{a}\abs{b}\), for some \(C > 0\).} \(\abs{-} \colon R \to \mathbb{R}_{\geq 0}\) with respect to which it is complete. We call a Banach ring \emph{non-Archimedean} if it satisfies the ultra-metric property \(\abs{a + b} \leq \max\{\abs{a}, \abs{b}\}\). To fix intuition for the rest of this article, one should think of three classes of examples: (1) \(\R\) and \(\C\) with the Euclidean norm (\emph{archimedean fields}); (2) \(\Q_p\) and \(\C_p\) with the \(p\)-adic norm (\emph{non-Archimedean fields}); (3) \(\Z\) with the Euclidean norm and \(\Z_p\) with the \(p\)-adic norm, which we think of as \emph{global} base rings for analytic geometry.

A \emph{Banach \(R\)-module} is an \(R\)-module with a norm with respect to which it is complete. We call a Banach \(R\)-module \emph{non-Archimedean} if its norm satisfies the  ultra-metric property. Implcitly, whenever the base Banach ring \(R\) is non-Archimedean, we will assume that the Banach modules over such rings are non-Archimedean as well. An \(R\)-linear map \(f \colon M \to N\) between Banach \(R\)-modules is called \emph{bounded} if there is a \(C>0\) such that \(\abs{f(m)}_{N} \leq C\abs{m}_M\) for all \(m \in M\). Denote by \(\mathsf{Ban}_R\) the category of Banach \(R\)-modules and bounded \(R\)-linear maps. This is a symmetric monoidal category with the completed projective tensor product \(- \haotimes - \colon \mathsf{Ban}_R \times \mathsf{Ban}_R \to \mathsf{Ban}_R\) defined as the completion of the algebraic tensor product \(M \otimes_R N\) with respect to the norm \[\norm{x} \defeq \inf \setgiven{\sum_{i=0}^n \norm{m_i}_M \norm{n_i}_N}{x = \sum_{i=0}^n m_i \otimes n_i}.\] The completed projective tensor product is universal for bounded bilinear maps into a Banach \(R\)-module, and there is an internal Hom given by the \(R\)-module \[\mathsf{Hom}(M,N) = \setgiven{T \colon M \to N \text{ }R\text{ -linear }}{\norm{T}< \infty}\] with the operator norm. The functor \(M \haotimes -\) is left adjoint to \(\mathsf{Hom}(M,-)\), or in other words, \(\mathsf{Ban}_R\) is a \emph{closed} symmetric monoidal category. 

The category of Banach \(R\)-modules is finitely complete and finitely cocomplete, and is \emph{quasi-abelian} in the sense that it is stable under pullbacks (resp. pushouts) of cokernels (resp. kernels) by arbitrary morphisms. Recall that a quasi-abelian category is said to have \emph{enough projectives} if for every object \(M\), there is a cokernel \(P \to M\) from a projective object \(P\). 

\begin{lemma}\cite{ben2020fr}
The category \(\mathsf{Ban}_R\) has enough flat projectives, and the tensor product of two projectives is projective.
\end{lemma}

\begin{proof}
This is well-known and we direct the reader to \cite{bambozzi2020sheafyness}*{Proposition 3.9, 3.11} for a complete proof of the fact that \(\mathsf{Ban}_R\) has enough flat projectives. For exposition, we spell out what the projectives are: let \((X, *)\) be a pointed set, and \(f \colon X \to \R_{\geq 0}\) a function that takes the value \(0\) only on the base point \(x = *\). For \(x \neq *\), consider the Banach \(R\)-module \(R_{f(x)} = R\) as an \(R\)-module, with the norm \(\abs{r}_{R_{f(x)}} = \frac{\abs{x}_R}{f(x)}\). Then the coproduct \[l^1(X) \defeq \coprod_{x \neq *}^{\leq 1} R_{f(x)}\] in the category \(\mathsf{Ban}_R^{\leq 1}\) of Banach \(R\)-modules with contracting \(R\)-linear maps is a projective object in \(\mathsf{Ban}_R\). It is then easy to see that any Banach \(R\)-module \(M\) is a quotient \(l^1(M) \to M\) of such a projective. To see that the completed projective tensor product of two projectives is projective, we first use the standard identification \(l^1(X) \haotimes l^1(Y) \cong l^1 (X \times Y)\), and then the fact that any projective object in \(\mathsf{Ban}_R\) is a retract of an \(l^1\)-space, and the fact that the tensor product commutes with finite colimits. 
\end{proof}

While \(\mathsf{Ban}_R\) is a useful starting point for analytic geometry - particularly because it contains, for instance, the affinoids of rigid analytic geometry - we are impeded by the lack of infinite limits and colimits. This entirely rules out overconvergent and Stein spaces. To remedy this, one needs to enlarge the category to include (in the very least) Fr\'echet spaces. To this end, consider the category \(\mathsf{Ind}(\mathsf{Ban}_R)\) of inductive systems of Banach \(R\)-modules. This is still a closed symmetric monoidal quasi-abelian category, but now has all limits and colimits. However, the category \(\mathsf{Ind}(\mathsf{Ban}_R)\) is too large, and it is sometimes convenient to work in the smaller concrete category \(\mathsf{Ind}^s(\mathsf{Ban}_R)\) of inductive systems of Banach \(R\)-modules with \emph{monomorphic} structure maps. This category is closed symmetric monoidal, quasi-abelian, complete and cocomplete, but has the convenient feature that when \(R\) is a Banach field (that is, a Banach ring containing \(\Q\)), it admits an alternative description in terms of \emph{bounded subsets}:

\begin{definition}\label{def:bornology}
A \emph{bornology} on a set \(X\) is a collection \(\bdd_X\) of its subsets is closed under taking finite unions and hereditary under inclusions. The elements of \(\bdd_X\) shall be called \emph{bounded subsets}, and the pair \((X, \bdd_X)\) a \emph{bornological set}. A map \(f \colon (X, \bdd_X) \to (Y, \bdd_Y)\) between bounded subsets is called \emph{bounded} if whenever \(S \in \bdd_X)\), \(f(S) \in \bdd_Y\). 
\end{definition}

Typically, one needs and uses some additional structure on the underlying set in applications. In our applications, we typically start with an \(R\)-linear structure which is sufficient to define seminorms that arise from suitable bounded subsets. 

\begin{definition}\label{def:bornological-module}
Let \(R\) be a Banach field with a nontrivial norm. A \emph{bornological \(R\)-module} \(M\) is a module with a bornology such that the addition and scalar multiplication maps \[+ \colon M \times M \to M, \quad R \times M \to M\] are bounded functions. 
\end{definition} 

In what follows, we describe how bornologies arise from modules over a nontrivially valued Banach field. Given such a module \(M\), we may define for each nonempty subset (which we may without loss of generality take to be convex lattices \footnote{A subset \(S \subset M\) being a lattice here means that \(R^{\leq 1} S = M\).})  \(S \in \bdd_M\), the seminorm \[\rho_S(x) \defeq \inf_{\lambda \in R}\setgiven{\abs{\lambda}}{x \in \lambda S},\] called the \emph{gauge seminorm}. The \(R\)-module \(M_S\) generated by \(S\) is a semi-normed \(R\)-module, whose unit ball is precisely \(S\). To distinguish between arbitrary bounded subsets, and those which arise as closed unit balls of seminorms, we refer to the latter as \emph{discs}. We call a bornology on an \(R\)-module \emph{convex} if the collection of bounded discs is cofinal in the original bornology (ordered by inclusion). We shall always assume that our bornologies are convex, and drop this adjective from here onwards. A bornological \(R\)-module is called \emph{separated} if for every bounded disc \(S\), the associated seminorm \(\rho_S\) is a norm. Finally, a bornological \(R\)-module \(M\) is called \emph{complete} if for every bounded disc \(S\), the associated seminormed space \((M_S, \rho_S)\) is a Banach \(R\)-module. Denote by \(\mathsf{Born}_R\) and \(\mathsf{CBorn}_R\) the categories of separated and complete bornological \(R\)-modules.

To relate the concrete definition of bornologies above with inductive systems, we observe that there are canonical functors 

\[
\mathrm{diss} \colon \mathsf{Born}_R \to \mathsf{Ind}(\mathsf{Norm}_R), \quad \mathsf{CBorn}_R \to \mathsf{Ind}(\mathsf{Ban}_R), 
\] taking a (complete) bornological \(R\)-module \(M\) to the formal inductive system \(``\mathrm{colim}" (M_S, \rho_S)\) of (complete) normed \(R\)-modules. We call this the \emph{dissection} functor.

\begin{theorem}\cite{kelly2022analytic}*{Proposition 2.14}\label{thm:dissection}
Let \(R\) be a Banach ring. Then the dissection functor \(\mathrm{diss} \colon \mathsf{CBorn}_R \to \mathsf{Ind}(\mathsf{Norm}_R)\) is faithful. If \(R\) is a nontrivially valued Banach field, then it is fully faithful, and the essential image can be identified with the essentially monomorphic inductive systems of normed \(R\)-modules. Similar claims hold for \(\mathsf{CBorn}_R\) and \(\mathsf{Ind}(\mathsf{Ban}_R)\).  
\end{theorem}

Using Theorem \ref{thm:dissection}, we will identify the categories \(\mathsf{CBorn}_R\) and \(\mathsf{Ind}^s(\mathsf{Ban}_R)\) from this point onwards. 

As already mentioned, the categories \(\mathsf{CBorn}_R\) and \(\mathsf{Ind}(\mathsf{Ban}_R)\) are quasi-abelian. This means that they have well-behaved derived categories, defined by taking the Verdier quotient of the triangulated homotopy category of unbounded chain complexes by the thick subcategory of exact chain complexes. Furthermore, being quasi-abelian, they have left and right \(t\)-structures. The following result shows that the distinction between inductive systems of Banach spaces and complete bornological spaces disappears upon taking derived categories:

\begin{theorem}\label{thm:elementary}
Let \(R\) be a Banach ring. 
\begin{enumerate}
\item The category of inductive systems of Banach \(R\)-modules \(\mathsf{Ind}(\mathsf{Ban}_R)\) is an elementary quasi-abelian category (in the sense of \cite{Schneiders:Quasi-Abelian});
\item The category of complete bornological \(R\)-modules \(\mathsf{CBorn}_R\) is an \textbf{Admon}-elementary\footnote{This means that \(\mathsf{Hom}(X,-)\) commutes with essentially monomorphic filtered colimits.} quasi-abelian category;
\item The categories \(\mathsf{LH}(\mathsf{Ind}(\mathsf{Ban}_R))\) and \(\mathsf{LH}(\mathsf{CBorn}_R)\) are Grothendieck abelian categories;
\item We have a monoidal equivalence of categories \[\mathsf{D}(\mathsf{Ind}(\mathsf{Ban}_R)) \simeq \mathsf{D}(\mathsf{LH}(\mathsf{Ind}(\mathsf{Ban}_R))) \simeq \mathsf{D}(\mathsf{LH}(\mathsf{CBorn}_R)) \simeq \mathsf{D}(\mathsf{CBorn}_R).\] 
\end{enumerate}
\end{theorem}

\begin{proof}
The first two claims are \cite{kelly2022analytic}*{Proposition 4.8}. That the left hearts are Grothendieck quasi-abelian is a consequence of \cite{bode2021six}*{Definition 2.2}. The monoidal derived equivalence between a symmetric monoidal quasi-abelian category and its left heart is purely formal (see \cite{bambozzi2020sheafyness}*{Corollary 2.11}). The monoidal derived equivalence between inductive systems and bornologies is induced by the adjoint pair \[\varinjlim \colon \mathsf{Ind}(\mathsf{Ban}_R) \rightleftarrows \mathsf{CBorn}_R \colon \mathrm{diss},\] whose proof is spelled out in \cite{bambozzi2020sheafyness}*{Proposition 3.16}. 
\end{proof}

We now turn to the \(\infty\)-categorical enhancement of the derived categories defined so far. By Theorem \ref{thm:elementary} and \cite{kelly2016homotopy}*{Theorem 4.3.58}, since \(\mathsf{Ind}(\mathsf{Ban}_R)\) is an elementary quasi-abelian category, the category \(\mathsf{Ch}(\mathsf{Ind}(\mathsf{Ban}_R))\) of unbounded complexes of ind-Banach \(R\)-modules has a \emph{projective model structure}, whose weak equivalences and fibrations are the quasi-isomorphisms and termwise cokernels, respectively. Furthermore, the model structure is stable, simplicial, combinatorial and projectively monoidal. Consequently, localising at the weak equivalences yields \[\bD(\mathsf{Ind}(\mathsf{Ban}_R)) \defeq L^H(\mathsf{Ch}(\mathsf{Ind}(\mathsf{Ban}_R))) =  N(\mathsf{Ch}(\mathsf{Ind}(\mathsf{Ban}_R)))[W^{-1}]\] a presentably stable symmetric monoidal \(\infty\)-category. In fact, as is shown in \cite{kelly2016homotopy}, $\mathrm{Ch}(\mathrm{CBorn}_{R})$ can also be equipped with the projective model structure such that the equivalence of Theorem \ref{thm:elementary} Part (4) can be boosted to a Quillen equivalence. We have even more structure: setting \(\mathcal{P}^0\) as the full subcategory of \(\mathsf{Ch}(\mathsf{Ind}(\mathsf{Ban}_R))\) generated by the compact projective generators under the formation of finite direct sums, tensor and symmetric powers, we have the following:

\begin{theorem}\cite{kelly2022analytic}
With notation as above, the quadruple \[(\mathsf{Ch}(\mathsf{Ind}(\mathsf{Ban}_R)), \mathsf{Ch}_{\geq 0}(\mathsf{Ind}(\mathsf{Ban}_R)), \mathsf{Ch}_{\leq 0}(\mathsf{Ind}(\mathsf{Ban}_R)), \mathcal{P}^0)\] is a model derived algebraic context. Consequently, the localisation \[(L^H(\mathsf{Ch}(\mathsf{Ind}(\mathsf{Ban}_R))), L^H(\mathsf{Ch}_{\geq 0}(\mathsf{Ind}(\mathsf{Ban}_R))), L^H(\mathsf{Ch}_{\leq 0}(\mathsf{Ind}(\mathsf{Ban}_R))), L^H(\mathcal{P}^0))\] is a derived algebraic context.  
\end{theorem}

In order not to overload notation, we will use the identification in the last part of Theorem \ref{thm:elementary} to unambiguously denote by \(\bD(R)\) the derived \(\infty\)-categories of ind-Banach and complete bornological \(R\)-modules. Furthermore, as the equivalence is monoidal, we will denote the derived commutative algebra objects in \(\bD(R)\) by \(\mathbf{DAlg}(R)\), and refer to its objects as \emph{derived complete bornological \(R\)-algebras}.  

\section{Derived analytic geometry from the bornological perspective - complex geometry and the rigid $G$-topology}\label{section:derived-borno-geom}

Building on \cite{ben2024perspective}, in this section, we describe several categories of analytic stacks defined relative to a base Banach ring using the topologies and the general process described in Section \ref{sec:der-context}. When we specialise to the case where our Banach ring \(R\) is  either a nontrivially valued, non-Archimedean Banach field \(K\) of characteristic zero or \(\C\), our category of analytic stacks contains classical complex and rigid analytic spaces as full subcategories. The starting point is the observation that the category \(\mathsf{Ind}(\mathsf{Ban}_R)\) contains the building blocks of different kinds of analytic geometry. To this end, we first note that \(\mathsf{Ind}(\mathsf{Ban}_K)\) contains \(\mathsf{Ban}_K\) as a monoidal full subcategory. It also contains Fr\'{e}chet spaces as a full subcategory by \cite{BaBK}, and nuclear Fr\'{e}chet spaces as a full \textit{monoidal} subcategory by \cite{Meyer:HLHA}. These are also full subcategories of \(\mathsf{CBorn}_K\). Finally, as \(\Q \subset R\), the canonical map \(\theta \colon \mathbf{DAlg}^{cn}(\mathbf{D}(K)) \simeq \mathbf{CAlg}(\bD_{\geq 0}(K))\) is an equivalence. 

To fix some notation, for the various classes of affinoids $\mathrm{Spec}(A)$ and analytic geometry contexts defined below, we will write
$$\mathcal{M}(A)=|\mathrm{Spec}(A)|_{Ber}.$$

\subsection{Rigid analytic spaces}

Let \(K\) be a nontrivially valued non-Archimedean Banach field.

\subsubsection{Tate and affinoid algebras}

Let $|K|_{>0}$ denote the subset of $\mathbb{R}_{>0}$ consisting of the image of the map $|-|:K^{*}\rightarrow\mathbb{R}_{>0}$.
For each \(n>0\), and a polyradius \(r = (r_1, \dotsc, r_n) \in |K|_{>0}\), define the \emph{Tate algebra} \[K \gen{r_1^{-1} x_1 , \dotsc, r_n^{-1} x_n} = \setgiven{\sum_I a_I x^I \in K[[x_1 ,\cdots, x_n]]}{\lim_{\abs{I} \to 0} \abs{a_I}r^I = 0}\] of convergent power series. This is a Banach \(K\)-algebra with the supremum norm \(\abs{\sum_I a_I x^I}_\infty = \sup \abs{a_I} r^I\). 

\begin{proposition}[\cite{ben2022analytification}]
    The natural map 
    $$\mathrm{Sym}(K)\rightarrow K \gen{r^{-1} x}$$
    is a homotopy epimorphism.
\end{proposition}

In particular, the functor
$$K\gen{-}:\underline{\mathbb{R}_{>0}}^{op}\rightarrow \mathbf{CAlg}(\bD_{\geq 0}(K))$$
defines a $\underline{\mathbb{R}_{>0}}$-generating class of algebras of homotopy polynomial type. This is the functor \(T\) from Definition \ref{def:generatingclass}.

An \emph{affinoid \(K\)-algebra} is a quotient \(A = K \gen{r_1^{-1} x_1 , \dotsc, r_n^{-1} x_n}/I\) of a Tate algebra for some \(n\) and polyradius \(r\) by an ideal $I$ which is automatically closed and finitely generated. 
\begin{definition}\label{def:derived-affinoid}
An algebra \(A \in \mathbf{CAlg}(\bD_{\geq 0}(K))\) is called a \emph{derived affinoid} if \(\pi_0(A)\) is an affinoid \(K\)-algebra, and \(\pi_n(A)\) is finitely generated as a \(\pi_0(A)\)-module for each \(n \geq 0\). We denote the full subcategory on derived affinoids by \(\mathbf{Afnd}_K \subset \mathbf{CAlg}(\bD_{\geq 0}(K))\). 
\end{definition}

In what follows, we study localisations a between affinoid \(K\)-algebras relative to the generating class \(K\langle - \rangle \colon |K|_{>0} \to  \mathbf{CAlg}(\bD_{\geq 0}(K))\). Later we will discuss \'{e}tale maps. The definition of a \(T\)-rational localisation from Definition \ref{defn:Trat} specialises to a \emph{derived rational localisation} \[A \to A \hat{\otimes}_K^L K\langle x_1r_1^{-1}, \dotsc, x_nr_n^{-1} \rangle // (gx_1 - f_1, \dotsc, gx_n f_n)\] in the sense of \cite{ben2022analytification}.

Now consider the the finite \(G\)-\(T\)-pre-topology \(\tau^{rat}\); this has as covers derived rational localisations that are finitely conservative. On derived affinoids, this topology is the same as the one considered in \cite{ben2024perspective}*{Section 9.2.4}, and \cite{soor2024derived} as the following lemma shows:

The following is \cite{ben2024perspective}*{Theorem 9.2.34}. We include a proof here for completeness.

\begin{lemma}\label{lem:Arun-top}
Let \(\bA_K = \mathbf{Afnd}_K^\op\). A family \(\{(\mathsf{Spec}(B_i) \to \mathsf{Spec}(A)\}\) in \(\bA_K\) is a \(\tau^{rat}\)-cover if and only if each \(A \to B_i\) is derived strong, and \(\{\mathsf{Spec}(\pi_0(B_i)) \to \mathsf{Spec}(\pi_0(A))\}\) are rational localisations that are finitely conservative.
\end{lemma}

\begin{proof}
Consider a finite family \(\{(\mathsf{Spec}(B_i) \to \mathsf{Spec}(A)\}\) in \(\bA_K\) such that \(\pi_0(A) \to \pi_0(B_i)\) are rational localisations of affinoids that are finitely conservative, and each \(A \to B_i\) is derived strong. By \cite{ben2017non}*{Theorem 5.37} $\{\mathrm{Spec}(\pi_{0}(B_{i}))\rightarrow\mathrm{Spec}(\pi_{0}(A))\}$ is a cover in the homotopy epimorphism topology. Thus $\pi_{0}(A)\rightarrow\prod_{i}\pi_{0}(B_{i})$ is descendable by Lemma \ref{lem:hepi-descendable}. Then by \cite{ben2024perspective}*{Lemma 4.5.88}, the maps \(\mathsf{Spec}(B_i) \to \mathsf{Spec}(A)\) are rational localisations, and by Lemma \ref{lem:derstrongconver} the map $A\rightarrow\prod_{i}B_{i}$ is descendable. Descendability in particular implies finite conservativity. 

Conversely suppose that \(\{(\mathsf{Spec}(B_i) \to \mathsf{Spec}(A)\}\) is a $\tau^{rat}$-cover of derived affinoids. By \cite{ben2024perspective}{Lemma 5.4.106} each $\pi_{0}(A)\rightarrow\pi_{0}(B_{i})$ is transverse to each $\pi_{n}(A)$. Moreover, each $\pi_{0}(A)\rightarrow\pi_{0}(B_{i})$ is a (discrete) rational localisation and thus a homotopy epimorphism. Hence by \cite{ben2024perspective}*{Proposition 2.6.160 (3)} each map $A\rightarrow B_{i}$ is derived strong. Then by  \cite{ben2024perspective}*{Proposition 2.3.87} we have $\pi_{0}(A)\hat{\otimes}_{A}^{\mathbb{L}}B_{i}\cong\pi_{0}(B_{i})$. Thus the collection \(\{(\mathsf{Spec}(\pi_{0}(B_i)) \to \mathsf{Spec}(\pi_{0}(A))\}\) is a base-change of the cover \(\{(\mathsf{Spec}(B_i) \to \mathsf{Spec}(A)\}\). Again by \cite{ben2017non}*{Theorem 5.37},  \(\{\mathsf{Spec}(\pi_{0}(B_i)) \to \mathsf{Spec}(\pi_{0}(A))\}\) corresponds to a cover of affinoids in the usual sense.\qedhere

\end{proof}

Let \(\mathbf{P}^{o-sm}\) be the collection of open smooth maps in \(\mathbf{CAlg}(\bD_{\geq 0}(K))^\op\).

\begin{proposition}\label{prop:rigid-geometry-tuple}
Let \(K\) be a non-trivially valued non-Archimedean Banach field and \(\bA_K\) the opposite category of \(\mathbf{Afnd}_K\). Then the tuple \((\mathbf{Aff}_{\bD_{\geq 0}(K)}, \tau^{rat}, \mathbf{P}^{o-sm}, \bA_K)\) is a strong relative geometry tuple. 
\end{proposition}

\begin{proof}
\comment{By Lemma \ref{lem:rational-loc-etale}, rational localisations are standard \'etale and therefore \'etale. To see that the topology \(\tau^{rat}\) is local, let \(f \colon A \to B\) be a morphism, and \(B \to B_i\) a \(\tau^{rat}\)-cover such that \(A \to B_i\) is \'etale. Then there is a \(\tau^{rat}\)-cover \(\{\mathsf{Spec}(B_{ij}) \to \mathsf{Spec}(B_i)\}\) such that the composition \(A \to B \to B_i \to B_{ij}\) is standard smooth. Since there are finitely many \(i\), taking their intersection yields a common \(\tau^{rat}\)-refinement of the original cover \(\{\mathsf{Spec}(B_i) \to \mathsf{Spec}(B)\}\).Consequently, \(f \colon A \to B\) is smooth.}
The class $\mathbf{P}^{sm}$ is local for $\tau^{rat}$ by definition.
  It is easy to see that derived affinoids are closed under pushouts by arbitrary morphisms (\cite{soor2024derived}*{Lemma 2.3}). The proof that smooth maps of affinoids are closed under pushouts is similar to that of Proposition \ref{prop:etale-top}. Finally, to see that \(\tau^{rat}\) is subcanonical, we observe that the maps \(A \to B_i\) in the covers are homotopy epimorphisms; then by Lemma \ref{lem:hepi-descendable}, the total map \(A \to \prod_i B_i = B\) is descendable, so that we have a symmetric monoidal equivalence \[\bD(A) \overset{\simeq}\to \varprojlim_{[m] \in \Delta} \bD(B^{\hat{\otimes} m+1}).\] The result now follows by evaluating this equivalence at the unit object to get \(A \simeq \varprojlim_{[m] \in \Delta} B^{\hat{\otimes}_A m+1}\) as required.  

  It is strong by construction.
\end{proof}
\comment{
\begin{proposition}
  The tuple \((\mathbf{Aff}_{\bD_{\geq 0}(K)}, \tau^{rat}, \mathbf{P}^{o-sm}, \bA_K)\) is tempered. 
\end{proposition}

\begin{proof}
    Let $\{\mathrm{Spec}(A_{i})\rightarrow\mathrm{Spec}(A)\}_{i=1}^{n}$ be a cover in $\tau$ with each $\mathrm{Spec}(A_{i_{1}}\otimes^{\mathbb{L}}_{A}\ldots\otimes^{\mathbb{L}}_{A}A_{i_{n}})$ a derived affinoid. 
\end{proof}
}

\begin{proposition}\label{prop:ratdomainstronglocal}
    Let $\phi:A\rightarrow B$ be a map of derived affinoids. Suppose there is a cover $\{\mathrm{Spec}(\psi_{i}):\mathrm{Spec}(B_{i}^{0})\rightarrow\mathrm{Spec}(\pi_{0}(B))\}_{i=1}^{n}$ in the rational topology such that each composite $\mathrm{Spec}(B_{i}^{0})\rightarrow\mathrm{Spec}(\pi_{0}(B))\rightarrow\mathrm{Spec}(\pi_{0}(A))$ is a rational localisation. Then there exists a cover $\{\mathrm{Spec}(B_{i})\rightarrow\mathrm{Spec}(B)\}_{i=1}^{n}$ such that
    \begin{enumerate}
        \item 
        each $\mathrm{Spec}(B_{i})\rightarrow\mathrm{Spec}(B)\rightarrow\mathrm{Spec}(A)$ is a rational domain embedding;
        \item 
        for each $1\le i\le n$, $\pi_{0}(B_{i})\cong B_{i}^{0}$.
    \end{enumerate}
\end{proposition}

\begin{proof}
    Write $B_{i}^{0}\cong\pi_{0}(B)\otimes^{\mathbb{L}}K<\lambda^{i}_{1},\ldots,\lambda_{n_{i}}^{i}>\big\slash\big\slash(g_{0}^{i}y^{\lambda^{i}_{1}}-g_{1}^{i},\ldots,g_{0}^{i}y^{\lambda^{i}_{n}}-g_{n}^{i})$, where each $g_{0}^{i},\ldots,g_{n}^{i}$ generates the unit ideal. Now define $B_{i}\defeq B\otimes^{\mathbb{L}}K<\lambda^{i}_{1},\ldots,\lambda_{n_{i}}^{i}>\big\slash\big\slash(g_{0}^{i}y^{\lambda^{i}_{1}}-g_{1}^{i},\ldots,g_{0}^{i}y^{\lambda^{i}_{n}}-g_{n}^{i})$. Then $\mathrm{Spec}(B_{i})\rightarrow\mathrm{Spec}(B)$ is a rational domain embedding. Moreover, we have $\pi_{0}(B_{i})\cong B^{i}_{0}$. The collection $\{\mathrm{Spec}(B_{i})\rightarrow\mathrm{Spec}(B)\}_{i\in\mathcal{I}}$ is a cover in the rational topology by Lemma \ref{lem:derstrongdescendable}. Finally, $A\rightarrow B$ is derived strong by Proposition \ref{prop:stronglocstrong}.
\end{proof}

\begin{proposition}\label{prop:smoothhtpymono}
     Let $f:X=\mathrm{Spec}(B)\rightarrow Y=\mathrm{Spec}(A)$ be a morphism of derived affinoids such that $f$ is in $n-\mathbf{P}^{o-sm}$ and is a homotopy monomorphism. Then $f$ is, derived rationally locally, a derived rational localisation.
\end{proposition}

\begin{proof}
Consider the corresponding map of algebras $\phi:A\rightarrow B$. Since $\phi$ is smooth, it is derived strong. Hence $\pi_{0}(A)\rightarrow\pi_{0}(A)\hat{\otimes}^{\mathbb{L}}_{A}B\cong\pi_{0}(B)$ is a homotopy monomorphism. By \cite{MR3626003}*{Theorem 5.31}, $\pi_{0}(A)\rightarrow\pi_{0}(B)$ is, rationally locally, a rational subdomain. Hence the claim follows from Proposition \ref{prop:ratdomainstronglocal}.
\end{proof}
In other words, a derived rigid space in the sense of \cite{soor2023quasicoherent}*{Definition 2.17} is the same as a $2$-scheme. We record this as a theorem.

\begin{theorem}\label{thm:derived-rigid}
We have a fully faithful embedding \(\mathbf{dRig}_K \hookrightarrow \mathbf{Stk}(\mathbf{Afnd}_K, \tau^{rat})\). 
\end{theorem}

\comment{
\begin{proof}
    Suppose that $X$ is a $1$-scheme. Pick an atlas $\{U_{i}\rightarrow X\}_{i\in\mathcal{I}}$ with each $U_{i}\rightarrow X$ a monomorphism in $\mathbf{P}$. Let $V\rightarrow X$ be any map with $V$ affine. Then $V\times_{X}U_{i}\rightarrow V$ is a monomorphism in $\mathbf{P}$ with each $V\times_{X}U_{i}$ a $0$-scheme. In particular, it has an atlas $\{V_{j}\rightarrow V\times_{X}U_{i}\}_{j\in\mathcal{J}}$ with each $V_{j}\rightarrow V\times_{X}U_{i}$ a monomorphism in $\mathbf{P}$. In particular, each $V_{j}\rightarrow V_{j}\times_{X}U_{i}\rightarrow U_{i}$ is a homotopy monomorphism in $\mathbf{P}$. In particular, it can be refined to a rational subdomain embedding. This means that each $U_{i}\rightarrow X$ is an analytic subspace in the sense of \cite{soor2023quasicoherent}*{Definition 2.17}, and hence $X$ is a derived rigid space in the sense of \cite{soor2023quasicoherent}*{Definition 2.17}.

Conversely, suppose that $X$ is a derived rigid space in the sense of \cite{soor2023quasicoherent}*{Definition 2.17}. Pick an effective epimorphism $\coprod_{i\in\mathcal{I}}U_{i}\rightarrow X$ with each $U_{i}$ affine, and each $U_{i}\rightarrow X$ an analytic subspace in the sense of \cite{soor2023quasicoherent}*{Definition 2.17}. Let $V\rightarrow X$ be any map with $V$ affine. Then $V\times_{X}U_{i}\rightarrow V$ is an analytic subspace. Thus there is an effective epimorphism $\coprod_{j_{i}\in\mathcal{J_{i}}}V_{j_{i}}\rightarrow V\times_{X}U_{i}$ with each $V_{j_{i}}\rightarrow U_{i}$ be a rational subdomain embedding. In particular it is a smooth homotopy monomorphism. We also need each $V_{j_{i}}\rightarrow V\times_{X}U_{i}$ to be representable by affines. Let $W\rightarrow V\times_{X}U_{i}$ be any morphism with $W$ affine. Then $W\times_{V\times_{X}U_{i}}V_{j_{i}}\cong W\times_{U_{i}}V_{j_{i}}$ is a fibre product of affines, and hence is affine. 
\end{proof}

\begin{proposition}\label{prop:anasubschaff}
    Let $X=\mathrm{Spec}(A)$ be an affinoid. A morphism $Y\rightarrow X$ is an analytic subspace of $X$ in the sense of \cite{soor2023quasicoherent}*{Definition 2.17 (ii)} if and only if it is a subscheme.
\end{proposition}

\begin{proof}
Suppose $Y\rightarrow X$ is a subscheme.
    Pick an atlas $\{V_{i}\rightarrow Y\}_{i\in\mathcal{I}}$. Then each $V_{i}\rightarrow Y\rightarrow X$ is a homotopy monomorphism in $\mathbf{P}$ between affinoids, and hence is rationally locally a rational embedding. Hence, by refining our atlas, we get an atlas $\{\tilde{V}_{j}\rightarrow Y\}_{j\in\mathcal{J}}$ of $Y$ such that each $\tilde{V}_{j}\rightarrow X$ is a rational embedding. 

    Conversely, suppose that $Y\rightarrow X$ is a subspace in the sense of \cite{soor2023quasicoherent}*{Definition 2.17 (ii)}. Then there exists an effective epimorphism $\coprod_{i\in\mathcal{I}}\mathrm{Spec}(A_{i})\rightarrow Y$ such that each composite $\mathrm{Spec}(A_{i})\rightarrow X$ is a rational subdomain embedding. We claim that this is an atlas. But this follows from the fact that for any pair of morphisms $V,W\rightarrow Y$ with $V,W$ affine, we have $V\times_{Y}W\cong V\times_{X}W$ is affine. 
\end{proof}

\begin{proposition}
    Let $X\in\mathbf{Stk}(\bA, \tau_{\vert \bA})$. A morphism $Y\rightarrow X$ is an analytic subspace in the sense of \cite{soor2023quasicoherent}*{Definition 2.17} if and only if it is a subscheme. 
\end{proposition}

\begin{proof}
    Suppose that $Y\rightarrow X$ is a subscheme. Let $U\rightarrow X$ be a morphism with $U$ affine. Then $U\times_{X}Y\rightarrow U$ is subscheme, and hence an analytic subspace of $U$ in the \cite{soor2023quasicoherent}*{Definition 2.17}. Therefore $Y\rightarrow X$ is an analytic subspace of $U$ in the \cite{soor2023quasicoherent}*{Definition 2.17}.

    Conversely, suppose that $Y\rightarrow X$ is an analytic subspace in the sense of \cite{soor2023quasicoherent}*{Definition 2.17}. For any $U\rightarrow X$ with $U$ affine, $U\times_{X}Y\rightarrow U$ is an analytic subscheme. Hence by Proposition \ref{prop:anasubschaff}, the map $Y\rightarrow X$ is in $\mathbf{P}$. Hence it is a subscheme by Proposition \ref{htpymonoPsubsch}.
\end{proof}

}

\begin{proposition}
    The tuple \(\mathbf{dBer}_{K}=(\mathbf{Aff}_{\bD_{\geq 0}(K)}, \tau^{rat}, \mathbf{P}^{o-sm}, \bA_K)\) is stably Berkovich.
\end{proposition}

\begin{proof}
Let $X$ be a qcqs derived scheme. Then $|X|\cong |t_{\le0}X|$. Hence we may assume that $X$ is discrete. The lattice of special subspaces of $X$ coincides with the lattice of closed subsets of the Berkovich spectrum $\mathcal{M}(X)$ which are of the form $\mathcal{M}(U)$, for some special subspace $U$ of $X$. The complements of such subspaces form a base for the compact Hausdorff space $\mathcal{M}(X)$. This gives that $\mathrm{An}_{\mathbf{dBer}_{K}}(X)$ is a Berkovich locale. We need to show that for a special subspace $U$ of $X$, and any morphism $f:Y\rightarrow X$ of qcqs schemes, the natural map
$$\mathcal{M}(U\times_{X}Y)\rightarrow\mathcal{M}(U)\times_{\mathcal{M}(X)}\mathcal{M}(Y)$$
is surjective. We may assume that everything is affine. Put $U=\mathrm{Spec}(B),X=\mathrm{Spec}(A),Y=\mathrm{Spec}(C)$. We shall write, for example,  $\mathcal{M}(A)=\mathcal{M}(\mathrm{Spec}(A))$.

We now use \cite{https://deepblue.lib.umich.edu/bitstream/handle/2027.42/176463/attilio_1.pdf?sequence=1}*{Proposition III.2.14}, though we must be careful. Consider the uniformisations $A\rightarrow A^{u}, B\rightarrow B^{u},C\rightarrow C^{u}$ (\cite{https://deepblue.lib.umich.edu/bitstream/handle/2027.42/176463/attilio_1.pdf?sequence=1}*{Proposition II.4.13}). We then get a morphism of Banach algebras
$$B\hat{\otimes}_{A}C\rightarrow B^{u}\hat{\otimes}_{A^{u}}C^{u}.$$
We get a commutative diagram of morphisms of Berkovich spaces.
\begin{displaymath}
    \xymatrix{
    \mathcal{M}(B^{u}\hat{\otimes}_{A^{u}}C^{u})\ar[d]\ar[r] &\mathcal{M}(B\otimes_{A}C)\ar[d]\\
    \mathcal{M}(B^{u})\times_{\mathcal{M}(A^{u})}\mathcal{M}(C^{u})\ar[r] & \mathcal{M}(B)\times_{\mathcal{M}(A)}\mathcal{M}(C).
    }
\end{displaymath}
Since morphisms of Banach algebras with uniform target are non-expanding (\cite{https://deepblue.lib.umich.edu/bitstream/handle/2027.42/176463/attilio_1.pdf?sequence=1}*{Lemma II.4.10}), by \cite{https://deepblue.lib.umich.edu/bitstream/handle/2027.42/176463/attilio_1.pdf?sequence=1}*{Proposition III.2.14}, the left-hand vertical map in the diagram above is surjective. The bottom horizontal map is a homeomorphism by \cite{https://deepblue.lib.umich.edu/bitstream/handle/2027.42/176463/attilio_1.pdf?sequence=1}*{ Proposition III.1.31}. Hence the right-hand vertical map is also surjective.
\end{proof}

We will really be interested in the \'{e}tale topology, which we will return to later, once we have introduced formal models.

Notice that in Proposition \ref{prop:rigid-geometry-tuple}, we in particular proved that the presheaf \[\mathbf{QCoh} \colon \mathbf{Afnd}_K \to \mathbf{CAlg}(\mathbf{Pr}_{st}^L)\] satisfies Cech descent. \comment{Later on we shall also see that the hypotheses of Theorem \ref{thm:hyperdescent} apply, so that we also have \(\tau^{rat}\)-hyperdescent. }

Finally, since the topology is also subcanonical, the functor \[\mathbf{QCoh} \colon \mathbf{Stk}(\bA, \tau_\bA^{rat})^\op \to \mathbf{CAlg}(\mathbf{Pr}_{st}^L)\] is a sheaf for the effective epimorphism topology that is right Kan extended from affinoids. In symbols, for \(X \in \mathbf{Stk}(\bA, \tau_\bA^{rat})^\op\), we have an equivalence \[\mathbf{QCoh}(X) \simeq \varprojlim_{\mathsf{Spec}(A) \to X} \mathbf{QCoh}(\mathsf{Spec}(A))\] of symmetric monoidal \(\infty\)-categories.



We will return to rigid analytic geometry from the perspective of formal schemes in Section \ref{sec:formal-geom}. Here we will also introduce the faithfully flat and \'{e}tale topologies in rigid geometry.

\subsection{Overconvergent analytic geometry}\label{subsec:overconv}

In this subsection, let \(R\) denote an arbitrary, nontrivially valued Banach ring. We consider the \emph{Washnitzer algebra}

\[
W_n(r)^\dagger  = \begin{cases}
\varinjlim_{\varrho > r} R \gen{\varrho_1^{-1} x_1 , \dotsc, \varrho_n^{-1} x_n} \qquad R \text{ non-Archimedean }\\
\varinjlim_{\varrho > r} R \{\varrho_1^{-1} x_1 , \dotsc, \varrho_n^{-1} x_n\} \qquad R \text{ archimedean},
\end{cases}
\] of \emph{overconvergent analytic functions} on a polydisc of radius \(r\), where the colimit is taken in the category of complete bornological \(R\)-algebras. We shall restrict $r$ to lie in $|K|_{>0}$.

\begin{proposition}[\cite{ben2022analytification}]
    The natural map 
    $$\mathrm{Sym}(R)\rightarrow W_{1}(r)$$
    is a homotopy epimorphism.
\end{proposition}

In particular, the functor
$$W_{1}(-):\underline{|K|_{>0}}^{op}\rightarrow\mathbf{DAlg}^{cn}(\bD(R)^\heartsuit)$$
defines an $\underline{|K|_{>0}}$-generating class of algebras of homotopy polynomial type.

An \emph{admissible dagger affinoid algebra} is a quotient \(W_n(r)^\dagger/I\) by an admissible ideal taken in the category \(\mathsf{CBorn}_R\) of complete bornological modules.

\begin{definition}\label{def:derived-dagger-stein}
A \emph{derived dagger affinoid algebra} is an object \(A \in \mathbf{CAlg}(\mathbf{D}_{\geq 0}(R))\) satisfying the following properties:
\begin{enumerate}
\item \(\pi_0(A)\) is an admissible dagger affinoid algebra;
\item each \(\pi_n(A)\) is a finitely generated \(\pi_0(A)\)-module. 
\end{enumerate}
\end{definition}  

Denote by \(\mathbf{Afnd}_R^\dagger\) the full subcategory of \(\mathbf{CAlg}(\mathbf{D}_{\geq 0}(\mathsf{CBorn}_R))\) of derived dagger affinoid algebras, and let \(\bA_R^\dagger\) denote its opposite category. In a manner completely analogous to Proposition \ref{prop:rigid-geometry-tuple}, the tuple \((\mathbf{Aff}_{\bD_{\geq 0}(R)}, \tau^{rat}, \mathbf{P}^{o-sm}, \bA_R^\dagger)\) is a relative geometry tuple, so that \[(\mathbf{Aff}_{\bD_{\geq 0}(R)}, \tau_{\bA_R^\dagger}^{rat}, \mathbf{P}_{\bA_R^\dagger}^{o-sm}, \bA_R^\dagger)\] is a strong relative geometry tuple. We shall refer to stacks relative to this relative geometry tuple as \emph{overconvergent analytic stacks}. 

\comment{We now isolate a full subcategory of \emph{derived dagger analytic spaces} \(\mathbf{dAn}_R\) - these are contained in the \(1\)-schemes for the relative geometry tuple associated to dagger affinoids. More concretely, we glue derived dagger affinoid algebras as follows: \edit{\textcolor{red}{I suppose I can erase the following definition as this is contained in the subsection "Derived spaces and the underlying topological space"?}}

\begin{definition}\label{def:global-analytic}
\begin{enumerate}
\item For a derived dagger affinoid \(X = \mathsf{Spec}(A) \in \bA_R^\dagger\),  we call a monomorphism \(U \subseteq X\) in \(\mathbf{Stk}(\bA_R, \tau_{\vert \bA_R^\dagger}^{rat})\) an \emph{analytic subspace} if there is a small family of representable functors \(U_i = \mathsf{Spec}(A_i) \in \bA_R^\dagger\) and an epimorphism \(\coprod_i U_i \to U\) such that the induced maps \(\mathsf{Spec}(A_i) \to \mathsf{Spec}(A)\) in \(\bA_R^\dagger\) are derived strong rational localisations. 
\item For an arbitrary \(X  \in \mathbf{Stk}(\bA_R^\dagger, \tau_{\vert \bA_R^\dagger})\), an inclusion \(U \subseteq X\) is called an \emph{open immersion} if for every map \(\mathsf{Spec}(A) \to X\) from a representable sheaf, the pullback \(U \times_X \mathsf{Spec}(A) \to \mathsf{Spec}(A)\) is an analytic subspace.
\item  An object \(X \in \mathbf{Stk}(\bA_R^\dagger, \tau_{\vert \bA_\C^\dagger})\) is called a \emph{derived dagger analytic space} if it admits a cover \((U_i \to X)\) by a small family of open immersions, where each \(U_i = \mathsf{Spec}(A_i) \in \mathbf{A}_R^\dagger\). We denote by \(\mathbf{dAn}_R\) full \(\infty\)-subcategory of \(\mathbf{Stk}(\bA_R^\dagger, \tau_{\vert \bA_R^\dagger})\) generated by derived complex analytic spaces. 
\item We call a derived dagger analytic space \(X \in \mathbf{dAn}_\C\) \emph{discrete} if it admits a small cover \(\coprod_{i\in I}\mathsf{Spec}(A_i) \to X\) by discrete dagger affinoid algebras \(A_i\). 
\end{enumerate}  
\end{definition}

\begin{lemma}\label{lem:colim}
Any derived dagger analytic space \(X \in \mathbf{dAn}_R\) is a colimit of derived dagger affinoid spaces. 
\end{lemma}

\begin{proof}
Consider the effective epimorphism \(Y = \coprod_i \mathsf{Spec}(A_i) \to X\) in \(\mathbf{Stk}(\bA_R^\dagger, \tau_{\bA_R^\dagger})\)  induced by a cover of derived dagger affinoid subspaces. By the universality of colimits, we know that \(Y \times_X Y\) is equivalent in \(\mathbf{Stk}(\bA_R^\dagger, \tau_{\bA_R^\dagger})\) to the coproduct \(\coprod_{i,j} \mathsf{Spec}(A_i) \times_X \mathsf{Spec}(A_j)\), so that we have a coequaliser diagram \[\coprod_{i,j} \mathsf{Spec}(A_i) \times_X \mathsf{Spec}(A_j) \rightrightarrows \coprod_{i} \mathsf{Spec}(A_i) \to X.\] Choosing a cover of \(X\) by derived dagger affinoids, and using that such affinoids are closed under fibre products, it is easy to see that the fibre product \(\mathsf{Spec}(A_i) \times_X \mathsf{Spec}(A_j)\) is a derived complex analytic space. Therefore by definition, we again have a cover \(\mathsf{Spec}(A_{ijk}) \to \mathsf{Spec}(A_i) \times_X \mathsf{Spec}(A_j)\) by dagger affinoid spaces, so that \(X\) is part of the coequaliser diagram \[\coprod_{ijk} \mathsf{Spec}(A_{ijk}) \rightrightarrows \coprod_{i} \mathsf{Spec}(A_i) \to X\] in \(\mathbf{Stk}(\bA_R^\dagger, \tau_{\bA_R^\dagger})\) as required.
\end{proof}

Let \(X \in \mathbf{dAn}_R\) be a derived dagger analytic space, and let \(\mathbf{An}_R(X)\) be the category of open immersions of \(X\). We define the \emph{small analytic site of \(X\)} as the category \(\mathbf{An}_R(X)\) with covering sieves generated by small families of open immersions \((U_i \to X)_i\) such that \(\coprod_i U_i \to X\) is an effective epimorphism. 

\begin{lemma}\label{lem:good-colim}
Let \(X \in \mathbf{dAn}_R\). Then we have an equivalence \[X \simeq \underset{\mathsf{Spec}(A_i) \to X \in \mathbf{An}(X)}\varinjlim \mathsf{Spec}(A_i)\] in \(\mathbf{Stk}(\bA_R^\dagger, \tau_{\vert \bA_R^\dagger})\). 
\end{lemma}

\begin{proof}
In the colimit presentation of Lemma \ref{lem:colim}, choose \(I\) to be the indexing set of all open immersions \(\mathsf{Spec}(A_i) \to X\) and \(K_{ij}\) the set of open immersions \(\mathsf{Spec}(A_{ijk}) \to \mathsf{Spec}(A_i \times_X A_j)\), where \(\mathsf{Spec}(A_{ijk}) \to \mathsf{Spec}(A_i)\) and \(\mathsf{Spec}(A_{ijk}) \to \mathsf{Spec}(A_j)\) are \'etale. 
\end{proof}}

Over a non-Archimedean, nontrivially valued Banach field \(K\), by \cite{grosse2000rigid}*{Theorem 2.27}, the full subcategory \(\mathbf{Rig}_K^{proper}\) of partially proper rigid analytic spaces is equivalent to the full subcategory of \(\mathbf{dAn}_K\) of discrete dagger analytic spaces that are \emph{partially proper}. Consequently, we get a fully faithfully embedding \[\mathbf{Rig}_K^{proper} \overset{\subset}\longrightarrow \mathbf{dAn}_K \subset \mathbf{Stk}_{geom}(\bA_K^\dagger, \tau_{\bA_K^\dagger}^{rat}).\] 
In the next subsection, we shall consider the complex analogue of this result to show that derived complex Stein spaces embed fully faithfully inside \(\mathbf{dAn}_\C\).

\comment{

Over either an Archimedean or non-Archimedean ring we can also define the \textit{Stein algebra of polyradius} \(\rho\in(\mathcal{R}\cup\{\infty\})^{n}\) by
$$\mathcal{O}(D^{n}_{<\rho,R})=\varprojlim_{r<\rho}W_n(r)^\dagger $$

\begin{proposition}\cite{ben2022analytification}
    The natural map 
    $$\mathrm{Sym}(\mathbb{I})\rightarrow \mathcal{O}(D^{n}_{<\rho,R})$$
    is a homotopy epimorphism.
\end{proposition}

Again the functor
$$\mathcal{O}(D^{n}_{<-,R}):(\underline{\mathbb{R}_{>0}\cup\{\infty\}})^{op}\rightarrow\mathbf{DAlg}^{cn}(\bD(R)^\heartsuit)$$
defines a $\underline{\mathbb{R}_{>0}\cup\{\infty}$-generating class of algebras of homotopy polynomial type.

\textcolor{red}{What is the difference between using derived Steins and derived dagger Steins? In complex analytic geometry, Porta-Yue-Yu's derived complex analytic spaces seem to embed inside geometric stacks relative to derived Steins. Also does one get something interesting in the non-Archimedean case?}

For an admissible affinoid dagger algebra \(A\), denote by \(A_n(r)^\dagger = A \haotimes_R W_n(r)^\dagger\) the completed tensor product with a Washnitzer algebra. Note that the tensor product of two affinoid dagger algebras is an affinoid dagger algebra. We call a morphism \(f \colon A \to B\) of affinoid dagger \(K\)-algebras is called a \emph{Weierstrass localisation} if \[B \cong A_n(r)^\dagger/(X_1 - f_1, \dotsc, X_n - f_n)\] for some \(f_1, \dotsc, f_n \in A\), \(n \in \N\) and polyradius \(r\). For an affinoid dagger algebra, let \(\mathcal{M}(A)\) denote its bornological Berkovich spectrum, defined as the set of equivalence classes of bounded homomorphisms \(A \to \bar{R}\) for valued field extensions \(\bar{R} \supset K\) equipped with the weak topology. This is a contravariant, fully faithful functor from the category of dagger Stein algebras to the category of locally ringed \(G\)-topological spaces. 

\begin{definition}
A \emph{dagger Stein algebra} is a sequential limit \[\cdots \to A_{n+1} \to A_n \to \cdots \to A_0\] of affinoid dagger algebras, where the structure maps \(A_{n+1} \to A_n\) are Weierstrass localisations, and \(\mathcal{M}(A_n)\) is contained in the interior of \(\mathcal{M}(A_{n+1})\).
\end{definition}

\begin{definition}\label{def:derived-dagger-stein}
A \emph{derived dagger Stein algebra} is an object \(A \in \mathbf{CAlg}(\mathbf{D}_{\geq 0}(\mathsf{CBorn}_R))\) satisfying the following properties:
\begin{enumerate}
\item \(\pi_0(A)\) is a dagger Stein algebra;
\item each \(\pi_n(A)\) is coadmissible as a \(\pi_0(A)\)-module. 
\end{enumerate}
\end{definition}  

Denote by \(\mathbf{dStnAlg}_R^\dagger\) the full subcategory of \(\mathbf{CAlg}(\mathbf{D}_{\geq 0}(\mathsf{CBorn}_R))\) of derived dagger Stein algebras. In the next subsection, we shall use these algebras as the building blocks of complex analytic spaces.}

\subsection{Complex analytic geometry}

In this subsection, we realise derived complex manifolds as a full subcategory of schemes relative to the geometry context \[(\mathbf{Aff}_{\bD_{\geq 0}(\C)},  \tau_{\bA_\C^\dagger}^{rat}, \mathbf{P}_{\bA_\C^\dagger}^{et}, \bA_\C^\dagger),\] where \(\bA_\C^\dagger\) is the full subcategory of derived affinoid algebras. We describe the building blocks more concretely.

The algebra \[\C\{\varrho_1^{-1} x_1 , \dotsc, \varrho_n^{-1} x_n\} = \setgiven{\sum_{\alpha \in \N^n}c_\alpha x^{\abs{\alpha}}}{\sum_{\alpha} \abs{c_\alpha} \varrho^{\alpha} < \infty}\] is a Banach algebra with respect to the \(l^1\)-norm describing analytic functions on an open polydisc \(D_n(r)\) of polyradius \(\varrho\) centered at \(0\). To get a well-behaved integration theory, we allow for analytic functions converging on larger radii (bounded below by a fixed polyradius \(r \in \R_{> 0}^n\)) which are allowed to depend on the function \[W_n(r)^\dagger = \varinjlim_{\varrho> r} \C\{\varrho_1^{-1} x_1 , \dotsc, \varrho_n^{-1} x_n\},\] where the colimits are taken along the subspace inclusions \(\C\{\varrho_1^{-1} x_1 , \dotsc, \varrho_n^{-1} x_n\} \hookrightarrow \C\{\varphi_1^{-1} x_1 , \dotsc, \varphi_n^{-1} x_n\}\) for \(\varrho > \varphi\). The algebra \(W_n(r)^\dagger\) is canonically a complete bornological \(\C\)-algebra. 

Now consider the closed polydisc \(\overline{D_n(r)} = \setgiven{z \in \C^n}{\abs{z} \leq r}\) centered at \(0\). We define the algebra of \emph{germs of holomorphic functions} on \(\overline{D_n(r)}\) as the bornological inductive limit \[\mathcal{O}(\overline{D(r)}) \defeq \varinjlim_{V \overset{\text{open}}\supset D(r)} \mathcal{O}(V) \cong \varinjlim_{\varrho > r} \mathcal{O}(D(\varrho)) \cong \varinjlim_{\varrho> r} \C\{\varrho_1^{-1} x_1 , \dotsc, \varrho_n^{-1} x_n\} = W_n(r)^\dagger.\] 

\begin{definition}\label{def:complex-stack}
A \emph{complex analytic stack} is a geometric stack relative to the strong relative geometry tuple \((\mathbf{Aff}_{\bD_{\geq 0}(\C)}, \tau_{\bA_\C^\dagger}^{rat}, \mathbf{P}_{\bA_\C^\dagger}^{o-sm}, \bA_\C^\dagger)\). 
\end{definition}

Exactly as in Proposition \ref{prop:smoothhtpymono}, one can proof the following.

\begin{proposition}\label{prop:smoothhtpymono}
     Let $f:X=\mathrm{Spec}(B)\rightarrow Y=\mathrm{Spec}(A)$ be a morphism of derived complex analytic stacks such that $f$ is in $n-\mathbf{P}^{o-sm}$ and is a homotopy monomorphism. Then $f$ is, derived rationally locally, a derived rational localisation.
\end{proposition}

The next result is similar to much of the content of \cite{clausenscholze3}*{Lecture XI}.

\begin{lemma}
    The tuple \(\mathbf{G}_{\mathbb{C}}=(\mathbf{Aff}_{\bD_{\geq 0}(\C)}, \tau_{\bA_\C^\dagger}^{rat}, \mathbf{P}_{\bA_\C^\dagger}^{o-sm}, \bA_\C^\dagger)\) is stably Berkovich. In fact, for any morphisms of qcqs schemes $X\rightarrow Z$, $Y\rightarrow Z$, the natural map
    $$|X\times_{Z}Y|_{Ber}\rightarrow |X|_{Ber}\times_{|Z|_{Ber}}|Y|_{Ber}$$
    is a homeomorphism.
\end{lemma}

\begin{proof}
 Let $X=\mathrm{Spec}(A)$ be affine. We may assume that it is discrete. Consider its Berkovich spectrum $\mathcal{M}(A)$ (\cite{bambozzi2014generalization}*{Definition 2.1.20}). By Ostrowski's Theorem, this is just the set $\mathrm{Hom}(A,\mathbb{C})$ of bounded algebra homomorphisms from $A$ to $\mathbb{C}$. Write $A\cong W_{n}(r)^{\dagger}\big\slash (f_{1},\ldots,f_{m})$. As sets, we have that
 $$\mathcal{M}(W_{n}(r)^{\dagger})\cong\overline{D}^{n}(r)$$
 is the closed polydisc of polyradius $r$. Indeed we have 
 $$\mathcal{M}(W_{n}(r)^{\dagger})=\varprojlim_{\varrho > r}\mathcal{M}(\C\{\varrho_1^{-1} x_1 , \dotsc, \varrho_n^{-1} x_n\}).$$
 Now, any map of algebras $\C\{\varrho_1^{-1} x_1 , \dotsc, \varrho_n^{-1} x_n\})\rightarrow\mathbb{C}$ is necessarily bounded by $1$. Since $\C\{\varrho_1^{-1} x_1 , \dotsc, \varrho_n^{-1} x_n\})$ is the symmetric algebra in the non-expanding Banach category on $\mathbb{C}_{\varrho_1}\oplus\ldots\oplus\mathbb{C}_{\varrho_n}$, a bounded map of algebras $\C\{\varrho_1^{-1} x_1 , \dotsc, \varrho_n^{-1} x_n\})\rightarrow\mathbb{C}$ is classified by a non-expanding map of modules $$\mathbb{C}_{\varrho_1}\oplus\ldots\oplus\mathbb{C}_{\varrho_n}\rightarrow\mathbb{C},$$
 which is the same as a tuple $(z_{1},\ldots,z_{n})$ of elements of $\mathbb{C}$, with $|z_{i}|\le \varrho_{i}$. Taking the limit over all polyradii $\varrho>r$ give the closed polydisc of polyradius $r$. Then $\mathcal{M}(W_{n}(r)^{\dagger}\big\slash I)$ is just the subset of $\overline{D}^{n}(r)$ consisting of those points $(x_{1},\ldots, x_{n})$ such that $f_{i}(x_{1},\ldots,x_{n})=0$ for all $1\le i\le m$. Let us now work out the topology. By definition, the topology is the weakest topology such that for all $a\in A$, the function
 $$\chi_{a}:\mathrm{Hom}(A,\mathbb{C})\rightarrow\mathbb{C},$$
 $\phi\mapsto\phi(a)$
 is continuous. Under the bijection with $Z(f_{1},\ldots,f_{m})$, this is the weakest topology such that for each $g\in W_{n}(r)^{\dagger}\big\slash (f_{1},\ldots,f_{m})$, the map
 $$Z(f_{1},\ldots,f_{m})\rightarrow\mathbb{C},$$
 $$(x_{1},\ldots,x_{n})\mapsto f(x_{1},\ldots,x_{n})$$
 is continuous. Consider the space $Z(f_{1},\ldots,f_{m})$ equipped with its topology as a compact subspace of $\mathbb{C}^{n}$. The evaluation maps $(x_{1},\ldots,x_{n})\mapsto g(x_{1},\ldots,x_{n})$ are continuous for this topology, so the identity $Z(f_{1},\ldots,f_{n})\rightarrow\mathcal{M}( W_{n}(r)^{\dagger}\big\slash (f_{1},\ldots,f_{m}))$ is continuous. By \cite{bambozzi2016dagger}*{Lemma 4.3}, the space $\mathcal{M}( W_{n}(r)^{\dagger}\big\slash (f_{1},\ldots,f_{m}))$ is compact, so the identity map is in fact a homeomorphism. Under this identification, inclusions $\mathcal{M}(B)\rightarrow\mathcal{M}(A)$, where $A\rightarrow B$ is a rational localisation, correspond to inclusions of certain Stein compact spaces. This includes, for example, closed polydiscs of finite radius. Complements of finite unions of closed polydiscs clearly form a basis for the topology on $Z(f_{1},\ldots,f_{n})$. Thus, we have $|X|_{Ber}\cong Z(f_{1},\ldots,f_{n})$ with the Euclidean topology.

 Finally, let $f:X\rightarrow Z$, $g:Y\rightarrow Z$ be morphisms of qcqs schemes. We claim that 
 $$|X\times_{Z}Y|_{Ber}\rightarrow |X|_{Ber}\times_{|Y|_{Ber}}|Z|_{Ber}$$
 is a homeomorphism. By the qcqs assumption, we may assume that everything is affine. Moreover, since both sides are compact Hausdorff spaces, it is enough to show that the map is a bijection. Put $X=\mathrm{Spec}(A)$, $Z=\mathrm{Spec}(C)$, $Y=\mathrm{Spec}(B)$. We then have
 $$|X\times_{Z}Y|_{Ber}\cong\mathrm{Hom}(A\hat{\otimes}_{C}B,\mathbb{C})\cong\mathrm{Hom}(A,\mathbb{C})\times_{\mathrm{Hom}(C,\mathbb{C})}\mathrm{Hom}(B,\mathbb{C})\cong |X|_{Ber}\times_{|Z|_{Ber}}|Y|_{Ber}$$
 as required.
\end{proof}

Let $A=W_{n}(r)^{\dagger}\big\slash(f_{1},\ldots,f_{m})$. Consider a localisation $A\hat{\otimes} W_{k}^{\dagger}(\delta)\big\slash (g_{0}X_{1}-g_{1},\ldots,g_{0} X_{k}-g_{k})$. This space is the subset of $Z(f_{1},\ldots,f_{m})\times\overline{D}^{k}(\delta)$ consisting of those points $(\underline{x},\underline{X})$ such that $g_{0}(\underline{x})X_{i}=g_{i}(\underline{x})$ for all $i$. Equivalently, this is the subset of $Z(f_{1},\ldots,f_{n})$ consisting of those $\underline{x}$ such that $g_{i}(\underline{x})\le\delta_{i} g_{0}(\underline{x})$.

Recall that classically, a complex analytic space refers to a locally ringed space which is locally given by \((X, \mathcal{O}_X)\), where \(X \subseteq D_n(\varrho)\) is the vanishing locus of a finite collection of holomorphic functions \(f_1, \dotsc, f_n\) on an open polydisc \(D_n(\varrho)\) in \(\C^n\), and \(\mathcal{O}_X(U)\) is the subring of holomorphic functions on \(U \subseteq X\) (i.e., locally it is a Stein space). 


In \cite{ben2024perspective} the category of derived Stein algebras $\mathbf{dStnAlg}_{\C}^{\dagger}$ was defined as the full subcategory of $\mathbf{DAlg}^{cn}_{\C}(\mathrm{CBorn}_{\C})$ consisting of those derived algebras $A$ such that $\pi_{0}(A)$ is isomorphic to the nuclear Fr\'{e}chet algebra $\mathcal{O}(X)$ of holomorphic functions on a Stein space $X$, and each $\pi_{n}(A)$ is the nuclear Fr\'{e}chet module of global sections of a coherent sheaf $\mathcal{F}_{n}$ on $\mathcal{O}(X)$. In loc. cit. it was shown that there is a topology $\tau_{St}^{rat}$ on $\mathbf{dStnAlg}_{\C}^{\dagger}$, generated by so-called open immersions of derived Stein algebras. To summarise, a collection of maps 
$$\{\mathrm{Spec}(B_{i})\rightarrow\mathrm{Spec}(A)\}_{i\in\mathcal{I}}$$
is a cover in $\tau_{St}^{rat}$ precisely if 
\begin{enumerate}
    \item 
    if $\pi_{0}(A)=\mathcal{O}(X)$, and $\pi_{0}(B_{i})=\mathcal{O}(U_{i})$, then the collection of maps of algebras $\{\mathcal{O}(X)\rightarrow\mathcal{O}(U_{i})\}_{i\in\mathcal{I}}$ corresponds to a cover 
    $$\{U_{i}\rightarrow X\}$$
    of $X$ by open Stein subspaces,
    \item 
    each map $A\rightarrow B_{i}$ is a derived strong homotopy epimorphism.
\end{enumerate}

Write $\mathbf{dStn}^{\dagger}_{\C}=(\mathbf{dStnAlg}^{\dagger}_{\C})^{op}$. For $A\cong\mathcal{O}(X)$ a (discrete) Stein algebra, we denote by $\mathrm{Coad}(A)\subset\mathrm{Mod}(A)$ the full subcategory consisting of those $A$-modules which are isomorphic to the module of global sections of some coherent sheaf on $X$, equipped with its canonical Fr\'{e}chet bornology. Such objects are called \textit{coadmissible }$A$-\textit{modules}. There is a more general theory of coadmissible modules over a general Fr\'{e}chet-Stein algebra, as well as derived variants, developed in \cite{ben2024perspective}*{Section 5.4}. An important point in that section, is the existence of \textit{geometric presentations} of derived dagger Steins. Let $A$ be a derived dagger Stein. A geometric presentation of $A$ is a tuple $(A,A_{n},\overline{A}_{n})$ where $n\in\mathbb{N}_{0}$ and 

\begin{enumerate}
\item 
each $\overline{A}_{n}$ is a derived dagger affinoid,
\item 
each $A_{n}$ is a derived dagger Stein,
\item 
there is a sequence of maps 
$$\ldots\rightarrow\overline{A}_{n+1}\rightarrow A_{n+1}\rightarrow\overline{A}_{n}\rightarrow \ldots$$
and compatible maps $A\rightarrow\overline{A}_{n}$, $A\rightarrow A_{n}$ such that
\begin{enumerate}
\item 
the natural maps
$$A\rightarrow\mathbf{lim}_{n}A_{n}$$
and
$$A\rightarrow\mathbf{lim}_{n}\overline{A}_{n}$$
are equivalences;
    \item all of the maps $A\rightarrow A_{n}$, $A\rightarrow\overline{A}_{n}$, $\overline{A}_{n+1}\rightarrow\overline{A}_{n}$, $A_{n+1}\rightarrow A_{n}$, $\overline{A}_{n}\rightarrow A_{n}$ are derived strong homotopy epimorphisms;
    \item 
    the induced maps $\mathcal{M}(\pi_{0}(\overline{A}_{n}))\rightarrow\mathcal{M}(\pi_{0}(\overline{A}_{n+1}))$ are Weierstrass localisations of dagger affinoid with image contained in the interior of $\mathcal{M}(\pi_{0}(\overline{A}_{n+1}))$.
\end{enumerate}
\end{enumerate}

Finally, it was shown that if $f:A\rightarrow B$ is a map of dagger Stein algebras, then we can find presentations $(A,A_{n},\overline{A}_{n})$, $(B,B_{n},\overline{B}_{n})$, such that $f$ is the limit of compatible sequences of maps $f_{n}:A_{n}\rightarrow B_{n}$, $\overline{f}_{n}:\overline{A}_{n}\rightarrow\overline{B}_{n}$.

\begin{lemma}
    A collection of maps of derived dagger Steins
    $$\{\mathrm{Spec}(B_{i})\rightarrow\mathrm{Spec}(A)\}$$
    is a cover in $\tau_{St}^{rat}$ if and only if 
    \begin{enumerate}
        \item 
        each $A\rightarrow B_{i}$ is a derived strong homotopy epimorphism,
        \item 
        if $M\in\mathrm{Coad}(A)$ is such that $\pi_{0}(B_{i})\hat{\otimes}_{\pi_{0}(A)}M\cong 0$ for all $i$ then $M\cong 0$.
    \end{enumerate}
\end{lemma}

\begin{proof}
We first note that by \cite{BaBK}*{ Theorem 5.7}, homotopy epimorphisms of discrete dagger Stein algebras correspond to open immersions of the corresponding Stein spaces. Thus if $A\rightarrow B$ is a homotopy epimorphism of derived dagger Steins, it is a derived strong homotopy epimorphism if and only if $\pi_{0}(A)\rightarrow\pi_{0}(B)$ is a homotopy epimorphism, i.e. $\pi_{0}(A)\rightarrow\pi_{0}(B)$ corresponds to an open immersion of Stein spaces.

Now, suppose 
$$\{\mathrm{Spec}(B_{i})\rightarrow\mathrm{Spec}(A)\}$$
is a cover in $\tau_{St}^{rat}$. Then each 
$$\{\mathrm{Spec}(\pi_{0}(B_{i}))\rightarrow\mathrm{Spec}(\pi_{0}(A)\}$$
corresponds to a usual cover of Steins by open Steins. The conservativity condition for coadmissible modules follows from descent for coherent sheaves on Steins.

Conversely, suppose we have a collection of morphisms $\{\mathrm{Spec}(B_{i})\rightarrow\mathrm{Spec}(A)\}$ such that each $A\rightarrow B_{i}$ is a derived strong homotopy monomorphism, and if $M\in\mathrm{Coad}(A)$ is such that $\pi_{0}(B_{i})\hat{\otimes}_{\pi_{0}(A)}M\cong 0$ for all $i$ then $M\cong 0$. Write $\pi_{0}(B_{i})=\mathcal{O}(U_{i})$ and $\pi_{0}(A)=\mathcal{O}(X)$. What remains to prove is that $\coprod_{i}U_{i}\rightarrow X$ is surjective. Let $x\in X$, and consider the corresponding maximal ideal $\mathfrak{m}_{x}$. Then $\mathcal{O}(X)\big\slash\mathfrak{m}_{x}$ is a coadmissible module. Thus for some $i$, $\mathcal{O}(U_{i})\big\slash(\mathcal{O}(U_{i})\mathfrak{m}_{x})$ is non-zero. This implies that $x$ lies in this $U_{i}$.
\end{proof}

To see why using dagger affinoids as building blocks of complex analytic spaces is a reasonable idea we prove the following:

\begin{theorem}\label{thm:complexgeomdagger}
The category \(\mathbf{dStn}^{\dagger}_{\C}\) embeds fully faithfully inside the full subcategory \(\mathbf{dAn}_\C \subset \mathbf{Stk}_{geom}(\mathbf{Aff}_{\bD_{\geq 0}(\C)}, \tau^{rat}, \mathbf{P}^{o-sm}, \bA_\C^\dagger)\) of derived analytic spaces.
\end{theorem}

 Let $\widetilde{(-)}$ denote the composite functor
$$\mathbf{Aff}_{\bD_{\geq 0}(\C)}\rightarrow\mathbf{Stk}(\mathbf{Aff}_{\bD_{\geq 0}(\C)})\rightarrow\mathbf{Stk}(\mathbf{A}_\C^{\dagger})\rightarrow\mathbf{Stk}(\mathbf{Aff}_{\bD_{\geq 0}(\C)}),$$
where the first functor is the Yoneda embedding, the second is restriction, and the third is Kan extension. We claim that the restriction of this functor to $\mathbf{dStn}^{\dagger}_{\C}$ is fully faithful, and the essential image consists of $0$-schemes.

\begin{proof}[Proof of Theorem \ref{thm:complexgeomdagger}]

   Let \(\mathsf{Spec}(A) \in \mathbf{dStn}^{\dagger}_\C\) be an affine Stein space. Consider a presentation $(A,A_{n},\overline{A}_{n})$. We claim that the collection of maps $\{\mathrm{Spec}(\overline{A}_{n})\rightarrow \widetilde{\mathrm{Spec}(A)}\}$ is a $0$-atlas. Let $f:A\rightarrow B$ be a map with $B$ dagger affinoid. By \cite{ben2024perspective}*{Corollary 5.4.112} there exists $m$ such that $f$ factors through
    $$A\rightarrow \overline{A}_{n}\rightarrow B$$
    for all $n\ge m$. For $n\ge m$, we then have the following pushout diagram
    \begin{displaymath}
        \xymatrix{
        A\ar[d]\ar[r] & \overline{A}_{n}\ar[d]\ar[r] & B\ar[d]\\
        \overline{A}_{n}\ar[r] & \overline{A}_{n}\ar[r] & B
        }
    \end{displaymath}
        where we have used that $A\rightarrow \overline{A}_{n}$ is a homotopy epimorphism. For $n<m$ we still have 
        $$B\hat{\otimes}^{\mathbb{L}}_{A}\overline{A}_{n}\cong B\hat{\otimes}^{\mathbb{L}}_{\overline{A}_{m}}\overline{A}_{n}$$
        which is a derived dagger affinoid by \cite{ben2024perspective}*{Corollary 4.5.66}. Clearly $\{\mathrm{Spec}(B\hat{\otimes}^{\mathbb{L}}_{A}\overline{A}_{n})\rightarrow\mathrm{Spec}(B)\}$ is a cover of $\mathrm{Spec}(B)$, as in fact some of the $\mathrm{Spec}(B\hat{\otimes}^{\mathbb{L}}_{A}\overline{A}_{n})$ are just $\mathrm{Spec}(B\hat{\otimes}^{\mathbb{L}}_{A}\overline{A}_{n})$. Moreover all maps $\mathrm{Spec}(B\hat{\otimes}^{\mathbb{L}}_{A}\overline{A}_{n})\rightarrow\mathrm{Spec}(B)$ are rational localisations. 

        \comment{Note that this also shows we have $\widetilde{\mathrm{Spec}}(A)\cong\colim_{n}\mathrm{Spec}(\overline{A}_{n})$, so that the functor \(\widetilde{(-)}\) above factorises through \(\mathbf{Stk}(\bA_\C^\dagger)\). }
        
        Finally, to show that this functor is fully faithful, let $B$ another dagger Stein algebra. We have 
    \begin{align*}
            \mathbf{Map}(\widetilde{\mathrm{Spec}}(A),\widetilde{\mathrm{Spec}}(B))&\cong\mathbf{Map}(\colim_{n}\mathrm{Spec}(\overline{A}_{n}),\widetilde{\mathrm{Spec}}(B))\\
            &\cong\lim_{n}\mathbf{Map}(\mathrm{Spec}(\overline{A}_{n}),\widetilde{\mathrm{Spec}}(B))\\
            &\cong\lim_{n}\mathbf{Map}(B,\overline{A}_{n})\\
            &\cong\mathbf{Map}(B,A),
\end{align*} which completes the proof.
\end{proof}

Next we will sketch a proof that topologies match up.
Now let $A$ be a derived dagger Stein, and $M\in\mathbf{Coad}(A)$. Recall this means that $M\cong\mathbf{lim}_{n}\overline{A}_{n}\otimes^{\mathbb{L}}_{A}M$, where each $\overline{A}_{n}\otimes^{\mathbb{L}}_{A}M$ is such that each $\pi_{k}(\overline{A}_{n}\otimes^{\mathbb{L}}_{A}M)$ is a finitely presented $\pi_{0}(\overline{A}_{n})$-module for all $n$ and $k$. Equivalently each $\pi_{k}(M)\in\mathrm{Coad}(\pi_{0}(A))$. Moreover by \cite{ben2024perspective}*{Section 4.5}, we have
$$B\hat{\otimes}_{A}^{\mathbb{L}}M\cong\mathbf{lim}_{n}(B_{n}\hat{\otimes}_{A_{n}}^{\mathbb{L}}M).$$
In particular
$$\widetilde{\mathrm{Spec}}(f)^{*}(M)\cong B\hat{\otimes}_{A}^{\mathbb{L}}M.$$

\begin{proposition}
If $g:A\rightarrow B$ is a homotopy epimorphism derived dagger Stein algebras, then the corresponding map $f:\widetilde{\mathrm{Spec}}(B)\rightarrow\widetilde{\mathrm{Spec}}(A)$ is a homotopy monomorphism.
\end{proposition}

\begin{proof}
This follows immediately frpm 
$$\widetilde{\mathrm{Spec}}(B)\times_{\widetilde{\mathrm{Spec}}(A)}\widetilde{\mathrm{Spec}}(B)\cong\widetilde{\mathrm{Spec}}(B\hat{\otimes}_{A}^{\mathbb{L}}B).$$
\comment{
Let $g:A\rightarrow B$ and $h:B\rightarrow C$ be derived strong maps of derived dagger Stein algebras. Write $\pi_{0}(A)\cong\mathcal{O}(\mathbb{C}^{l})\big\slash I$. $\pi_{0}(B)\cong\mathcal{O}(Y)$, $\pi_{0}(C)\cong\mathcal{O}(Z)$. Let the coordinates of $\mathbb{C}^{l}$ be $z_{1},\ldots,z_{l}$. Consider the ideal $T$ of $\pi_{0}(B)\hat{\otimes}^{\mathbb{L}}\pi_{0}(C)\cong\mathcal{O}(Y\times Z)$ generated by $\{f(z_{j})-g(z_{j})\}$. This is a finitely generated, and hence closed, ideal. The map
$$\pi_{0}(B)\hat{\otimes}\pi_{0}(A)\hat{\otimes}\pi_{0}(C)\rightarrow T$$
is clearly onto set-theoretically, and therefore an admissible epimorphism since everything is Fr\'{e}chet. Thus the relative tensor product is just $\mathcal{O}(Y\times Z)\big\slash T$, which is derived Stein. It remains to prove that each $\pi_{n}(B\hat{\otimes}_{A}^{\mathbb{L}}C)$ is coadmissible. Here is where we use derived strength. We have
$$\pi_{*}(B)\hat{\otimes}^{\mathbb{L}}_{\pi_{*}(A)}\pi_{*}(C)\cong(\pi_{0}(B)\otimes_{\pi_{0}(A)}\pi_{0}(C))\otimes$$
}
\comment{
 Pick presentations $(A,A_{n},\overline{A}_{n})$, $(B,B_{n},\overline{B}_{n})$, and $(C,C_{n},\overline{C}_{n})$. We may assume that $g$ and $h$ have presentations $(g,g_{n},\overline{g}_{n})$ and $(h,h_{n},\overline{h}_{n})$. Byu 
Now we have 
$$B\hat{\otimes}_{A}^{\mathbb{L}}C\cong\mathbf{lim}_{n}(B\hat{\otimes}_{A}C_{n})\cong\mathbf{lim}_{m,n}(B_{m}\hat{\otimes}_{A}C_{n}).$$
By cofinality this limit is equivalent to 
$$\mathbf{lim}_{n}(B_{n}\hat{\otimes}_{A}C_{n})\cong\mathbf{lim}_{n}(B_{n}\hat{\otimes}_{A_{n}}C_{n})$$
where we have used that $A\rightarrow A_{n}$ is a homotopy epimorphism. Now there are cofinal maps of diagram
$$\overline{B}_{n}\hat{\otimes}^{\mathbb{L}}_{\overline{A}_{n}}\overline{C}_{n}\rightarrow B_{n}\hat{\otimes}_{A_{n}}^{\mathbb{L}}C_{n}$$,
$$B_{n+1}\hat{\otimes}_{A_{n+1}}^{\mathbb{L}}C_{n+1}\rightarrow\overline{B}_{n}\hat{\otimes}^{\mathbb{L}}_{\overline{A}_{n}}\overline{C}_{n}.$$
Thus we have 
$$B\hat{\otimes}_{A}^{\mathbb{L}}C\cong\mathbf{lim}_{n}\overline{B}_{n}\hat{\otimes}^{\mathbb{L}}_{\overline{A}_{n}}\overline{C}_{n}.$$
Moreover 
$$(B\hat{\otimes}_{A}^{\mathbb{L}}C,B_{n}\hat{\otimes}^{\mathbb{L}}_{A_{n}}C_{n},\overline{B}_{n}\hat{\otimes}^{\mathbb{L}}_{\overline{A}_{n}}\overline{C}_{n})$$
is a presentation of $B\hat{\otimes}_{A}^{\mathbb{L}}C$ as a derived dagger Stein.

As for the homotopy epimorphism claim, for $A\rightarrow B$ a map of derived dagger Steins, we have that $B\hat{\otimes}_{A}^{\mathbb{L}}B$ is a derived dagger Stein, and $$\widetilde{\mathrm{Spec}}(B)\times_{\widetilde{\mathrm{Spec}}(A)}\widetilde{\mathrm{Spec}}(B)\cong\widetilde{\mathrm{Spec}}(B\hat{\otimes}_{A}^{\mathbb{L}}B).$$
Since $\widetilde{\mathrm{Spec}}(-)$ is fully faithful, and hence conservative, the claim immediately follows.
}
\end{proof}

\begin{lemma}
  If a map $g:A\rightarrow B$ of derived dagger Steins is a derived strong homotopy epimorphism then $f:\widetilde{\mathrm{Spec}}(B)\rightarrow\widetilde{\mathrm{Spec}}(A)$ is in $\mathbf{P}^{rat}$.
\end{lemma}

\begin{proof}
Suppose $g$ is a derived strong homotopy epimorphism. Consider an arbitrary map $A\rightarrow C$ with $C$ derived dagger affinoid. We pick presentations $(A,A_{n},\overline{A}_{n})$ and $(B,B_{n},\overline{B}_{n})$ as derived dagger Steins and, by reordering if necessary, assume that the map $A\rightarrow C$ factors through each $\overline{A}_{n}$. $C=\overline{A}_{n}$.
Now $\widetilde{\mathrm{Spec}}(\overline{A}_{n}\hat{\otimes}^{\mathbb{L}}_{A}B)$ has an atlas given by $\{\mathrm{Spec}(\overline{B}_{l}\hat{\otimes}_{A}^{\mathbb{L}}\overline{A}_{n})\rightarrow\widetilde{\mathrm{Spec}}(\overline{A}_{n}\hat{\otimes}^{\mathbb{L}}_{A}B)\}$. This is just the base-change of the atlas $\{\mathrm{Spec}(\overline{B}_{l})\rightarrow\widetilde{\mathrm{Spec}}(B)\}.$ Consider the map $\overline{A}_{n}\rightarrow\overline{B}_{l}\hat{\otimes}_{A}^{\mathbb{L}}\overline{A}_{n}$. Now for $k\ge l,n$ $\overline{B}_{l}\hat{\otimes}^{\mathbb{L}}\overline{A}_{n}\cong\overline{B}_{l}\hat{\otimes}_{\overline{A}_{k}}^{\mathbb{L}}\overline{A}_{n}$. Thus it is a dagger affinoid. Moreover $\overline{A}_{n}\rightarrow\overline{B}_{l}\hat{\otimes}_{A}^{\mathbb{L}}\overline{A}_{n}$ is a derived strong homotopy epimorphism of derived dagger affinoids. Thus on the level of $\pi_{0}$ it
 is necessarily a dagger affinoid open immersion by \cite{bambozzi2016dagger}*{Theorem 5.7}. \qedhere

        \end{proof}

Define a topology $\tilde{\tau}_{rat}$ on $\mathbf{Stk}(\mathbf{A}^{\dagger}_{\C})$ by declaring that a collection of maps $\{\mathcal{X}_{i}\rightarrow\mathcal{Y}\}$ is a cover precisely if 
\begin{enumerate}
    \item 
    each $\mathcal{X}_{i}\rightarrow\mathcal{Y}$ is a homotopy monomorphism of stacks and is in $\mathbf{P}^{rat}$;
    \item 
   $\coprod_{i}\mathcal{X}_{i}\rightarrow\mathcal{Y}$ is an epimorphism of stacks.
\end{enumerate}

\begin{lemma}\label{lem:pratderstrong}
    Let $\widetilde{\mathrm{Spec}}(f):\widetilde{\mathrm{Spec}}(B)\rightarrow\widetilde{\mathrm{Spec}}(A)$ be in $\mathbf{P}^{rat}$. Then $A\rightarrow B$ is derived strong. Moreover there is a cover $\mathrm{Spec}(B_{i})\rightarrow\mathrm{Spec}(B)$ in $\tau^{tat}_{St}$ such that each $A\rightarrow B_{i}$ is a homotopy epimorphism.
\end{lemma}

\begin{proof}[Sketch proof]
First suppose that we have a presentation
$$(f:A\rightarrow B,f_{n}:A_{n}\rightarrow B_{n},\overline{f}_{n}:\overline{A}_{n}\rightarrow\overline{B}_{n}).$$
We claim it suffices to prove that each $\overline{f}_{n}$ is derived strong. Indeed suppose this is the case. We want 
$$\pi_{m}(B)\cong\pi_{0}(B)\hat{\otimes}^{\mathbb{L}}_{\pi_{0}(A)}\pi_{m}(A).$$
But we have 
$$\pi_{m}(\overline{B}_{n})\cong\pi_{0}(\overline{B}_{n})\hat{\otimes}^{\mathbb{L}}_{\pi_{0}(B)}\pi_{m}(B)$$
since $B\rightarrow\overline{B}_{n}$ is derived strong. Also by the assumption that $\overline{A}_{n}\rightarrow\overline{B}_{n}$ is derived strong and the fact that each $A\rightarrow\overline{A}_{n}$ is derived strong, we have
\begin{multline*}
    \pi_{m}(\overline{B}_{n})\cong\pi_{0}(\overline{B}_{n})\hat{\otimes}^{\mathbb{L}}_{\pi_{0}(\overline{A}_{n})}\pi_{m}(\overline{A}_{n})\\
    \cong\pi_{0}(\overline{B}_{n})\hat{\otimes}^{\mathbb{L}}_{\pi_{0}(A)}\pi_{m}(A)\cong\pi_{0}(\overline{B}_{n})\hat{\otimes}^{\mathbb{L}}_{\pi_{0}(B)}\pi_{0}(B)\hat{\otimes}^{\mathbb{L}}_{\pi_{0}(A)}\pi_{m}(A).
\end{multline*}
Now descent for coadmissible $B$-modules on the cover $\{\mathrm{Spec}(\overline{B}_{n})\rightarrow\widetilde{\mathrm{Spec}}(B)\}$ gives that 
$$\pi_{m}(B)\cong\pi_{0}(B)\hat{\otimes}^{\mathbb{L}}_{\pi_{0}(A)}\pi_{m}(A).$$

Now let us prove that each $\overline{f}_{n}$ is derived strong. By assumption the map $A\rightarrow\overline{B}_{n}$ factors through $\overline{A}_{n}$. Thus we have 
$$\overline{B}_{n}\hat{\otimes}^{\mathbb{L}}_{A}\overline{A}_{n}\cong B_{n}\hat{\otimes}^{\mathbb{L}}_{\overline{A}_{n}}\overline{A}_{n}\cong\overline{B}_{n}.$$
Hence by pullback the map $\mathrm{Spec}(\overline{B}_{n})\rightarrow\mathrm{Spec}(\overline{A}_{n})$ is a homotopy monomorphism which is in $\mathbf{P}^{rat}$. At long last we are reduced to proving that if $\mathrm{Spec}(D)\rightarrow\mathrm{Spec}(C)$ is a map of dagger affinoids which is in $\mathbf{P}^{rat}$ then it is derived strong. But this follows by a descent argument similar to the one above. 

For the final claim, by the derived strong result above we may restrict to the discrete case. For each $n$, we may find a cover $\{\mathrm{Spec}(D_{n_{i}})\rightarrow\mathrm{Spec}(\overline{B}_{n})\}$ by dagger affinoids such that each composition $\mathrm{Spec}(D_{n_{i}})\rightarrow \mathrm{Spec}(\overline{A}_{n})$ is a homotopy epimorphism, i.e., $\mathcal{M}(\overline{A}_{n})\rightarrow\mathcal{M}(D_{n_{i}})$ is a localisation. Now, topologically, it is clear that we may find a cover of $\mathcal{M}(B)$ by opens $U_{i}$ such that the composition $U_{i}\rightarrow\mathcal{M}(B)\rightarrow\mathcal{M}(A)$ is an open immersion.
\qedhere

\comment{
    Consider the map $\mathrm{Spec}(\overline{A}_{n})\rightarrow\widetilde{\mathrm{Spec}}(A)$. There is an atlas $\{\mathrm{Spec}(C_{j})\rightarrow\mathrm{Spec}(\overline{A}_{n})\times_{\widetilde{\mathrm{Spec}}(A)}\widetilde{\mathrm{Spec}}(B)$ such that each map $\mathrm{Spec}(C_{j})\rightarrow\mathrm{Spec}(\overline{A}_{n})$ is a rational localisation, and hence a derived strong map. 
    }
\end{proof}

\begin{lemma}
  A collection of maps $\{f_{i}:\mathrm{Spec}(B_{i})\rightarrow \mathrm{Spec}(A)\}$  of derived Steins is a cover in $\tau^{rat}_{St}$ if and only if $\{\widetilde{\mathrm{Spec}}(A_{i})\rightarrow\widetilde{\mathrm{Spec}}(B)\}$ is a cover in $\tilde{\tau}_{rat}.$ 
  
  \comment{
  Conversely, if $\{\widetilde{\mathrm{Spec}}(B_{i})\rightarrow\widetilde{\mathrm{Spec}}(A)\}$ is a cover in $\tilde{\tau}_{rat}$, then for each $\widetilde{\mathrm{Spec}}(B_{i})$ there is a cover of Stein spaces $\{\mathrm{Spec}(\mathcal{O}(V_{j_{i}}))\rightarrow \mathrm{Spec}(\mathcal{O}(U_{i}))\}$ for each $i$, such that the corresponding collection
$$\{\mathrm{Spec}(\mathcal{O}(V_{j_{i}}))\rightarrow \mathrm{spec}(\mathcal{O}(X))\}_{i,j_{i}}$$
is a cover of derived Steins by open derived Stein subspaces.
}
\end{lemma}

\begin{proof}
Suppose $\{f_{i}:\mathrm{Spec}(B_{i})\rightarrow \mathrm{Spec}(A)\}$ is a cover in $\tau^{rat}_{St}$. Let $f_{i}=\mathrm{Spec}(g_{i})$
    Each $A\rightarrow B_{i}$ is a derived strong homotopy epimorphism, so $\{\widetilde{\mathrm{Spec}}(\mathcal{O}(U_{i}))\rightarrow\widetilde{\mathrm{Spec}}(\mathcal{O}(X))\}$ consists of elements of $\mathbf{P}^{rat}$. It remains to prove that the map $\coprod_{i\in\mathcal{I}}\widetilde{\mathrm{Spec}}(B_{i})\rightarrow\widetilde{\mathrm{Spec}}(A)$ is an effective epimorphism. Fix presentations $(A,A_{n},\overline{A}_{n})$, $(B_{i},B_{i,m},\overline{B}_{i,m})$. Fix $n$ and consider the collection $$\{\mathrm{Spec}(\overline{A}_{n}\hat{\otimes}_{A}\overline{B}_{i,m})\rightarrow\mathrm{Spec}(\overline{A}_{n})\}_{i\in\mathcal{I},m\in\mathbb{N}}.$$ Each of the maps 
    $$\overline{A}_{n}\rightarrow\overline{A}_{n}\hat{\otimes}_{A}\overline{B}_{i,m}$$
    is a derived strong homotopy epimorphism of derived dagger affinoids, and hence is a rational localisation of derived dagger affinoids.  Let $X,U_{i}$ be Stein spaces such that $\pi_{0}(A)\cong\mathcal{O}(X),\pi_{0}(B_{i})\cong\mathcal{O}(U_{i})$. Write $\mathcal{O}(X)_{n}=\pi_{0}(A_{n})$, and $\mathcal{O}(U_{i})_{m}=\pi_{0}(B_{i,m})$.     
    By \cite{BaBK}, we have $\mathcal{M}(\mathcal{O}(U_{i}))=\bigcup_{m}\mathcal{M}(\mathcal{O}(U_{i})_{m})$. We also have $\mathcal{M}(\pi_{0}(\overline{A}_{n}\hat{\otimes}_{A}\overline{B}_{i,m}))=\mathcal{M}(\mathcal{O}(X)_{n})\cap\mathcal{M}(\mathcal{O}(U_{i})_{m})$, where the intersection is taken inside $\mathcal{M}(\mathcal{O}(X))$. Clearly then the collection $\{\mathcal{M}(\pi_{0}(\overline{A}_{n}\hat{\otimes}_{A}\overline{B}_{i,m}))\}_{i\in\mathcal{I},m\in\mathbb{N}}$ covers $\mathcal{M}(\mathcal{O}(X)_{n})$. By compactness of $\mathcal{M}(\mathcal{O}(X)_{n})$ there is a finite subcover $\{\mathcal{M}(\pi_{0}(\overline{A}_{n}\hat{\otimes}_{A}\overline{B}_{k_{n},m_{k_{n}}}))\}_{1\le k_{n}\le d_{n}}$. It follows immediately that 
$$\{\mathrm{Spec}(\overline{A}_{n}\hat{\otimes}_{A}\overline{B}_{k_{n},m_{k_{n}}})\rightarrow\mathrm{Spec}(\overline{A}_{n})\}_{1\le k_{n}\le d_{n}}$$
is a cover in the rational topology on dagger affinoids. We also get a commutative diagram
\begin{displaymath}
    \xymatrix{
\coprod_{1\le k_{n}\le d_{n}}\mathrm{Spec}(\overline{A}_{n}\hat{\otimes}_{A}\overline{B}_{k_{n},m_{k_{n}}})\ar[d]\ar[r] & \coprod_{i\in\mathcal{I}}\widetilde{\mathrm{Spec}}(B_{i})\ar[d]\\
\mathrm{Spec}(\overline{A}_{n})\ar[r] &\widetilde{\mathrm{Spec}}(A).  
    }
\end{displaymath}
Now the right-hand map being an effective epimorphism easily follows from the fact that
any map $\mathrm{Spec}(C)\rightarrow\mathrm{Spec}(A)$ with $\mathrm{Spec}(C)$ a derived dagger affinoid factors through some $\mathrm{Spec}(\overline{A}_{n})$. 

    \comment{
    Potentially after re-indexing, we may assume that $g_{i}=\mathbf{lim}_{n}(g_{i,n}:A_{n}\rightarrow B_{i,n})$ for each $i$.

    that we can also assume that $\{\mathrm{Spec}(\mathcal{O}(U_{i})_{m})\rightarrow\mathrm{Spec}(\mathcal{O}(X)_{m})\}$ is a cover for each $m$. Indeed each map $\mathcal{O}(X)_{m}\rightarrow\mathcal{O}(U_{i})_{m}$ is is a homotopy epimorphism by the $2$-out-of-$3$ property of homotopy epimorphism. It remains to show conservativity. Let $x\in\mathcal{M}(\mathcal{O}(X))$. Then $x$ must be in some $\mathcal{M}(\mathcal{O}(U_{i}))\subset\mathcal{M}(\mathcal{O}(X))$. But This $x$ must be contained in some $\mathcal{M}(\mathcal{O}(U_{i})_{m})$. Now 
    
    Let $h:\mathcal{O}(X)\rightarrow A$ be a map with $A$ a derived dagger affinoid.  Now $\mathcal{O}(X)\rightarrow A$ factors through some $\mathcal{O}(X)_{n}$. We claim that }

    Conversely, suppose that 
    $$\{\widetilde{\mathrm{Spec}}(f_{i}):\widetilde{\mathrm{Spec}}(B_{i})\rightarrow\widetilde{\mathrm{Spec}}(A)\}$$ is a cover in $\tilde{\tau}^{rat}$. By Lemma \ref{lem:pratderstrong} each map $f_{i}:A\rightarrow B_{i}$ is a derived strong homotopy epimorphism. Thus, base changing $\{\widetilde{\mathrm{Spec}}(f_{i}):\widetilde{\mathrm{Spec}}(B_{i})\rightarrow\widetilde{\mathrm{Spec}}(A)\}$ along the map $\widetilde{\mathrm{Spec}}(A)\rightarrow\widetilde{\mathrm{Spec}}(\pi_{0}(A))$, we find that 
    $$\{\widetilde{\mathrm{Spec}}(\pi_{0}(f_{i})):\widetilde{\mathrm{Spec}}(\pi_{0}(B_{i}))\rightarrow\widetilde{\mathrm{Spec}}(\pi_{0}(A))\}$$
    is also a cover in $\tau^{rat}$. 
    Write $\mathcal{O}(X)=\pi_{0}(A)$, $\mathcal{O}(U_{i})=\pi_{0}(B_{i})$. It remains to show that $\{U_{i}\rightarrow X\}$ is a cover of Stein spaces.
    
    Let $M\in\mathrm{Coad}(\mathcal{O}(X))$. We need to show that if
    $$\mathcal{O}(U_{i})\hat{\otimes}^{\mathbb{L}}_{\mathcal{O}(X)}M\cong 0$$
    for all $i$ then $M\cong 0$. But $M\in\mathbf{QCoh}(\widetilde{\mathrm{Spec}}(\mathcal{O}(X)))$. Thus by descent for quasi-coherent descent along effective epimorphisms, we have that if $\mathcal{O}(U_{i})\hat{\otimes}^{\mathbb{L}}_{\mathcal{O}(X)}M\cong 0$ for all $i$, then $M\cong 0$, as required.
\end{proof}

\subsubsection{A comparison with Porta and Yue Yu}

In \cite{ben2024perspective}*{Section 9.2.1.1} we defined the category $\mathbf{dStn}^{gf}_{K}$ of globally finitely embeddable Stein spaces to consists of those derived Stein spaces $X$, in the sense of Porta and Yue Yu, such that $t_{\le0}X$ is a globally finitely embeddable Stein space, and each $\pi_{n}(\mathcal{O}_{X})$ is a globally finitely generated coherent sheaf on $t_{\le0}X$. Recall that $t_{\le0}X$ being globally finitely generated means, by definition, that there is a closed embedding $t_{\le0}X\rightarrow K^{n}$ for some $n$, and the kernel of $\mathcal{O}(K^{n})\rightarrow\mathcal{O}(t_{\le0}X)$ is a globally finitely generated coherent sheaf. Let $\mathbf{dStnAlg}^{\dagger,gf}$ denote the full subcategory of $\mathbf{DAlg}(\mathsf{CBorn}_{\mathbb{C}})$ consisting of those derived algebras $A$ with $\pi_{0}(A)\cong\mathcal{O}(X)$ for some globally finitely embeddable Stein space $X$, and each $\pi_{n}(A)$ is a finitely generated $\pi_{0}(A)$-module. It was shown in \cite{ben2024perspective} that there is an equivalence of categories
$$(\mathbf{dStn}^{gf}_{\mathbb{C}})^{op}\cong\mathbf{dStnAlg}^{\dagger,gf},$$
and moreover that the topologies on either side match up. Thus the stacks of Porta and Yue Yu, and least those which are locally globally finitely embeddable, embed fully faithfully inside our derived analytic stacks. With a little bit more effort we can extend this to all of Porta and Yue Yu's stacks, but at least we can rest assured that we have all discrete complex analytic spaces.

\comment{
A derived Stein space $X\in\mathbf{dStn}_{K}$ is said to be $n$-\textit{truncated} if the canonical map $t_{\le n}X\rightarrow X$ is an equivalence. Let $\mathbf{dStn}_{K}^{trunc}$ denote the category of derived Stein spaces which are $n$-truncated for some $n$. Since for an arbitrary derived Stein space $X$ the canonical map
$$\colim_{n}t_{\le n}X\rightarrow X$$
is an equivalence, the category $\mathbf{dStn}_{K}^{trunc}$ is dense in $\mathbf{PreStk}(\mathbf{dStn}_{K}^{op})$. By restriction, we get a fully faithful functor

$$\mathbf{Stk}(\mathbf{dStn}_{K}^{op},\tau^{rat})\rightarrow \mathbf{Stk}((\mathbf{dStn}^{trunc}_{K})^{op},\tau^{rat}).$$

Now let $X$ be an $m$-truncated derived Stein space. Since any discrete Stein space may be covered by globally finitely embedabble subspaces, and any coherent sheaf is locally finitely generated, We may then find an open cover $\{U_{i}\rightarrow X\}$ with each $U_{i}$ an $m$-truncated derived Stein space, and such that each $U_{i}$ is globally finitely embeddable as a derived Stein space. It follows that the restriction functor
$$\mathbf{Stk}((\mathbf{dStn}^{trunc}_{K})^{op},\tau^{rat})\rightarrow \mathbf{Stk}((\mathbf{dStn}^{gf,trunc}_{K})^{op},\tau^{rat})$$
is an equivalence. Now we return to the fully faithful functor
$$(\mathbf{dStn}^{gf,trunc}_{K})^{op}\rightarrow\mathbf{DAlg}(\mathrm{CBorn}_{K}).$$

First consider the restricted Yoneda embedding 
$$\mathbf{dStn}_{K}\rightarrow\mathbf{Stk}(\mathbf{dStn}^{<\infty,f}_{K},\tau^{rat}).$$
We claim that this is fully faithful. First observe that for any derived Stein spaces $X,Y$ we have 
$$\mathbf{Map}_{\mathbf{dStn}_{K}}(t_{\le n}X,Y)\cong\mathbf{Map}_{\mathbf{dStn}_{K}}(t_{\le n}X,t_{\le n}Y).$$
We then have 
$$\mathbf{Map}_{\mathbf{dStn}_{K}}(X,Y)\cong\mathbf{Map}_{\mathbf{dStn}_{K}}(\colim_{n}t_{\le n}X,Y)\cong\mathbf{lim}_{n}\mathbf{Map}_{\mathbf{dStn}_{K}}(t_{\le n}X,t_{\le n}Y).$$

}

\comment{\begin{lemma}\label{lem:complex-mfd-dagger}
Any (classical) complex analytic space has a cover by dagger affinoid spaces. 
\end{lemma}

\begin{proof}
By definition, for any \(x\) in a complex analytic space \(X\), there is a neighbourhood \(U_x\) that is homeomorphic to an affine analytic set \[V_x = \setgiven{z \in U \subset \C^n}{f_1(z) = \dotsc = f_n(z) = 0, \quad f_1, \dotsc, f_n \in \mathcal{O}_{\C^n}(U)}.\]
Denoting by \(a\) the image of \(x\) in \(V_x\), there is a neighbourhood basis \((\overline{D_i(r_i)})_{i \in \N}\) of closed polydiscs centred at \(a\) in \(\C^n\). Therefore, the collection \((\overline{D_i(r_i)} \cap V_x)\) is a neighbourhood basis of \(a\) in \(V_x\), and these are clearly dagger affinoid spaces being in the Berkovich spectrum of the dagger algebra \(W_n(r_i)^\dagger/ (f_1, \dotsc, f_n) \cong \mathcal{O}(\overline{D_i(r_i)})/ (f_1, \dotsc, f_n)\).     
\end{proof}}

\comment{\subsubsection{The underlying topological space}\edit{\textcolor{red}{This is probably also already contained in what you're writing in Section 2?}}

We now associate a topological space to a derived complex analytic space. For \(X \in \mathbf{dAn}_\C\), consider the poset of open immersions \(\mathbf{An}_\C(X)\). By \cite{lurie2009higher}*{Section 6.4}, this is a locale with respect to the relation \(U \leq V\) if and only if \(U \subseteq V\) is a rational open immersion, with joins and meets given by coproduct and fibre products, respectively.  Furthermore, for any analytic subspace \(U \subseteq X\), we may refine a cover of \(X\) by dagger affinoids to a cover of \(U\) by dagger affinoids, so that \(\mathbf{An}_\C(X)\) has a basis given by objects in \(\mathbf{A}_\C^\dagger\), which are quasi-compact. Consequently, the locale \(\mathbf{An}_\C(X)\) is locally compact, so that by a version of the Stone duality, we may identify \(\mathbf{An}_\C(X)\) with its space of points. More precisely, setting \(\abs{X} = \mathrm{pt}(\mathbf{An}_\C(X))\) as the set of completely prime ultrafilters, we have a homeomorphism \[\Omega(\abs{X}) \simeq \mathbf{An}_\C(X),\] where \(\Omega\) is the functor from the category of  topological spaces to locales sending a topological space to its locale of open subsets. By general nonsense (see \cite{soor2024derived}*{Remark 2.28}), it follows that \(\abs{X}\) is a locally spectral topological space. Finally, for any \(\infty\)-category \(\mathbf{D}\) admitting small limits, we have an equivalence of \(\infty\)-categories \[\mathbf{Shv}(\mathbf{An}_\C(X), \mathbf{D}) \simeq \mathbf{Shv}(\abs{X}, \mathbf{D}).\]
classical in the sense of Definition \ref{def:classical}.
}

\comment{Finally, it remains to relate the space \(\abs{X}\) to the underlying topological space of a discrete complex analytic space. For an affinoid dagger algebra \(A\), let \(\mathcal{M}(A)\) denote its \emph{bornological Berkovich spectrum}, defined as the set of bounded multiplicative seminorms on \(A\) compatible with the Euclidean absolute value. We equip this set with the weak topology, that is, the smallest topology such that \(\mathcal{M}(A) \to \R_{\geq 0}\), \(\norm{-} \mapsto \norm{f}\) is continuous for all \(f \in A\). Since affinoid dagger algebras are filtered colimits of Banach algebras, the Berkovich spectrum is nonempty, compact and Hausdorff \cite{bambozzi2014generalization}*{Theorem 2.1.30}. Furthermore, it does not depend on the specific representation of \(A\) as a colimit.  

\begin{lemma}\label{lem:Berkovich-spectrum}\cite{clausenscholze3}*{Theorem 11.8(2)}
Let \(A = W_n(r)^\dagger/I\) be an affinoid dagger algebra, so that \(I\) is given by finitely many holomorphic functions, and \(X = V(I)\in \C^n\) their vanishing locus. Then \(\mathcal{M}(A)\) is homeomorphic to \(X\), where the latter is equipped with the Euclidean topology.   
\end{lemma}

\begin{corollary}\label{cor:MA=An}
Let \(X = \mathsf{Spec}(A)\) be a discrete affinoid dagger algebra. Then $|X|_{Ber}$ is homeomorphic to \(\mathcal{M}(A)\).
\end{corollary}

\begin{proof}
With \(X\) a discrete derived affinoid space, the poset of open immersions \(\mathbf{An}(X)\) is equivalent to the poset of open subsets of \(X\) in the analytic topology induced by \(\C^n\) using \cite{aristov2020open}. The result now follows from Lemma \ref{lem:Berkovich-spectrum}. 
\end{proof}}

\section{Formal geometry and \'{e}tale rigid analytic geometry as bornological geometry}\label{sec:formal-geom}

 In this section, we realise formal geometry in the sense of Raynaud as geometry relative to bornological modules over the ring of integers $R$ of a nontrivially valued, non-Archimedean Banach field $K$. We also show that the abstractly defined \'etale topology from Section \ref{sec:der-context} coincides with the classical \'etale topology. Much of this should be generalisable to Tate rings, or at least strongly Noetherian ones, using the work of \cite{bambozzi2020sheafyness}.

\subsection{Homotopical algebra for contracting Banach spaces}

We recall here \cite{kelly2016homotopy} a category that will play an important role in future both in this article and in future work. Consider the subcategory \(\mathsf{SNrm}_{R}^{\leq 1}\) whose objects are seminormed \(R\)-modules, and whose morphisms are contractive \(R\)-linear maps. Note that for non-Archimedean base rings, we always assume the seminorms to be non-Archimedean. Similarly, one can define categories \(\mathsf{Nrm}_{R}^{\leq 1}\) and \(\mathsf{Ban}_{R}^{\leq 1}\) of contracting normed and Banach \(R\)-modules. These are all complete and cocomplete categories.


We equip all these categories with the so-called \textit{strong exact structure} introduced in \cite{kelly2016homotopy}, whereby a kernel-cokernel pair
\begin{displaymath}
\xymatrix{
    0\ar[r]& A\ar[r]^{i} & B\ar[r]^{p} & C\ar[r] & 0
    }
\end{displaymath}
is a short exact sequence precisely if for all $c\in C$ there is $b\in B$ with $p(b)=c$ and $||b||=||c||$. The closed symmetric monoidal structures on $\mathsf{SNrm}_{R}$, $\mathsf{Nrm}_{R}$ and $\mathrm{Ban}_{R}$ restrict to ones on $\mathsf{SNrm}_{R}^{\le1}$, $\mathsf{Nrm}_{R}^{\le1}$ and $\mathsf{Ban}_{R}^{\le1}$. Moreover all of these categories have enough projectives. Indeed the objects $l^{1}(X)$ are still projective, and in fact for every object $V$ there is a strong epimorphism $l^{1}(|V|)\rightarrow V$, where $|V|$ denotes the underlying normed set of $V\setminus 0$.

A projective generating set in $\mathsf{SNrm}_{R}^{\le1}$ is given by 
$$\{R_{\delta}:\delta\in\mathbb{R}_{\ge0}\}$$
and in both $\mathsf{Nrm}_{R}^{\le1}$ and $\mathsf{Ban}_{R}^{\le1}$ is given by
$$\{R_{\delta}:\delta\in\mathbb{R}_{>0}\}.$$
The tensor product of two projectives is still projective. Moreover projectives are flat. Indeed tensoring with $R_{\delta}$ just rescales the norm by $\delta$ which is clearly an exact functor. It then remains to observe that direct sums are exact. 

The categories $\mathsf{SNrm}_{R}^{\le1}$, $\mathsf{Nrm}_{R}^{\le1}$, and $\mathsf{Ban}_{R}^{\le1}$ are also all quasi-abelian categories, and thus may be equipped with the quasi-abelian exact structure. 

We have the usual separation functor
$$\mathrm{Sep}:\mathsf{SNrm}_{R}^{\le1}\rightarrow\mathsf{Nrm}_{R}^{\le1},$$
and completion functor
$$\mathrm{Cpl}:\mathsf{Nrm}_{R}^{\le1}\rightarrow\mathsf{Ban}_{R}^{\le1}.$$
We also have their composition, the separated completion functor
$$\overline{\mathrm{Cpl}}=\mathrm{Cpl}\circ\mathrm{Sep}:\mathsf{SNrm}_{R}^{\le1}\rightarrow\mathsf{Ban}_{R}^{\le1}.$$

\begin{lemma}\label{lem:cplsepexact}
$\mathrm{Cpl}$ and $\overline{\mathrm{Cpl}}$ are exact for the quasi-abelian exact structures. $\mathrm{Sep}$, $\mathrm{Cpl}$, and $\overline{\mathrm{Cpl}}$ are exact for the strong exact structures.
\end{lemma}

\begin{proof}
The only non-trivial claim is that $\mathrm{Sep}$ is exact for the strong exact structure. Indeed, for the quasi-abelian exact structure this is false. So let 
\begin{displaymath}
\xymatrix{
0\ar[r] & X\ar[r]^{i} & Y\ar[r]^{q} & Z\ar[r] & 0
}
\end{displaymath}
be a strong exact sequence in $\mathrm{SNrm}_{R}^{\le1,nA}$. Let $[y]\in Y\big\slash\{y:\rho_{Y}(y)=0\}$ be such that $[q(y)]=0$. This means precisely that $\rho_{Z}(q(y))=0$. Let $\tilde{y}$ be such that $q(y)=q(\tilde{y})$ and $\rho_{Y}(\tilde{y})=\rho_{Z}(q(y))=0$. Then $y=\tilde{y}+i(x)$, and $[y]=[i(x)]$. $i$ is an isometry, so clearly $X\big\slash\{\rho_{X}(x)=0\}\rightarrow Y\big\slash\rho_{Y}(y)=0\}$. This completes the proof.
\end{proof}

For a semi-normed module $M$ and $\delta>0$ we write $B_{M}(0,\delta)$ for the open ball centred at zero with radius $\delta$ in $M$. Similarly for $\delta\ge0$ we write $\overline{B}_{M}(0,\delta)$ for the closed ball of radius $\delta$.

\begin{lemma}
    $\mathsf{SNrm}_{R}^{\le1}$ and $\mathsf{Nrm}_{R}^{\le1}$ are monoidal $\mathbf{AdMon}$-elementary exact categories. In particular, they present derived algebraic contexts \(\bD(\mathsf{SNrm}_{R}^{\le1})\) and \(\bD(\mathsf{Nrm}_{R}^{\le1})\) respectively. 
\end{lemma}

\begin{proof}
It remains to prove that for any collection of modules $\{M_{i}\}_{i\in\mathcal{I}}$ and any $\delta$, the natural map
$$\coprod_{i\in\mathcal{I}}\mathrm{Hom}(R_{\delta},M_{i})\rightarrow\mathrm{Hom}(R_{\delta},\bigoplus_{i\in\mathcal{I}}^{\le1}M_{i})$$
is an isomorphism. This amounts to showing that $B_{\bigoplus_{i\in\mathcal{I}}^{\le1}M_{i}}(0,\delta)=\coprod_{i\in\mathcal{I}}B_{M_{i}}(0,\delta)$.  But for $(m_{i})\in\coprod_{i\in\mathcal{I}}^{\le1}M_{i}$ with $\rho((m_{i}))\le\delta$, we have $\rho_{i}(m_{i})\le\delta$ for each $i$, and this completes the proof.
\end{proof}

The category $\mathsf{Ban}_{R}^{\le1}$ is not $\mathbf{AdMon}$-elementary. Because of the completion involved in the direct sum, mapping out of $R$ does not commute with coproducts. For the strong exact structure, it is nevertheless projectively monoidal with enough projectives, and kernels. By \cite{kelly2016homotopy} the projective model structure exists on $\mathrm{Ch}(\mathsf{Ban}_{R}^{\le1})$, is monoidal, satisfies the monoid axiom, and $\mathrm{Ch}(\mathsf{Ban}_{R}^{\le1})$ is a combinatorial model category. Thus it presents a closed, presentable, stable, symmetric monoidal $\infty$-category $\mathbf{D}(\mathsf{Ban}_{R}^{\le1})$. 

We shall use the notation $\mathbf{DAlg}^{cn}(\mathrm{Ban}_{R})$ to mean the full subcategory of $\mathbf{DAlg}^{cn}(\mathrm{SNrm}_{R})$ consisting of those derived algebras $X$ whose underlying object is in $\mathbf{Ch}(\mathrm{Ban}_{R})$.

\begin{proposition}
Filtered colimits in $\mathsf{SNrm}_{R}^{\le1,nA}$ and $\mathsf{Ban}_{R}^{\le1}$ are exact for the quasi-abelian exact structures. In fact in $\mathsf{SNrm}_{R}^{\le1}$ filtered colimits commute with kernels.
\end{proposition}

\begin{proof}
    By Lemma \ref{lem:cplsepexact} it suffices to prove the claim for $\mathsf{SNrm}_{R}^{\le1}$. Clearly filtered colimits commute with kernels. Let $F:\mathcal{I}\rightarrow\mathsf{SNrm}_{R}^{\le1}$ be a filtered diagram. Then $\colim_{i\in\mathcal{I}}F(i)$ is the semi-normed $R$-module whose underlying module is the algebraic colimit, and with semi-norm given by
   $$\rho([x])=\mathrm{inf}_{i\rightarrow j\in\mathcal{I}}\{\rho_{j}(\alpha_{ij}(x))\}$$
    where $x\in F(i)$, $\alpha_{ij}:F(i)\rightarrow F(j)$ is the structure map, and $\rho_{j}$ is the semi-norm on $F(j)$. Let $\eta:F\rightarrow G$ be a natural transformation of diagrams $\mathcal{I}\rightarrow\mathrm{SNrm}_{R}^{\le1,nA}$. Let $K:\mathcal{I}\rightarrow\mathrm{SNrm}_{R}^{\le1,nA}$ be defined by $K(i)=\mathrm{Ker}(\eta(i):F(i)\rightarrow G(i))$. Clearly the underlying module of $\colim_{i\in\mathcal{I}}K$ coincides with the underlying module of $\mathrm{Ker}(\colim_{i\in\mathcal{I}}F\rightarrow\colim_{i\in\mathcal{I}}G)$. Thus it remains to prove that the semi-norms coincide. Let $k\in K_{i}$. Then
    $$\rho([k])=\mathrm{inf}_{i\rightarrow j\in\mathcal{I}}\rho_{j}(\alpha_{ij}(k)).$$
    Since if $k\in K(i)$ then $\alpha_{ij}(k)\in K(j)$ this clearly coincides with the norm on the kernel.
\end{proof}

\begin{lemma}
    In $\mathsf{SNrm}_{R}^{\le1}$, $\mathsf{Nrm}_{R}^{\le1}$, and $\mathsf{Ban}_{R}^{\le1}$ each object of the form $R_{\delta}$ is $\aleph_{1}$-compact. In particular $\aleph_{1}$-filtered colimits are exact in the strong exact structures. 
\end{lemma}

\begin{proof}
Consider first the semi-normed case. Let $F:\mathcal{I}\rightarrow\mathsf{SNrm}_{R}^{\le1}$ be an $\aleph_{1}$-filtered diagram. We need to show that $\overline{B}_{\colim_{\mathcal{I}}F(i)}(0,\delta)=\colim_{\mathcal{I}}\overline{B}_{F(i)}(0,\delta)$. Let $x\in F(i)$ with $\rho([x])=\delta$. Therefore $\mathrm{inf}_{i\rightarrow j}\rho_{j}(\alpha_{ij}(x))=\delta$.  Pick a countable subset $\{i\rightarrow j_{k}:k\in\mathbb{N}\}$ such that $\mathrm{inf}_{k\in\mathbb{N}}\rho_{j}(\alpha_{ij_{k}}(x))=\delta$. Let $j_{k}\rightarrow u$ be maps such that the compositions $i\rightarrow j_{k}\rightarrow u$ all coincide for all $k$. Then $\rho_{u}(\alpha_{ui}(x))\le\rho_{j}(\alpha_{ij_{k}}(x))$ so $\rho_{u}(\alpha_{ui}(x))\le\delta$, and this completes the proof.

Next we claim that $\mathsf{Nrm}_{R}^{\le1}$ and $\mathsf{Ban}_{R}^{\le1}$ are closed in $\mathsf{SNrm}_{R}^{\le1}$ under $\aleph_{1}$-filtered colimits. If all $F(i)$ are separated then the proof above shows that $\colim_{i\in\mathcal{I}}F(i)$ is also separated. Next let $F(i)$ all be Banach spaces and let $([x_{n}])$ be a Cauchy sequence in $\colim_{i\in\mathcal{I}}F(i)$. By a similar argument to the above we may assume that all $x_{n}$ live in the same $M_{i}$. Moreover they converge there.
\end{proof}

\begin{corollary}
$\mathsf{SNrm}_{R}^{\le1}$, $\mathsf{Nrm}_{R}^{\le1}$, and $\mathsf{Ban}_{R}^{\le1}$ are locally $\aleph_{1}$-presentable.
\end{corollary}

\begin{proof}
    $\mathrm{LH}(\mathsf{SNrm}_{R}^{\le1})$ is a locally finitely presentable abelian category. In particular it is also locally $\aleph_{1}$-presentable. $\mathrm{SNrm}_{R}^{\le1,nA}$ is a reflective subcategory closed under $\aleph_{1}$-filtered colimits so is also locally $\aleph_{1}$-presentable. In turn $\mathrm{Nrm}_{R}^{\le 1}$ and $\mathrm{Ban}_{R}^{\le 1}$ and reflective subcategories closed under $\aleph_{1}$-filtered colimits, so are themselves $\aleph_{1}$-presentable. 
\end{proof}

\comment{
Now consider $\mathsf{Ban}_{R}^{\le1}$ with the strong exact structure. It is projectively monoidal with enough projectives, and kernels. By \cite{kelly2016homotopy} the projective model structure exists on $\mathrm{Ch}(\mathsf{Ban}_{R}^{\le1})$ and it is monoidal. Moreover it satisfies the monoid axiom, and $\mathrm{Ch}(\mathsf{Ban}_{R}^{\le1})$ is a combinatorial model category. Thus it presents a closed, presentable, stable, symmetric monoidal $\infty$-category $\mathbf{D}(\mathsf{Ban}_{R}^{\le1})$. 
    }

\comment{\subsubsection*{The quasi-abelian flat model structure}
\textcolor{red}{Jack comment: it turns out this is no longer needed. I would consider removing it for preprint version.}
Finally we consider the quasi-abelian model structures on our categories. Since filtered colimits in quasi-abelian exact structures are exact we get the following.

\begin{lemma}
In the quasi-abelian exact structures on $\mathsf{SNrm}_{R}^{\le1}$, $\mathsf{Nrm}_{R}^{\le1}$, and $\mathsf{Ban}_{R}^{\le1}$ filtered colimits of flat objects are flat.
\end{lemma}

Although there are not enough projectives for the quasi-abelian exact structure, the objects $R_{\delta}$ still generated our categories and are flat. Thus we have enough flat objects. Moreover, each category has exact filtered colimits are is $\aleph_{1}$-presentable. In particular they are purely $\aleph_{1}$-accessible and weakly elementary. By \cite{kelly2024flat} Corollary 6.3 the flat model structure exists on $\mathrm{Ch}(\mathrm{SNrm}_{R}^{\le1})$, $\mathrm{Ch}(\mathrm{Nrm}_{R}^{\le1})$, and $\mathrm{Ch}(\mathrm{Ban}_{R}^{\le1})$. Moreover they are all combinatorial monoidal model categories satisfying the monoid axiom. In particular they present closed, symmetric, presentable, monoidal $(\infty,1)$-categories which we denote by \(\mathbf{D}_f(\mathsf{SNrm}_R^{\le1})\), \(\mathbf{D}_f(\mathsf{Nrm}_R^{\le1})\) and $\mathbf{D}_f(\mathsf{Ban}_R^{\le1})$ respectively.}

We conclude with a discussion which will be useful for discussing adic completeness later.

\begin{lemma}
    Products in $\mathsf{SNrm}_{R}^{\le1}$ are strongly exact.
\end{lemma}

\begin{proof}
    Clearly they commute with kernels. It remains to prove they commute with cokernels. Now the product of a family $\{M_{i}\}_{i\in\mathcal{I}}$, $\prod_{i\in\mathcal{I}}^{\le1}M_{i}$ is the sub-module of the product $\prod_{i\in\mathcal{I}}M_{i}$ consisting of those tuples $(m_{i})_{i\in\mathcal{I}}$ where $\mathrm{sup}_{i\in\mathcal{I}}||m_{i}||<\infty$. The norm is given by the supremum of the norms of the entries. From this description it is clear that the product of strong epimorphisms is strong.
\end{proof}

\comment{
\begin{corollary}
    Let 
    $$\ldots\rightarrow M_{n+1}\rightarrow M_{n}\rightarrow\ldots\rightarrow M_{1}\rightarrow M_{0}$$ be a sequence in $\mathsf{SNrm}_{R}^{\le1}$ with each $M_{n+1}\rightarrow M_{n}$ being an admissible epimorphism for the quasi-abelian exact structure. Then the map
    $$\lim_{n}M_{n}\rightarrow\mathbb{R}\varprojlim_{n}M_{n}$$
    is an equivalence in $\bD_f(\mathsf{SNrm}_{R}^{\le1})$. 
\end{corollary}

\begin{proof}
We may pass to the left heart $\mathrm{LH}(\mathrm{SNrm}_{R}^{\le1})$. By \cite{kelly2024flat} this is a Grothendieck abelian category. Then the claim simply reduces to the Mittag-Leffler condition in such a category. 
\end{proof}
}

\subsubsection{Filtering by open balls}

Let $R$ be a non-Archimedean Banach ring. Note that for $\delta>0$, both the open ball $B_{R}(0,\epsilon)$ and the closed ball $\overline{B}_{R}(0,\epsilon)$ are closed ideals of $R$.

 \begin{proposition}\label{prop:Rlimballs}
Let $(\epsilon_{n})_{n\ge 1}$ be a sequence in $(0,1)$ decreasing to $0$. Then
$$R\cong \varprojlim_{n}^{\le1}R\big\slash B_{R}(0,\epsilon_{n})$$
\end{proposition}

\begin{proof}
Let $([r_{n}])$ be a sequence in $\lim^{\le1}R\big\slash B_{R}(0,\epsilon_{n})$. Then $r_{n+1}=r_{n}+s_{n}$ with $s_{n}\in B_{R}(0,\epsilon_{n})$. Put $s_{0}=0$. The sequence $(s_{n})$ converges to zero, so $\sum_{n=0}^{\infty}s_{n}$ converges to some $s$. The projection of $s$ to each $R\big\slash B_{R}(0,\epsilon_{n})$ is $r_{n}$, so the morphism is surjective. Note that $r$ being in the kernel would mean $|r|<\epsilon_{n}$ for all $n$, so $r=0$. Thus the projection is injective. It remains to compute norms. Let $\epsilon_{n}\ge|r|$ and either $|r|<\epsilon_{n-1}$ or $|r|\ge1$. For $m\le n$ its projection module to $R\big\slash B_{R}(0,\epsilon_{m})$ then has norm $|r|$ by the non-Archimedean property. If $|r|\ge1$ this is true for all $n$, and if $|r|<\epsilon_{n-1}$ then for $m\ge n-1$ its reduction modulo $B_{R}(0,\epsilon_{m})$ is zero. In either case the projection to the limit has norm $|r|$.  
\end{proof}

The situation we will typically have in mind is the following. A ring $R$ will contain some ideal $I$, and the norm on $R$ will be an $I$-adic one, say $|r|=\delta^{m}$ if $r\in I^{m}$ but $r\notin I^{m+1}$. where $0<\delta<1$. Then $B_{R}(0,\delta^{m})=I^{m}$.

\begin{proposition}
The functor
$$\mathsf{Norm}_{R}^{\le 1} \rightarrow\mathsf{Norm}_{R}^{\le 1}$$
sending 
$$M\mapsto M\big\slash B_{M}(0,\epsilon_{n})$$
is exact for the strong exact structure. 
\end{proposition}

\begin{proof}
Let $M_{\bullet}$ be acyclic. By definition this means that for each $\delta>0$, $B_{M_{\bullet}}(0,\delta)$ is acyclic. Next we claim that $M\rightarrow M\big\slash B_{M}(0,\delta)$ is always a strong epimorphism. Indeed if $m\in M$ satisfies $||m||\ge\delta$, then $||[m]||=||m||$ in $M\big\slash B_{M}(0,\delta)$. Thus the sequence
$$B_{M_{\bullet}}(0,\epsilon_{n})\rightarrow M_{\bullet}\rightarrow M_{\bullet}\big\slash B_{M_{\bullet}}(0,\epsilon_{n}),$$
is short exact, so the right-hand term is also acyclic. 
\end{proof}

\begin{proposition}
    Let $M \in \mathsf{Norm}_{R}^{\le1}$. Let $\epsilon_{n}$ be a sequence in $(0,1)$ decreasing to $0$. Then $\varprojlim_{n}M\big\slash B_{M}(0,\epsilon_{n})$ is isomorphic to the completion of $M$.
\end{proposition}

\begin{proof}
First we show that $\lim_{n}M\big\slash B_{M}(0,\epsilon_{n})$ is complete. It suffices to prove that each $M\big\slash B_{M}(0,\epsilon_{n})$ is complete. Let $(a_{n})$ be a Cauchy sequence. The limit of $||a_{m+1}-a_{m}||$ tends to zero. Thus for all sufficiently large $m$ we have $||a_{m+1}-a_{m}||<\epsilon_{n}$, which clearly means that $a_{m+1}=a_{m}$, and the sequence is eventually constant. In particular it is convergent. Now there is a natural map 
$$M\rightarrow\varprojlim_{n}M\big\slash B_{M}(0,\epsilon_{n})$$
given by the projections. As in Proposition \ref{prop:Rlimballs}, this morphism is an isometry. Again, as in Proposition \ref{prop:Rlimballs}, it is surjective whenever $M$ is complete. What remains to check is that the natural map 
$$M\big\slash B_{M}(0,\delta)\rightarrow\hat{M}\big\slash B_{\hat{M}}(0,\delta)$$
is an isometry. An argument similar to Proposition \ref{prop:Rlimballs} shows that it is surjective. Since $M\rightarrow\hat{M}$ is an isometry, it is in fact easy to see that $M\rightarrow\hat{M}\big\slash B_{\hat{M}}(0,\delta)$ is a strong epimorphism. The kernel is evidently $B_{M}(0,\delta)$.
\end{proof}

\subsubsection*{The convex subcategory}

A useful subcategory of $\bD(\mathsf{Ban}_{R}^{\le1})$ is the following. Let $\bD(\mathsf{Ban}_{R}^{\le1})^\circ$ denote the full subcategory of $\mathbf{D}(\mathsf{Ban}_{R}^{\le1})$ consisting of those complexes equivalent to ones which can be written as transfinite extensions of bounded below complexes of the form $l^{1}(V_{\bullet})$ where each $V_{n}$ is a discrete normed set. Likewise we define the full subcategory \(\bD(\mathsf{SNorm}_{R}^{\le1})^\circ\) (resp. \(\bD(\mathsf{Norm}_{R}^{\le1})^\circ\)) of \(\bD(\mathsf{SNorm}_{R}^{\le1})^\circ\) (resp. \(\bD(\mathsf{Norm}_{R}^{\le1})^\circ\)) consisting of of those complexes equivalent to ones which can be written as transfinite extensions of bounded below complexes of the form \(\tilde{l}^1(V) \cong \bigoplus_{V} R\) for a discrete semi-normed set \(V\) (resp. \(\overline{l}^1(V) \cong \bigoplus_{V} R\) for a discrete normed set \(V\)), where the coproduct is taken in the category \(\mathsf{SNorm}_R^{\leq 1}\) (resp. \(\mathsf{Norm}_R^{\leq 1}\)). We call these categories the \textit{convex subcategories} of \(\bD(\mathsf{Ban}_{R}^{\le1})\), \(\bD(\mathsf{Norm}_R^{\leq 1})\), and \(\bD(\mathsf{SNorm}_{R}^{\le1})\). The term convexity here really relates to rings of integers of number fields. In $\mathsf{Ind}(\mathsf{Ban}_{K})$ when $K$ is a field, the set $\{l^{1}(V):V\textrm{ a discrete set}\}$ is a generating set of projectives. (Recall that $V$ is discrete if the norm of every element is $1$).

\begin{proposition}
$\mathbf{D}(\mathsf{Ban}_{R}^{\le1})^\circ$ is a monoidal subcategory of $\mathbf{D}(\mathsf{Ban}_{R}^{\le1})$. 
\end{proposition}

\begin{proof}
This follows from the fact that $l^{1}(V)\hat{\otimes} l^{1}(W)\cong l^{1}(V\times W)$ and the product of discrete normed sets is discrete.
 \end{proof}

We now define functors which help us identify the convex categories defined above with ``completions'' of usual categories of modules. Recall that for a commutative, unital ring \(R\), we had defined in \ref{ex:initial} the derived algebraic context \(\mathbf{C}_R\) of derived \(R\)-algebras. Denote by \(\underline{R}\) the underlying ring. Define the functor \[\mathrm{disc}_{\underline{R}}^{\le1} \colon \mathbf{Mod}_{\underline{R}}^{ffg} \to \bD(\mathsf{Norm}_R^{\leq 1}), \quad \underline{R}^{\oplus n} \mapsto R^{\oplus n},\] which extends by sifted cocompletion to a colimit-preserving functor \[\mathrm{disc}_{\underline{R}}^{\le1} \colon \mathbf{Mod}_{\underline{R}} \to \bD(\mathsf{Norm}_R^{\leq 1}).\] It is easy to see that this functor is fully faithful. Similarly, if \(R\) is a Banach ring, we have a discretisation functor \[\widehat{\mathrm{disc}}^{\le1}_{\underline{R}} \colon \mathbf{Mod}_{\underline{R}} \to \bD(\mathsf{Ban}_R^{\leq 1})\] taking values in the contracting Banach derived category. 

The following describes the essential images of the discretisation functors.

 \begin{lemma}\label{lem:convex-norm}
     Let $R$ be a normed ring such that $|r|\le 1$ for all $r\in R$. Then the functor
      $$\mathrm{disc}_{R}^{\le1}:\mathbf{C}_{R}\rightarrow \mathbf{D}(\mathrm{Norm}_{R}^{\le1})^\circ$$
      is an equivalence of categories.
 \end{lemma} 

 \begin{proof}
There is a Quillen adjunction
$$\adj{\mathrm{disc}_{R}^{\le1}}{\mathrm{Ch}(\mathrm{Mod}_{R})}{\mathrm{Ch}(\mathrm{Norm}_{R}^{\le1})}{|-|}$$
where both sides are equipped with the projective model structure, and $|-|$ denotes the underlying module functor. $|-|$ is evidently exact. Thus for any dg-projective complex $P_{\bullet}$ in $\mathrm{Ch}(\mathrm{Mod}_{R})$, the unit $\eta_{P}$ of this adjunction is an equivalence. Hence $\mathrm{disc}_{R}^{\le1}$ is fully faithful. A dg-projective $P_{\bullet}$ may be written as a transfinite extension of bounded below complexes of projectives which are in fact free modules. Such a free module $\bigoplus_{v\in V}R$ is sent to $\overline{l}^{1}(V)$ by $\mathrm{disc}_{R}^{\le1}$. Since $\mathrm{disc}_{R}^{\le1}$ commutes with colimits, it follows that the essential image of $\mathrm{disc}_{R}^{\le1}$ is contained in  $\mathbf{D}(\mathrm{Norm}_{R}^{\le1})^\circ$. We now check that we get every object of  $\mathbf{D}(\mathrm{Norm}_{R}^{\le1})^\circ$. It suffices to prove that we get every bounded below complex of projectives of the form $l^{1}(V_{\bullet})$, with each $V_{n}$-discrete. 

 The underlying \(\underline{R}\)-module of a projective $l^{1}(V)$, with $V$-discrete,is $\bigoplus_{v\in V}R$. Moreover any $l^{1}(V)\rightarrow l^{1}(W)$ with \(V\) and \(W\) discrete normed sets is bounded by the assumption on the norm on \(R\). Hence we clearly have $l^{1}(V_{\bullet})\cong \mathrm{disc}_{R}^{\le1}(|l^{1}(V_{\bullet})|)$.
 \end{proof}

 \begin{corollary}
 Let $R$ be a Banach ring such that any sequence which converges to zero is eventually constant. Then the functor 
 $$\widehat{\mathrm{disc}}_{R}^{\le1}:\mathbf{C}_{R}\rightarrow\mathbf{D}(\mathsf{Ban}_{R}^{\le1})^\circ$$
 is an equivalence of categories.
 \end{corollary}

 \begin{proof}
     In this case $\overline{l}^{1}(V)=l^{1}(V)$. The same proof as Lemma \ref{lem:convex-norm} yields the desired equivalence. 
 \end{proof}

  \begin{lemma}
     Let $R$ be a Banach ring.  Then $\mathbf{D}(\mathsf{Ban}_{R}^{\le1})^\circ$ is closed under coproducts in $\mathbf{D}(\mathsf{Ban}_{R}^{\le1})$.  
 \end{lemma}

 \begin{proof}
      This follows from the fact that $\bigoplus_{i\in\mathcal{I}}l^{1}(V_{i})\cong l^{1}(\bigcup_{i\in\mathcal{I}}V_{i})$.
 \end{proof}

\subsection{Rings with pseudo-uniformisers}

 \begin{definition}
     Let $R$ be a Banach ring. A \textit{pseudo-uniformiser of }$R$ is an element $\pi\in R$ such that
     \begin{enumerate}
     \item 
     $|\pi^{n}|=|\pi|^{n}$ for all non-negative integers $n$,
     \item
     $|\pi|<1$.
     \end{enumerate}
 \end{definition}

 \begin{definition}
    Let $R$ be a Banach ring with pseudo-uniformiser $\pi$. A normed $R$-module $M$ is said to be $\pi$-adic if 
$$||m||=|\pi|^{\sf{max}\{l:m\in\pi^{l}M\}}.$$
$R$ is said to be a $\pi$-adic if it $\pi$-adic as a module over itself.
\end{definition}

\begin{proposition}\label{prop:algepi}
    Let $M$ and $N$ be $\pi$-adically normed. If $\phi:M\rightarrow N$ is surjective then it is an admissible epimorphism.
\end{proposition}

\begin{proof}
Note that with the $\pi$-adic norms all morphisms are bounded by $1$. 
     Let $n\in N$ with $||n||=|\pi|^{l}$, so that $n=\pi^{l}\tilde{n}$ for some $||\tilde{n}||=1$. Let $\tilde{m}$ be such that $g(\tilde{m})=\tilde{n}$. Let $m=\pi^{l}\tilde{m}$. Then $g(m)=n$ and $||m||\le|\pi|^{l}=||n||$. On the other hand $||n||=||g(m)||\le||m||$. Hence $||m||=||n||$. 
\end{proof}

\begin{corollary}
    Let $R$ be a $\pi$-adic Banach ring. A sequence
    $$0\rightarrow K\rightarrow M\rightarrow N\rightarrow 0$$
    of $\pi$-adic normed $R$-modules with $N$ $\pi$-torsion free is exact in the strong exact structure if and only if it is algebraically exact.
\end{corollary}

\begin{proof}
Note that with the $\pi$-adic norms all morphisms are bounded by $1$. Evidently, strong exactness implies algebraic exactness.
    Suppose the sequence is algebraically exact. By the Lemma above the map $g:M\rightarrow N$ is an admissible epimorphism Since $N$ is $\pi$-torsion free, the restriction of the norm on $M$ to $K$ is the $\pi$-adic one. 
\end{proof}

 Let $R$ be a Banach ring with pseudo-uniformiser $\pi$. We equip $\underline{R}$ with the $\pi$-\textit{adic semi-norm} whereby 
 $$|r|_{\pi}=|\pi|^{\max\{l\in\mathbb{N}_{0}\cup\{\infty\}: r\in\pi^{l}R\}}.$$
Let $R_{\pi}$ denote the ring $\underline{R}$ equipped with the $\pi$-adic semi-norm. This is a $\pi$-adic semi-normed ring. Suppose that for all $r\in R, |r|\le 1$. Let $r=\pi^{l}u$. Then $|r|\le |\pi|^{l}|u|\le|\pi|^{l}$. Hence the map $R_{\pi}\rightarrow R$ is bounded.

The category of non-Archimedean Banach rings has an initial object, $\mathbb{Z}_{triv}$, the integers equipped with the discrete norm.

\begin{definition}
Let $R$ be a Banach ring with pseudo-uniformiser $\pi$ such that for all $r\in R$ we have $|r|\le1$. $R$ is said to be \textit{almost adic} if the map $R_{\pi}\rightarrow R$ is an isomorphism in $\sf{Ban}_{\mathbb{Z}_{\mathrm{triv}}}$.
\end{definition}

\begin{remark}
    Note we \textit{are not} requiring the map be an isometry. Thus the induced map $\mathbf{D}(\sf{Ban}_{R}^{\le1})\rightarrow\mathbf{D}(\sf{Ban}_{R_{\pi}}^{\le1})$ will not be an equivalence in general. However, importantly for us, $\mathbf{D}(R)\rightarrow\mathbf{D}(R_{\pi})$ will be an equivalence.
\end{remark}

\begin{lemma}
    Suppose that $R$ is a Banach ring with pseudo-uniformiser $\pi$ such that $|r|\le1$ for all $r\in R$, and that for all $l$, $\overline{B}_{R}(0,|\pi|^{l})=(\pi^{l})$. Then the map $R_{\pi}\rightarrow R$ is an isomorphism of normed rings.
\end{lemma}

\begin{proof}
    We already know that the identity map $R_{\pi}\rightarrow R$ is bounded. We claim that $R\rightarrow R_{\pi}$ is also bounded. In fact we claim that $|r|_{\pi}\le\frac{1}{\pi}|r|$. Indeed, suppose $|r|_{\pi}=|\pi|^{l}$. Then $r\in B_{R}(0,|\pi|^{l})$ but $r\notin B_{R}(0,|\pi|^{l+1})$. Thus $|r|>|\pi|^{l+1}=|\pi| |r|_{\pi}$, as required.
\end{proof}

 \begin{example}
     Let $K$ be a non-trivially valued non-Archimedean Banach field, and $\mathcal{O}_{K}$ its ring of integers. Then $\mathcal{O}_{K}$ has a pseudo-uniformiser $\pi$ such that the norm on $\mathcal{O}_{K}$ is almost adic. Indeed pick any non-zero element $\pi$ with $|\pi|<1$. 
 \end{example}

 \begin{lemma}\label{lem:torsion free}
     Let $R$ be a $\pi$-adic Banach ring which is $\pi$ torsion-free. Then
     $$\mathbf{D}(\mathsf{Ban}_{R}^{\le1})^{\circ}$$
     coincides with the category consisting of those complexes isomorphic to a complex $M_{\bullet}$ where 
     \begin{enumerate}
     \item
      the norm on each $M_{n}$ is $\pi$-adic,
     \item 
     each $M_{n}$ is torsion-free.
     \end{enumerate}
 \end{lemma}

 \begin{proof}
 We write $\delta=|\pi|$.
 First let us show that for $V$ a discrete normed set, $l^{1}(V)$ is equipped with the $\pi$-adic norm. Let $(r_{v})_{v\in V}\in l^{1}(V)$. Note that $||(r_{v})_{v\in V}||=\mathrm{max}\{|r_{v}|_{v\in V}\}$. Thus $||(r_{v})_{v\in V}||=\delta^{l}$ for some $l$. Then $|r_{v_{0}}|=\delta^{l}$ for some $v_{0}$. Thus $r_{v_{0}}=\pi^{l}s_{v_{0}}$ with $|r|=1$. We write every $r_{v}=\pi^{l_{v}}s_{v}$ with $|s_{v}|=1$. Note $l_{v}\ge l$ for all $v$. Thus $(r_{v})=\pi^{l}(\pi^{l_{v}-l}s_{v})$. Moreover $||(\pi^{l_{v}-l}s_{v})||=1$ since $|(s_{v_{0}})|=1$. This shows that the norm coincides with the adic one. Moreover, $l^{1}(V)$ is clearly $\pi$-torsion-free.

 Now let $M$ be an $\pi$-adic Banach $R$-module. Let $D(M)=\{m:||m||=1\}$. We claim there exists a strong epimorphism $l^{1}(D(M))\rightarrow M$. Indeed let $m\in M$ with $||m||=\delta^{l}$ so that $m=\pi^{l}\tilde{m}$ for some $\tilde{m}\in D(M)$. The element $\pi^{l}\delta_{\tilde{m}}$ maps to $m$ and has norm $\delta^{l}$.

 Now consider a complex $M_{\bullet}$ of such modules. Since $\mathbf{Ch}(\mathrm{Ban}_{R}^{\le1,nA})^{\circ}$ is closed under colimits we may assume that it is bounded below.  Pick a strict epimorphism $f:l^{1}(D(M_{0}))\rightarrow M_{0}$. Let $K$ denote the kernel of this map. We claim that the $\pi$-adic norm on $M_{0}$ restricts to the $\pi$-adic norm on $K$. Indeed let $k\in K$ be non-zero, with norm $\delta^{n}$. Write $k=\pi^{n}m$ with $||m||=1$. Then $0=f(k)=\pi^{n}f(m)$. But $M_{0}$ is $\pi$-torsion free, so $f(m)=0$, and $m\in K$. Continuing in this way we construct a resolution of $M_{\bullet}$ by modules of the form $l^{1}(V)$, with $V$ discrete. \qedhere
 
 \end{proof}

 \begin{corollary}
         Let $R$ be $\pi$-adic and $\pi$-torsion free as a Banach ring. Then $\mathbf{D}(\mathsf{Ban}_{R}^{\le1})^\circ$ is closed under products in $\mathbf{D}(\mathsf{Ban}_{R}^{\le1})$ - in particular it is a reflective subcategory.
 \end{corollary}

 \begin{proof}
  Let ${M_{i}:i\in\mathcal{I}}$ be a collection of $\pi$-torsion free $\pi$-adic $R$-modules. The product $\prod_{i\in\mathcal{I}}M_{i}$ consists of tuples $(m_{i})_{i\in\mathcal{I}}$ such that $\mathrm{sup}_{i\in\mathcal{I}}||m_{i}||_{i}<\infty$, with norm given by the supremum. Note that in fact the maximum will be achieved. Suppose that $||(m_{i})||=\delta^{l}$. Then $\forall i\in\mathcal{I}$, \(m_{i}=\delta^{l_{i}}n_{i}\) for some $l_{i}\ge l$ and some $n_{i}$ with $||n_{i}||=1$, and for some $i_{0}$, $l_{i_{0}}=l$. Thus $(m_{i})=\pi^{l}(\pi^{l_{i}-l}n_{i})$, and $\pi^{l_{i_{0}}-l}n_{i_{0}}=n_{i_{0}}$. Hence $(m_{i})\in\pi^{l}\prod_{i\in\mathcal{I}}M_{i}$.
 \end{proof}

 \begin{theorem}\label{thm:tilde-nuc}
    Let $R$ be a $\pi$-adic Banach ring.
     The natural functor
     $$\mathbf{D}(\mathsf{Ban}_{R}^{\le1})^\circ\rightarrow\varprojlim_{n}\mathbf{D}(\mathsf{Mod}_{R\big\slash (\pi^{n})})$$
     is an equivalence of categories. In particular it coincides with the category of derived \(\dvgen\)-adically complete modules.
 \end{theorem}

\begin{proof}
We have that both $\mathbf{D}(\mathrm{Ban}_{R}^{\le1})^\circ$ and $\varprojlim\mathbf{D}(\mathsf{Mod}_{R\big\slash (\pi^{n})}$ are reflective subcategories of $\mathbf{Ch}(\mathrm{SNorm}_{R}^{\le1})^\circ$. The left adjoint to the first inclusion sends a dg-projective complex $\tilde{l}^{1}(V_{\bullet})$ to $l^{1}(V_{\bullet})$, and the left adjoint to the second inclusion sends $\tilde{l}^{1}(V_{\bullet})$ to $\mathbf{lim}_{n}(R\big\slash (\pi^{n}))\hat{\otimes}^{\mathbb{L}}\tilde{l}^{1}(V_{\bullet})$. We claim that term-wise this is just the completion. Now, since $\tilde{l}^{1}(V_{\bullet})$ is a dg-projective complex, it is dg-flat, so that $$(R\big\slash (\pi^{n}))\hat{\otimes}^{\mathbb{L}}\tilde{l}^{1}(V_{\bullet})\cong (R\big\slash (\pi^{n}))\hat{\otimes}\tilde{l}^{1}(V_{\bullet})\cong \tilde{l}^{1}(V_{\bullet})\big\slash\pi^{n}\tilde{l}^{1}(V_{\bullet})\cong \tilde{l}^{1}(V_{\bullet})\big\slash B_{\tilde{l}^{1}(V_{\bullet})}(0,|\pi|^{n}).$$
Thus $\mathbf{lim}_{n}(R\big\slash (\pi^{n}))\hat{\otimes}^{\mathbb{L}}\tilde{l}^{1}(V_{\bullet})$ just computes the completion, as required.
\end{proof}


\comment{
\begin{lemma}
Let $R$ be any Banach ring. Then
    $\mathbf{D}(\mathsf{SNrm}_{R}^{\le1})$ is compactly generated and rigid over $\mathbf{D}(\mathsf{Mod}_R)$.
\end{lemma}

\begin{proof}
       It suffices to prove that the unit is compact, and that all objects are nuclear. To check the former it suffices to prove that the closed ball of radius $1$ in any contracting sum of modules is the sum of the closed balls of radius $1$ in each of the factors.
        
        To check the latter statement it suffices to prove that 
        $$\mathbf{Map}(R_{\delta},N)\cong \mathbf{Map}(R_{\delta},R)\otimes N$$
        for all $\delta>0$. But we have 
        $$\mathbf{Map}(R_{\delta},N)_{m}\cong (N_{m})_{\frac{1}{\delta}}$$
        and 
        $$\mathbf{Map}(R_{\delta},R)\otimes N\cong R_{\frac{1}{\delta}}\otimes N,$$
        which completes the proof.


\end{proof}
}

\begin{definition}
    Let $R$ be a Banach ring with pseudo-uniformiser $\pi$. An $R$-module $M$ is said to be $\pi$-\textit{adically flat} if each $M\big\slash\pi^{k}M$ is algebraically flat as an $\underline{R\big\slash(\pi^{k})}$ module.
\end{definition}

\begin{proposition}
    Let $R$ be a Banach ring with pseudo-uniformiser $\pi$ which is $\pi$-torsion free. Let $M$ and $N$ be $\pi$-adic normed $R$-modules which are $\pi$-torsion free. If $M\otimes_{R}N$ is $\pi$-torsion free then the tensor product norm on $M\otimes_{R}N$ coincides with the $\pi$-adic one.
\end{proposition}

\begin{proof}
    We have that the tensor product norm is given by the separation of the semi-norm $\rho(v)=\mathrm{inf}\{\mathrm{max}||x_{i}\otimes y_{i}||:v=\sum_{i}x_{i}\otimes y_{i}\}$.  Let us prove that $\rho(\pi v)=|\pi|\rho(v)$. In general, clearly $\rho(\pi v)\le |\pi|\rho(v)$. When $\rho(v)=0$ it is clear that $\rho(\pi v)=0=|\pi|\rho(v)$. Suppose that $\rho(v)\neq0$. We claim that $\rho(\pi v)\neq0$. Indeed, if $\rho(\pi v)=0$, then for all $l>0$ we could find a representation $\pi v=\sum x_{i}\otimes y_{i}$ with $\mathrm{max}\{||x_{i}|| ||y_{i}||\}=|\pi|^{l}$. This means we can write $x_{i}\otimes y_{i}=\pi^{l}\tilde{x}_{i}\otimes\tilde{y}_{i}$ for some $\tilde{x}_{i},\tilde{y}_{i}$. Then $v=\pi^{l-1}\sum\tilde{x}_{i}\otimes\tilde{y}_{i}$. Thus $\rho(v)\le |\pi|^{l-1}$. Since this is true for all $l$, we have $\rho(v)=0$.
    
    Now assume $\rho(\pi v)\neq0$, i.e. $\rho(v)\neq0$. Write $\rho(\pi v)=|\pi|^{l}$ for some $l$. Clearly $l\ge 1$. Then there is some representation $\pi v=\sum_{i}x_{i}\otimes y_{i}$, where
    $\mathrm{max}_{i}||x_{i}|| ||y_{i}||=|\pi|^{l}$. Write $x_{i}=\pi^{s_{i}}\tilde{x}_{i}, y_{i}=\pi^{t_{i}}\tilde{y}_{i}$. Then $s_{i}+t_{i}=l_{i}\ge l$ for all $i$, and for some $i_{0}$, $s_{i_{0}}+t_{i_{0}}=l$. We may write $x_{i}\otimes y_{i}=\pi^{l}\tilde{x}_{i}\otimes\tilde{y}_{i}$. Then $\pi v=\pi^{l}\sum_{i}\tilde{x}_{i}\otimes\tilde{y}_{i}=\pi^{l}\tilde{v}$, where $\tilde{v}=\sum_{i}\tilde{x}_{i}\otimes\tilde{y}_{i}$. Note we must have $\rho(\tilde{v})=1$. Otherwise the proof above shows it would be divisible by $\pi$. Thus $v=\pi^{l-1}\tilde{v}$, and so $\rho(v)\le |\pi|^{l-1}$, and $|\pi|\rho(v)\le\rho(\pi v)$. 
\end{proof}

\begin{lemma}\label{lem:flattorsfree}
Let $R$ be a Banach ring with pseudo-uniformiser $\pi$, which is $\pi$-torsion free. Let $M$ and $N$ be $\pi$-torsion free $R$-modules. Equip $R,M$, and $N$ with the $\pi$-adic norms. If $N$ is flat as an $R$-module then $N\otimes_{R}M$ is $\pi$-torsion free, and we have $N\otimes^{\mathbb{L}}_{R}M\cong N\otimes_{R}M$ in $\mathrm{Norm}_{R}^{\le1}$. In particular, we also have $N\hat{\otimes}^{\mathbb{L}}_{R}M\cong N\hat{\otimes}_{R}M$ in $\mathrm{Ban}_{R}^{\le1}$.
\end{lemma}

\begin{proof}
Pick a projective resolution $\overline{l}^{1}(V_{\bullet})\rightarrow M$ in $\mathrm{Norm}_{R}^{\le1}$. Consider
$$N\otimes_{R}\overline{l}^{1}(V_{\bullet})\rightarrow N\otimes_{R}M.$$
The complex $N\otimes_{R}\overline{l}^{1}(V_{\bullet})$ computes $N\otimes^{\mathbb{L}}_{R}M$, so it suffices to prove that the augmented complex
$$N\otimes_{R}\overline{l}^{1}(V_{\bullet})\rightarrow N\otimes_{R}M$$
is acyclic. Since $N$ is flat $N\otimes_{R}M$ is $\pi$-torsion-free. Thus its norm is the $\pi$-adic one. It now remains to prove that the augmented complex is algebraically acyclic. But this follows from flatness of $N$.
\end{proof}

\subsubsection{Derived quotients by uniformisers}

\begin{lemma}\label{lem:htpy-inv-lim}
Let $M\in\mathbf{D}(\mathsf{SNorm}_{R}^{\le1})$ be such that each $\pi_{n}(M)$ is $\pi$-torsion free (that is, multiplication by $\pi$ is an admissible monomorphism). Then the map
$$\pi_{n}(M)\big\slash\pi^{k}\rightarrow\pi_{n}(M\big\slash\big\slash(\pi^{k}))$$
is an equivalence for each $k$.
\end{lemma}

\begin{proof}
$M\big\slash\big\slash\pi^{k}$ is computed as the cone of the `multiplication by $\pi^{k}$' map
$$\times\pi^{k}:M\rightarrow M.$$
Let $C$ denote the cone.
    We have a long exact sequence
    \begin{displaymath}
        \xymatrix{
\ldots\ar[r] &\pi_{n+1}(C)\ar[r] & \pi_{n}(M)\ar[r]^{\times\pi^{k}} &\pi_{n}(M)\ar[r] &\pi_{n}(C)\ar[r] & \ldots.
        }
    \end{displaymath}
    By assumption the map $\times\pi^{k}:\pi_{n}(M)\rightarrow \pi_{n}(M)$ is an admissible monomorphism. Thus the sequence 
    \begin{displaymath}
        \xymatrix{
        0\ar[r] & \pi_{n}(M)\ar[r]^{\times\pi^{k}} &\pi_{n}(M)\ar[r] &\pi_{n}(C)\ar[r] &0
        }
    \end{displaymath}
    is short exact, which completes the proof.
\end{proof}

\begin{lemma}\label{lem:htpy-inverse-lim2}
    Let $M\in\mathbf{D}(\mathsf{Norm}_{R}^{\le1})$ be such that each $\pi_{n}(M)$ is $\pi$-torsion free and derived $\pi$-adically complete. Then the map
    $$M\rightarrow\varprojlim_{k}M\big\slash\big\slash\pi^{k}$$
    is an equivalence, and 
    $$\pi_{*}(M)\cong\varprojlim_{k}\pi_{*}(M)\big\slash\pi^{k}.$$
\end{lemma}

\begin{proof}
    We have $\pi_{n}(M\big\slash\big\slash\pi^{k})\cong\pi_{n}(M)\big\slash\pi^{k}$ by Lemma \ref{lem:htpy-inv-lim}. We therefore have that each map $\pi_{n}(M\big\slash\big\slash\pi^{k+1})\rightarrow\pi_{n}(M\big\slash\big\slash\pi^{k})$ is an epimorphism. By the long exact sequence for the homotopy groups of $\lim_{k}M\big\slash\pi^{k}$ and a Mittag-Leffler argument, we have 
    $$\pi_{*}(\varprojlim_{k}M\big\slash\big\slash\pi^{k})\cong\varprojlim_{k}\pi_{*}(M\big\slash\big\slash\pi^{k})\cong\varprojlim_{k}\pi_{*}(M)\big\slash\pi^{k}\cong\pi_{*}(M),$$
    as required. 
\end{proof}

\begin{lemma}
 Let $A\in\mathbf{DAlg}^{cn}(\bD_{\geq 0}(\mathsf{Norm}_{R}^{\le1}))$ and let $M$ be an $A$-module in $\mathbf{D}(\mathsf{Norm}_{R}^{\le1})$. Suppose that each $\pi_{n}(M)$ is $\pi$-torsion free.
    \begin{enumerate}
    \item 
  If $M$ is derived strong as an $A$-module then each $M\big\slash\big\slash\pi^{k}$ is derived strong as an $A\big\slash\big\slash\pi^{k}$-module.
\item 
If $\pi_{*}(A),\pi_{0}(M)\in\mathbf{D}_{\ge0}(\sf{Ban}_{A}^{\le1})^{\circ}$ and each $\pi_{*}(A)\otimes_{\pi_{0}(A)}\pi_{0}(M)$ is $\pi$-torsion free then the converse is true, that is if each $M\big\slash\big\slash\pi^{n}M$ is derived strong as an $A\big\slash\big\slash\pi^{k}$-module, then $M$ is derived strong as an $A$-module.
    \end{enumerate}
\end{lemma}

\begin{proof}
The first claim follows from \cite{ben2024perspective}*{Corollary 2.3.89}, and the fact that $\pi_{n}(M)$ is transverse to $\pi_{0}(A)\big\slash(\pi^{k})$ over $\pi_{0}(A)$.

To prove the partial converse claim, we are required to prove that $\pi_{*}(A)\hat{\otimes}^{\mathbb{L}}_{\pi_{0}(A)}\pi_{0}(M) \cong \pi_*(M)$ for all \(* \geq 0\). We have \(\pi_*(A//\dvgen^k) \otimes^{\mathbb{L}}_{\pi_0(A//\dvgen^k)} \pi_0(M//\dvgen^k M) \cong \pi_*(M//\dvgen^k M)\) for all \(* \geq 0\)  by the derived strength hypothesis on the reductions mod \(\dvgen\).  By Lemma \ref{lem:htpy-inv-lim}, we have \[(\pi_*(A) \hat{\otimes}_{\pi_0(A)} \pi_0(M))/\dvgen^k \cong \pi_*(M) / \dvgen^k\] for all \(* \geq 0\). Now we can conclude using Lemma \ref{lem:flattorsfree} and Lemma \ref{lem:htpy-inverse-lim2}. 
\end{proof}

\subsubsection*{Relationship to bornological modules}

The natural functors
$$\mathsf{SNrm}_{R}^{\le1}\rightarrow\mathsf{Ind}(\mathsf{SNrm}_{R}^{\le1}), \mathsf{Nrm}_{R}^{\le1}\rightarrow\mathsf{Ind}(\mathsf{Nrm}_{R}^{\le1})\text{ and } \mathsf{Ban}_{R}^{\le1}\rightarrow\mathsf{Ind}(\mathsf{Ban}_{R}^{\le1})$$
are clearly exact. Moreover they sends projectives to projectives and are strongly monoidal. We get induced strongly monoidal functors of $\infty$-categories
$$\bD(\mathsf{SNrm}_{R}^{\le1})\rightarrow \bD(\mathsf{Ind}(\mathsf{SNrm}_{R}^{\le1})), \qquad \bD(\mathsf{Nrm}_{R}^{\le1}) \rightarrow \bD(\mathsf{Ind}(\mathsf{Nrm}_{R}^{\le1})).$$

\comment{Clearly this sends objects in the kernel of 
$$\bD(\mathsf{SNrm}_{R}^{\le1}) \rightarrow \bD_f(\mathsf{SNrm}_R^{\le1})$$
to trivial objects, and therefore factors through a strongly monoidal functor
$$\bD_f(\mathsf{SNrm}_R^{\le1})\rightarrow \bD(\mathsf{Ind}(\mathsf{SNrm}_{R}^{\leq 1})).$$ We have a similar factorisation for the contracting normed and Banach categories. }

We can summarise the relationship between the categories constructed in this section through the following diagram:
\begin{displaymath}
    \xymatrix{
    \bD(\mathsf{SNrm}_{R}^{\le1}) \ar[d]\ar[r] & \bD(\mathsf{Ind}(\mathsf{SNrm}_{R}^{\leq 1})) \ar[d]\\
    \bD(\mathsf{Nrm}_{R}^{\le1}) \ar[d]\ar[r] & \bD(\mathsf{Ind}(\mathsf{SNrm}_{R}^{\leq 1})) \ar[d]\\
    \bD(\mathsf{Ban}_{R}^{\le1}) \ar[r] &  \bD(\mathsf{Ind}(\mathsf{Ban}_{R}^{\leq 1}))
    }
\end{displaymath}

\begin{proposition}
    The functor 
    $$i:\mathbf{D}(\mathsf{Ban}_{R}^{\le1})\rightarrow \mathbf{D}(\mathsf{Ind}(\mathsf{Ban}_R))$$
    \begin{enumerate}
    \item
    commutes with finite limits and colimits, 
    \item 
    commutes with geometric realisations of objects in $\mathbf{D}_{\ge0}(\mathsf{Ban}_{R}^{\le1})$,
\item
is $t$-exact.
\end{enumerate}
\end{proposition}

\begin{proof}
    \begin{enumerate}
        \item 
    At the level of complexes, the functor clearly commutes with finite sums of complexes, and taking mapping cones.
    \item 
    At the level of complexes, it commutes with totalisations of non-negatively-graded double complexes, as only finite sums in each degree are necessary.
    \item 
    The functor at the level of exact categories is $t$-exact, and also commutes with both kernels and cokernels. \qedhere
    \end{enumerate}
\end{proof}

\begin{lemma}\label{lem:adic-subcat}
Let $R$ be a Banach ring such that $|r|\le1$ for all $r\in R$. Then the functor
$$\mathbf{D}_{\ge0}(\mathsf{Ban}_{R}^{\le1})^\circ\rightarrow \mathbf{D}_{\ge0}(\mathsf{Ind}(\mathsf{Ban}_R))$$
is strongly monoidal and fully faithful. 
\end{lemma}

\begin{proof}
    For full faithfullness we reduce to showing that any bounded map $f:l^{1}(V)\rightarrow l^{1}(W)$ is bounded by $1$. But any such bounded map is determined by the images of the elements $v\in V$. Moreover if $||f(v)||\le 1=||v||$ for all $v$ this implies that $||f||\le1$. This clearly follows from the fact that any element of $l^{1}(W)$ has norm bounded by $1$. 
\end{proof}

We also have the following useful fact.

\begin{lemma}
    Let $M, N\in\mathsf{Ban}^{\le1}_{R}$. If $M$ is transverse to $N$ in the strong exact structure, then it is transverse to $N$ in $\mathsf{Ban}_{R}$.
\end{lemma}

\begin{proof}
    There exists a strong resolution $P_{\bullet}\rightarrow M$ by objects $P_{n}$ which are projective in the strong exact structure on $\mathsf{Ban}^{\le1}_{R}$ and the quasi-abelian structure on $\mathsf{Ban}_{R}$. In particular $N\hat{\otimes}_{R}P_{\bullet}$ computes the derived tensor product in the strong exact structure, and in the quasi-abelian structure on $\mathsf{Ban}_{R}$. Hence if it is strongly quasi-isomorphic to $M\hat{\otimes}_{R}N$, then it is certainly quasi-isomorphic to $M\hat{\otimes}_{R}N$ in the quasi-abelian structure. 
\end{proof}

\subsubsection{Derived admissible algebras}

Let $R$ be an almost adic Banach ring with pseudo-uniformiser $\pi$. The \emph{affinoid Tate algebra in} $n$ variables is the free commutative monoid on $R^{\oplus n}$, computed in $\sf{Ban}_{R}^{\le1}$. We denote it by $R \gen{x_1, \dotsc, x_n}$, and we have 
 \[R \gen{x_1, \dotsc, x_n} \defeq \varprojlim_n R/(\pi^{k})[x_1, \dotsc, x_n],\] viewed as an algebra in \(\bD_{\geq 0}(\mathsf{Ban}_R^{\leq 1})^\circ \).

\begin{definition}[\cite{bosch2014lectures}*{Section 7.3}]\label{def:formal-adject}
Let $A$ be an \(R\)-algebra. We say that 

\begin{enumerate}
\item \(A\) is \emph{topologically of finite type} if it admits an epimorphism 
$$R \gen{x_1, \dotsc, x_n}\rightarrow A$$
in $\mathrm{Comm}(\mathrm{LH}^{\mathrm{CBorn}_{R}})$.
\item \(A\) is \emph{topologically of finite presentation} if it is isomorphic to \(R\gen{x_1, \dotsc, x_n}/I\), where \(I\) is a finitely generated ideal (and hence closed).
\item \(A\) is \emph{admissible} if it is topologically of finite presentation and $\pi$-torsion free. 
\end{enumerate}
\end{definition}

\begin{lemma}
Let A be an \(R\)-algebra which is topologically of finite presentation. Then A is coherent as an \(R\)-algebra.
\end{lemma}
\begin{proof}
See for instance \cite{bosch2014lectures}*{Corollary 7.3.6}.
\end{proof}

\comment{\begin{lemma}\label{lem:pushout-adm-formal}
Let $f:A\rightarrow B$ be a \(\dvgen\)-adically \'etale map in $\mathbf{CAlg}(\mathbf{Ch}_{\ge0}(\mathrm{Ban}_{R}^{\le1}))$. If $C$ is an admissible $R$-algebra and $A\rightarrow C$ a map, then $C\hat{\otimes}_{A}^{\mathbb{L}}B$ is an admissible $R$-algebra.
\end{lemma}
\begin{proof}
Clear. 
\end{proof}}
   
\begin{definition}
We call \(A \in \mathbf{DAlg}^{cn}(\bD_{\geq 0}(\mathsf{Ban}_R^{\leq 1}))\) a \emph{derived adic ring topologically of finite type (resp. of finite presentation/ almost admissible)} if \(\pi_0(A)\) is an \(R\)-algebra topologically of finite type (resp. of finite presentation/ admissible), and \(\pi_n(A)\) is coherent as a \(\pi_0(A)\)-module. $A$ is said to be \textit{admissible} if it is almost admissible, and each $\pi_{n}(A)$ is $\pi$-torsion free.
\end{definition}   

\begin{lemma}\label{lem:fpform}
    Let $A$ be an adic ring topologically of finite presentation. Let $M$ be a finite, free $A$-module, and $N\subset M$ a finitely generated submodule. Then $N$ is finitely generated as an object of $\mathrm{Mod}_{A}(\mathrm{LH(CBorn_{R})})$. Moreover, $\mathrm{coker}(N\rightarrow M)\cong M\big\slash N$ (where the left-hand side is computed in the left heart, and the right-hand side in Banach spaces).
\end{lemma}

\begin{proof}
  The sub-module $N$ is finitely generated, so it is closed with respect to the restriction of the $\pi$-adic topology. Moreover the restriction of the $\pi$-adic topology is the $\pi$-adic topology on $N$ (\cite{bosch1}*{Lemma 1.2}). There exists an algebraic  epimorphism of modules $A^{\oplus k}\rightarrow N$. Since everything has the $\pi$-adic norm here, this is an admissible epimorphism by Proposition \ref{prop:algepi}. Thus the composite $A^{\oplus k}\rightarrow N\rightarrow M$ is an admissible morphism.
\end{proof}

\begin{proposition}\label{prop:strongadc}

\begin{enumerate}
      Let $A\rightarrow B$ and $B\rightarrow C$ be maps with $A,B,C$ derived adic rings topologically of finite presentation. If $\pi_{0}(B)$ is of finite presentation as a $\pi_{0}(A)$-module, then $B\hat{\otimes}_{A}^{\mathbb{L}}C$ is topologically of finite presentation.
    \item 
    Let $f:A\rightarrow B$, $g:A\rightarrow C$ be maps of derived almost admissible adic rings of finite presentation. Suppose that as a map of rings $f$ is flat. Then $B\hat{\otimes}_{A}^{\mathbb{L}}C$ is a derived almost admissible adic ring as long as $\pi_{0}(B)\hat{\otimes}_{\pi_{0}(A)}\pi_{0}(C)$ admits no torsion. 
\end{enumerate}
\end{proposition}

\begin{proof}
\begin{enumerate}
\item
    We use \cite{ben2024perspective}*{Lemma 4.5.55}. We need to show that if $f:A\rightarrow B$, $g:A\rightarrow C$ are maps of discrete admissible affinoids, then the tensor product in the left heart, $B\otimes^{LH}_{A}C$, is also a discrete admissible affinoid. Write $B\cong A<x_{1},\ldots,x_{n}>\big\slash I$ with $I$ a finitely generated (and hence admissible) ideal. Now we can just use Lemma \ref{lem:fpform}.
    


    \item 
    We know that $B\hat{\otimes}_{A}^{\mathbb{L}}C$ is a derived adic ring topologically of finite presentation. Thus everything interesting is happening in $\pi_{0}$, and we may assume everything is discrete. We need to show that $B\hat{\otimes}_{A}C$ is $\pi$-torsion free. This follows from Lemma \ref{lem:flattorsfree}.\qedhere
    \end{enumerate}
\end{proof}

\begin{remark}
    Let $A$ be a derived adic ring topologically of finite type over $R$. Let $f_{1},\ldots,f_{m}:R\rightarrow\pi_{0}(A)\hat{\otimes}R\gen{x_{1},\ldots,x_{m}}$. Then $A\hat{\otimes} R\gen{x_{1},\ldots,x_{m}}\big\slash\big\slash(f_{1},\ldots,f_{n})$ is also topologically of finite type over $R$. Indeed we have 
    $$A\hat{\otimes} R\gen{x_{1},\ldots,x_{m}}\big\slash\big\slash(f_{1},\ldots,f_{n})\cong A\hat{\otimes} R \gen{x_{1},\ldots,x_{m}} \hat{\otimes}^{\mathbb{L}}_{R\gen{x_{1},\ldots,x_{n}}}R,$$
    which is a pushout of derived adic rings topologically of finite type over $R$.
\end{remark}

\subsubsection{Inverting $\pi$}

Let $R$ be a $\pi$-adic Banach ring which is $\pi$-torsion free. Note that $|\pi r|=|\pi| |r|$ for all $r\in R$. Consider $R[\frac{1}{\pi}]$ with the norm $|\frac{r}{\pi^{n}}|=\frac{|r|}{|\pi^{n}|}$. Then this is a non-Archimedean, non-trivially valued Banach field $K$, and $\mathcal{O}_{K}\cong R$. Consider the functor
$$K\otimes^{\mathbb{L}}_{R}(-):\mathbf{D}(R)\rightarrow\mathbf{D}(K).$$
This is a strong symmetric monoidal functor. In particular we get an induced functor
$$K\otimes^{\mathbb{L}}_{R}(-):\mathbf{DAlg}(R)\rightarrow\mathbf{DAlg}(K).$$

\begin{lemma}[c.f \cite{mathewsimplicialrings}*{Proposition 2}, \cite{ben2024perspective}*{Lemma 4.5.67/ Theorem 4.5.68}]
    Let $A$ be either derived adic ring topologically of finite presentation over $R$, or a derived affinoid $K$-algebra. Then we may write $A$ as a geometric realisation $A\cong|R<x_{m1},\ldots,x_{mn_{m}}>|$, where $m$ denotes the simplicial degree in the adic ring case, and as $A\cong|K<x_{m1},\ldots,x_{mn_{m}}>|$ in the affinoid case.
\end{lemma}

\begin{proof}
  The proof works exactly as in \cite{mathewsimplicialrings}*{Proposition 2} where in the adic case, instead of Noetherianity, we use that each $R<y_{1},\ldots,y_{l}>$ is coherent, and all finitely generated ideals are admissible. 
\end{proof}

\comment{
\begin{theorem}
    $$K\otimes^{\mathbb{L}}_{R}(-):\mathbf{DAlg}(R)\rightarrow\mathbf{DAlg}(K)$$
    restricts to an equivalence of categories 
    $$K\otimes^{\mathbb{L}}_{R}(-):\mathbf{DAfnd}_{R}^{cn,adm}\rightarrow\mathbf{DAfnd}_{K}^{cn}.$$
\end{theorem}

\begin{proof}
    First observe that objects of the form $R<x_{1},\ldots,x_{m}>$ are flat, and therefore transverse to $R\rightarrow K$. Hence we have 
    $$K\otimes^{\mathbb{L}}_{R}|R<x_{m1},\ldots,x_{mn_{m}}>|\cong |K<x_{m1},\ldots,x_{mn_{m}}|.$$
    By the classical Raynaud Theorem,
    $$K\hat{\otimes}^{\mathbb{L}}_{R}(-):\mathbf{DAfnd}_{R}^{cn,adm}\rightarrow\mathbf{DAfnd}_{K}^{cn}$$
    is essentially surjective. It remains to prove that it is fully faithful. The proof is similar to \cite{ben2024perspective}*{Bit about embedding Lawvere theories}. Let $|R<x_{\bullet}>|$ and $|R<y_{\bullet}>|$ be presentations of derived admissible adic $R$-algebras. We need to show that the map 
    $$\mathbf{Map}(|R<x_{\bullet}>|,|R<y_{\bullet}>|)\rightarrow\mathbf{Map}(|K<x_{\bullet}>|,|K<y_{\bullet}>|).$$
By a colimit argument we may assume that $|R<x_{\bullet}>|=R<x>$ is just discrete with one variable. Write $B=|R<y_{\bullet}>|$. 
\comment{
First we prove that the map
$$\mathbf{Map}(R<x>,B)\rightarrow\mathbf{Map}(K<x>,K\otimes^{\mathbb{L}}_{R}B)$$
is a homotopy monomorphism. Indeed consider the commutative diagram
\begin{displaymath}
    \xymatrix{
    \mathbf{Map}(R[x],B)\ar[d]\ar[r] & \mathbf{Map}(K[x],K\otimes^{\mathbb{L}}_{R}B)\ar[d]\\
    \mathbf{Map}(R<x>,B)\ar[r] & \mathbf{Map}(K<x>,K\otimes^{\mathbb{L}}_{R}B).
    }
\end{displaymath}
Both vertical maps are homotopy monomorphisms. We claim that the left-hand vertical map is an isomorphism. Indeed we have $ \mathbf{Map}(R[x],B\cong |B|$, where $|B|=|R<y_{\bullet}>|$ is the geometric realisation of the underlying rings. On the other hand we also have $\mathbf{Map}_{\mathbf{Dalg}(\mathrm{Ban}^{\le1}_{R})}(R<x>,B)\cong |B|$. But we have $\mathbf{Map}_{\mathbf{Dalg}(\mathrm{Ban}^{\le1}_{R})}(A,B)\cong\mathbf{Map}_{\mathbf{DAlg}(R)}(i(A),i(B))$ for any algebras $A,B$ in the convex subcategory. 

Next we prove that the map 
$$\mathbf{Map}(R[x],B)\rightarrow \mathbf{Map}(K[X],K\otimes^{\mathbb{L}}_{R}B)$$
is a homotopy monomorphism. This is the same as saying that $|B|\rightarrow |K\otimes_{R}^{\mathbb{L}}B|$ is a homotopy monomorphism. 
}
\textcolor{red}{cotangent complex argument}. Consider the Postnikov presentation $B\cong\mathbf{lim}_{n}B_{\le n}$ of $B$. Then $K\otimes^{\mathbb{L}}B$ has Postnikov presentation $\mathbf{lim}_{n}K\otimes_{\mathbb{L}}B_{\le n}$. Therefore by pulling out limits, we reduce to the case that $B\cong B_{\le n}$ for some $n$. Now the proof is by induction on $n$. For $n=0$ this is the classical Raynaud localisation for affinoids. Suppose the claim has been proven whenever $B\cong B_{\le n}$ for some $n$, and consider $B$ with $B\cong B_{\le n+1}$. Now we have $B\cong B_{\le n}\oplus_{d}\pi_{n+1}(B)[n+2]$. We have 
$$\mathbf{Map}_{R<x>\big\backslash B\big\slash B_{\le n}}(R<x>,B)\cong\Omega\mathbf{Map}(\mathbb{L}_{B_{\le n}\big\slash R<x>},\pi_{n+1}(B)[n+2]).$$
We also have

\end{proof}
}

\subsubsection{Categories of modules, and the cotangent complex}

Again let $R$ be a $\pi$-adic Banach ring which is $\pi$-torsion free. 

Let $A$ be a derived admissible adic $R$-algebra. Consider its derived category $\mathbf{D}(A)$. Let $\mathbf{D}^{coh}(A)$ denote the full subcategory of $\mathbf{D}(A)$ consisting of those modules $M$ such that each $\pi_{n}(M)$ is a coherent $A$-module. Consider the affinoid algebra $A[\frac{1}{\pi}]=K\hat{\otimes}^{\mathbb{L}}_{R}A$. We get the free-forgetful adjunction
$$\adj{A[\frac{1}{\pi}]\hat{\otimes}_{A}^{\mathbb{L}}(-)}{\mathbf{D}(A)}{\mathbf{D}(A[\frac{1}{\pi}])}{|-|}.$$

Let $\mathbf{D}^{coh}(A)$ denote the full subcategory of $A$-modules $M_{\bullet}$ such that each $\pi_{n}(M_{\bullet})$ is coherent as a $\pi_{0}(A)$-module. Similarly let $\mathbf{D}^{coh}(A)$ denote the full subcategory of $A[\frac{1}{\pi}]$-modules $N_{\bullet}$ such that each $\pi_{n}(N_{\bullet})$ is coherent (i.e. finitely presented) as a $\pi_{0}(A)[\frac{1}{\pi}]$-module.

Let $M_{\bullet}\in\mathbf{D}_{\ge0}^{coh}(A)$. We have a spectral sequence $\pi_{*}(A)[\frac{1}{\pi}]\hat{\otimes}_{\pi_{*}(A)}^{L}\pi_{*}(M_{\bullet})\cong\pi_{*}(M_{\bullet})[\frac{1}{\pi}]\Rightarrow\pi_{*}(A[\frac{1}{\pi}]\hat{\otimes}_{A}^{L}M_{\bullet})$. Each $\pi_{*}(M_{\bullet})[\frac{1}{\pi}]$ is finitely generated over $\pi_{0}(A)[\frac{1}{\pi}]$. Since $\pi_{0}(A)[\frac{1}{\pi}]$ is Noetherian, we have that $\pi_{*}(A[\frac{1}{\pi}]\hat{\otimes}_{A}^{L}M_{\bullet})$ is finitely presented. 

Thus $A[\frac{1}{\pi}]\hat{\otimes}_{A}^{L}(-)$ restricts to a functor
$$\mathbf{D}_{\ge0}^{coh}(A)\rightarrow\mathbf{D}_{\ge0}^{coh}(A[\frac{1}{\pi}]).$$

In fact we have $\mathbf{D}_{\ge0}^{coh}(A)\subseteq\mathbf{Rec}(\{A^{\oplus n}\})$, where $\mathbf{Rec}(\{A^{\oplus n}\})$ is the free cocompletion by geometric realisations, and similarly for $\mathbf{D}_{\ge 0}^{coh}(A[\frac{1}{\pi}]).$

\begin{lemma}
Let $M\in\mathbf{Mod}_A(\mathbf{D}(\sf{Ban}_{R}^{\le1}))$ be such that each $\pi_{m}(M)$ is a $\pi$-adically normed $\pi_{0}(A)$-module. Then we have 
$$\pi_{*}(M[\frac{1}{\pi}])\cong\pi_{*}(M)[\frac{1}{\pi}].$$
\end{lemma}

\begin{proof}
We may assume we are working over the ground ring $R$. By an easy truncation argument, we may also assume that everything is discrete. Thus we are reduced to proving  that, whenever $M$ is a $\pi$-adic $R$-module, it is transverse to $R[\frac{1}{\pi}]$. Now $M$ may be resolved by modules of the form $l^{1}(V)$ for $V$ a discrete normed set. Such objects are $\pi$-adic. Hence we are reduced to proving that, if
$$0\rightarrow K\rightarrow L\rightarrow M\rightarrow 0$$
is a strong exact sequence with $L$ and $M$ (and automatically $K$) being $\pi$-adic, the sequence
$$0\rightarrow K[\frac{1}{\pi}]\rightarrow L[\frac{1}{\pi}]\rightarrow M[\frac{1}{\pi}]\rightarrow 0$$
is strong exact. The sequence is evidently algebraically exact. It remains to prove that $ L[\frac{1}{\pi}]\rightarrow M[\frac{1}{\pi}]$ is a strong epimorphism. Since completion is exact, we can ignore completions. Let $\frac{m}{\pi^{k}}\in M$, which has norm $\frac{||m||}{|\pi|^{k}}$. Pick an element $l$ of $L$ mapping to $m$, which has the same norm as that of $m$. Then $\frac{l}{\pi^{k}}$ has the same norm as $\frac{m}{\pi^{k}}\in M$, and maps to it.
\end{proof}

\begin{lemma}\label{lem:formalmodelmods}
Let $M\in\mathbf{D}^{coh}_{\ge0}(A)$ be arbitrary, and $N\in\mathbf{D}^{coh}(A[\frac{1}{\pi}])$ have bounded homology. Let $f:M[\frac{1}{\pi}]\rightarrow N$ be a morphism. Then there exists an $\tilde{N}\in\mathbf{D}^{coh}(A)$ and a map $\tilde{f}:M\rightarrow\tilde{N}$ such that $f\cong\tilde{f}[\frac{1}{\pi}]$. Moreover we may assume that each $\pi_{m}(N)$ is $\pi$-torsion free. 
\end{lemma}

\begin{proof}
    This can be proven identically to \cite{antonio2018derived}*{Corollary A.2.2}, without the $\pi$-torsion free claim. For the $\pi$-torsion free claim we argue that for any $N\in\mathbf{D}^{coh}(A)$ of bounded homological dimension, there exists a $\tilde{N}\in\mathbf{D}^{coh}(A)$ concentrated in the same homological interval, together with a map $N\rightarrow\tilde{N}$, such that 
\begin{enumerate}
\item
each $\pi_{m}(\tilde{N})$ is $\pi$-torsion free;
\item
the map $N[\frac{1}{\pi}]\rightarrow\tilde{N}[\frac{1}{\pi}]$ is an equivalence.
\end{enumerate}
Let $N$ be concentrated in degress $[n,0]$. We prove the claim by induction on $n$. When $n=0$, so that $N$ is just a discrete $\pi_{0}(A)$-module, we just let $\tilde{N}=N\big\slash N_{\pi-torsion}$. Since the torsion submodule is coherent, $\tilde{N}$ is coherent.  Suppose we have proven the claim for some $n$, and let $N$ be concentrated in degrees $[n+1,0]$. We have a fibre-cofibre sequence
$$\pi_{n+1}(N)\rightarrow N\rightarrow \tau_{\le n}N.$$
Pick maps $\pi_{n+1}(N)\rightarrow\widetilde{\pi_{n+1}(N)}$, $\tau_{\le n}N\rightarrow\widetilde{\tau_{\le n}N}$ satisfying the requirements, guaranteed by the inductive hypothesis. Then we can glue these to a map $N\rightarrow\tilde{N}$ also satisfying the requirements.
\end{proof}

\begin{proposition}
    Let $A\rightarrow B$ be a map of derived topologically finitely presented adic $R$-algebras. Let $\mathbb{L}^{\le1}_{B\big\slash A}$ denote the cotangent complex computed in $\mathbf{Mod}_B(\mathbf{D}(\sf{Ban}_{R}^{\le1}))$, and let $\mathbb{L}_{B\big\slash A}$ denote the cotangent complex computed in the category $\mathbf{Mod}_B(\mathbf{D}(R))$. The natural map 
    $$i(\mathbb{L}^{\le1}_{B\big\slash A})\rightarrow\mathbb{L}_{B\big\slash A}$$
    is an equivalence.
\end{proposition}

\begin{proof}
    It suffices to prove this in the absolute case. That is, we may assume $A=R$. We may write $A\cong |R<x_{m1},\ldots,x_{mn_{m}}>|$. Thus it in fact suffices to prove the claim for rings of the form $R<y_{1},\ldots,y_{l}>$. But in this case we have $\mathbb{L}^{\le1}_{B}\cong B^{\oplus l}$, and $\mathbb{L}_{i(B)}\cong i(B)^{\oplus l}$, where in the latter case we use that $\mathrm{Sym}(R^{\oplus n})\rightarrow R<y_{1},\ldots,y_{l}>$ is a homotopy epimorphism. The map $i(B^{\oplus l})\rightarrow i(B)^{\oplus l}$ is the obvious identification.
    \end{proof}

    In particular, we can unambiguously refer to a map $A\rightarrow B$ of derived topologically finitely presented adic $R$-algebras being formally \'{e}tale, without specifying whether we are talking about $(\sf{Ban}_{R}^{\le1}))$, or $\mathbf{D}(R)$.

\begin{lemma}\label{lem:cotanderivedcomplete}
    Let $A\rightarrow B$ be a map of derived adic $R$-algebras topologically of finite type. Then $\mathbb{L}_{B\big\slash A}$ is a coherent $B$-module. If $A$ and $B$ are in fact admissible, then $\mathbb{L}_{B\big\slash A}$ is derived $\pi$-adically complete.
\end{lemma}

\begin{proof}
    It suffices to prove this in the absolute case, i.e. $A=R$. Write $\pi_{0}(B)\cong R<x_{1},\ldots,x_{n}>\big\slash I$ in $\mathrm{Comm}(\mathrm{Ban}_{R}^{\le1})$ with $I$ finitely generated. By projectivity, the map $R<x_{1},\ldots,x_{n}>\rightarrow \pi_{0}(B)$ lifts to a map $R<x_{1},\ldots,x_{n}>\rightarrow B$. Note that $\mathbb{L}_{R<x_{1},\ldots,x_{n}>\big\slash R}\cong R<x_{1},\ldots,x_{n}>^{\oplus n}$ so is coherent and derived $\pi$-adically complete. It suffices to prove that $\mathbb{L}_{B\big\slash R<x_{1},\ldots,x_{n}>}$ is coherent in the topologically finitely presented case, and derived $\pi$-adically complete in the admissible case. This follows from \cite{ben2024perspective}*{Lemma 2.4.123}.
\end{proof}

The next two results, and their proofs, are completely analogous to \cite{antonio2018derived}*{ Theorem 4.4.4}.

\begin{corollary}
   Let $K=R[\frac{1}{\pi}]$, and let $g:A\rightarrow B$ be a map of derived affinoid $K$-algebras. Suppose there is a formal model $\tilde{g}:\tilde{A}\rightarrow\tilde{B}$, i.e., $\tilde{g}$ is a map of derived admissible adic $R$-algebras such that $\tilde{g}[\frac{1}{\pi}]\cong g$.
Let $M$ be a coherent $B$-module which is homologically bounded. Any $A$-derivation $d:B\rightarrow M$ admits a formal model. Precisely, there is a coherent $\tilde{B}$-module $\tilde{M}$, and an $\tilde{A}$-derivation $\tilde{d}:\tilde{B}\rightarrow\tilde{M}$, such that $d\cong\tilde{d}[\frac{1}{\pi}]$. Moreover, we may assume that each $\pi_{m}(M)$ is $\pi$-torsion free.
\end{corollary}

\begin{proof}
    An $A$-derivation $d:B\rightarrow M$ is equivalent to a map $f:\mathbb{L}_{B\big\slash A}\rightarrow M$. Now we have $\mathbb{L}_{B\big\slash A}\cong\mathbb{L}_{\tilde{B}\big\slash\tilde{A}}[\frac{1}{\pi}]$, and we may use Lemma \ref{lem:formalmodelmods}.
\end{proof}

\begin{corollary}
    Every derived affinoid algebra $A$ has a derived admissible formal model. In fact, if $\tilde{A}_{0}$ is a given formal model for $\pi_{0}(A)$, then there exists a formal model $\tilde{A}$ for $A$ with $\pi_{0}(\tilde{A})\cong\tilde{A}_{0}$.
\end{corollary}

\begin{proof}
    Consider the Postnikov tower
    $A\cong\mathbf{lim}A\le n$ where $$A_{\le n+1}\cong A_{\le n}\oplus_{d_{n}}\Omega(\pi_{n+1}(A)[n+2])$$ for some derivation $d_{n}:A_{\le n}\rightarrow\pi_{n+1}(A)[n+2]$. Note each $A_{\le n}$ is a derived admissible adic $R$-algebra, so that $\mathbb{L}_{A_{\le n}}$ is coherent. We inductively construct a compatible system of formal models $\ldots\rightarrow\tilde{A}_{\le n+1}\rightarrow\tilde{A}_{\le n}\rightarrow\ldots.$ Suppose $\tilde{A}_{\le n}$ has been constructed. Given $\mathbb{L}_{\tilde{A}_{\le n}}$ a formal model for $\mathbb{L}_{A_{\le n}}$, we can find a formal model $\tilde{d}_{n}:\mathbb{L}_{\tilde{A}_{\le n}}\rightarrow\tilde{M}_{n+1}[n+2]$ for $d_{n}$. We may even assume that $\tilde{M}_{n+1}[n+2]$ is concentrated in degree $n+2$ and is $\pi$-torsion free. Define $\tilde{A}_{\le n+1}\defeq\tilde{A}_{\le n}\oplus_{\tilde{d}_{n}}\Omega\tilde{M}_{n+1}[n+2]$. Then we define $\tilde{A}\defeq\mathbf{lim}\tilde{A}_{\le n}.$
\end{proof}

One can similarly prove a relative version of this

\begin{corollary}
    Any map $g:A\rightarrow B$ of derived admissible adic $R$-algebras admits a formal model. In fact given a formal model $\tilde{g}_{0}:\tilde{A}_{0}\rightarrow\tilde{B}_{0}$ of $\pi_{0}(g):\pi_{0}(A)\rightarrow\pi_{0}(B)$, there is a formal model $\tilde{g}:\tilde{A}\rightarrow\tilde{B}$ of $g$ with $\pi_{0}(\tilde{g})\cong\tilde{g}_{0}$.
\end{corollary}


\begin{definition}\label{def:adic-formal-et}
    A map $f:A\rightarrow B$ in $\mathbf{DAlg}^{cn}(\sf{Norm}_{R}^{\le1})^{\circ}$ is said to be \textit{adically formally \'{e}tale} if each map $A\big\slash\big\slash(\pi^{k})\rightarrow B\big\slash\big\slash(\pi^{k})$ is a formally \'{e}tale map of underlying derived algebras.
\end{definition}

\begin{corollary}
    Let $f:A\rightarrow B$ be an adically formally \'{e}tale map of derived admissible adic $R$-algebras. Then it is formally \'{e}tale.
\end{corollary}

\begin{proof}
    This follows immediately from Lemma \ref{lem:cotanderivedcomplete}.
\end{proof}

\subsection{The topologies}

As in the dagger affinoid setting, we isolate a category of derived admissible formal schemes. The base category over which everything takes place is \(\mathbf{Aff}_{\bD(R)} = \mathbf{DAlg}^{cn}(\bD(R))\). To topologise this category, we consider the following:


\begin{definition}
    Let $R$ be a Banach ring with pseudo-uniformiser $\pi$. Let $f:A\rightarrow B$ be a map in either $\mathbf{DAlg}^{cn}(\bD(R))$ or in $\mathbf{DAlg}^{cn}(\bD(\sf{Ban}_{R}^{\le1}))$. We say $f$ is \textit{adically finitely presented} if it is of the form $A \to A\hat{\otimes}^\mathbb{L}R\gen{x_{1},\ldots,x_{n}}\big\slash\big\slash(f_{1},\ldots,f_{m})$ where $f_{1},\ldots,f_{m}:R\rightarrow\pi_{0}(A)\hat{\otimes} R\gen{x_{1}, \ldots,x_{n}}$.
\end{definition}

In complete analogy with the classical setting, we define derived adic topologies.

\begin{definition}
Fix a pre-topology $\tau$ on derived connective rings which is descendable and consists of algebraically finitely presented flat maps.
The \textit{adic }$\tau$-pre-topology on $\mathbf{Aff}_{\bD(\mathsf{SNorm}_R^{\le1})}$ consists of covers
$$\{\mathrm{Spec}(B)\rightarrow\mathrm{Spec}(A)\}$$
such that 
\begin{enumerate}
\item
$A\rightarrow B$ is descendable,
\item
$A\rightarrow B$ is adically finitely presented,
\item 
for any map $A\rightarrow C$ with $C$ a derived topologically finitely presented affinoid,
$$\{\mathrm{Spec}((R\big\slash\pi^{n}R\otimes^{\mathbb{L}}(C\hat{\otimes}_{A}^{\mathbb{L}}B))\rightarrow\mathrm{Spec}((R\big\slash\pi^{n}R\otimes^{\mathbb{L}}C)\}$$
is a cover in $\tau$.
\end{enumerate}
\end{definition}
We denote this topology by $\tau_{ad}$. In particular, for $\tau$ consisting of \'{e}tale covers, this gives the \textit{adic \'{e}tale pre-topology}, which we denote by \(\tau_{ad}^{\'{e}t}\). For $\tau$ consisting of faithfully flat covers, we get the \textit{adic faithfully flat topology}, which we denote by $\tau_{ad}^{ff}$. We have a further topology sitting in between:

\begin{definition}
\begin{enumerate}
    \item 
    A map $A\rightarrow B$ is said to be \textit{formally rig-\'{e}tale} if $\mathbb{L}_{B\big\slash A}[\frac{1}{\pi}]\cong 0$.
    \item 
    A morphism $A\rightarrow B$ is said to be \textit{rig-\'{e}tale} (c.f. \cite{MR4466640}*{Definition 1.3.3}) if it is formally rig-\'{e}tale, $\pi_{0}(A)\rightarrow\pi_{0}(B) $ is algebraically flat, and $\pi_{0}(A)\cong \pi_{0}(B)<x_{1},\ldots,x_{n}>\big\slash(f_{1},\ldots,f_{n})$, where the Jacobian determinant $det(J)$ generates an open ideal.
    \item
    The \textit{rig-\'{e}tale} topology, $\tau_{ad}^{rig-\'{e}t}$, is the topology consisting of covers $\{\mathrm{Spec}(B_{i})\rightarrow\mathrm{Spec}(A)\}$ in $\tau_{ad}^{ff}$ where each $A\rightarrow B_{i}$ is rig-\'{e}tale. 
    \end{enumerate}
\end{definition}
In particular if $A\rightarrow B$ is (formally) rig-\'{e}tale, then $A[\frac{1}{\pi}]\rightarrow B[\frac{1}{\pi}]$ is (formally) \'{e}tale, and rig-\'{e}tale morphisms are formally rig-\'{e}tale. 

\begin{lemma}\label{lemma:strongrigetale}
Let $A\rightarrow B$ be a map of derived admissible affinoid $\mathcal{O}_{K}$-algebras which is derived strong. Then it is formally rig-\'{e}tale if and only if $\pi_{0}(A)\rightarrow\pi_{0}(B)$ is formally rig-\'{e}tale.
\end{lemma}

\begin{proof}
If $A\rightarrow B$ is formally rig-\'{e}tale then the base change $\pi_{0}(A)\rightarrow\pi_{0}(A)\hat{\otimes}^{\mathbb{L}}_{A}B\cong\pi_{0}(B)$ is (formally) rig-\'{e}tale. 
Conversely, suppose that $\pi_{0}(A)\rightarrow\pi_{0}(B)$ is formally rig-\'{e}tale. Since the functor $M\mapsto M[\frac{1}{\pi}]$ is transverse to $\pi$-adic modules, $A[\frac{1}{\pi}]\rightarrow B[\frac{1}{\pi}]$ is still derived strong. Moreover we have $\pi_{0}(A[\frac{1}{\pi}]\rightarrow B[\frac{1}{\pi}])\cong(\pi_{0}(A)[\frac{1}{\pi}]\rightarrow\pi_{0}(B)[\frac{1}{\pi}])$ is \'{e}tale by assumption, and the result follows.
\end{proof}

\comment{
\begin{remark}
    If $A\rightarrow B$ is a cover in $\tau_{ad}$, and $A\rightarrow C$ is any map with $C$ a derived adic algebra topologically of finite presentation then $C\hat{\otimes}_{A}^{\mathbb{L}}B$ is automatically a derived adic algebra topologically of finite presentation. Moreover it is a derived adic admissible algebra if $C$. Finally, if $C$ is discrete, then by derived strength $C\hat{\otimes}_{A}^{\mathbb{L}}B$ is also a discrete admissible adic $R$-algebra.
\end{remark}
}

Let us give the most important examples of covers in this topology.

\begin{definition}
Let $R$ be a Banach ring with $|r|\le1$ for all $r\in R$, and a pseudo-uniformiser $\pi$. A map $f:A\rightarrow B$ in $\mathbf{DAlg}^{cn}(\bD(\mathsf{SNorm}_{R}^{\le1}))$ is said to be \textit{adically descendable of order} $d$ if each map 
$$A\big\slash\big\slash \pi^{n}A\rightarrow B\big\slash\big\slash \pi^{n}B$$
is descendable of order $d$.
\end{definition}

The credit for the following belongs to Lucas Mann, who proved it in the condensed setting (\cite{mann2022p}*{Theorem 2.7.2}).

\begin{theorem}
Let $f:A\rightarrow B$ be adically descendable of index $d$. Suppose that $A$ is adically complete. Then $f$ is descendable.
\end{theorem}

\begin{proof}
We need to check that for sufficiently large $d'$ the map $K^{d'}_{f}\rightarrow A$ is zero. Now we have the factorisation
$$K^{d'}_{f}\rightarrow K^{d'}_{f\big\slash \pi^{n}f}\rightarrow A\big\slash \pi^{n}A.$$
Consider 
$$\mathbf{Map}(K^{d}_{f},A)\cong\mathbb{R}^1\varprojlim_{n}\mathbf{Map}(K^{d}_{f},A\big\slash (\pi^{n}A)).$$
We get a short exact sequence 
$$0\rightarrow\mathbb{R}^{1}\varprojlim\mathbf{Map}(K^{m}_{f},A\big\slash \pi^{n}A)\rightarrow\pi_{0}\mathbf{Map}(K^{m}_{f},A)\rightarrow\varprojlim\pi_{0}\mathbf{Map}(K^{m}_{f},A\big\slash \pi^{n}A)\rightarrow0.$$
Thus $K^{m}_{f}\rightarrow A$ lives in the image of $\mathbb{R}^{1}\varprojlim\mathbf{Map}(K^{m}_{f},A\big\slash\pi^{n}A) \to \pi_{0}\mathbf{Map}(K^{m}_{f},A)$, and we may conclude exactly as in \cite{mann2022p}*{Theorem 2.7.2}. \qedhere
\end{proof}



\begin{lemma}\label{lem:etale-pi_0}
    Let $\tau$ be either the \'{e}tale or faithfully flat topologies on rings. A map $f:A\rightarrow B$ of derived admissible adic rings gives rise to a cover in $\tau_{ad}$ if and only if 
    \begin{enumerate}
    \item 
    $f$ is derived strong,
    \item 
    and $\pi_{0}(f):\pi_{0}(A)\rightarrow\pi_{0}(B)$ gives rise to an \'{e}tale/ faithfully flat cover in the usual sense of adic rings.
    \end{enumerate}
\end{lemma}

\begin{proof}
    Suppose that $f$ is derived strong and $\pi_{0}(f):\pi_{0}(A)\rightarrow\pi_{0}(B)$ gives rise to an \'{e}tale/ faithfully flat cover in the usual sense of adic rings. Then each $\pi_{0}(f)\big\slash\big\slash(\pi^{k})\cong\pi_{0}(f)\big\slash(\pi^{k})$ is an \'{e}tale/ faithfully flat cover of finite presentation. Moreover the map $f\big\slash\big\slash(\pi^{k}):A\big\slash\big\slash(\pi^{k})\rightarrow B\big\slash\big\slash(\pi^{k})$ is strongly flat in the sense of \cite{toen2008homotopical}*{Definition 2.2.2.3}. Thus it is a derived faithfully flat map. By \cite{mathew2016galois}*{Proposition 3.31} it is descendable, and in fact is descendable of index $2$. Thus $A\rightarrow B$ is adically descendable of index $2$. Since $A$ is adically complete, $f$ is descendable.

    Conversely suppose that $f:A\rightarrow B$ gives rise to a cover in $\tau_{ad}$. Then each $f\big\slash\big\slash(\pi^{k})$ is a cover in $\tau$. Moreover, each $f\big\slash\big\slash(\pi^{k})$ is derived strong, by the torsion-free assumptions on $\pi_{*}(A)$ and $\pi_{*}(B)$. Thus $f$ is derived strong, and also by base-changing along $A\big\slash\big\slash(\pi^{k})\rightarrow\pi_{0}(A\big\slash\big\slash(\pi^{k}))\cong\pi_{0}(A)\big\slash(\pi^{k})$ gives that $\pi_{0}(A)\big\slash(\pi^{k})\rightarrow \pi_{0}(B)\big\slash(\pi^{k})$ is a cover in $\tau$, as required. 
\end{proof}

Denote by \(\bA_R^{form}\) the full subcategory of \(\mathbf{Aff}_{\bD(R)}\) consisting of derived admissible adic rings. Note that since the inclusion
$$\mathbf{D}_{\ge0}(\sf{Ban}_{R}^{\le1})^{o}\rightarrow\mathbf{D}_{\ge0}(R)$$
is strongly monoidal (and fully faithful), it sends descendable maps to descendable maps, so that  we may indeed view \(\mathbf{A}_R^{form}\) as a full subcategory of \(\mathbf{Aff}_R\) without changing the adic \(\tau\)-topology. The main reason to embed derived admissible algebras into bornological algebras is to be able to view these as \emph{nuclear} bornological algebras in a sense that will be made precise in the next section.  By Proposition \ref{prop:strongadc} it is easy to see that \((\mathbf{Aff}_{\bD(R)}, \tau_{ad}^{rig-et},\mathbf{P}^{sm}, \bA_R^{form})\) is a relative geometry tuple. 

\subsection{Faithfully flat descent}

In this section we prove some faithfully flat descent results in the formal, and rigid, settings. In the condensed setting, similar results along these lines have been obtained in \cite{mikami2023fppf}, \cite{mann2022p}, and \cite{anschutz2024descent}.
.
\begin{definition}
    A map $f:A\rightarrow B$ of bornological algebras is said to be \textit{algebraically flat} if the map of underlying algebras is flat.
\end{definition}

\begin{lemma}\cite{de1996etale}*{Observation 3.1.1}
Let $A\rightarrow B$ be a discrete standard \'{e}tale map of affinoids. Then it is algebraically flat. 
\end{lemma} 

\begin{lemma}
    Let $f:A\rightarrow B$ be an algebraically faithfully flat map of discrete affinoids over a non-trivially valued Banach field $K$. Then 
    $A\rightarrow B$ satisfies descent for $\mathbf{QCoh}$.
\end{lemma}

\begin{proof}
    We may find a rational cover $A\rightarrow A_{i}$ such that each map $A_{i}\rightarrow A_{i}\hat{\otimes}_{A}^{\mathbb{L}}B\cong A_{i}\hat{\otimes}_{A}B$ has a faithfully flat formal model $\tilde{A}_{i}\rightarrow\tilde{B}_{i}$ (\cite{MR4735655}*{Lemma 4.4}, \cite{MR1225983}*{Theorem 4.1}). Such a map is adically faithfully flat, and hence is descendable. Since the rational topology satisfies descent, we win. 
\end{proof}

\subsubsection{\'{E}tale maps and the \'{e}tale topology in rigid geometry}

Let us now discuss the \'{e}tale topology for rigid analytic spaces.
A \(T\)-standard \'etale map as in Definition \ref{def:T-standard-et} specialises to morphisms \(A \to B\) such that \(B\) is isomorphic to an \(A\)-algebra of the form \(A \hat{\otimes}_K^L K\gen{x_1, \dotsc, x_n}//(f_1, \dotsc, f_n)\) with invertible Jacobian.

\begin{lemma}\label{lem:discetaletrans}
   Let $A\rightarrow B$ be a discrete standard \'{e}tale map of discrete affinoids. Then $B\hat{\otimes}_A^{\mathbb{L}} B\cong B\hat{\otimes}_{A}B$ is an equivalence.
\end{lemma}

\begin{proof}
   The transversality condition $B\hat{\otimes}_A^\mathbb{L} B\cong B\hat{\otimes}_{A}B$ is local for the rational topology. Indeed for a localisation $A\rightarrow A_{i}$ we have
   $$(B\hat{\otimes}_{A}^{\mathbb{L}}A_{i})\hat{\otimes}^{\mathbb{L}}_{A_{i}}(B\hat{\otimes}_{A}^{\mathbb{L}}A_{i})\cong(B\hat{\otimes}_A^\mathbb{L} B)\hat{\otimes}^{\mathbb{L}}_{A}A_{i}.$$
   We also have 
 $$(B\hat{\otimes}_{A}A_{i})\hat{\otimes}_{A_{i}}(B\hat{\otimes}_{A}^{\mathbb{L}}A_{i})\cong(B\hat{\otimes}_A B)\hat{\otimes}_{A}A_{i}\cong(B\hat{\otimes}_A B)\hat{\otimes}^{\mathbb{L}}_{A}A_{i}$$
 where in the very last step we have used transversality of affinoids to rational localisations. Thus we may assume that $A\rightarrow B$ has a flat formal model $\tilde{A}\rightarrow\tilde{B}$ (\cite{MR1225983}*{Corollary 5.10}). Now the claim follows from Lemma \ref{lem:flattorsfree}
\end{proof}

 \begin{lemma}\label{lem:rigid-pi_0-etale}
     Let $A$ be an affinoid \(K\)-algebra, and let $B\cong A \langle x_{1},\ldots,x_{n}\rangle \big\slash(f_{1},\ldots,f_{n})$ where the determinant of the Jacobian of the $f_{i}$s is a unit in $B$. Then the diagonal map $B\hat{\otimes}_{A}^{\mathbb{L}}B\rightarrow B$ is a Zariski open immersion in $\mathbf{DAlg}(\sf{CBorn}_{K})$, and hence a localisation. In particular, the map $A\rightarrow B$ is formally \'{e}tale and the map
     $$A \hat{\otimes}_K^{\mathbb{L}} K \langle x_{1},\ldots,x_{n} \rangle \big\slash\big\slash (f_{1},\ldots,f_{n})\rightarrow B$$
     is an equivalence. 
 \end{lemma}

\begin{proof}
Consider the map $A\rightarrow A\hat{\otimes}_K^{\mathbb{L}} K \langle x_{1},\ldots,x_{n}\rangle \big\slash\big\slash (f_{1},\ldots,f_{n})$. By Proposition \ref{prop:Tstandard} this is formally \'{e}tale. In particular, the map 
$$A\rightarrow B\cong \pi_{0}(A\hat{\otimes}_K^L K \gen{x_{1},\ldots,x_{n}}\big\slash\big\slash (f_{1},\ldots,f_{n}))$$ is discrete formally \'{e}tale (i.e, $0\cong\Omega^{1}_{B\slash A}$). We have $0\cong\Omega^{1}_{B\slash A}\cong coker(I^{2}\rightarrow I)$ where $I$ is the kernel of $B\hat{\otimes}_{A}B\rightarrow B$. Now $B\hat{\otimes}_{A}B$ is affinoid, so that all ideals are admissible and we have $I^{2}=I$. The diagonal being a Zariski open immersion now follows from \cite{ben2024perspective}*{Lemma 2.6.154}. Since $B\hat{\otimes}_A^L B\cong B\hat{\otimes}_{A}B$ this implies that $A\rightarrow B$ is formally \'{e}tale. Now we have a commutative diagram
\begin{displaymath}
   \xymatrix{
    & A\ar[dl]\ar[drr] & & \\
A\hat{\otimes}^{\mathbb{L}} K \gen{x_{1},\ldots,x_{n}}\big\slash\big\slash (f_{1},\ldots,f_{n}) &&& B
}
\end{displaymath}
Both diagonal maps are formally \'{e}tale, and after applying $\pi_{0}$ to the bottom algebras, there exists an isomorphism of $A$-algebras between them. Thus it lifts uniquely, up to a contractible choice, to an equivalence over $A$
$$A\hat{\otimes}_K^{\mathbb{L}} K\gen{x_{1},\ldots,x_{n}}\big\slash\big\slash (f_{1},\ldots,f_{n})\rightarrow A$$
 \end{proof}

 \begin{corollary}\label{cor:standard-etale-characterisation}
  A map $A\rightarrow B$ of derived affinoids is standard \'{e}tale if and only if
   \begin{enumerate}
       \item 
       $A\rightarrow B$ is derived strong
       \item
        $\pi_{0}(A)\rightarrow\pi_{0}(B)$ is a discrete \'{e}tale map of affinoids.
   \end{enumerate}
 \end{corollary}

 \begin{proof}
Note that for any standard $\mathsf{Afnd}_{K}$-\'{e}tale map $A\rightarrow B$ of derived bornological algebras the map $\pi_{0}(A)\rightarrow\pi_{0}(B)$ is a discrete standard $\mathsf{Afnd}_{K}$-\'{e}tale map.

For the converse, by \cite{ben2024perspective}*{Proposition 2.6.160}, and the fact that discrete \'{e}tale maps are transverse to finitely presented modules, we know that $A\rightarrow B$ is formally \'{e}tale. Write $\pi_{0}(B)\cong\pi_{0}(A)<x_{1},\ldots,x_{n}>\big\slash(f_{1},\ldots,f_{n})$. Consider the diagram
\begin{displaymath}
    \xymatrix{
& A\ar[dl]\ar[dr] & \\
A<x_{1},\ldots,x_{n}>\big\slash\big\slash (f_{1},\ldots,f_{n}) & & B
     }
\end{displaymath}
Both $A\rightarrow A<x_{1},\ldots,x_{n}>\big\slash\big\slash (f_{1},\ldots,f_{n})$ and $A\rightarrow B$ are formally \'{e}tale. Thus the isomorphism $\pi_{0}(A<x_{1},\ldots,x_{n}>\big\slash (f_{1},\ldots,f_{n}))\cong\pi_{0}(B)$ lifts (uniquely up to a contractible choice) to an equivalence $A\cong B$ by \cite{ben2024perspective}*{Corollary 2.1.36}. 
 \end{proof}

 \begin{corollary}\label{cor:diagonal-loc}
     Let $A\rightarrow B$ be an \'{e}tale map of derived affinoids. Then $B\hat{\otimes}_{A}^{\mathbb{L}}B\rightarrow B$ is a Zariski open immersion, and hence a rational localisation.
 \end{corollary}

 \begin{proof}
This follows immediately from \cite{ben2024perspective}*{ Lemma 2.6.164}. 
 \end{proof}

\comment{
\begin{remark}\label{rem:diagonal-loc}
The lemma and two corollaries above works for any generating class of bornological algebras for which all ideals are closed and finitely generated. In particular this is true for dagger affinoids.
\end{remark}
}
\comment{\edit{There's a second proof.}

\begin{proof}
$\mathrm{Spec}(B)\rightarrow\mathrm{Spec}(B\hat{\otimes}_{A}^{\mathbb{L}}B)$ is a local homotopy monomorphism, and therefore is formally \'{e}tale. Thus $A\rightarrow B$ is formally \'{e}tale. 
\end{proof}}

\begin{definition}\label{def:etale-desc}
    Let $K$ be a non-trivially valued non-Archimedean Banach field. Let $\tau$ be the rig-\'e{tale} pre-topology on simplicial commutative bornological $\mathcal{O}_{K}$-algebras, and consider the corresponding adic topology $\tau_{ad}$. We denote by $\tau_{pre}^{rig-\'{e}t}$ the pre-topology generated by maps $\mathrm{Spec}(B)\rightarrow\mathrm{Spec}(A)$
    such that 
    \begin{enumerate}
        \item 
     $A\rightarrow B$ is $\mathrm{Tate}_K$-\'{e}tale, that is, \(T\)-standard \'etale for \(T(x) = K\gen{x}\).
        \item 
    $A\rightarrow B$ is descendable. 
    \end{enumerate}
Denote by $\tau^{rig-\'{e}t}$ the topology generated by $\tau_{pre}^{rig-\'{e}t}$. 
\end{definition}

By construction we have the following.

\begin{theorem}\label{etaledescentadic}
    $\mathbf{QCoh}$ satisfies descent for $\tau_{pre}^{rig-\'{e}t}$, and hence $\tau^{rig-\'{e}t}$.
\end{theorem}

To describe the \'etale topology, it is convenient to define the following intermediate pre-topology $\tau_{pre}^{rig-\'et}\subset\tau^{rig-\'{e}t}_{ref}\subset\tau^{rig-\'{e}t}$, where we say that a $\mathrm{Afnd}_{K}$-\'{e}tale map $A\rightarrow B$ is a cover if there exists a rational cover $A\rightarrow A'$ such that $A'\rightarrow A'\hat{\otimes}^{\mathbb{L}}_{A}B$ is a descendable map. 

Let $\mathcal{M}(B)\rightarrow\mathcal{M}(A)$ be an \'{e}tale cover of affinoids in the usual sense. Then rationally locally, this \'{e}tale cover has a faithfully flat formal model. Faithfully flat covers of formal models, and hence their base changes are descendable. Thus the pre-topology on affinoids consisting of descendable \'{e}tale maps generates the same topology as the usual \'{e}tale topology. We denote the usual \'{e}tale topology of discrete affinoids by $\tau^{rig-\'{e}tale}_{disc}$. 

\begin{lemma}\label{lem:classical-etale}
    Let $A\rightarrow B$ be a map of derived affinoid algebras. It corresponds to a cover in $\tau_{ref}^{rig-\'{e}t}$ if and only if 
    \begin{enumerate}
        \item 
        it is derived strong,
        \item 
        $\pi_{0}(A)\rightarrow\pi_{0}(B)$
        corresponds to a cover in $\tau^{rig-\'{e}tale}_{disc}$.
    \end{enumerate}
\end{lemma}

\begin{proof}
Suppose $f:A\rightarrow B$ is a cover in $\tau_{ref}^{rig-\'{e}t}$. By Corollary \ref{cor:standard-etale-characterisation} $f$ is derived strong. Let $A\rightarrow A'$ be a rational cover. In particular it is also derived strong. Finally by transversality of finitely presented modules to \'{e}tale maps, the map $A'\rightarrow A'\hat{\otimes}_{A}^{\mathbb{L}}B$ is also derived strong. By assumption it is descendable. Now taking pushouts along the map $A\rightarrow\pi_{0}(A)$ we get a commutative diagram
\begin{displaymath}
    \xymatrix{
\pi_{0}(A)\ar[d]\ar[r]^{\pi_{0}(f)}& \pi_{0}(B)\ar[d]\\
\pi_{0}(A')\ar[r] & \pi_{0}(A')\hat{\otimes}^{\mathbb{L}}_{\pi_{0}(A)}\pi_{0}(B)\cong\pi_{0}(A')\hat{\otimes}_{\pi_{0}(A)}\pi_{0}(B).
    }
\end{displaymath}
The left-hand vertical map is a rational cover, and the bottom horizontal map is descendable and \'{e}tale.

Conversely suppose $f:A\rightarrow B$ is derived strong and $\pi_{0}(A)\rightarrow\pi_{0}(B)$ corresponds to an \'{e}tale cover of discrete affinoids. In particular this is an \'{e}tale, and hence faithfully flat, cover in the discrete sense. We can find a rational cover $\pi_{0}(A)\rightarrow A'_{0}$, and an adically faithfully flat and formally rig-\'{e}tale formal model $\tilde{A}'_{0}\rightarrow \tilde{B}'_{0}$ of the map $A'_{0}\rightarrow\ A'_{0}\hat{\otimes}^{L}_{\pi_{0}(A)}\pi_{0}(B)$. By Postnikov extension, we extend $\pi_{0}(A)\rightarrow A'_{0}$ to a derived strong rational cover $A\rightarrow A'$ as follows. Let $d^{A}_{n}:A_{\le n}\rightarrow\pi_{n+1}(A)[n+2]$ be the canonical derivation. We then just iteratively construct $A'_{\le n+1}=A'_{\le n}\oplus_{A'_{\le n}\hat{\otimes}^{\mathbb{L}}_{A_{\le n}}d^{A}_{n}}(A'_{\le n}\hat{\otimes}^{\mathbb{L}}_{A_{\le n}}\pi_{n+1}(A)[n+2]).$
By construction we have that $A_{\le n+1}\rightarrow A'_{\le n+1}$ is derived strong. Taking limits of the Postnikov towers we get a derived strong map $A\rightarrow A'$ such that $\pi_{0}(A)\rightarrow\pi_{0}(A')$ corresponds to a rational localisation cover. Thus $A\rightarrow A'$ corresponds to a rational localisation cover. Now pick a formal model $\tilde{A}'$ of $A'$ such that $\pi_{0}(\tilde{A}')\cong\tilde{A}'_{0}$. By a similar process we extend $\tilde{A}'_{0}\rightarrow\tilde{B}'_{0}$ to a derived strong map $\tilde{A}'\rightarrow\tilde{B}'$ of derived admissible adic $R$-algebras. Note then each $\tilde{A}'\big\slash\big\slash(\pi^{k})\rightarrow\tilde{B}'\big\slash\big\slash(\pi^{k})$ is also derived strong. Moreover $\pi_{0}(\tilde{A}'\big\slash\big\slash(\pi^{k}))\rightarrow\pi_{0}(\tilde{B}'\big\slash\big\slash(\pi^{k}))$ is a faithfully flat map of finitely presented algebras. Thus $\tilde{A}'\big\slash\big\slash(\pi^{k})\rightarrow\tilde{B}'\big\slash\big\slash(\pi^{k})$ is a faithfully flat map of derived algebras. In particular it is descendable. Thus $\tilde{A}'\rightarrow\tilde{B}'$ is in fact a cover in the adically faithfully flat topology, and so is descendable. By Lemma \ref{lemma:strongrigetale} it is further formally rig-\'{e}tale. 

Finally, we claim that $\tilde{A}'\rightarrow\tilde{B}'$ is a formal model for $A'\rightarrow A'\hat{\otimes}^{\mathbb{L}}_{A}B$. Now since tensoring with $R[\frac{1}{\pi}]$ preserves homology of $\pi$-adic objects, $\tilde{A}'[\frac{1}{\pi}]\rightarrow\tilde{B}'[\frac{1}{\pi}]$ is also derived strong, and $\pi_{0}(\tilde{A}'[\frac{1}{\pi}]\rightarrow\tilde{B}'[\frac{1}{\pi}])\cong(\pi_{0}(A)\rightarrow\pi_{0}(B))$ is formally \'{e}tale. Hence  $\tilde{A}'[\frac{1}{\pi}]\rightarrow\tilde{B}'[\frac{1}{\pi}]$ is formally \'{e}tale. Consider the diagram
\begin{displaymath}
    \xymatrix{
    & A'\ar[dl]\ar[dr] &\\
    \tilde{B}'[\frac{1}{\pi}]& & A'\hat{\otimes}^{\mathbb{L}}_{A}B.
    }
\end{displaymath}
Both legs are formally \'{e}tale maps. The isomorphism $\pi_{0}(\tilde{B}'[\frac{1}{\pi}])\cong\pi_{0}(A'\hat{\otimes}_{A}^{\mathbb{L}}B)$ then lifts, uniquely up to a contractible choice, to an isomorphism of $A'$-algebras $\tilde{B}'[\frac{1}{\pi}]\rightarrow A'\hat{\otimes}^{\mathbb{L}}_{A}B$. All this work is to say that $A\rightarrow B$ has a rational refinement to a descendable \'{e}tale cover. Hence it is a cover in $\tau_{ref}^{rig-\'{e}t}$.
\end{proof}

\comment{
\begin{lemma}\label{lem:faithfullyflathuber}
    Let $f:A\rightarrow B$ be a morphism of discrete affinoid algebras.
    \begin{enumerate}
        \item 
    $\mathrm{Spa}(f):\mathrm{Spa}(B,B^{0})\rightarrow\mathrm{Spa}(A,A^{0})$ is flat if and only if $A\rightarrow B$ is algebraically flat.
    \item 
    If $\mathrm{Spa}(f):\mathrm{Spa}(B,B^{0})\rightarrow\mathrm{Spa}(A,A^{0})$ is flat and surjective then $A\rightarrow B$ is algebraically faithfully flat.
    \item 
    If $A\rightarrow B$ is \'{e}tale and algebraically faithfully flat then $\mathrm{Spa}(f):\mathrm{Spa}(B,B^{0})\rightarrow\mathrm{Spa}(A,A^{0})$ is surjective.
    \end{enumerate}
\end{lemma}

\begin{proof}
\begin{enumerate}
    \item 
Since $A$ and $B$ are topologically finite type $K$-algebras, by\cite{zavyalov2024quotients}*{Lemma B.4.3 and B.4.6}, flatness of $A\rightarrow B$ coincides precisely with flatness of $\mathrm{Spa}(B,B^{0})\rightarrow\mathrm{Spa}(A,A^{0})$.
\item 
 If $\mathrm{Spa}(B,B^{0})\rightarrow\mathrm{Spa}(A,A^{0})$ is surjective then by \cite{zavyalov2024some}*{Lemma 7.2}, $f$ is faithfully flat.
Suppose $A\rightarrow B$ is \'{e}tale and faithfully flat. Let $(k,k^{+})$ be an affinoid field, and let $\mathrm{Spa}(k,k^{+})\rightarrow\mathrm{Spa}(A,A^{0})$ be a morphism. This corresponds precisely to a morphism $g:A\rightarrow k$ such $g(A^{0})\subset k^{+}$. Consider $\tilde{g}:B\rightarrow B\hat{\otimes}_{A}k$. Writing 
 $$B\cong A<x_{1},\ldots,x_{m}>\big\slash(f_{1},\ldots,f_{n}),$$ where we may assume that the $f_{i}$ are power bounded, we have $B\hat{\otimes}_{A}k\cong k<x_{1},\ldots,x_{m}>\big\slash(f_{1},\ldots,f_{n})$. We equip this with the ring of definition 
 $$B\hat{\otimes}_{A}k^{+}\defeq k^{+}<x_{1},\ldots,x_{m}>\big\slash(f_{1},\ldots,f_{n}).$$

 Now it suffices to prove that $\mathrm{Spa}(B\hat{\otimes}_{A}k,B\hat{\otimes}_{A}k^{+})\rightarrow\mathrm{Spa}(k,k^{+})$ is surjective. Hence we may assume we are working with a morphism of Huber pairs $(A,A^{0})\rightarrow (k,k^{+})$ where $(k,k^{+})$ is an affinoid field. An element of $\mathrm{Spa}(k,k^{+})$ corresponds to a valuation subring $k^{+}\subset R\subset k^{\circ}$. The associated value group is $\Gamma_{R}=k^{*}\big\slash R^{*}$ with order $a R^{*}\le bR^{*}$ if $ab^{-1}$. The valuation is $|a|_{R}=aR^{*}$ if $a\neq0$, and $|0|_{R}=0$. Since $k\rightarrow A$ is \'{e}tale, $A$ is isomorphic to a finite product of (separable) extensions of $k$. Now we may extend the valuation on $k$ to one on each of the field extensions by \cite{adicnotes}*{Proposition I.1.2.3}, and this gives the required lift. 
 \qedhere
 
 
\end{enumerate}

\end{proof}
}

\begin{theorem}\label{thm:refetalecoverequivalent}
  Let $K$ be a non-trivially valued Banach field, and let $f:A\rightarrow B$ be a morphism between topologically finitely presented $K$-algebras.
   \begin{enumerate}
       \item The following are equivalent.
    \begin{enumerate}
        \item 
        The set $\{\mathrm{Spec}(f):\mathrm{Spec}(B)\rightarrow\mathrm{Spec}(A)\}$ is an algebraically faithfully flat cover.
        \item 
        The map $\mathcal{M}(f):\mathcal{M}(B)\rightarrow\mathcal{M}(A)$ is a flat surjection.
    \end{enumerate}
    \item
The following are also equivalent.
      \begin{enumerate}
        \item 
        The set $\{\mathrm{Spec}(f):\mathrm{Spec}(B)\rightarrow\mathrm{Spec}(A)\}$ is an \'{e}tale cover.
        \item 
        The map $\mathcal{M}(f):\mathcal{M}(B)\rightarrow\mathcal{M}(A)$ is an \'{e}tale surjection.
            \comment{            \item 
    The map $\mathrm{Spa}(f):\mathrm{Spa}(B,B^{0})\rightarrow\mathrm{Spa}(A,A^{0})$ is an \'{e}tale surjection.}
    \end{enumerate}
    \end{enumerate}
\end{theorem}

\begin{proof}
\begin{enumerate}
    \item 
 Clearly if $\{\mathrm{Spec}(f):\mathrm{Spec}(B)\rightarrow\mathrm{Spec}(A)\}$ is a faithfully flat cover then $f$ is flat. Since $\mathrm{Spec}(f)$ satisfies descent, the functor $\mathrm{Spec}(f)^{*}$ is conservative. Now a point $x\in\mathcal{M}(A)$ corresponds to a bounded map $A\rightarrow\overline{\kappa}(x)$ where $\overline{\kappa}(x)$ is a Banach valued field. Again we have that $B\hat{\otimes}_{A}\overline{\kappa}(x)$ is non-zero. Thus its Berkovich spectrum is non-empty, so there is some $y\in\mathcal{M}(B\hat{\otimes}_{A}\overline{\kappa}(x))$ mapping to $x$. The image of this $y$ in $\mathcal{M}(B)$ maps to $x$. 

Conversely, suppose $\mathcal{M}(f):\mathcal{M}(B)\rightarrow\mathcal{M}(A)$ is a flat surjection. Since flatness is a local property, this implies that $f:A\rightarrow B$ is flat. Surjectivity clearly implies faithfulness.
\item 
This is similar to the first part, using that \'{e}taleness is a local property.
\comment{
That (a)$\Leftrightarrow$(b) clearly follows similarly to (1). That (a)$\Leftrightarrow$(b) follows immediately from Lemma \ref{lem:faithfullyflathuber}.}\qedhere
\end{enumerate}
\end{proof}

\begin{definition}\label{def:rigid-stack}
A \emph{rigid analytic stack over} $K$ is a geometric stack relative to the strong relative geometry tuple \((\mathbf{Aff}_{\bD_{\geq 0}(K)}, \tau_{\bA}^{\'{e}t}, \mathbf{P}^{sm}, \bA_K)\). 
\end{definition}

\subsubsection{Coherent hyperdescent}

For $A$ a derived affinoid algebra, let $\mathbf{Coh}(A)$ denote the category of $A$-modules such that each $\pi_{m}(A)$ is a finitely presented $\pi_{0}(A)$-module. In \cite{ben2024perspective}*{Corollary 2.3.109} it was shown that for $A\rightarrow B$ a map of derived affinoids, and $M\in\mathbf{Coh}_{+}(A)$  $n$-connective for some $n$, then $B\otimes_{A}^{\mathbb{L}}M\in\mathbf{Coh}_{+}(A)$.

\begin{lemma}
    The assignment
    $$\mathbf{Coh}_{+} \colon \mathbf{A}_{K}^{op}\rightarrow\mathbf{Cat},$$
    $$\mathrm{Spec}(A)\mapsto\mathbf{Coh}_{+}(A)$$
    satisfies hyperdescent for the \'{e}tale topology.
\end{lemma}

\begin{proof}
   This is an immediate consequence of Theorem \ref{etaledescentadic}, Theorem \ref{thm:refetalecoverequivalent}, \cite{ben2024perspective}*{Lemma 7.2.34}, and \cite{conrad2003descent}.
\end{proof}

Then by \cite{ben2024perspective}*{Proposition 7.2.41} we get the following.

\begin{corollary}
    Let $\mathrm{Spec}(A)$ be a derived affinoid. Then $\mathbf{Map}(-,\mathrm{Spec}(A))$ is a hypersheaf for the \'{e}tale topology on $\mathbf{Aff}_{\mathbf{D}_{\ge0}(K)}.$
\end{corollary}

\comment{
\begin{lemma}
The sheafification functor
$$\mathbf{Stk}(\mathbf{Aff}_{\bD(K)}, \tau_{\bA}^{rat})\rightarrow\mathbf{Stk}(\mathbf{Aff}_{\bD(K)}, \tau_{\bA}^{\'{e}t})$$
restricts to a fully faithful functor on $\mathbf{dAn}_K$.
\end{lemma}

\begin{proof}
Let $\tilde{X}$ denote the \'{e}tale stackification of a derived analytic space $X$. By forgetting, we identify this with a stack for $\tau^{rat}_{\mathcal{A}}$. We need to show that 

Let $X$ be an object of $\mathbf{dAn}_K$. We claim that it is already a sheaf for the \'{e}tale topology. It suffices to prove that for any affine scheme $\mathrm{Spec}(A)$, the map
$$\pi_{0}\mathbf{Map}_{\mathbf{dAn}}(\mathrm{Spec}(A),X)\rightarrow\pi_{0}\mathbf{Map}_{\mathbf{Stk_{\'{e}t}}}(\mathrm{Spec}(A),\tilde{X})$$
is an equivalence. Let $\{U_{i}\rightarrow X\}$ be a rational cover by derived affinoids. Let $f:\mathrm{Spec}(A)\rightarrow\tilde{X}$ be a map. There exists an \'{e}tale cover $\{\mathrm{Spec}(B_{j})\rightarrow\mathrm{Spec}(A)\}$ such that each $\mathrm{Spec}(B_{j})\rightarrow X$ factors through some $U_{j(i)}\rightarrow X$. The image $V_{j}$ of $\mathrm{Spec}(B_{j})\rightarrow\mathrm{Spec}(A)$ is open, and factors through $U_{j(i)}$. These $V_{j}$ then give an open cover of $\mathrm{Spec}(A)$. This gives that $f$ is the image of a map in $\pi_{0}\mathbf{Map}_{\mathbf{dAn}}(\mathrm{Spec}(A),X)$. Now let $f_{1},f_{2}:\mathrm{Spec}(A)\rightarrow X$ be maps whose images in $\mathbf{Map}_{\mathbf{Stk_{\'{e}t}}}(\mathrm{Spec}(A),\tilde{X})$ are equal. Then there exists an \'{e}tale cover $\{h_{j}\mathrm{Spec}(B_{j})\rightarrow\mathrm{Spec}(A)\}$ such that $f_{1}\circ h_{i}=f_{2}\circ h_{j}$ for all $i$. But the images of the $h_{j}$ give an open cover $\tilde{h}_{j}$ such that $f_{1}\tilde{h}_{j}=f_{2}\tilde{h}_{j}$ for all $j$. Hence these maps agree as maps to $X$. 
\end{proof}
}
\comment{\subsection{The dagger Stein case}

\textcolor{red}{Jack comment: not sure this is necessary, just use homotopy monomorphisms for dagger Steins.}

\subsubsection{\'{E}tale maps of non-Archimedean dagger Steins}

Let $A\rightarrow B=A<\lambda_{1}x_{1},\ldots,\lambda_{n}x_{n}>^{\dagger}\slash(f_{1},\ldots,f_{n})$ be a standard \'{e}tale map of dagger affinoids. That is $J(f_{1},\ldots,f_{n})$ is a unit after base changing to $B$. We may write $A\cong colim_{i\in\mathcal{I}} A_{i}\big\slash I$ where each $A_{i}$ is an affinoid. Then $B\cong colim_{\underline{\rho}<\underline{\lambda}} colim_{i} A_{i}<\rho_{1}x_{1},\ldots,\rho_{n}x_{n}>\big\slash(f_{1},\ldots,f_{n}) $. Now $J(f)$ is a unit in $B$. Since the structure maps are monomorphisms it must be a unit in one of the $A_{i}<\rho_{1}x_{1},\ldots,\rho_{n}x_{n}>\big\slash(f_{1},\ldots,f_{n})$. Thus the map 
$$A_{i}\rightarrow A_{i}<\rho_{1}x_{1},\ldots,\rho_{n}x_{n}>\big\slash(f_{1},\ldots,f_{n})$$
is formally \'{e}tale. As a colimit of formally \'{e}tale maps  $A\rightarrow B$ is formally \'{e}tale. Moreover any such map is transverse to $A$-modules $M$ which are finitely presented over $A<\gamma_{1}x_{1},\ldots,\gamma_{m}x_{m}>^{\dagger}$ for some $m$ and some $\underline{\gamma}$. 

Now let $A\rightarrow B$ be an \'{e}tale map of dagger Steins. There is a homotopy monomorphism cover $A\rightarrow C_{i}$ such that each $C_{i}\rightarrow C_{i}\hat{\otimes}_{A}B$ is standard \'{e}tale. By descent for transversality $A\rightarrow B$ is transverse to $A$-modules $M$ which are finitely presented over $A<\gamma_{1}x_{1},\ldots,\gamma_{m}x_{m}>^{\dagger}$ for some $m$ and some $\underline{\gamma}$. 

\begin{definition}
    A map $A\rightarrow B$ of dagger Steins is said to be an \textit{\'{e}tale cover} if
    \begin{enumerate}
        \item 
        it is dagger \'{e}tale
        \item 
        it is descendable.
    \end{enumerate}
\end{definition}

By \cite{ben2024perspective} we have the following.

\begin{proposition}
    Let $A\rightarrow B$ be an \'{e}tale cover of dagger Steins. Then the map $\mathcal{M}(B)\rightarrow\mathcal{M}(A)$ of Berkovich spaces is surjective. 
\end{proposition}

\begin{example}
        If $A\rightarrow B$ is a cover in the finite $G$-topology then it is an \'{e}tale cover.
\end{example}
}

\subsubsection{Bornological adic spaces}

In this subsection we give a rudimentary definition of bornological (analytic) adic spaces. Really here we should work with Banach Tate rings, but we postpone that to future work.
Let $\mathbb{Z}_{triv}$ denote the non-Archimedean Banach ring of integers with $|n|=1$ for all $n\in\mathbb{Z}$ with $n\neq0$. We consider the strong geometry tuple
$$(\mathbf{Aff}_{\bD(\mathbb{Z}_{triv})},\tau_{\mathbf{A}}^{\'{e}t},\mathbf{P}^{sm}_{\mathbf{A}},\mathbf{A}),$$
where $\mathbf{A}$ consists of those algebras which are derived affinoids over some non-trivially valued Banach field $K$.
Note that for any non-trivially valued Banach field $K$, 

$$(\mathbf{Aff}_{D_{\ge0}(\mathbb{Z}_{triv})},\tau_{\mathbf{A}}^{\'{e}t},\mathbf{P}^{sm}_{\mathbf{A}},\mathbf{A}_{K})$$
is a strong geometry tuple, so that $$\mathbf{dAn}_{K}\subseteq\mathbf{Stk}_{geom}(\mathbf{Aff}_{
\bD(\mathbb{Z}_{triv})},\tau_{\mathbf{A}}^{\'{e}t},\mathbf{P}^{sm}_{\mathbf{A}},\mathbf{A}).$$

\begin{definition}
    Let $X$ be a derived scheme over $\mathbb{Z}_{triv}$. $X$ is said to be a \textit{derived analytic space} if there is a cover $\{U_{i}\rightarrow X\}$ such that each $U_{i}\cong\mathrm{Spec}(A_{i})$, with $A_{i}$ being a derived affinoid over a non-trivially valued Banach field $K_{i}$. 
\end{definition}

The full subcategory of $\mathbf{Stk}_{geom}(\mathbf{Aff}_{\bD(\mathbb{Z}_{triv})},\tau_{\mathbf{A}}^{\'{e}t},\mathbf{P}^{sm}_{\mathbf{A}},\mathbf{Aff}_{\bD(\mathbb{Z}_{triv})})$ consisting of derived analytic spaces is denoted $\mathbf{dAn}_{\mathbb{Z}_{triv}}$. 

\begin{lemma}
   Let $f:X\rightarrow Z$ be a smooth map of schemes in the geometry tuple  $(\mathbf{Aff}_{\bD(\mathbb{Z}_{triv})},\tau_{\mathbf{A}}^{\'{e}t},\mathbf{P}^{sm}_{\mathbf{A}},\mathbf{A})$, and $g:Y\rightarrow Z$ any morphism with $Y$ a derived analytic space. Then $Y\times_{X}Z$ is a derived analytic space.
\end{lemma}

\begin{proof}
    The question is local, so we may assume that $X\cong\mathrm{Spec}(B),Y\cong\mathrm{Spec}(C)$, and $Z\cong\mathrm{Spec}(A)$ are all affine, and that $f$ is standard smooth. That is, $B\cong A<x_{1},\ldots,x_{m+n}>\big\slash\big\slash (f_{1},\ldots,f_{m})$. Thus 
    $$B\hat{\otimes}_{A}^{\mathbb{L}}C\cong C<x_{1},\ldots,x_{m+n}>\big\slash\big\slash (f_{1},\ldots,f_{m}).$$   
\end{proof}

\subsection{Descent for \'etale covers of analytic spaces}

\subsubsection{Underlying spaces and effective epimorphisms}

\comment{
Let $X$ be a pre-stack. Denote by $\mathcal{O}_{X}$ the object of $\mathbf{QCoh}(X)$ whose value on a map $\mathrm{Spec}(A)\rightarrow X$ is $A$. 

If $X$ is a derived analytic space then $\mathcal{O}_{X}\in\mathbf{Coh}(X)$. Suppose that $\{f_{i}:U_{i}=\mathrm{Spec}(A_{i})\rightarrow X\}$ is an atlas for $X$. Let $M\in\mathbf{QCoh}(X)$. We say that $M$ is \textit{coherent} if each $\pi_{n}(M)$ is coherent as a $t_{\le0}X$-module. Note in the affinoid/ dagger affinoid/ formal cases this is the same as each $\pi_{n}(M)$ being finitely presented.
}

\begin{definition}
A morphism $f:X\rightarrow Y$ of derived analytic spaces is said to be \textit{derived strong} if there are covers $\{\mathrm{Spec}(B_{j})\rightarrow X\}$, $\{\mathrm{Spec}(A_{i})\rightarrow Y\}$, such that each map $\mathrm{Spec}(B_{j})\rightarrow Y$ factors through some $\mathrm{Spec}(A_{i})$, and the map $A_{i}\rightarrow B_{j}$ is derived strong.
\end{definition}

\begin{proposition}
    Let $\mathrm{Spec}(B)\rightarrow\mathrm{Spec}(A)$ be a map of derived affinoids. It is derived strong as a map of analytic spaces if and only if $A\rightarrow B$ is derived strong as a map of algebras.
\end{proposition}

\begin{proof}
    If $A\rightarrow B$ is derived strong as a map of algebras then clearly $\mathrm{Spec}(B)\rightarrow\mathrm{Spec}(A)$ is derived strong as a map of analytic spaces. 

    Conversely suppose that $\mathrm{Spec}(B)\rightarrow\mathrm{Spec}(A)$ is derived strong as a map of analytic spaces. Fix atlases $\{\mathrm{Spec}(B_{j})\rightarrow\mathrm{Spec}(B)\}$ and $\mathrm{Spec}(A_{i})\rightarrow\mathrm{Spec}(A)\}$ such that each $\mathrm{Spec}(B_{j})\rightarrow\mathrm{Spec}(A)$ factors through some $\mathrm{Spec}(A_{i})\rightarrow\mathrm{Spec}(A)$, and $A_{i}\rightarrow B_{j}$ is derived strong. Note that each map $B\rightarrow B_{j}$ and $A\rightarrow A_{i}$ is also derived strong. In particular $A\rightarrow B_{j}$ is derived strong. We then have  
    $$\pi_{n}(B_{j})\cong\pi_{0}(B_{j})\otimes^{\mathbb{L}}_{\pi_{0}(A)}\pi_{n}(A)\cong\pi_{0}(B_{j})\otimes_{\pi_{0}(B)}^{\mathbb{L}}\pi_{0}(B)\otimes^{\mathbb{L}}_{\pi_{0}(A)}\pi_{n}(A).$$
    We also have 
    $$\pi_{n}(B_{j})\cong\pi_{0}(B_{j})\otimes^{\mathbb{L}}_{\pi_{0}(B)}\pi_{n}(B).$$
    Thus by descent we have 
    $$\pi_{0}(B)\otimes^{\mathbb{L}}_{\pi_{0}(A)}\pi_{n}(A)\cong\pi_{n}(B),$$
    as required.
\end{proof}

\begin{lemma}
    Let $X\rightarrow Y$ be a morphism. Suppose there exists an atlas $\{\mathrm{Spec}(C_{k})\rightarrow X\}$ such that each $\mathrm{Spec}(B_{i})\rightarrow Y$ is derived strong. Then $X\rightarrow Y$ is derived strong.
\end{lemma}

\begin{proof}
Pick covers $\{\mathrm{Spec}(B_{j_{k}})\rightarrow \mathrm{Spec}(C_{k})\}$, $\{\mathrm{Spec}(A_{i})\rightarrow Y\}$, such that each map $\mathrm{Spec}(B_{j_{k}})\rightarrow Y$ factors through some $\mathrm{Spec}(A_{i})$, and the map $A_{i}\rightarrow B_{j_{k}}$ is derived strong. Then $\{\mathrm{Spec}(B_{j_{k}})\rightarrow X\}_{k,j_{k}}$ is a cover satisfying the required properties. 
\end{proof}

\begin{lemma}
    Let $f:X\rightarrow Y$ be a morphism. Suppose there exists a cover $\{\mathrm{Spec}(A_{i})\rightarrow Y\}$ such that each $X\times_{Y}\mathrm{Spec}(A_{i})\rightarrow\mathrm{Spec}(A_{i})$ is derived strong. Then $f$ is derived strong.
\end{lemma}

\begin{proof}
    Straightforward.
\end{proof}

\begin{example}
    Let $X$ be a derived analytic space and $U\rightarrow X$ an open immersion. Then this morphism is derived strong.
\end{example}

\begin{lemma}\label{lem:strongtruncationdan}
    Let $X\rightarrow Y$ be a derived strong map of derived analytic spaces. Then the map $t_{\le0}X\rightarrow X\times_{Y}t_{\le0}Y$ is an equivalence.
\end{lemma}

\begin{proof}
This question is local, and reduces to the result for affines \cite{ben2024perspective}.
\end{proof}

\begin{lemma}
    A morphism $f:X\rightarrow Y$ of derived analytic spaces is smooth (resp. \'{e}tale) if and only if
    \begin{enumerate}
        \item 
        it is derived strong,
        \item 
        $t_{\le0}f: t_{\le0}X\rightarrow t_{\le0 }Y$ is an \'{e}tale (resp. smooth) map of discrete analytic spaces.
    \end{enumerate}
\end{lemma}

\begin{proof}
This is again a local question, and so reduces to the case for affines \cite{ben2024perspective}. \qedhere
\comment{
    Suppose $f$ is \'{e}tale. Pick atlases $U_{i}=\mathrm{Spec}(B_{i})\rightarrow X$ and $V_{j}=\mathrm{Spec}(A_{j})\rightarrow Y$ such that $f$ lifts to \'{e}tale maps $f_{i}:U_{i}\rightarrow V_{j(i)}$ of derived affinoids. Then we have each $f_{i}$ is derived strong, and each $t_{\le0}f_{i}$ is an \'{e}tale map of discrete analytic spaces. Now $t_{\le0}U_{i}\rightarrow t_{\le0}X$ and $t_{\le0}V_{j}\rightarrow t_{\le0}Y$ are atlases for $t_{\le0}X$ and $t_{\le0} Y$ respectively. Moreover $t_{\le0}f_{i}$ are lifts of $t_{\le0}f$. Thus $t_{\le0}f$ is an \'{e}tale map of discrete analytic spaces. Moreover we have $(t_{\le0}f_{i})^{*}\pi_{n}(\mathcal{O}_{v_{i}})\cong$

    Conversely suppose that $f$ is derived strong and $t_{\le0}f: t_{\le0}X\rightarrow t_{\le0 }Y$ is an \'{e}tale map of discrete analytic spaces. 
    }
\end{proof}

\begin{lemma}\label{lem:smootheffepi}
    Let $f:X\rightarrow Y$ be a smooth morphism of derived rigid analytic spaces. Then $f$ is an effective epimorphism if and only if
        $t_{\le0}X\rightarrow t_{\le0}Y$ is a smooth effective epimorphism.
\end{lemma}

\begin{proof}
  Since $f$ is smooth it is derived strong.

First suppose that $f$ is an effective epimorphism.
The fibre-product $t_{\le0}X\cong X\times_{Y}t_{\le0}Y\rightarrow t_{\le0}Y$ is a smooth effective epimorphism.

Conversely suppose that $t_{\le0} f$ is a smooth effective epimorphism. Let $U\rightarrow Y$ be a morphism with $U=\mathrm{Spec}(A)$ a derived affinoid. Consider the map on truncations $t_{\le0}U\rightarrow t_{\le0} Y$. We may find an \'{e}tale cover $\{V^{0}\rightarrow t_{\le0} U\}$, with $V^{0}$ discrete, such that $V^{0}\rightarrow t_{\le0}Y$ lifts to a map $V^{0}\rightarrow t_{\le0} X$. Equivalently we have a map $V^{0}\rightarrow t_{\le0}U\times_{t_{\le0}Y}t_{\le0}X\cong t_{\le0}(U\times_{Y}X)$.  As in Proposition \ref{lem:classical-etale} we may lift $V^{0}\rightarrow t_{\le0}U$ to an \'{e}tale cover of derived affinoids $g:V\rightarrow U$. We need to show that we can lift the map $V^{0}\rightarrow t_{\le0}U\times_{t_{\le0}Y}t_{\le0}X\cong t_{\le0}(U\times_{Y}X)$ to a map $V\rightarrow U\times_{Y}X$. Now $U\times_{Y}X\rightarrow U$ is smooth. Since $U\times_{Y}X$ is a scheme, it has an obstruction theory by \cite{kelly2022analytic} 
Proposition 8.12. We may then inductively climb the Postnikov tower of $V$ to lift
$t_{\le0}V\cong V^{0}\rightarrow t_{\le0} U$ to a map $V\rightarrow U\times_{Y}X$, as required. 
\end{proof}

\begin{corollary}
    Let $f:X\rightarrow Y$ be a smooth morphism of derived rigid analytic spaces. Then $f$ is an effective epimorphism if and only if $|f|_{Ber}:|X|_{Ber}\rightarrow |Y|_{Ber}$ is.
\end{corollary}

\begin{proof}
This follows immediately from Lemma \ref{lem:smootheffepi}, Theorem \ref{thm:refetalecoverequivalent}, and Corollary \ref{cor:equivalenttruncation}.
\comment{
Suppose $|f|_{Ber}$ is a surjection. Since the underlying topological space of a derived analytic space coincides with the underlying space of its classical truncation, we may assume by Lemma \ref{lem:smootheffepi} that $X$ and $Y$ are discrete. Moreover the question is local so we may assume that $X\cong\mathrm{Spec}(B)$ and $Y\cong\mathrm{Spec}(A)$ are affinoids. Now by \cite{soor2024derived}*{Remark 2.28 (iv)}, $|X|$ and $|Y|$ are the Huber spectra of $B$ and $A$ respectively. Hence we may use
Lemma \ref{lem:faithfullyflathuber}.
}
\end{proof}

\comment{
\begin{theorem}
    Let $f:\mathcal{X}\rightarrow\mathcal{Y}$ be an effective epimorphism . Then $\mathbf{QCoh}$ satisfies  descent for $i(f):i(\mathcal{X})\rightarrow i(\mathcal{Y})$ 
\end{theorem}

\begin{proof}
By picking an atlas in the $G$-topology we may assume that $\mathcal{Y}=U=\mathrm{Spec}(A)\}$ is affinoid. Pick an atlas in the $G$-topology $\{U_{i}\rightarrow\mathcal{X}\}_{i\in\mathcal{I}}$. The composite cover $\{U_{i}=\mathrm{Spec}(A_{i})\rightarrow U\}_{i\in\mathcal{I}}$ is an \'{e}tale cover of affinoids by affinoids. In particular we may assume that $\mathcal{I}$ is finite and that $\{\mathrm{Spec}(\prod_{i\in\mathcal{I}}A_{i})\rightarrow\mathrm{Spec}(A)\}$ is a usual \'{e}tale cover of affinoids by affinoids. We know this satisfies descent by Theorem \ref{etaledescentadic}.
\end{proof}
}
\comment{
\begin{theorem}
    Let
    \begin{displaymath}
        \xymatrix{
        W\ar[r]\ar[d] & V\ar[d]^{f}\\
        U\ar[r]^{j} & X
        }
    \end{displaymath}
be a Nisnevich square of derived rigid analytic spaces. Then it is a complete, regular $\mathbf{QCoh}$-descent square.
\end{theorem}

\begin{proof}
    The square is a pullback square by definition. Moreover $\{V\coprod U\rightarrow X\}$ is a universal $\mathbf{QCoh}$-descent morphism by Theorem \ref{thm:adicdescendable}. Now consider the square
      \begin{displaymath}
        \xymatrix{
        W\ar[r]\ar[d] & V\ar[d]\\
        W\times_{U}W\ar[r] & V\times_{X}V
        }
    \end{displaymath}
    Now $W\cong U\times_{X}V$ so $W\times_{U}W\cong V\times_{X}V\times_{X}U$. Thus $V\times_{X}V\times_{X}U\rightarrow V\times_{X}V$ is just the pullback of the map $U\rightarrow X$ along $V\times_{X}V\rightarrow X$ and hence is a monomorphism.

    The real content of this statement is that 
    $$V\coprod V\times_{X}V\times_{X}U\rightarrow V\times_{X}V$$
    is still an \'{e}tale cover. Now $V\times_{X}V\rightarrow X$ is in fact an open immersion of derived rigid analytic spaces. In particular it is derived strong. Thus it suffices to prove that  
        \begin{displaymath}
        \xymatrix{
        t_{0}(W)\ar[r]\ar[d] & t_{0}(V)\ar[d]\\
        t_{0}(W)\times_{t_{0}(U)}t_{0}(W)\ar[r] & t_{0}(V)\times_{t_{0}(X)}t_{0}(V)
        }
    \end{displaymath}
    is a Nisnevich square, which is done in \cite{andreychev2023k}.
\end{proof}
}
\section{Nuclear bornological modules and rigidification}

In this section, we introduce a subcategory of the derived \(\infty\)-category \(\mathbf{D}(R)\) of complete bornological \(R\)-modules that will play a crucial role for the rest of this article. Recall that an object \(M \in \mathsf{Ind}(\mathsf{Ban}_R)\) is called \textit{nuclear} if for any Banach \(R\)-module \(X\), we have an isomorphism \[X^\vee \haotimes M \to \mathsf{Hom}(X,M)\] in \(\mathsf{Ind}(\mathsf{Ban}_R)\), where \(X^\vee = \underline{\mathsf{Hom}}(X,R)\). A morphism \(f \colon M \to N\)  is called \textit{nuclear} or \textit{trace class} if there is a bounded \(R\)-linear map \(R \to M^\vee \haotimes N\) such that \(f\) is the composition \[M \cong R \haotimes M \to M \haotimes M^\vee \haotimes N \to N.\] That is, \(f\) is in the image of the canonical morphism \[\underline{\Hom}(R, M^\vee \haotimes N) \to \underline{\Hom}(M,N).\]

We call a complete bornological \(R\)-module \(M\) \emph{nuclear} if \(\mathsf{diss}(M)\) is nuclear in \(\mathsf{Ind}(\mathsf{Ban}_R)\). This definition - after some technical embellishments - generalises to chain complexes of bornological modules to yield a so-called rigid category of chain complexes of (infinitely) \textit{nuclear modules}. To set things up, we first record the following result:

\begin{lemma}\cite{ben2020fr}*{Lemma 4.14-4.16}\label{lem:nuclear-characracterisation} 
An inductive system \(M \in \mathsf{Ind}(\mathsf{Ban}_R)\) is nuclear if and only if there is a directed set \(I\) such that for each \(i \in I\), there is a \(j \geq i\) such that the structure map \(M_i \to M_j\) is trace-class.
\end{lemma}

In what follows, we will need a stronger form of nuclearity in which the trace-class maps in the inductive system representing a nuclear module factorise further into trace-class maps. Denote by \(\mathsf{Tr}^\Q\) the collection of bounded maps \(f \colon X_0 \to X_1\) between Banach spaces such that there is a diagram \(F \colon [0,1] \cap \Q \to \mathsf{Ban}_R\) with \(f = F(0) \to F(1)\), and for any \(i< j \in [0,1] \cap \Q\), \(F(i) \to F(j)\) is trace-class.  We shall call such a trace-class map \(f \colon X_0 \to X_1\) \emph{factorisably trace-class}.

\begin{definition}\label{def:infinitely-nuclear}
An \textit{infinitely nuclear} module is an inductive system \(M \in \mathsf{Ind(Ban)}_R\), such that there exists a representing directed set \(I\) where for each \(i \in I\), there is a \(j > i\) such that the structure map \(M_i \to M_j\) is factorisably trace-class. 
\end{definition}

\begin{example}[Disc algebras]\label{ex:disc-algebras}
Let \(r\) and \(\varrho \in \R_{>0}\) be radii. Then the algebra of analytic functions \[\mathcal{O}(D_{<r}^1) = \varprojlim_{\varrho > r} R\gen{\frac{x}{\varrho}}\] on an open disc of radius \(r\), and the dagger algebra of overconvergent analytic functions \[ R\gen{\frac{x}{r}}^\dagger = \varinjlim_{\varrho > r} R\gen{\frac{x}{\varrho}}\] on a disc of radius \(r\) are \(\Q\)-nuclear.  For the dagger algebra this is immediate as the structure maps \[R\gen{\frac{x}{\varrho}}  \to R\gen{\frac{x}{\tau}}\] are trace-class, and one can always find some radius \(\alpha\) between \(\varrho\) and \(\tau\) such the map above factorises as trace-class maps \[R\gen{\frac{x}{\varrho}} \to R\gen{\frac{x}{\alpha}} \to  R\gen{\frac{x}{\tau}}.\] One can do something similar for the algebra of analytic functions, but one first needs to write \(\mathcal{O}((D_{<r}^1)\) as a colimit of disc algebras. This is done in \cite{ben2020fr}*{Corollary 6.9}, and the structure maps are again \(\R^n\) indexed trace-class maps. 
\end{example}

Let \(\mathsf{Nuc}(R)\) and \(\mathsf{Nuc}^{\infty}(R)\) denote the full subcategories of \(\mathsf{Ind}(\mathsf{Ban}_R)\) of nuclear and \(\Q\)-nuclear modules. There are related categories of nuclear and infinitely-nuclear objects in the category \(\bD(R)\). 

\subsection{Nuclear objects in general stable categories}

More generally, this discussion can be repeated in any stable, compactly generated, closed symmetric monoidal $\infty$-category $\mathbf{C}$. Let us fix such a category.

\begin{definition}\label{def:scholze-nuclear}
Let \(\bC\) be a presentable, closed symmetric monoidal stable \(\infty\)-category with compact unit object \(1\). 
\begin{itemize}
\item We call a morphism \(f \colon M \to N\) in \(\bC\) \textit{trace-class} if it is in the image of the canonical map \(\pi_0\mathbf{Map}(R, M^\vee \otimes N) \to \pi_0\mathbf{Map}(M, N)\).
\item We call \(M \in \bC\) \textit{nuclear} if for any compact object \(X\), the canonical map \[\mathbf{Map}(1, \underline{\mathbf{Map}}(X,R) \otimes M) \to \mathbf{Map}(X,M)\] is an equivalence.
\item We call \(M \in \bC\) \textit{strongly nuclear} if for any compact object \(X\), the canonical map \[\underline{\mathbf{Map}}(X,R) \otimes M \to \underline{\mathbf{Map}}(X,M)\] is an equivalence.
\item We call \(M \in \bC\) \textit{basic nuclear} if it is a sequential colimit \[M \simeq \varinjlim_n M_n\] of compact objects of \(\bC\), where the transition maps are trace-class.
\item We call \(M \in \bC\) \(\Q\)-\textit{basic nuclear} if it is a \(\Q\)-colimit as above, with trace-class transition maps.
\end{itemize}
\end{definition}

The following was proven in the one-categorical case in \cite{ben2020fr}*{Lemma 4.15}, and chain complexes of such in \cite{ben2024perspective}*{Corollary 3.1.52}. The same proof as  \cite{ben2020fr}*{Lemma 4.15} works in the compactly generated stable $\infty$-categorical setting as well, as established in \cite{aoki2025very}*{Proposition 7.11}.

\begin{proposition}[\cite{ben2020fr}*{Lemma 4.15},\cite{aoki2025very}*{Proposition 7.11}]
    If $\mathbf{C}$ is compactly generated then nuclear objects and strongly nuclear objects coincide.
\end{proposition}

By \cite{clausenscholze3}*{Theorem 8.6}, the basic nuclear objects of \(\bC\) are the \(\omega_1\)-compact objects of \(\bC\). Denote by \(\mathbf{Nuc}(\bC)\) the full subcategory of \(\bC\) generated under colimits of the basic nuclear objects. By definition, \(\mathbf{Nuc}(\bC)\) is \(\omega_1\)-compactly generated. 

\comment{
\begin{lemma}[\cite{ben2024perspective}*{Corollary 2.1.56}]\label{lem:nuclear-relative}
    Suppose that every nuclear object of $\mathbf{C}$ is strongly nuclear. Let $A\in\mathbf{Comm}(\mathbf{C})$ have strongly nuclear underlying object in $\mathbf{C}$. Then any nuclear $A$-module is strongly nuclear.
\end{lemma}
}

We now study the full subcategory of \(\bC\) generated under colimits of \(\Q\)-basic nuclear objects, which we denote by \(\mathbf{Nuc}^\infty(\mathbf{C})\).

\comment{\begin{lemma}\label{lem:abstract-nuc}
Let \(\bC\) be a compactly generated, closed symmetric monoidal, stable \(\infty\)-category with compact unit object, and exact filtered colimits. Then the full \(\infty\)-subcategory \(\mathbf{Nuc}^\infty(\bC)\) is \(\omega_1\)-compactly generated and the \(\omega_1\)-compact objects are precisely the \(\Q\)-basic nuclear objects. 
\end{lemma}

\begin{proof}
We first observe that the \(\Q\)-basic nuclear objects are \(\omega_1\)-compact objects in \(\mathbf{Nuc}^\infty(\bC)\). This holds because every \(\Q\)-basic nuclear object may be written as a sequential colimit of compact objects of \(\bC\), so that the transition maps are in particular trace-class and therefore compact. Furthermore, since the factorisably trace-class maps form a factorisable, precompact ideal, and since the filtered colimits in \(\bC\) are exact, the same proof as in \cite{nkp}*{Proposition 2.3.14} shows that the full subcategory of \(\Q\)-basic nuclear objects is closed under countable colimits. Finally, since the \(\Q\)-basic nuclear objects generate \(\mathbf{Nuc}^\infty(\bC)\), its \(\omega_1\)-compact objects are countable colimits of such objects, proving that the \(\omega_1\)-compact objects agree with the \(\Q\)-basic nuclear objects in \(\mathbf{Nuc}^\infty(\bC)\). 
\end{proof}}

We will be particularly interested in the following setting. 
    Let $\underline{\mathbf{C}}=(\mathbf{C},\mathbf{C}_{\ge0},\mathbf{C}_{\le0},\mathbf{C}^{0})$ be a derived algebraic context. As a monoidal $\infty$-category it is equivalent to $\mathbf{Ch}(\mathbf{C}^{\heart})$. Moreover $\mathbf{C}^{\heart}\cong\mathcal{P}_{\Sigma}(\mathbf{C}^{0})\cong\mathrm{Ind}(\mathbf{C}^{\heart,cpct}).$ In this setup, we establish the following version of \cite{clausenscholze3}*{ Lemma 8.7.}, in derived algebraic contexts, for both basic nuclear and basic infinitely nuclear objects.

\begin{proposition}\label{prop:basic-nuclear-explicit}
Let $\underline{\mathbf{C}}$ be a derived algebraic context. Every \(M \in \mathbf{Nuc}(\mathbf{C})^{\omega_1}\) may be represented by a sequential colimit  \(M \simeq \varinjlim_n M_n\) of chain complexes with finitely many terms, where each \(M_{n,k}\) is a compact projective. Each map in the inductive system \(M_n \to M_{n+1}\) is given termwise by injective trace-class maps.  
\end{proposition}

\begin{proof}
Denote the monoidal unit of $\mathbf{C}$ by $R$. 
By a purely formal argument, we may reduce \(M\) to a sequential colimit of finite complexes with trace-class transition maps \(M_i \to M_{i+1}\), whose terms \(M_i\) are compact projective generators. It therefore remains to show that these maps are term-wise trace-class morphisms of objects of $\mathbf{C}^{\heart}$. Since projective objects in \(\mathbf{C}^{\heart}\) are flat, we first observe that the derived dual \(\mathbf{RHom}(P, R)\) is concentrated in degree zero. Now by definition, a trace-class morphism \(M \to N\) in \(\mathbf{C}\) comes from a morphism in \(H_0(\mathbf{RHom}(R, M^\vee \otimes_R^L N))\). Reducing to the case that the terms in \(M\) and \(N\) are compact projective generators of the form \(P\in\mathbf{C}^{0}\), and using that the derived dual is concentrated in degree \(0\), and that \(N\) is degreewise flat, we have that \[H_0(\mathbf{RHom}(R, M^\vee \otimes_R^L N)) \cong H_0(\underline{\Hom}(R, \underline{\Hom}(M, R) \otimes_R N)),\] where \(\underline{\Hom}(M, R)\) is the usual Hom-complex. This yields the desired result for basic nuclear complexes. 

\end{proof}

\begin{remark}\label{rem:basic-nuclear-explicit}
    Note that it is unclear whether the analogous claim as in the proposition above holds for \(\Q\)-basic nuclear complexes. We will later show that, in the case of bornological spaces, any chain complex which may be represented as a \(\Q\)-indexed colimit of complexes of compact projectives with termwise trace-class maps is \(\Q\)-basic nuclear, the converse would require representing a \(\Q\)-indexed inductive system with trace-class transition maps by an actual chain complex of compact projectives with termwise factorisably trace-class chain maps. 
\end{remark}

\comment{
We also have the following useful consequence.

\begin{corollary}
Let $\underline{\mathbf{C}}$ be a derived algebraic context. Nuclear (and hence infinitely nuclear) objects of $\mathbf{C}$ are strongly nuclear.
\end{corollary}

\begin{proof}
    Let $M_{\bullet}$ be nuclear. We need to show that for any compact projective $P$, the natural map
    $$P^{\vee}\otimes^{\mathbb{L}}M_{\bullet}\rightarrow\mathbf{Map}(P,M_{\bullet})$$
    is an equivalence. By Proposition \ref{prop:basic-nuclear-explicit} we may assume that $M_{\bullet}$ is degree-wise basic nuclear in the heart $\mathrm{Ind}(\mathbf{C}^{\heart,cpct}).$ Then we have $P^{\vee}\otimes^{\mathbb{L}}M_{\bullet}\cong P^{\vee}\otimes M_{\bullet}$, and $\mathbf{Map}(P,M)\cong\mathrm{Hom}(P,M_{\bullet})$. Thus it will suffice to prove that the map 
    $$P^{\vee}\otimes M_{\bullet}\rightarrow\mathrm{Hom}(P,M_{\bullet})$$
    is an isomorphism of complexes. This is \cite{ben2020fr}*{Lemma 4.15}. 
\end{proof}
}

We now investigate the formal properties of the categories of nuclear and infinitely nuclear objects in $\mathbf{C}$. In general, these categories are \emph{not} compactly generated although \(\mathbf{C}\) is. They are however $\omega_{1}$-presentable. Let us return to the general setting of a fixed compactly generated symmetric monoidal stable $\infty$-category $\mathbf{C}$. Fix also a set $\mathcal{S}$ of compact generators.

\begin{proposition}
    Let $F:\mathbf{C}\rightarrow\mathbf{D}$ be a colimit-preserving strongly monoidal functor between presentable, closed, symmetric monoidal, stable $\infty$-categories. Then $F$ sends (factorisably) trace-class morphisms to (factorisably) trace-class morphisms. In particular it sends (infinitely) nuclear objects to (infinitely) nuclear objects.
\end{proposition}

\begin{proof}
Let us first explain why $F$ preserving (factorisably) trace class maps implies it preserves (infinitely) nuclear objects. Let \(M \in \mathbf{Nuc}(\mathbf{C})\) (resp.  \(M \in \mathbf{Nuc}^{\infty}(\mathbf{C})\)), then we may write \(M\) as a  $\colim_{i\in\mathcal{I}}M_{i}$ with each $M_{i}\rightarrow M_{j}$ being (factorisably) trace class. Since $F$ commutes with colimits, the claim is clear. Now we show why  $F$ preserves trace-class maps, let \(f \colon P \to Q\) be such a map. Then it is in the image of 
$$\mathbf{Map}(R_{\mathbf{C}},P^{\vee}\otimes^{\mathbb{L}}Q)\rightarrow\mathbf{Map}(P,Q).$$
Now the claim follows from the commutativity of the diagram
\begin{displaymath}
    \xymatrix{
    \mathbf{Map}(R_{\mathbf{C}},P^{\vee}\otimes^{\mathbb{L}}Q)\ar[d]\ar[r]&  \mathbf{Map}(P,Q)\ar[dd]\\
    \mathbf{Map}(F(R_{\mathbf{C}}),F(P^{\vee}\otimes^{\mathbb{L}}Q))\ar[d] & \\
    \mathbf{Map}(R_{\mathbf{D}},\underline{\mathbf{Map}}(F(P),R_{\mathbf{D}})\otimes^{\mathbb{L}}F(Q))\ar[r] & \mathbf{Map}(F(P),F(Q)).
    }
\end{displaymath}
where the bottom map on the left is induced by map 
$F(\underline{\mathbf{Map}}(P,R_{\mathbf{C}}))\rightarrow\underline{\mathbf{Map}}(F(P),F(R_{\mathbf{C}}))\cong\underline{\mathbf{Map}}(F(P),R_{\mathbf{D}})$ defined as follows. We have $F(\underline{\mathbf{Map}}(P,R_{\mathbf{C}}))\rightarrow\underline{\mathbf{Map}}(F(P),F(R_{\mathbf{C}}))\cong\underline{\mathbf{Map}}(F(P),R_{\mathbf{D}})$ defined as follows. We have $F(P)\otimes^{\mathbb{L}}F(\underline{\mathbf{Map}}(P,R_{\mathbf{C}}))\cong F(P\otimes^{\mathbb{L}} \underline{\mathbf{Map}}(P,R_{\mathbf{C}}))\rightarrow F(R_{\mathbf{C}})\cong R_{\mathbf{D}}.$ By adjunction, this gives the required map. 
\end{proof}

\begin{corollary}[\cite{ben2024perspective}*{Proposition 2.1.53}]
    Let $A\rightarrow B$ be a map in \(\mathbf{Comm}(\mathbf{C})\). If $P\rightarrow Q$ is trace-class in $\mathbf{Mod}(A)$ then $B\otimes^{\mathbb{L}}_{A}P\rightarrow B\otimes_{A}^{\mathbb{L}}Q$ is trace-class in $\mathbf{Mod}(B)$. In particular if $M$ is (infinitely) nuclear as an $A$-module. Then $B\otimes_{A}^{\mathbb{L}}M$ is (infinitely) nuclear as a $B$-module.
\end{corollary}

\comment{
Write
$$A_{0}\cong\colim_{n}A_{0n}$$
with each $A_{0n}\rightarrow A_{0n+1}$ being factorisably trace class.
By the discussion of Sullivan modules in \cite{kelly2019koszul}*{Subsection 3.9} may write
$$A\cong\colim_{n}B_{n},$$
where $B_{n}=A_{0}$, and 
$$B_{n+1}\cong B_{n}\hat{\otimes}_{\alpha}\mathrm{Sym}(V_{n+1})$$
where $V$ is free of finite rank. Now the map $V_{1}[1]\rightarrow A_{0}$ factors through some $A_{0m_{1}}$. By rescaling we may as well assume that this is 

and the subscript $\alpha$ denotes a twisted differential. Now suppose we have shown that $B_{n}$ is infinitely nuclear as an algebra, so that $B_{n}\cong\colim_{i}B_{n}^{i}$, with each $B_{n}^{i}\rightarrow B_{n}^{j}$ being factorisably trace class. 

}
\comment{

    $$A\cong\colim_{n}A_{n}$$
    where $A_{0}\cong W_{n}(r)^{\dagger}$, and each $A_{n+1}$ is isomorphic to a pushout of one of the following forms:
    \begin{displaymath}
      \xymatrix{
       \Sigma^{k_{n}}\mathrm{Sym}(K)\ar[d]\ar[r] &A_{n}\ar[d]\\
      K\ar[r] & A_{n+1}
        }
    \end{displaymath}
    or 
     \begin{displaymath}
        \xymatrix{
       \mathrm\ar[d]\ar[r] &A_{n}\ar[d]\\
       \Sigma^{k_{n}}\mathrm{Sym}(K) \ar[r] & A_{n+1}
        }
    \end{displaymath}
    where $k_{n}>0$. By Dold-Kan we may pass to complexes. The second diagram is just 
    $$\mathrm{Sym}(S^{k_{n}}(K))\otimes^{\mathbb{L}}A_{i}.$$
    Since $\mathrm{Sym}(S^{k_{n}}(K))$ has only zero differentials and is free of finite rank in each degree, it is clear that 
    $$\mathrm{Sym}(S^{k_{n}}(K))\otimes^{\mathbb{L}}A_{i}\rightarrow\mathrm{Sym}(S^{k_{n}}(K))\otimes^{\mathbb{L}}A_{j}$$
    is still factorisably trace-class. 
    
    We will show inductively that each $A_{n}$ can be written as a colimit of derived dagger affinoids with nuclear transition maps. 
Now $A_{0}$ is an infinitely nuclear algebra. Suppose we have shown that $A_{n}$ is an infinitely nuclear algebra, and consider $A_{n+1}$. For the second case, we first show that $\Sigma^{k_{n}}W_{1}(r_{n})^{\dagger}$ is an infinitely nuclear algebra. We can interpret this algebra as $W(\partial\Delta^{k_{n}}(r_{n}))^{\dagger}$, where $\partial\Delta^{k_{n}}(r_{n})$ is the $|K|_{>0}$-sorted simplicial set, whose underlying simplicial set is $\partial\Delta^{k_{n}}$, and with all non-degenerate simplices given value $r_{n}$. 
By writing as a colimit over larger radii, this is an infinitely nuclear algebra.
Let us now handle the first case.

Write $A_{n}=\colim_{i}A_{n}^{i}$ as an infinitely nuclear algebra. Here we are killing a cell in homotopical degree $k_{n}-1$. We may assume this cell comes from some $A^{i}_{n}$. Possibly after rescaling the radius, we can assume that $W_{1}(\partial\Delta^{k_{n}}(r_{n}))^{\dagger}\rightarrow A_{n}$ factors through $A_{n}^{i}$. In particular, by shifting we can write 
$$A_{n+1}\cong\colim_{\mathcal{I}} (K\hat{\otimes}^{\mathbb{L}}_{W_{1}(\partial\Delta^{k_{n}}(r_{n}))^{\dagger}}A^{i}_{n}).$$

}

\comment{
\begin{proposition}\label{prop:nuclear-relative}
Let \(A \in \mathbf{Comm}(\mathbf{C})\) such that the underlying object in \(\mathbf{C}\) is strongly nuclear. We then have an equivalence
\[\mathbf{Nuc}(A) \simeq \mathbf{Mod}_A(\mathbf{Nuc}(\mathbf{C})).\]
If in addition $A$ is infinitely nuclear as an algebra, then we have an equivalence
\[\mathbf{Nuc}^{\infty}(A) \simeq \mathbf{Mod}_A(\mathbf{Nuc}^{\infty}(\mathbf{C})).\]

\end{proposition}

\begin{proof}
The nuclear case is a consequence of \cite{ben2024perspective}*{Lemma 2.1.55}. The infinitely nuclear case is \cite{aoki2025very}*{Theorem 7.45}. \qedhere

Let \(M \in \mathbf{Nuc}^{\infty}(A)\) and let \(A\) be an algebra whose underlying object in \(\mathbf{C}\) is strongly nuclear. We may write it as a colimit $\colim_{i\in\mathcal{I}}A\otimes^{\mathbb{L}}P_{i}$ where the $P_{i}$ are compact objects of $\mathbf{C}$, and the morphisms $A\otimes^{\mathbb{L}}P_{i}\rightarrow A\otimes^{\mathbb{L}}P_{j}$ are factorisably trace class. Now the claim that $M$ is in $\mathbf{Nuc}^{\infty}(\mathbf{C})$ is a consequence of Lemma \ref{lem:forgetn}.

Conversely, let $M$ be an $A$-module whose underlying object in \(\mathbf{C}\) is infinitely nuclear. Since $A$ is infinitely nuclear, we have that for all $n$, $A^{\otimes^{\mathbb{L}}n}\otimes^{\mathbb{L}}M$ is infinitely nuclear as an $A$-module. This gives a simplicial object whose colimit computes $A\otimes^{\mathbb{L}}_{A}M\cong M$. Since infinitely nuclear objects are closed under colimits, this implies that $M$ is infinitely nuclear as an $A$-module.
For a compact object \(P \in \bD(A)\), we then have by definition of nuclearity an isomorphism \begin{equation}\label{eq:nuc}
\mathbf{Hom}_A(A, P^{\vee_A} \hat{\otimes}_A^L M)) \simeq \mathbf{Hom}_A(P, M),
\end{equation}  where \(P^{\vee_A} = \underline{\mathbf{Hom}}_A(P, A)\). It suffices to take \(P = l^1(X,A) \defeq l^1(X) \hat{\otimes}^\mathbb{L} A\), since \(\mathbf{Hom}_A(-, A)\) commutes with finite direct sums, and any compact object in \(\bD(A)\) is a finite complex of finite direct sums of terms of the form \(l^1(X,A)\). Now since \(A\) is nuclear over \(R\), we have \[P^{\vee_A} = \underline{\mathbf{Hom}}_A(l^1(X,A),A) \cong \underline{\mathbf{Hom}}(l^1(X), A) \cong \underline{\mathbf{Hom}}(l^1(X), R) \hat{\otimes}_R^L A\] by Lemma \ref{lem:nuclear-relative}. Consequently, Equation \ref{eq:nuc} is equivalent to \[\mathbf{Hom}(R, l^1(X)^\vee \hat{\otimes}^\mathbb{L} M) \cong \mathbf{Hom}(l^1(X), M),\] showing that \(M\) is nuclear in \(\bD(R)\). Conversely, suppose \(M\) is nuclear in \(\bD(R)\) with an \(A\)-module structure, then as the tensor product of trace-class maps is trace-class, we have that \(A^{\hat{\otimes}^{\mathbb{L}n}} \hat{\otimes}^\mathbb{L} M \) is nuclear in \(\bD(A)\). Taking the colimit of the resulting simplicial diagram computes $A\hat{\otimes}^{\mathbb{L}}_{A}M\cong M$. Since the category of nuclear $A$-modules is closed under colimits, this shows that $M$ is nuclear as an $A$-module.  \qedhere
\end{proof}
}

\subsection{Nuclear objects and rigidification}

Recall that a morphism \(f \colon X \to Y\) in a stable \(\infty\)-category \(\bC\) is called \emph{compact} if for any morphism \(Y \to \varinjlim_i Z_i\) into a filtered colimit, the composite \(X \overset{f}\to Y \to Z\) factors through a finite stage \(X \to Z_i\). This is a relative version (or rather, a generalisation) of the notion of a compact object. Indeed, an object \(X \in \bC\) is compact if and only if the identity \(1_X \colon X \to X\) is a compact morphism.

\begin{theorem}[Dualisable \(\infty\)-categories]\cite{nkp}*{Theorem 2.9.2}\label{thm:dualisable}
Let \(\bC\) be a presentable, stable \(\infty\)-category. Then the following are equivalent:

\begin{enumerate}
\item \(\mathbf{C}\) is the retract in \(\mathsf{Pr}_{\mathsf{St}}^L\) of a compactly generated stable \(\infty\)-category;
\item The colimit functor \(\mathsf{Ind}(\mathbf{C}) \to \mathbf{C}\) has a left adjoint;
\item Filtered colimits in \(\bC\) distribute over products, that is, \[\prod_K \mathrm{colim}_I F \simeq \mathrm{colim}_{K^I} \prod_K F\] for a filtered \(\infty\)-category \(I\);
\item The category \(\bC\) is generated under colimits by compactly exhaustible objects, that is, objects \(X \in \bC\) that admit a representation \(X \cong \varinjlim_{n \in \N} X_n\) as a sequential colimit, where the transition maps \(X_n \to X_{n+1}\) are compact.   
\end{enumerate}
\end{theorem}

\begin{definition}\label{def:dualisable-cat}
A stable, presentable \(\infty\)-category satisfying the equivalent conditions of Theorem \ref{thm:dualisable} is called a \emph{dualisable} category. 
\end{definition}

\begin{example}
Any compactly generated, presentable stable \(\infty\)-category is dualisable.  
\end{example}

In what follows, we recall a recipe from \cite{nkp} to produce dualisable categories that are not necessarily compactly generated.

\comment{
\begin{definition}\label{def:ideal}
Let \(\bC\) be a presentable \(\infty\)-category, and \(S\) a collection of morphisms.

\begin{itemize}
\item we call \(S\) a \emph{two-sided ideal} in \(\bC\) if for composable morphisms \(f\), \(g\) and \(h \in \bC\), suppose \(g \in S\), then \(fgh \in S\).
\item we call \(f \in S\) \emph{factorisable} if it can be written as composition \(f = gh\), where \(g\) and \(h \in S\). Denote by \(S^\infty \subseteq S\) the class of factorisable morphisms.  
\item we call \(S\) \emph{accessible} if any morphism in \(S\) factorises through a \(\kappa\)-compact object for some cardinal \(\kappa\).
\item we call an ideal \(S\) \emph{precompact} if it is accessible, contains the identity on the initial object and for a diagram \(F_0 \leftarrow F_1 \rightarrow F_2\) of functors on \([0,1] \cap \Q \to \bC\), the pushout \(F_0 \coprod_{F_1} F_2\) maps \(0 \to 1\) to \(S\).  
\end{itemize}

\end{definition}

The compact morphisms in a dualisable category are a factorisable precompact ideal. Now suppose \(S\) is a precompact ideal in a presentable \(\infty\)-category \(\bC\), we have the following result due to Clausen:}

\begin{theorem}\cite{nkp}*{Theorem 2.7.4, Addendum 2.7.5}\label{thm:clausen-rigid}
Let \(\bC\) be a presentable, stable \(\infty\)-category and \(S\) a precompact ideal. Then there exists a terminal dualisable category \((\bC, S)^{\mathrm{core}}\) with a left adjoint functor \((\bC, S)^{\mathrm{core}} \to \bC\) that sends compact morphisms in \((\bC, S)^{\mathrm{core}}\) to \(S\).  Furthermore, if every morphism in \(S\) is  compact in \(\bC\), then the functor \((\bC, S)^{\mathrm{core}} \to \bC\) is fully faithful.  
\end{theorem}

\begin{proof}
We only spell out the construction of the category \((\bC, S)^{\mathrm{core}}\) and the universal functor \((\bC, S)^{\mathrm{core}} \to \bC\) as these will be used in the proof of Theorem \ref{thm:rigidification=nuc}. Let \(j \colon \bC \to \mathbf{Ind}(\bC)\) denote the Yoneda embedding. Define by \((\bC, S)^{\mathrm{core}}\) the full subcategory of \(\mathbf{Ind}(\bC)\) generated under colimits by objects of the form \[X \cong \varinjlim_{\alpha \in \Q} jX_\alpha\] such that the structure maps \(jX_{\alpha} \to j X_{\beta}\) belong to \(S\). The desired initial functor is the restriction \((\bC,S)^{\mathrm{core}} \to \bC\) of the colimit functor \(k \colon \mathbf{Ind}(\bC) \to \bC\). \qedhere 
\end{proof}
 
\begin{lemma}\cite{nkp}*{Lemma 4.4.5, Proposition 4.4.12}\label{lem:precompact-trace-class}
Let \(\bC\) be a stable, presentable \(\infty\)-category with a closed, symmetric monoidal structure and compact tensor unit. Then the class of trace-class morphisms in \(\bC\) is a precompact ideal. Furthermore, if \(\bC\) is dualisable, then every trace-class morphism is compact. 
\end{lemma}

\begin{definition}\label{def:rigid}
A dualisable category with a closed symmetric monoidal structure and compact unit is called \textit{rigid} if its compact and trace-class morphisms coincide.
\end{definition}

There is also a notion of local rigidity, and (local) rigidity relative to a base locally presentable symmetric monoidal category (\cite{nkp}*{Section 4.2}).
The following is (a special case of) \cite{nkp}*{Theorem 4.3.1}.

\begin{proposition}
    Let $\mathbf{C}$ be a locally presentable, stable, symmetric monoidal $\infty$-category. If $\mathbf{D}$ is a locally rigid $\mathbf{C}$-algebra, then it is a dualisable $\mathbf{C}$-module.
\end{proposition}

Applying Theorem \ref{thm:clausen-rigid} to the precompact ideal \(S\) of trace-class morphisms and combining with Lemma \ref{lem:precompact-trace-class}, we get the following:

\begin{theorem}\cite{nkp}*{Theorem 4.4.17}\label{thm:rigidification}
Let \(\bC\) be a stable, presentable \(\infty\)-category with a closed symmetric monoidal structure. Then there is a terminal rigid symmetric monoidal category \(\bC^{\mathrm{rig}}\), and a symmetric monoidal functor \(\bC^{\mathrm{rig}} \to \bC\), such that if \(\bD\) is a rigid symmetric monoidal \(\infty\)-category, the canonical map \[\mathbf{Fun}^{\otimes}(\bD, \bC^{\mathrm{rig}}) \to \mathbf{Fun}^{\otimes}(\bD, \bC)\] is an equivalence. Furthermore, if the monoidal unit of \(\bC\) is compact, then the functor \(\bC^{\mathrm{rig}} \to \bC\) is fully faithful. 
\end{theorem}

The category \(\bC^{\mathrm{rig}}\) is called the \emph{rigidification} of \(\bC\). The following result that we have learned from Dustin Clausen provides a verifiable criterion for when the subcategory of nuclear objects \(\mathbf{Nuc}(\bC)\) is rigid. It also justifies our distinction between nuclear and infinitely nuclear objects. 

\begin{corollary}\label{cor:Clausen}
Let \(\bC\) be a compactly generated, closed symmetric monoidal stable \(\infty\)-category with compact unit object. Then \(\mathbf{Nuc}(\bC)\) is rigid if and only if every trace-class morphism is factorisably trace-class.
\end{corollary}

\begin{proof}
Suppose \(\mathbf{Nuc}(\bC)\) is rigid, then the compact and trace-class morphisms coincide. By \cite{nkp}*{Proposition 2.4.4}, any compact morphism factorises through two compact morphisms, proving the claim in one direction. Conversely, if every trace-class morphism is factorisably trace-class, then the basic and \(\Q\)-basic nuclear objects in \(\bC\) coincide, so that by Theorem \ref{thm:rigidification}, it is already rigid. 
\end{proof}

\comment{
\begin{corollary}\label{thm:rigidification=nuc}
We have an equivalence of \(\infty\)-categories 
\[\mathbf{D}(R)^\mathrm{rig} \simeq \mathbf{Nuc}^\infty(R).\] In particular, \(\mathbf{Nuc}^\infty(R)\) is \(\omega_1\)-compactly generated and the \(\omega_1\)-compact objects are the \(\Q\)-basic nuclear objects in \(\bD(R)\). 
\end{corollary}

\begin{proof}
We first note that the unit object \(R\) in \(\bD(R)\) is compact, so that the rigidification \(\bD(R)^\rig\) is a full subcategory of \(\bD(R)\), which we identify with the essential image of the inclusion. Now let \(M \cong \varinjlim_{\alpha \in \Q} M_\alpha\) be a \(\Q\)-basic nuclear chain complex of complete bornological \(R\)-modules, so that the transition maps \(M_\alpha \to M_\beta\) are trace-class. Consequently, \(M\) may be identified with the inductive system \(\varinjlim_{\alpha \in \Q} j(M_\alpha) \in \mathbf{D}(R)^\mathrm{rig}\), where the transition maps lie in the precompact ideal of trace-class maps in \(\bD(R)\). Since \(\bD(R)^\mathrm{rig}\) and \(\mathbf{Nuc}^\infty(R)\) are generated under colimits inside \(\mathbf{C})\) by \(\varinjlim_{\alpha \in \Q} j(M_\alpha)\) and \(\varinjlim_{\alpha} M_\alpha\), the result follows.\qedhere 
\end{proof} 
}

\begin{corollary}\label{thm:rigidification=nuc}
Let $\mathbf{C}$ be a compactly generated, presentable, closed symmetric monoidal $\infty$-category with compact unit object $R$.
We have an equivalence of \(\infty\)-categories 
\[\mathbf{C}^\mathrm{rig} \simeq \mathbf{Nuc}^\infty(\mathbf{C}).\] In particular, \(\mathbf{C}^{\rig}\simeq \mathbf{Nuc}^\infty(\mathbf{C})\) is \(\omega_1\)-compactly generated and the \(\omega_1\)-compact objects are the \(\Q\)-basic nuclear objects in \(\mathbf{C}\). 
\end{corollary}

\begin{proof}
We first note that the unit object \(R\) in \(\mathbf{C}\) is compact, so that the rigidification \(\mathbf{C}^\rig\) is a full subcategory of \(\mathbf{C}\), which we identify with the essential image of the inclusion. Now let \(M \cong \varinjlim_{\alpha \in \Q} M_\alpha\) be a \(\Q\)-basic nuclear object, so that the transition maps \(M_\alpha \to M_\beta\) are trace-class. Consequently, \(M\) may be identified with the inductive system \(\varinjlim_{\alpha \in \Q} j(M_\alpha) \in \mathbf{C}^\mathrm{rig}\), where the transition maps lie in the precompact ideal of trace-class maps in \(\mathbf{C}\). Since \(\mathbf{C}^\mathrm{rig}\) and \(\mathbf{Nuc}^\infty(\mathbf{C})\) are generated under colimits inside \(\bD(R)\) by \(\varinjlim_{\alpha \in \Q} j(M_\alpha)\) and \(\varinjlim_{\alpha} M_\alpha\), the result follows.\qedhere 
\end{proof}

\comment{
\begin{lemma}\label{lem:nuclear-relative}
Let \(N \in \bD(R)\) be a nuclear object in \(\bD(R)\). Then \[\underline{\mathbf{Hom}}(M,N) \cong \underline{\mathbf{Hom}}(M,R) \haotimes^{\mathbb{L}} N\] for any compact object \(M\) in \(\bD(R)\). 
\end{lemma}

\begin{proof}
This can be deduced immediately from \cite{ben2020fr} Lemma 4.14, and Lemma 4.15. For completeness, we give a proof here.
It suffices to prove the result for the compact projective object \(M = l^1(X)\), for which \(\underline{\mathbf{Hom}}(l^1(X), -)\) commutes with filtered colimits, and a fortiori \(\omega_1\)-filtered colimits. It consequently suffices to prove the statement for objects \(N \in \bD(R)\) that are basic nuclear in \(\bD(R)\).  Writing \(N\) as a colimit \(\varinjlim_{n \in \N} P_n\) with trace-class transition maps \(f_n \colon  P_n \to P_{n+1}\), we have canonical maps \[\underline{\mathbf{Hom}}(l^1(X), P_n) \to \mathbf{Hom}(l^1(X),R) \hat{\otimes} P_{n+1}\] induced by the coevaluation-evaluation factorisation \(P_n \to P_n \hat{\otimes} P_n^\vee \hat{\otimes} P_{n+1} \to P_{n+1}\) of \(f_n\). Taking colimits yields the desired result.
\end{proof}
}

We now show that the assignments \(\mathsf{Spec}(A) \mapsto \mathbf{Nuc}(A)\) and \(\mathsf{Spec}(A) \mapsto \mathbf{Nuc}^{\infty}(A)\) are sheaves of \(\infty\)-categories for the topologies defined in Section \ref{sec:der-context}, under certain restrictions on $A$. For $\mathbf{Nuc}$, this was established in \cite{ben2024perspective}*{Section 7.2.4.2} for the finite homotopy monomorphism topology. Here we prove descent for the descendable topology on quasi-coherent sheaves.

\comment{
Let \(\mathbf{Q} \colon \bA^\op \to \mathsf{Pr}_{st}^{L}\) be a presheaf, and \(\mathbf{N} \subseteq \mathbf{Q}\) a sub-presheaf. We say that \(\mathbf{N}\) is \textit{local} if whenever \(\{f_i \colon U_i \to U\}\) is a cover in a Grothendieck pre-topology \(\tau\) and \(M \in \mathbf{Q}(U)\) is such that \(f_i^*(M) \in \mathbf{N}(U_i)\), then \(M \in \mathbf{N}(U)\). 

\begin{lemma}\cite{ben2024perspective}*{Lemma 7.1.4}\label{lem:nuc-descent-general}
Suppose that \(\mathbf{Q}\) satisfies Cech (hyper)descent and \(\mathbf{N}\) is local. Then \(\mathbf{N}\) satisfies Cech (hyper)descent. 
\end{lemma} 
 

\begin{theorem}\label{thm:nuc-descent}
Let \(\bA\) be a full subcategory of \(\mathbf{Comm}(\mathbf{C})^{op}\) closed under fibre products, such that for every \(\mathsf{Spec}(A) \in \bA\), the underlying chain complex of \(A\) is nuclear (resp. infinitely nuclear) in \(\bD(R)\). Then the functor \[\bA^\op \to \mathbf{CAlg}(\mathsf{Pr}_{st}^{L}), \quad A \mapsto \mathbf{Nuc}(A)\] is a \(\tau_{\vert \bA}^{hm}\)-sheaf of \(\infty\)-categories. 
\end{theorem}

\begin{proof}
For the nuclear case this \cite{ben2024perspective}*{Section 7.2.4.2}. Here we give a different proof that works for both $\mathbf{Nuc}$ and $\mathbf{Nuc}^{\infty}$. We prove the $\mathbf{Nuc}^{\infty}$ case, the $\mathbf{Nuc}$ case being similar

By Theorem \ref{cor:QCoh-descent}, the category \(\mathbf{QCoh}\) satisfies Cech descent, so by Lemma \ref{lem:nuc-descent-general}, it remains to show that the assignment \(A\mapsto \mathbf{Nuc}^{\infty}(A)\) is a local subpresheaf of \(\mathbf{QCoh}\) for the homotopy Zariski topology.  Let \(X = \mathsf{Spec}(A)\) and \(U = \mathsf{Spec}(A_i) \in \bA_{/X}\). Then we first observe that the localisation functor \(L_U \colon \bD(A) \to \bD(A_i)\) restricts to nuclear objects \(\mathbf{Nuc}(A) \to \mathbf{Nuc}(A_i)\). 

What we have shown so far is that localisation descends to a functor between nuclear modules. It remains to show that if \(M \in \mathbf{D}(A)\) such that \(L_{U_i}(M) \in \mathbf{Nuc}(A_i)\) for each \(i\), then \(M \in \mathbf{Nuc}(A)\). Let \(X \in \mathbf{D}(A)\) be a compact object and \(X_A = X \hat{\otimes}^\mathbb{L} A\). Then the natural map \(\underline{\mathbf{Hom}}_A(X_A, A) \hat{\otimes}_A^L M \to \underline{\mathbf{Hom}}_A(X_A, M)\) is equivalent to the map \[\underline{\mathbf{Hom}}(X,A) \hat{\otimes}_A^L M \to \underline{\mathbf{Hom}}(X,M).\] Now since by hypothesis \(M \hat{\otimes}_A^L A_i \in \mathbf{Nuc}(A_i)\) and \(A_i \in \mathbf{Nuc}(A)\), we have

\begin{multline*}
\underline{\mathbf{Hom}}(X,A) \hat{\otimes}_A^L A_i \hat{\otimes}_A M \simeq \underline{\mathbf{Hom}}(X,A_i) \hat{\otimes}_A M \simeq \underline{\mathbf{Hom}}(X,A_i) \hat{\otimes}_{A_i}^L A_i \hat{\otimes}_A M \\
\simeq \underline{\mathbf{Hom}}(X, A_i \hat{\otimes}_A^L M) \simeq \underline{\mathbf{Hom}}_{A_i}(A_i \hat{\otimes}^\mathbb{L} X, A_i \hat{\otimes}_A^L M)
\end{multline*}

Similarly, for any \(A_{(i_0,\cdots,i_n)} \defeq A_{i_0} \hat{\otimes}_A \dotsc \hat{\otimes}_A A_{i_n}\) with \((i_0, \cdots,i_n) \in I^n\), we have \[\mathbf{Hom}(X,A) \hat{\otimes}_A^L A_{(i_0,\cdots,i_n)} \hat{\otimes}_A^L M \simeq \mathbf{Hom}(A_i \hat{\otimes}^\mathbb{L} X, A_{(i_0,\cdots,i_n)} \hat{\otimes}_A^L M).\] Taking inverse limits and using that the limit is finite, we get 

\[\mathbf{Hom}(X^{\vee_A}, M) \simeq \mathbf{Hom}(X,M)\] which shows that \(M\) is nuclear as an \(R\)-module. The result now follows from Proposition \ref{prop:nuclear-relative}. \qedhere
\end{proof}

For our applications, we will start with an algebra \(A \in \mathbf{Alg}(\mathsf{CBorn}_R)\) that is either nuclear or \(\Q\)-nuclear as an \(R\)-module, build the derived \(\infty\)-category \(\mathbf{D}(A)\) of \(A\)-modules, and then consider the full \(\infty\)-subcategory \(\mathbf{Nuc}^\infty(A) \defeq \mathbf{Mod}_A(\bD(R)^\rig)\) of \(\mathbf{D}(A)\) of \emph{infinitely nuclear \(A\)-modules}.

}

\begin{lemma}
    Let $\mathbf{C}$ be a rigid category, and let $A\in\mathbf{Comm}(\mathbf{C})$. Then $\mathbf{Mod}_{A}(\mathbf{C})$ is rigid.
\end{lemma}

\begin{proof}
It suffices to prove that $\mathbf{Mod}_{A}(\mathbf{C})$ is rigid over $\mathbf{C}$.
    Let $\mathcal{S}$ be a set of atomic generators of $\mathbf{C}$. We claim that $\{A\otimes^{\mathbb{L}}S:S\in\mathcal{S}\}$ is a set of atomic generators of $\mathbf{Mod}_{A}(\mathbf{C})$. Indeed, it is certainly generating. Let $Y$ be an arbitrary element of $\mathbf{Mod}_{A}(\mathbf{C})$, and $V$ an arbitrary element of $\mathbf{C}$. Then 
    $$\mathbf{Map}_{A}(A\otimes^{\mathbb{L}}P,Y)\otimes^{\mathbb{L}}V\cong\mathbf{Map}(P,Y)\otimes^{\mathbb{L}}V\cong\mathbf{Map}(P,Y\otimes^{\mathbb{L}}V)\cong\mathbf{Map}_{A}(A\otimes^{\mathbb{L}}P,Y\otimes^{\mathbb{L}}V),$$
    where we have used that $P$ is atomic in $\mathbf{C}$. Note in particular that the unit of $\mathbf{Mod}_{A}(\mathbf{C})$ is $\mathbf{C}$-atomic. Moreover, we have by \cite{MR3381473}*{B. 1.5}, 
    $$\mathbf{Mod}_{A}(\mathbf{C})\otimes_{\mathbf{C}}\mathbf{Mod}_{A}(\mathbf{C})\cong\mathbf{Mod}_{A}(\mathbf{Mod}_{A}(\mathbf{C}))\cong\mathbf{Mod}_{A\otimes A}(\mathbf{C}).$$
    The functor $\mathbf{Mod}_{A\otimes A}(\mathbf{C})\rightarrow\mathbf{Mod}_{A}(\mathbf{C})$ is just forgetting along multiplication. The fact that this is an internal left adjoint follows from the projection formula for modules. \qedhere
    \end{proof}

Note that this \textit{does not imply} that for a general compactly generated presentable closed symmetric monoidal stable category $\mathbf{C}$, and $A\in\mathbf{Comm}(\mathbf{C}^{\rig})$, that $\mathbf{Mod}_{A}(\mathbf{C})^{\rig}\cong\mathbf{Mod}_{A}(\mathbf{C}^{\rig})$. Aoki has shown this is the case in the following circumstance:

\begin{definition}[c.f. \cite{aoki2025very}*{Section 7.5}]
    Say that $A\in\mathbf{Comm}(\mathbf{C})$ is \textit{infinitely nuclear as an algebra} if it can be written as $A\cong\colim_{i}A_{i}$, where each $A_{i}\in\mathbf{Comm}(\mathbf{C})$, and each $A_{i}\rightarrow A_{j}$ is factorisably trace class as a map in $\mathbf{C}$.
\end{definition}

\begin{example}
Let $A$ be a discrete dagger affinoid over a Banach ring $R$. Then $A$ is infinitely nuclear as an algebra, as observed in \cite{aoki2025very}*{Example 7.48} over $\mathbb{C}$ (though the same argument works over any Banach ring).
\end{example}

\begin{proposition}\label{prop:nuclear-relative}
Let \(A \in \mathbf{Comm}(\mathbf{C})\) such that the underlying object in \(\mathbf{C}\) is strongly nuclear. We then have an equivalence
\[\mathbf{Nuc}(\mathbf{Mod}(A)) \simeq \mathbf{Mod}_A(\mathbf{Nuc}(\mathbf{C})).\]
If in addition $A$ is infinitely nuclear as an algebra, then we have an equivalence
\[\mathbf{Nuc}^{\infty}(\mathbf{Mod}(A)) \simeq \mathbf{Mod}_A(\mathbf{Nuc}^{\infty}(\mathbf{C})).\]

\end{proposition}

\begin{proof}
The nuclear case is a consequence of \cite{ben2024perspective}*{Lemma 2.1.55}. The infinitely nuclear case is \cite{aoki2025very}*{Theorem 7.45}. 
\end{proof}


   \comment{ 
\begin{lemma}\label{lem:maxime1}
    Let \(\bC\) be a presentable closed symmetric monoidal stable \(\infty\)-category with a compact unit object. Then for \(A \in \mathbf{CAlg}(\bC)\) with \(A \in \bC^\rig\), the category \(\mathbf{Mod}_A(\bC^{rig})\) is rigid. Furthermore, if \(A\), \(B \in \bC^{\rig}\) and \(A \to B\) is descendable, then we have an equivalence \[\mathbf{Mod}_A(\bC^\rig) \to \varprojlim_{[n] \in \Delta} \mathbf{Mod}_{B^\otimes_A{^n}}(\bC^\rig)\] in both the category of presentable monoidal categories (with left adjoint functors as morphisms), and the full subcategory consisting of dualisable categories.
\end{lemma}

\begin{proof}
       This is clear as under the compact unit hypothesis, \(\bC^\rig \) embeds fully faithfully and monoidally inside \(\bC\). Thus descendability follows from Proposition \ref{prop:descfull}.\qedhere
       \comment{
       The first statement is clear. For the second statement, it suffices to take the case where \(A = 1_\bC\) and \(B = A\). Descendability now says that we have an equivalence \[\bC \to \varprojlim_{n \in \Delta} \mathbf{Mod}_{A^{\otimes n}}(\bC)\] of symmetric monoidal categories. Taking rigidifications, we get an equivalence \[\mathbf{Mod}_A^\rig(\bC) \to \varprojlim_{[n] \in \Delta}^\mathrm{dual} \mathbf{Mod}_{B^{\otimes_A n}}^\rig(\bC)\] in the category of dualisable categories. To see that this dualisable inverse limit agrees with the inverse limit taken in \(\mathbf{Pr}_{st}^L\), we invoke the criteria in \cite{ramzi2024dualizable}*{Corollary 4.5}. By passing to adjoints, we may view this as a colimit in \(\mathbf{Pr}_{st}^L\) of dualisable categories, which is dualisable. Furthermore, the base change maps \(\mathbf{Mod}_{B^{\otimes_A n}}^\rig(\bC) \to \mathbf{Mod}_{B^{\otimes_A n -1}}^\rig(\bC)\) are internal left adjoints as the tensor units \(B^{\otimes_A n}\) are compact. The last condition follows from \cite{mathew2016galois}*{Proposition 3.20}.}
\end{proof}
}

\begin{lemma}\label{lem:maxime1}
Let $\mathbf{C}$ be a compactly generated closed symmetric monoidal stable $\infty$-category with compact unit object $R$. Let $A\rightarrow B$ be a descendable map in $\mathbf{Comm}(\mathbf{C})$.
\begin{enumerate}
    \item 
    If both $A$ and $B$ are in $\mathbf{Nuc}^{\infty}(\mathbf{C})$, then the map
    \[\mathbf{Nuc}^{\infty}(A) \to \varprojlim_{[n] \in \Delta} \mathbf{Nuc}^{\infty}(B^{\otimes_A n})(\bC)\] 
    is an equivalence. Moreover, in this case, we in fact have an equivalence
     \[\mathbf{Nuc}^{\infty}(A) \simeq \varprojlim^{dual}_{[n] \in \Delta}\mathbf{Nuc}^{\infty}(B^{\otimes^{\mathbb{L}}_A n}).\] 
    \item
    If both $A$ and $B$ are in $\mathbf{Nuc}(\mathbf{C})$, then the map
    \[\mathbf{Nuc}(A) \to \varprojlim_{[n] \in \Delta}  \mathbf{Nuc}^{\infty}(B^{\otimes_A n})\] 
    is an equivalence.
\end{enumerate}
\end{lemma}

\begin{proof}
\begin{enumerate}
    \item 
   The category $\mathbf{Nuc}^{\infty}(\mathbf{C})$ is a full monoidal subcategory of $\mathbf{C}$ closed under colimits. Thus the claim follows from Proposition \ref{prop:descfull}. For the second claim, we use \cite{ramzi2024dualizable}*{Corollary 4.5}. The term on the left is rigid, and in particular dualisable. As maps of module categories arising from maps of algebras, everything is an internal left adjoint. Finally, because the map is descendable, the limit is absolute for exact functors. Thus the limit is in fact the limit in dualisable categories.
        \item
     Again we observe that $\mathbf{Nuc}(\mathbf{C})$ is a full monoidal subcategory of $\mathbf{C}$ closed under colimits. Thus the claim again follows from Proposition \ref{prop:descfull}. \qedhere
\end{enumerate}

\end{proof}

 \comment{To set things up, recall that when a presentably symmetric monoidal \(\infty\)-category \(\bC\) has a compact unit object, then its rigidification \(\bC^\rig\) is a full subcategory of \(\bC\), so that it makes sense to specify an object in \(\bC^\rig\). When \(\bC\) does not have a compact unit, the canonical functor \(k \colon \bC^\rig \to \bC\) is no longer fully faithful. In this case, we say that an algebra \(A\) has a \(\tilde{A}\) \emph{lift} to \(\bC^\rig\) if the canonical functor \(k(\tilde{A}) = A\). Note that if the unit is compact, then \(A\) has a lift to \(\bC^\rig\) if and only if \(A \in \bC^\rig\).

\begin{lemma}\label{lem:maxime1}
    Let \(\bC\) be a presentably symmetric monoidal stable \(\infty\)-category. 
    \begin{enumerate}
    \item Suppose for \(A \in \mathbf{Alg}(\bC)\), there exists a lift \(\tilde{A} \in \mathbf{Alg}(\bC^\rig)\), then \(\mathbf{Mod}_{\tilde{A}}(\bC^\rig)\) is rigid.
    \item Suppose \(A \to B\) in \(\mathbf{CAlg}(\bC)\) is such that there exists a descendable lift \(\tilde{A} \to \tilde{B}\) in \(\mathbf{CAlg}(\bC^\rig)\), then we have an equivalence \[\mathbf{Mod}_{\tilde{A}}(\bC^\rig) \to \varprojlim_{[n] \in \Delta} \mathbf{Mod}_{\tilde{B}^\otimes_{\tilde{A}} n}(\bC^\rig)\] of symmetric monoidal \(\infty\)-categories. 
    \end{enumerate}
\end{lemma}

\begin{proof}
    Clear.
\end{proof}}

\comment{The proof of Theorem \ref{thm:nuc-descentcent} does not directly adapt to show that the presheaf \(A \mapsto \mathbf{Nuc}^\infty(A)\) satisfies descent, as it is unclear in what generality \(\mathbf{Nuc}^\infty\) is local in the sense of Lemma \ref{lem:nuc-descent-general}. In this case, we have the following:

\begin{lemma}\label{lem:maxime}
    Let \(\bC\) be a closed symmetric monoidal presentable stable \(\infty\)-category. Suppose \(A\) is a descendable algebra in \(\mathbf{CAlg}(\bC)\) with a lift \(\tilde{A} \in \mathbf{CAlg}(\bC^\rig)\), we have an equivalence \[\bC^\rig \to \varprojlim_{[n] \in \Delta} \mathbf{Mod}_{\tilde{A}^\otimes n}(\bC^\rig)\] of symmetric monoidal \(\infty\)-categories.
\end{lemma}

\begin{proof}
    Descendability of the algebra says that we have an equivalence \[\bC \to \varprojlim_{n \in \Delta} \mathbf{Mod}_{A^{\otimes n}}(\bC)\] of \(\infty\)-categories. Taking rigidifications, we get an equivalence \[\mathbf{Mod}_A^\rig(\bC) \to \varprojlim_{[n] \in \Delta}^\mathrm{dual} \mathbf{Mod}_{B^{\otimes_A n}}^\rig(\bC)\] in the category of dualisable categories. To see that this dualisable inverse limit agrees with the inverse limit taken in \(\mathbf{Pr}_{st}^L\), we invoke the criteria in \cite{ramzi2024dualizable}*{Corollary 4.5}. By passing to adjoints, we may view this as a colimit in \(\mathbf{Pr}_{st}^L\) of dualisable categories, which is dualisable. Furthermore, the base change maps \(\mathbf{Mod}_{B^{\otimes_A n}}^\rig(\bC) \to \mathbf{Mod}_{B^{\otimes_A n -1}}^\rig(\bC)\) are internal left adjoints as the tensor units \(B^{\otimes_A n}\) are compact. The last condition follows from \cite{mathew2016galois}*{Proposition 3.20}.
\end{proof}}

\comment{

\begin{theorem}\label{thm:rigidification-descent}
Let \(\bA\) be a full subcategory of \(\mathbf{Comm}(\mathbf{C})^{op}\) closed under fibre products, such that for every \(\mathsf{Spec}(A) \in \bA\), the underlying chain complex of \(A\) is infinitely nuclear in \(\bD(R)\). Then for any descendable map \(A \to B\) in \(\bA\), we have an equivalence \[\mathbf{Nuc}^\infty(A) \to \varprojlim_{[n] \in \Delta} \mathbf{Nuc}^\infty (B^{\haotimes_A^{\mathbb{L}} n})\] of symmetric monoidal \(\infty\)-categories.
\end{theorem}

\begin{proof}
The unit object \(R\) in \(\bD(R)\) is compact so that the rigidification \(\bD(R)^\rig\) is a full subcategory. Since \(A\) and \(B\) are \(\Q\)-nuclear, they have lifts to \(\bD(R)^\rig\), which we still denote by \(A\) and \(B\), using the fully faithful embedding \(\bD(R)^\rig \to \bD(R)\). Since \(B\) is descendable over \(A\), the result follows from Lemma \ref{lem:maxime1}. \comment{Let \(\{\mathsf{Spec}(A_i) \to \mathsf{Spec}(A)\}_{i=0}^n\) be a covering of \(X = \mathsf{Spec}(A)\), and \(Y = \coprod_{i = 0}^n \mathsf{Spec}(A_i) \to X\) the canonical map. By Lemma \ref{lem:hepi-descendable}, the map \(Y \to X\) is descendable. Now since \(A\) is in particular nuclear over \(R\), we have \(\mathbf{Nuc}(A) \simeq \mathbf{Mod}_A(\mathbf{Nuc}(R))\) by Proposition \ref{prop:nuclear-relative}. Applying the rigidification functor, we get  \(\mathbf{Nuc}^\infty(A) \simeq \mathbf{Mod}_A(\mathbf{Nuc}^\infty(R))\), where we have used that \(A \in \mathbf{Nuc}^\infty(R)\). Using that \(B = \prod_{i =0}^n A_i \in \mathbf{Nuc}^\infty(R)\), we get that \(B \in \mathbf{Nuc}^\infty(A)\). The conclusion in the statement of the theorem now follows from Lemma \ref{lem:maxime}. }\qedhere
\end{proof}
}

This descent result allows us to globalise. For simplicity let $\underline{\mathbf{C}}$ be a derived algebraic context, and let $\mathbf{A}\subset\mathbf{Aff}^{cn}(\underline{\mathbf{C}})=(\mathbf{DAlg}(\mathbf{C}_{\ge0}))^{op}$ be a subcategory closed under pushouts. Let $\tau$ be a pre-topology on $\mathbf{A}$ which is contained in the descendable topology. Suppose that for any $\mathrm{Spec}(A)\in\mathbf{A}$, we have that the underlying object of $A$ is in $\mathbf{Nuc}^{\infty}(\mathbf{C})$. For $\mathcal{X}$ a stack, we may write
\[\mathbf{Nuc}^\infty(\mathcal{X}) = \varprojlim_{\mathsf{Spec}(A) \to \mathcal{X}\in\mathbf{A}_{\big\slash\mathcal{X}}} \mathbf{Nuc}^\infty(A),\] 
and with this definition and  Lemma \ref{lem:maxime1} we get descent along effective epimorphisms. This comes in handy particularly when $\mathcal{X}$ is geometric, i.e., it has an atlas.

Note that $\mathbf{Nuc}^{\infty}(X)$ defined in this way \textit{is not necessarily rigid}. However we still have 
$$\mathbf{Nuc}^{\infty}(X)\rightarrow\mathbf{QCoh}(X)$$
is fully faithful. Indeed it is just the full subcategory of $\mathbf{QCoh}(X)$ consisting of those $\mathcal{F}$, such that for any $g:\mathrm{Spec}(A)\rightarrow X$, $g^{*}\mathcal{F}\in\mathbf{Nuc}^{\infty}(A)$. Later we will prove that for a qcqs scheme, $\mathbf{Nuc}^{\infty}(X)$ is rigid, and when $X$ a so-called infinitely nuclear qcqs scheme, then there is a natural equivalence $\mathbf{Nuc}^{\infty}(X)\cong\mathbf{QCoh}(X)^{\rig}.$ 

\comment{
\begin{lemma}\label{lem:Nucdualisable}
The category $\mathbf{Nuc}^{\infty}(\mathcal{X})$ is dualisable. \textcolor{red}{something is fishy here, as this not true in the algebraic case. Where does the proof fo wrong}
\end{lemma}

\begin{proof}
Write
\[\mathbf{Nuc}^\infty(\mathcal{X}) = \varprojlim_{\mathsf{Spec}(A) \to \mathcal{X}\in\mathbf{A}_{\big\slash\mathcal{X}}} \mathbf{Nuc}^\infty(A),\] 
    By passing to adjoints, we may view this as a colimit in \(\mathbf{Pr}_{st}^L\) of dualisable categories.  To see that the adjoint diagram is still a diagram in \(\mathbf{Pr}_{st}^L\), note that the functors are now the right adjoints in an adjunction
$$\adj{h^{*}}{\mathbf{Nuc}^{\infty}(Z)}{\mathbf{Nuc}^{\infty}(W)}{h_{*}},$$
where $h:W\rightarrow Z$ is map of affines. Thus $h_{*}$ commutes with colimits.
\end{proof}
}

We will be particularly interested in the specialisation of this result to analytic contexts. Recall that we had defined the finite \(G\)-\(T\)-topology \(\tau^{rat}\) whose covers were given by finitely conservative families \(\{\mathsf{Spec}(A_i) \to \mathsf{Spec}(A)\}\) such that \(A \to A_i\) were \(T\)-rational localisations for some homotopy polynomial Lawvere theory \(T\). Since \(T\)-rational localisations are homotopy epimorphisms, by Lemma \ref{lem:hepi-descendable}, the total morphism \(A \to \prod_{i=0}^n B_i\) is descendable for a finite subfamily. Restricting the topology to a full subcategory \(\bA\) closed under pushouts, we see that \(\mathbf{QCoh} \colon \bA^\op \to \mathbf{CAlg}(\mathbf{Pr}_{st}^L)\) is a \(\tau^{rat}\)-sheaf of \(\infty\)-categories. Likewise, since the morphisms in an \'etale cover are descendable, we have that \[\mathbf{QCoh} \colon \bA^\op \to \mathbf{CAlg}(\mathbf{Pr}_{st}^L)\] is a \(\tau_{\bA}^{et}\)-sheaf. 

Recall from Section \ref{section:derived-borno-geom} that we defined a category of derived dagger \(R\)-analytic spaces \(\mathbf{dAn}_R\), which when \(R = \C\) and \(R = K\) a non-Archimedean Banach field contains the categories of complex and rigid Stein spaces as full subcategories. The local model of this category is the category \(\mathbf{Afnd}_R^\dagger\) of derived dagger affinoid \(R\)-algebras, which are nuclear as well as infinitely nuclear over \(R\). Denoting \(\bA_R^\dagger = (\mathbf{Afnd}_R^\dagger)^\op\) As a consequence, we have the following: 

\begin{corollary}\label{cor:nuc-analytic-et-descent}
Let \(R\) be a Banach ring and \(\bA = \mathbf{Afnd}_R^\dagger\) the full subcategory of \(\mathsf{Aff}_{\bD(R)}^{cn}\) of derived dagger affinoid algebras. Then the assignment \[\bA^\op \to \mathbf{CAlg}(\mathsf{Pr}_{st}^{L}), \quad A \mapsto \mathbf{Nuc}^\infty(A)\] is a \(\tau^{rat}\) and \(\tau^{et}\)-sheaf of \(\infty\)-categories.
\end{corollary}



\comment{Finally, it remains to consider the rigidity of the category \(\mathbf{Nuc}^\infty(A)\) when \(A\) is an admissible affinoid. Note that since the underlying \(R\)-module of \(A\) is Banach, such an algebra can never be nuclear except in uninteresting cases. This is needed to make the proofs of the previous two propositions work. To remedy this, we consider the functor \[(-)^{nuc} \colon \mathsf{CBorn}_R \to \mathsf{Nuc}(R)\] introduced in \cite{Meyer-Mukherjee:HL} that equips a complete bornological \(R\)-module with the \emph{compactoid bornology}. This is the universal way to turn a complete bornological \(R\)-module into a nuclear module, that is, we have a natural isomorphism \[\mathsf{Hom}_{\mathsf{Nuc}(R)}(M,N^{nuc}) \cong \mathsf{Hom}_R(M,N),\] exhibiting \(\mathsf{Nuc}(R)\) as a coreflective subcategory of \(\mathsf{CBorn}_R\). This functor is exact, so that it extends to a functor of derived categories.  }

\begin{remark}
Note that at this point, we cannot consider \(\mathbf{Nuc}^\infty\)-descent for rigid analytic spaces. The problem here is that when viewed as a complete bornological \(K\)-module, affinoid \(K\)-algebra is not a nuclear object in \(\bD(K)\). This needs to be forced using the compactoid bornology, which will be taken up in the next section.
\end{remark}  

\comment{Instead, it is most convenient to view \(\bA_R^{form}\) as a subcategory of \(\bD(\mathsf{Ban}_R^{\le 1})^\circ \simeq \varprojlim_n \bD(R/\dvgen^n)\)\edit{put reference}, whose \(\omega_1\)-compact objects are sequential colimits of morphisms \(X \to Y\) such that \(X/\dvgen^n \to Y/\dvgen^n\) are compact in \(\bD(R/\dvgen^n)\) for each \(n\). We shall call these morphisms \emph{compactoid} following \cite{Meyer-Mukherjee:HL}. Let \(\widetilde{\mathbf{Nuc}}(R) \defeq \varprojlim^{\mathrm{dual}}\bD(R/\dvgen^n) = (\bD(\mathsf{Ban}_R^{\le 1})^\circ)^{\rig}\).

We now apply Lemma \ref{lem:maxime1} to our setting. Recall that the category \(\bD(\mathsf{Ban}_R^{\leq 1})\) is only atomic generated as a \(\bD(\mathsf{Ban}_R^{\leq 1})^\circ\)-module and the Banach modules \(R_\delta\) for \(\delta >0\) are the \(\bD(\mathsf{Ban}_R^{\leq 1})^\circ\)-atomic generators. In particular, the tensor unit \(R\) is only an atomic object (and not compact). 

\begin{lemma}\label{lem:affinoid-nuc}
    Any derived admissible adic \(R\)-algebra \(A \in \bD(\mathsf{Ban}_R^{\leq 1})^\circ\) has a lift to \(\widetilde{\mathbf{Nuc}(R)}\). 
\end{lemma}

\begin{proof}
    By definition, an admissible affinoid \(R\)-algebra is a quotient of the Tate algebra of the form \(A = R\gen{x_1, \dotsc, x_n}//(f_1, \dotsc, f_k)\) for nonzero elements \(f_i\). Since we are working in the category \(\bD(\mathsf{Ban}_R^{\leq 1})^\circ\) of derived \(p\)-complete \(R\)-modules, \(R\gen{x} = \coprod_{\N} R\), which may be written as the filtered colimit \[colim_{J \subset \N, \abs{J}<\infty} R^{\abs{J}},\] where the structure maps are clearly compactoid. This has a lift, namely \(\widetilde{R\gen{x}} = colim_{J \subset \N, \abs{J}<\infty} jR^{\abs{J}}\) in \(\widetilde{\mathbf{Nuc}(R)}\). Consequently, the Tate algebra \(\widetilde{R\gen{x}} \in \widetilde{\mathbf{Nuc}}(R)\). By \cite{nkp}*{Lemma 2.7.3}, sequential colimits of compactoid maps are closed under finite colimits so that the quotient \(\widetilde{R\gen{x}//(f)}\) belongs to \(\widetilde{\mathbf{Nuc}(R)}\). Since \(\widetilde{\mathbf{Nuc}(R)}\) is symmetric monoidal, \[\widetilde{R\gen{x_1}//(f_1)} \hat{\otimes} \cdots \hat{\otimes} \widetilde{R\gen{x_n}//(f_n)} = \widetilde{R\gen{x_1,\dotsc,x_n}}//(f_1, \dotsc, f_n)\] is an object of \(\widetilde{\mathbf{Nuc}}(R)\) that lifts \(A = R\gen{x_1,\dotsc, x_n}//(f_1, \dotsc,f_n)\). Consequently, any derived affinoid \(R\)-algebra has a lift to \(\widetilde{\mathbf{Nuc}(R)}\).
\end{proof}

\begin{theorem}\label{thm:rigidification-formal-desc}
    The functor \[(\bA_R^{form})^{\op} \to \mathbf{CAlg}(\mathbf{Pr}_{st}^L), \quad A \mapsto \widetilde{\mathbf{Nuc}}(A) \defeq \mathbf{Mod}_A(\widetilde{\mathbf{Nuc}(R))}\] is a \(\tau_{ad}^{rat}\) and \(\tau_{ad}^{et}\)-sheaf of \(\infty\)-categories.
\end{theorem}

\begin{proof}
    By Lemma \ref{lem:maxime1} and \ref{lem:affinoid-nuc}, it suffices to check that the lift corresponding to a descendable map \(A \to B\) remains descendable in \(\widetilde{\mathbf{Nuc}(R)}\).  we need to check that when \(B\) is a derived rational localisation or \'etale over \(A\), the corresponding lifts to \(\widetilde{\mathbf{Nuc}(R)}\) are en there exists a descendable lift \(\tilde{A} \to \tilde{B}\) in \(\widetilde{\mathbf{Nuc}(R)}\). \comment{that is descendable over a lift \(\tilde{A}\), it lies in the rigidification \(\widetilde{\mathbf{Nuc}}(A)\). Since rational localisations are in particular \'etale, and \'etale maps are in particular finitely presented, the result follows if we prove that if \(B\) is finitely presented over \(A\), then \(B \in \widetilde{\mathbf{Nuc}}(A)\).} The same argument as in the proof of Theorem \ref{thm:rigidification-descent} reveals that it actually suffices to prove that if \(\tilde{A} \in \widetilde{\mathbf{Nuc}}(R)\), then \(\tilde{B} \in \widetilde{\mathbf{Nuc}}(R)\).  It remains to show that if \(B\) is a derived admissible adic \(A\) algebra that is adically finitely presented and descendable, then it lifts to a descendable map \(\tilde{A} \to \tilde{B}\) in \(\widetilde{\mathbf{Nuc}(R)}\). Since \(B\) is in particular finitely presented over \(A\), it is a cokernel of the form \[A \hat{\otimes}^\mathbb{L} R \gen{x_1, \dotsc, x_n} \overset{(f_1, \dotsc,f_k)}\to A \hat{\otimes}^\mathbb{L} R \gen{x_1, \dotsc, x_n}\] for elements \(f_1, \dotsc, f_k \in \pi_0(A \hat{\otimes}^\mathbb{L} R \gen{x_1, \dotsc, x_n})\). Let \[\tilde{A} = j(A \hat{\otimes} R\gen{x_1, \dotsc,x_n}//(f_1, \dotsc,f_n)) \in \widetilde{\mathbf{Nuc}(R)}.\] Then \(\tilde{B} = \mathrm{coker}(\tilde{A} \overset{(f_1, \dotsc,f_k)}\to \tilde{A})\) is a lift of \(B\) in \(\widetilde{\mathbf{Nuc}(R)}\).
\end{proof}}

\subsection{Internal models for rigidification}

The identification in Corollary \ref{thm:rigidification=nuc} is largely formal and inaccessible from the viewpoint of functional analysis. In this section, we show that in the complex setting, the rigidification \(\bD(\C)^\rig\) contains (complexes of) strongly dual nuclear Frechet spaces as a full subcategory, while in the non-Archimedean setting, we identify the rigidification with the stabilisation of nuclear bornological \(R\)-modules. 

\subsubsection{The complex case}

Throughout this subsection, let \(R = \C\). Recall that a trace-class map \(f\colon M \to N\) between complex Banach spaces is said to be \emph{of order \(0< p \leq 1\)} if it can be written as \begin{equation}\label{eqref:standard-rep}
f(x) = \sum_{n = 0}^\infty \lambda_n f_n(x) g_n
\end{equation} for bounded sequences \((f_n) \in M'\) and \((g_n) \in N\), and \(\lambda_n \in l^p\).  We say that \(f\) has order \(0\) if it is of \emph{order \(p\)} for every \(p \in (0,1]\).  By \cite{grothendieck1955produits}*{Chapitre 2, §1.2 Corollaire pp9}, any such map has a standard representation as in Equation \eqref{eqref:standard-rep} with \((\lambda_n)\) a \emph{rapidly decreasing sequence}, that is, for each \(k\), the sequence \((n^k \lambda_n)\) is bounded.

\begin{theorem}\label{thm:Grothendieck}
A bounded linear map \(f \colon M \to N\) between \(\C\)-Banach spaces is factorisably trace-class if and only if it is of order zero. 
\end{theorem}
\begin{proof}
Suppose \(f \colon M \to N\) is factorisably trace-class, for each \(n\), we may write \(f\) as a composition of \(n\) trace-class maps. Then by \cite{grothendieck1955produits}*{Chapitre 2, §1.3 Corollaire pp 10}, it has order at most \(\frac{2}{n-1}\). Since \(n\) was arbitrary, \(f\) has order \(0\). Conversely, if \(f\) has order zero, we have a representation \[f(x) = \sum_{n=0}^\infty \lambda_n f_n(x) g_n\] for bounded sequences \((f_n) \in M'\) and \((g_n) \in N\), and a rapidly decreasing sequence \((\lambda_n)\) of complex numbers. Then \(f\) is a composition of \[g \colon M \to c_0, \quad x \mapsto \lambda_n^{1/2}f_n(x)\] and \[h \colon c_0 \to N, \quad \xi \mapsto \sum_{n=0}^\infty \lambda_n^{1/2} \xi_n g_n.\] The map \(g\) is trace-class as it is a diagonal sequence and the sequence \((\lambda_n^{1/2})\) is absolutely summable. The map \(h\) is trace-class as we may write \(\xi_n = e_n(\xi)\), where \(e_n\) is the sequence in \(l^1\) with \(1\) in position \(n\) and zeroes elsewhere. These maps are of order \(0\) again by \cite{grothendieck1955produits}*{Chapitre 2, §1.2 Corollaire pp9} as the sequence \((\lambda_n^{1/2})\) is still rapidly decreasing.
\end{proof}

\begin{corollary}\label{cor:factorisably-injective}
Let \(f \colon X \to Y\) be an injective trace-class map of order zero. Then \(f\) factorises through two injective trace-class maps of order zero.
\end{corollary}

\begin{proof}
Suppose \(f \colon X \to Y\) is injective and of order zero, then in the factorisation from Theorem \ref{thm:Grothendieck}, the map \(g\) is injective, so that \(h\) is injective. 
\end{proof}

\begin{corollary}
    Let $X\in\mathrm{CBorn}_{\mathbb{C}}$. Then $X$ is in \(\mathsf{Nuc}_{b,\omega_0}^\Q(\C)\) if and only if $X$ can be written as 
    $$X\cong\colim_{n\in\mathbb{N}}X^{n}$$
    where each $X^{n}\rightarrow X^{n+1}$ is injective and factorisably trace class. 
\end{corollary}


Theorem \ref{thm:Grothendieck} says that \(\mathsf{Nuc}^\infty(\C)\) consists of inductive systems \(M \colon I \to \mathsf{Ban}_\C\) where the structure maps \(M_i \to M_j\) are trace-class of order zero. Likewise, \(\Q\)-nuclear bornological \(\C\)-vector spaces \(\mathsf{Nuc}_b^\infty(\C)\) are inductive systems of Banach spaces with injective factorisably trace-class maps. Noting that in the proof of Theorem \ref{thm:Grothendieck} we may replace \(c_0\) with any \(l^p\)-space for \(0<p<\infty\), we may arrange any \(\Q\)-nuclear bornological vector space as a sequential colimit of Hilbert spaces with injective order \(0\) trace-class structure maps. These are called \emph{dual strongly nuclear Fr\'echet spaces}.

\subsubsection*{Forcing \(\Q\)-nuclearity}

There is a canonical way to pass from complete bornological \(\C\)-vector spaces to  \(\Q\)-nuclear modules. Consider a bornological \(\C\)-vector space \((X, \bdd_X)\). We define the \(\Q\)-nuclear bornology as follows: a subset \(S \in \bdd_X^\rig\) is bounded if and only if there is an increasing sequence of complete bounded discs \((B_n)\) such that \(S \subset B_0 \subset B_1 \dotsc \), and the induced linear maps \(X_{B_n} \to X_{B_{n+1}}\) are trace-class of order zero for each \(n\). By construction, the bornological vector space \((X, \bdd_X^\rig)\) is \(\Q\)-nuclear.

\begin{theorem}
The inclusion \(\mathsf{Nuc}_b^\infty(\C) \to \mathsf{CBorn}_\C\) is left adjoint to the functor that assigns to a complete bornological \(\C\)-vector space \((X, \bdd_X)\) the nuclear bornological vector space \((X, \bdd_X^\rig)\). 
\end{theorem}

\begin{proof}
We first verify that the given assignment is indeed functorial. To see this, consider a bounded linear map \(f \colon (X, \bdd_X) \to (Y, \bdd_Y)\) of complete bornological vector spaces. Then for any \(S \in \bdd_X^\rig\), there is a sequence of increasing, complete bounded discs \((B_n)\) containing \(S\) such that the inclusion \(X_{B_n} \to X_{n+1}\) is trace-class of order zero. Equivalently, by Theorem \ref{thm:Grothendieck},  for each \(n\), there is a rapidly decreasing sequence \((\lambda_k^{(n)})\) in \(V_{B_{n+1}}\) such that \(V_{B_n}\) is contained in the disked hull of \((\lambda_k^{(n)})\). To see that \(f\) restricts to a bounded map \((X, \bdd_X^\rig) \to (Y, \bdd_Y^\rig)\), it remains to see that the sequence of bounded discs \((f(B_n))\) witnesses that \(f(S) \in \bdd_X\). Now by hypothesis, we may write each element in \(V_{B_n}\) as \(x = \sum_{n=0}^\infty \lambda_n x_n\) for some \(l^1\)-sequence \(\lambda_n\) and a rapidly decreasing sequence \(x_n \in V_{B_{n+1}}\). Then by the boundedness of \(f\), it follows that \((f(x_n)) \in W_{B_{n+1}}\) is rapidly decreasing, from which the boundedness of \(f\) follows. 

Next we wish to show that any morphism \(f \colon X \to Y\) of complete bornological \(\C\)-vector spaces, where \(X\) is \(\Q\)-nuclear, factors as \(X \to Y^\rig \to Y\). By \(\Q\)-nuclearity, \(X\) is a filtered colimit of Banach spaces \(X_n\) with order-zero trace-class transition maps. By Theorem \ref{thm:Grothendieck}, the map \(X_{n} \to X_{n+2}\) is trace-class of order zero, and therefore so is the composition with \(X_{n+2} \to Y_{n+2}\). Consequently, by a standard argument, the factorisation \(Y_n \to Y_{n+2}\) is trace-class. Finally, since the map \(Y_{n} \to Y_{n+2}\) further factorises through the order zero map \(X_{n+1} \to X_{n+2} \to Y_{n+2}\), it is itself of order zero. \qedhere 
\end{proof}

\comment{

We may also define 
$$(-)^{rig,IB}:\mathsf{Ind(Ban_{\mathbb{C}})}\rightarrow\mathsf{Nuc}^{\mathbb{Q}}(\mathbb{C})$$
defined as follows. For a Banach space $M$, $(M)^{rig}$ is a $\mathbb{Q}$-nuclear bornological space, which we regard as an object of $\mathsf{Nuc}^{\mathbb{Q}}(\mathbb{C})\subset\mathsf{Ind(Ban_{\mathbb{C}})}$. Now $\mathsf{Nuc}^{\mathbb{Q}}(\mathbb{C})$ is closed under filtered colimits in $\mathsf{Ind(Ban_{\mathbb{C}})}$. Thus the extension by filtered colimits of the functor
$$(-)^{rig}:\mathsf{Ban}_{\mathbb{C}}\rightarrow\mathsf{Ind(Ban_{\mathbb{C}})}$$ to a functor
$$(-)^{rig,IB}:\mathsf{Ind(Ban_{\mathbb{C}})}\rightarrow\mathsf{Ind(Ban_{\mathbb{C}})}$$
in fact lands in $\mathsf{Nuc}^{\mathbb{Q}}(\mathbb{C})$. Let $X=\colim_{\mathcal{I}}X_{i}$ be a $\mathbb{Q}$-nuclear object of $\mathsf{Ind(Ban_{\mathbb{C}})}$, where $\mathcal{I}$ is $\aleph_{1}$-filtered, and $X_{i}$ is basic $\mathbb{Q}$-nuclear. Further let $Y$ be 

Consequently, there exists a sequence of rapid decay \((\lambda_n)\) such that \(x = \sum_{n=0}^\infty \lambda_n x_n'(x) y_n\) for bounded sequences \((x_n)' \in X_{B_n}'\) and \((y_n) \in X_{B_{n+1}}\). We do this for inductive systems, but the same procedure applies to bornologies. Given  \(F \colon I \to \mathsf{Ban}_\C\) in \(\mathsf{Ind}(\mathsf{Ban}_\C)\), let \(I^F\) be the directed subset of \(I\) of those \(i \in I\) for which there is an increasing sequence \((j_n) \in I\) with \(i < j_1 < j_2 < \cdots\) such that each \(F(j_n) \to F(j_{n+1})\) is trace-class. Note that if \(i \in I^F\), then the elements \((j_n)\) in the chosen sequence for which \(F(i_n) \to F(i_{n+1})\) are trace-class belong to \(I^F\), so that by Lemma \ref{lem:nuclear-characracterisation}, the restriction \(F \colon I^F \to \mathsf{Ban}_\C\) lies in \(\mathsf{Nuc}(\C)\). This construction is functorial: let \(f \colon M \to N\) be a morphism with indexing sets \(I\) and \(J\). We may then represent \(f\) by a family of bounded \(\C\)-linear maps \(f_i \colon M_i \to N_{k(i)}\) for a function \(k \colon I \to J\), such that for any \(j \geq i\), there is a \(k \geq  \max\{k(i), k(j)\}\) in \(J\) such that the following diagram

\[
\begin{tikzcd}
M_i \arrow{r}{f_i} \arrow{d}{} & N_{k(i)} \arrow{dr} & \\
M_j \arrow{r}{f_j} & N_{k(j)} \arrow{r}{} & N_k,
\end{tikzcd}
\] commutes. Now suppose \(i \in I^M\); then by definition there is a sequence \(i < j_1 < j_2 < \cdots\)  such that \(M_{j_n} \to M_{j_{n+1}}\) are trace-class. From the proof of Lemma \ref{lem:nuclear-coreflective}, it follows that the maps \(\mathsf{Im}(f_n) \to \mathsf{Im}(f_{n+2})\) are trace-class, so that we have a commuting diagram 

\[
\begin{tikzcd}
M_{i_n} \arrow{r}{f_{i_n}} \arrow{d}{} & \mathsf{Im}(f_{i_{n}}) \subseteq N_{k(i_n)} \arrow{dr} \arrow{d} & \\
M_{i_{n+2}} \arrow{r}{f_{i_{n+2}}} & \mathsf{Im}(f_{i_{n+2}}) \arrow{r}{} & N_{k(i_{n+2})},
\end{tikzcd}
\] with trace-class vertical maps for each \(n\). This yields a morphism of inductive systems from \(F \colon I^F \to \mathsf{Ban}_\C\) to \(\mathsf{Im}(f) \colon J^G \to \mathsf{Ban}_\C\), whose composition with the canonical map from \(\mathsf{Im}(f)\) to \(G \colon J^G \to \mathsf{Ban}_\C\) yields the definition of the assignment on morphisms.  In fact, this is the universal way to turn an ind-Banach module into a nuclear module:

\begin{lemma}\label{lem:nuclear-coreflective}
The category \(\mathsf{Nuc}(\C)\) is a co-reflective subcategory of \(\mathsf{Ind(Ban)}_\C\), that is, the inclusion is left adjoint to the functor \(\mathsf{Ind(Ban)}_\C \to \mathsf{Nuc}(\C)\) that takes an ind-Banach module \(F \colon I \to \mathsf{Ban}_\C\) to the nuclear module \(F \colon I^{F} \to \mathsf{Ban}_\C\). Furthermore, the category of nuclear modules is closed under countable inverse limits, arbitrary  filtered colimits and tensor products. 
\end{lemma}

\begin{proof}
Let \(f\colon M \to N\) be a morphism in \(\mathsf{Ind(Ban)}_\C\), where \(M \colon I \to \mathsf{Ban}_\C\) is a nuclear ind-Banach \(\C\)-module, and \(N \colon J \to \mathsf{Ban}_\C\) is an arbitrary ind-Banach module. Then for each \(i\), there is a \(j>i\) such that  \(M_i \to M_j\) is trace-class. We then have a commuting diagram 
\[
\begin{tikzcd}
M_i \arrow{r}{f_i} \arrow{d}{} & N_{k(i)} \arrow{dr} & \\
M_j \arrow{r}{f_j} & N_{k(j)} \arrow{r}{} & N_k,
\end{tikzcd}
\] representing the morphism \(f\), where the maps \(M_i \to M_j\) are trace-class. By \cite{hogbe2011nuclear}*{Section 3.1, Remark 1}, we may without loss of generality take these structure maps to be polynuclear - that is, they factorise through two nuclear maps. Since each \(f(M_i) \subseteq N_{k(i)}\) is a quotient of \(M_i\), and since the composition of a bounded linear map with a polynuclear map is polynuclear, the map \(M_i \to M_j \to f(M_j) \subseteq N_{k(j)}\) is polynuclear. Consequently, by \cite{hogbe2011nuclear}*{Section 2.3, Proposition 1b}, its factorisation through \(f(N_i)\) yields a trace-class map \(f(N_i) \to f(N_j)\), showing that \(f(N) \cong ``colim" f(N_i)\) is nuclear.  The second claim is \cite{BaBK}*{Proposition 3.51, Proposition 3.52 and Lemma 4.18}.
\end{proof}

As in the nuclear case, we may exhibit the category \(\mathsf{Nuc}^\Q(\C) \subset \mathsf{Ind(Ban)}_\C\) of  \(\Q\)-nuclear spaces as a co-reflective subcategory. Given an inductive system \(M \colon I \to \mathsf{Ban}_\C\), consider the directed subset \(I^{\mathrm{rig}}\) of \(i \in I\) such that there is an increasing sequence \(i=j_0 < j_1 < \cdots \) such that \(M_{j_n} \to M_{j_{n+1}}\) is trace-class of order \(0\). Then \(M^\rig \defeq M \colon I^\rig \to \mathsf{Ban}_\C\) is a \(\Q\)-nuclear space by construction, and we get a functor \[(-)^\rig \colon \mathsf{Ind}(\mathsf{Ban}_\C) \to \mathsf{Nuc}^\Q(\C), \quad M \mapsto M^\rig.\]

\begin{proposition}\label{prop:infinitely-nuc}
The functor \((-)^\mathrm{rig}\) is right adjoint to the inclusion \(\mathsf{Nuc}^\Q(\C) \subset \mathsf{Ind}(\mathsf{Ban}_\C)\). 
\end{proposition}

\begin{proof}

\end{proof}}

Recall that the definition of basic nuclear and infinitely basic nuclear chain complexes entailed \emph{sequential} colimits with trace-class and factorisably trace-class transition maps. This puts a size restriction on the categories of complete bornological spaces and its full subcategories of nuclear and \(\Q\)-nuclear objects. Indeed, the restriction of \(\mathbf{Nuc}^\infty(\C)^{\omega_1}\) to the subcategory of complexes concentrated in degree \(0\) should consist of inductive systems \(M \in \mathsf{Ind}(\mathsf{Ban}_\C)\) (or complete bornological spaces) that can be written as an \(\N\)-indexed colimit \(M \cong ``colim_{n \in \N}" M_n\)  of order-zero (injective) trace-class maps \(M_n \to M_{n+1}\) between Banach spaces, which is smaller than the category \(\mathsf{Nuc}^\infty(\C)\). We call such nuclear ind-Banach (resp. bornological) spaces \emph{sequentially \(\Q\)-nuclear}, and denote the full subcategory of such spaces by \(\mathsf{Nuc}_{\omega_0}^\Q(\C)\) (resp. \(\mathsf{Nuc}_{b,\omega_0}^\Q(\C)\)).

\comment{

\begin{lemma}\label{lem:nucbornind}
    Every object $D$ of \(\mathsf{Nuc}_{\omega_0}^\Q(\C)\) fits into an exact sequence
    $$0\rightarrow X\rightarrow Y\rightarrow 0$$
    with $X,Y\in\mathsf{Nuc}_{b,\omega_0}^\Q(\C)$.
\end{lemma}
\begin{proof}
Let $D=\colim_{n} D_{n}$ with $D_{n}\rightarrow D_{n+1}$ being factorisably trace class, and each $D_{n}$ being a Hilbert space. For each $D_{n}$, let $\tilde{D}_{n}=\mathrm{Ker}(D_{n}\rightarrow V)$. Note that this is just $\colim_{m>n_{0}}\mathrm{Ker}(D_{n}\rightarrow D_{m})$ for some sufficiently large $n_{0}$. Now, each $\mathrm{Ker}(D_{n}\rightarrow D_{m})$ is a Banach space. Moreover  the maps $\mathrm{Ker}(D_{n}\rightarrow D_{m+1})$ are evidently injective. Thus $\tilde{D}_{n}$ is in fact a complete bornological space, and moreover a bornologically closed subspace of $D_{n}$. In particular it is a Banach space, and since $D{n}$ is a Hilbert space, $\tilde{D}_{n}$ is further a Hilbert space. Write $D'_{n}$ for $D_{n}\big\slash\tilde{D}_{n}$. This is also a Hilbert space, and is the orthogonal complement of $\tilde{D}_{n}$ in $D_{n}$. Now $D'_{n}\rightarrow D'_{n+1}$ is a monomorphism, and factors through $D_{n}\rightarrow D_{n+1}$

    This  is immediately implied by the proof of \cite{clausenscholze3}*{Lemma 8.13}, where one works with maps which are $p$-Schatten for all $p$ at once, not just for some given $p$.
\end{proof}
}

\comment{
\begin{definition}\label{def:metrizable}
A bornological vector space \(V\) is said to be \emph{bornologically metrisable} if its collection of bounded discs is \(\omega_1\)-filtered. An inductive system \(X \in \mathsf{Ind}(\mathsf{Ban}_\C)\) is \emph{ind-metrisable} if there exists an \(\omega_1\)-filtered category \(I\) such that \(X \cong ``colim_{i \in I}" X_i\).  \edit{comment on the permeance properties of metrisable bornological spaces?}
\end{definition}

We may now describe the category that has the right size to compare with the rigidification of \(\bD(\C)\):

\begin{definition}\label{def:strongly-nuc}
    We call a bornological vector space \emph{sequentially (\(\Q\)-)nuclear} if it is a sequential colimit of Banach spaces with injective, (order zero) trace-class structure maps. Denote the full subcategory of sequentially (\(\Q\)-) nuclear bornological vector spaces by \(\mathsf{Nuc}_{b, \omega_0}(\C)\) (resp. \(\mathsf{Nuc}_{b, \omega_0}^{\Q}(\C)\)). Similarly, one can define sequentially nuclear and \(\Q\)-nuclear inductive systems, denoted by \(\mathsf{Nuc}_{\omega_0}(\C)\) and \(\mathsf{Nuc}_{\omega_0}^\Q(\C)\), respectively. 
\end{definition}

The full subcategories \(\mathsf{Nuc}_{b, \omega_0}(\C)\)  and \(\mathsf{Nuc}_{b, \omega_0}^{\Q}(\C)\)  may be identified with the full subcategories of \emph{metrisable} nuclear (resp. \(\Q\)-nuclear) bornological \(\C\)-vector spaces. Similarly, the full subcategories \(\mathsf{Nuc}_{\omega_0}(\C)\) and \(\mathsf{Nuc}_{\omega_0}^\Q(\C)\) of \(\mathsf{Ind}(\mathsf{Ban_\C})\) are the metrisable nuclear inductive systems. }

\begin{lemma}
    Any object $X$ of $\mathsf{Nuc}^{\infty}(\mathbb{C})$ may be written as an $\omega_{1}$-filtered colimit of objects of $\mathsf{Nuc}^{\mathbb{Q}}_{\omega_{0}}(\mathbb{C})$. Conversely, any object of this form is $\mathbb{Q}$-nuclear. Finally, objects of $\mathsf{Nuc}^{\mathbb{Q}}_{\omega_{0}}(\mathbb{C})$ are $\omega_{1}$-compact. This in particular implies that 
    $$\mathsf{Nuc}^{\mathbb{\infty}}(\mathbb{C})\cong\mathsf{Ind}_{\omega_{1}}(\mathsf{Nuc}^{\mathbb{Q}}_{\omega_{0}}(\mathbb{C})).$$
\end{lemma}

\begin{proof}
    Much of this is close to being tautological, but we summarise it for completeness. That $\mathsf{Nuc}^{\mathbb{Q}}_{\omega_{0}}(\mathbb{C})$ consists of $\omega_{1}$-compact objects of $\mathsf{Ind(Ban_{\mathbb{C}})}$ is obvious. Moreover, $\mathsf{Nuc}^{\infty}(\mathbb{C})$ is closed under filtered colimits in $\mathsf{Ind(Ban_{\mathbb{C}})}$, so in particular $\omega_{1}$-filtered colimits of objects of $\mathsf{Nuc}^{\mathbb{Q}}_{\omega_{0}}(\mathbb{C})$ are still in $\mathsf{Nuc}^{\infty}(\mathbb{C})$. Conversely, suppose that $X\cong\colim_{\mathcal{I}}M_{i}$ is a filtered colimit of Banach spaces with each $M_{i}\rightarrow M_{j}$ being factorisably trace-class. Let $\{i_{n}\}\subset\mathcal{I}$ be countable. Without loss of generality we may assume 
    $i_{1}\le i_{2}\le\ldots$. We therefore get a factorisation 
    $$M_{i_{j}}\rightarrow\colim_{n=1}^{\infty}M_{i_{n}}\rightarrow X.$$
    The object $\colim_{n=1}^{\infty}M_{i_{n}}$ is in $\mathsf{Nuc}^{\mathbb{Q}}_{\omega_{0}}(\mathbb{C})$. Then $X$ is a filtered colimit of these objects, and an easy diagonal argument shows that this is an $\omega_{1}$-filtered colimit. 
\end{proof}

We may now describe explicitly the infinitely basic nuclear complexes of bornological modules, in the $\infty$-categorical sense, in terms of chain complexes whose entries lie in the categories just defined.

We now study permeance properties of the category \(\mathsf{Nuc}_{\omega_0}^{\Q}(\C)\) of sequentially \(\Q\)-nuclear ind-Banach and bornological spaces. 

\begin{proposition}\label{prop:properties-infinitely-nuclear}
\begin{enumerate}
    \item 
    The full subcategory \(\mathsf{Nuc}_{b,\omega_0}^{\Q}(\C)\) of \(\mathsf{CBorn}_\C\) is closed under 

\begin{enumerate}
\item Closed subspaces;
\item Countable products;
\item Tensor products;
\item Countable colimits.
\end{enumerate} 
\item  The full subcategory \(\mathsf{Nuc}_{\omega_0}^{\Q}(\C)\) of \(\mathsf{Ind}(\mathsf{Ban}_\C)\)  is closed under 
 \begin{enumerate}
\item Closed subspaces;
\item Countable products;
\item Tensor products;
\item Countable colimits.
\end{enumerate} 
\end{enumerate}

\end{proposition}

\begin{proof}
Let \(F\) be a \(\Q\)-nuclear bornological \(\C\)-module and \(E\) a closed submodule. We may write \(F \cong \varinjlim_{\alpha \in \Q} F_\alpha\) with injective trace-class maps between Banach spaces, so that we have an isomorphism \[E \cong \varinjlim_{\alpha \in \Q} E \cap F_{\alpha},\] with injective transition maps \(E \cap F_{\alpha} \to E \cap F_{\beta} \subseteq F_\beta\) with \(\alpha < \beta\). It remains to prove that these transition maps are trace-class. By hypothesis, we may find \(\alpha < \gamma < \beta\) such that we have a factorisation \(F_\alpha \to F_\gamma \to F_\beta\) into trace-class maps, so that the transition maps \(F_\alpha \to F_\beta\) are polynuclear. Consequently, by \cite{hogbe2011nuclear}*{Proposition 1(a), Section 2:3}, the map \(E \cap F_\alpha \to E \cap F_\beta\) is trace-class, as required. This proves (1).

To see (2), let \((E_n)_{n \in \N}\) be a collection of \(\Q\)-nuclear bornological \(\C\)-modules. By hypothesis, we may write each \(E_n\) as a colimit \(E_n \cong colim_{\alpha \in \Q} E_{n,\alpha}\) of Banach \(\C\)-modules with injective trace-class structure maps \(E_{n,\alpha} \to E_{n,\beta}\). The products \(E_\alpha = \prod_{n \in \N} E_{n,\alpha}\) of these Banach modules over varying \(\alpha\) generate the bornology on \(E = \prod_{n \in \N} E_n\) so that we have a colimit \[E \cong \varinjlim_{\alpha \in \Q} l^\infty(\N, E_\alpha)\] of Banach spaces of bounded sequences in \(E_\alpha\) for each \(\alpha\). That the transition maps are trace-class follows from the proof of \cite{hogbe2011nuclear}*{Theorem 4.1.1}. Statement (3) follows from choosing presentations as \(\Q\)-indexed inductive systems, and observing that the tensor product of two trace-class maps is trace-class.   

To prove (4), it suffices to prove that \(\mathsf{Nuc}^\Q(\C)\) has countable sums and cokernels. Note that since \(\mathsf{Nuc}^\Q(\C)\) is closed under kernels, we in fact only need to show that it is closed under quotients by admissible subobjects. 
The existence of countable colimits and finite coproducts follow from the representation of an infinitely nuclear space as a \(\Q \)-indexed colimit and (2), respectively. For quotients by closed subspaces, let \(E \subseteq F\) be a closed subspace of an infinitely nuclear bornological \(\C\)-modules. Writing \(F \cong \varinjlim_{\alpha \in \Q} F_\alpha\) with trace-class transition maps \(F_\alpha \to F_\beta\) with \(\alpha < \beta\), we have that \[\frac{F}{E} \cong colim_{\alpha \in \Q} \frac{F_\alpha}{E \cap F_\alpha},\] and the structure maps \(\frac{F_\alpha}{E \cap F_\alpha} \to \frac{F_\beta}{E \cap F_\beta}\) are trace-class by \cite{hogbe2011nuclear}*{Theorem 4.1.1}. \qedhere
\end{proof}

We remark that the proof of Proposition \ref{prop:properties-infinitely-nuclear} in particular demonstrates that the limits and colimits are computed in the underlying category \(\mathsf{CBorn}_\C\).

\comment{
\begin{lemma}
Let $X_{\bullet}\in\mathrm{Ch}(\mathsf{Ind}(\mathsf{Ban}_\C))$ be a complex such that each $X_{n}$ is in \(\mathsf{Nuc}_{\omega_0}^{\Q}(\C)\). Then it is quasi-isomorphic to a complex $\tilde{X}_{\bullet}$ where each $\tilde{X}_{n}$ is in  \(\mathsf{Nuc}_{b,\omega_0}^{\Q}(\C)\).
\end{lemma}

\begin{proof}
    
\end{proof}
}

\comment{\begin{lemma}
    Let $M_{\bullet}\rightarrow N_{\bullet}$ be a factorisably trace class map between bounded complexes of projectives. Then $\mathrm{LH}_{n}(M_{\bullet})\rightarrow\mathrm{LH}_{n}(N_{\bullet})$ is given by a map of complexes
    \begin{displaymath}
        \xymatrix{
        \mathrm{coim}(d_{n+1}^{M_{\bullet}})\ar[d]\ar[r] & \mathrm{coim}(d_{n+1}^{N_{\bullet}})\\
        \mathrm{Ker}(d_{n}^{M_{\bullet}})\ar[r] &\mathrm{Ker}(d_{n}^{N_{\bullet}}).
        }
    \end{displaymath}
\end{lemma}

\begin{proof}
    
\end{proof}

\begin{proof}
    Note that each $\mathrm{LH}_{n}(M_{\bullet})$ is in fact basic nuclear, so is in particular in $\mathrm{Ind(Ban}_{\mathbb{C}})$. It in fact suffices to prove the following:
\end{proof}

\begin{proposition}[c.f. \cite{clausenscholze3}*{Lemma 8.7}]
    Every basic $\mathbb{Q}$-nuclear object in $\mathbf{D}(\mathbb{C})$ is represented by a complex which is a sequential colimit of complexes
    $$\colim_{n}C_{\bullet}^{(n)}$$
    such that
    \begin{enumerate}
        \item 
        each term $C_{k}^{(n)}$ of each complex $C_{\bullet}^{(n)}$ is a compact projective;
        each map $C_{\bullet}^{(n)}\rightarrow C_{\bullet}^{(n+1)}$ is term-wise given by a \textcolor{red}{(factorisably?)} trace-class map.
    \end{enumerate}
\end{proposition}

\begin{proof}
Let $f:M_{\bullet}\rightarrow N_{\bullet}$ be factorisably trace-class, with $M_{\bullet}$ and $N_{\bullet}$ complexes such that $M_{\bullet}$ is a finite complex of projectives, and $N_{\bullet}$ is $dg$-flat. Then there is a diagram
    $$F:[0,1]\rightarrow\mathbf{D}(C)$$
    such that $F(0)=M_{\bullet}$, $F(1)=N_{\bullet}$, and for $0<q<1$, the map

    Recall that compact projectives are internally projective (i.e. $\mathbb{R}\underline{Hom}(c,-)\cong\underline{\mathrm{Hom}}(c,-)$ and flat. Let $M_{\bullet}\rightarrow N_{\bullet}$ be factorisably trace class. 
    
\end{proof}

\begin{proposition}
    Let $X\in\mathbf{D}(\mathbb{C})$. The following are equivalent.
    \begin{enumerate}
        \item 
        $X$ is basic $\mathbb{Q}$-nuclear.
        \item $X$ can be represented by a complex of strongly dual nuclear Fr\'{e}chet spaces.
        \item Each homology group $\mathrm{LH}_{n}$ is isomorphic to a quotient of strongly dual nuclear Fr\'{e}chet spaces.
        \item 
        $X$ lies in the smallest countable cocomplete stable subcategory generated by the DNF spaces.
    \end{enumerate}
\end{proposition}

\begin{proposition}\label{prop:basic-nuclear}
Let \(X \in \bD(\C)\). Then the following are equivalent:
\begin{enumerate}
\item \(X\) is basic \(\Q\)-nuclear;
\item \(X\) can be represented by a chain complex whose entries lie in \(\mathsf{Nuc}_{b, \omega_0}^{\Q}(\C)\);
\item \(X \in \bD(\mathsf{Nuc}_{b, \omega_0}^{\Q}(\C))\).
\end{enumerate}

\end{proposition}

\begin{proof}
We first note that any basic nuclear object in \(\bD(\C)\) written as a sequential colimit of Banach spaces with order zero trace-class transition maps may be replaced by a cofinal sequence of strongly dual nuclear Fr\'echet spaces. The proof now works in exactly the same way as \cite{clausenscholze3}*{Theorem 8.15}, replacing dual nuclear Fr\'echet spaces by strongly dual nuclear Fr\'echet spaces. 
\end{proof}

}
\begin{corollary}\label{cor:infinitely-nuc-quasi-ab}
\begin{enumerate}
    \item 
    The category \(\mathsf{Nuc}_{b}^{\infty}(\C)\) is an exact full subcategory of \(\mathsf{CBorn}_\C\), with the exact structure restricted from \(\mathsf{CBorn}_\C\), closed under colimits and finite limits. In particular, monomorphic filtered colimits are exact.
    \item 
    The category \(\mathsf{Nuc}^{\infty}(\C)\)  is an exact full subcategory of \(\mathsf{Ind(Ban_\C)}\), with the exact structure restricted from \(\mathsf{Ind(Ban_\C)}\), closed under colimits and finite limits. In particular, filtered colimits are exact.
\end{enumerate}
\end{corollary}

 In particular, the category \(\mathsf{Nuc}^{\infty}(\C)\) is an $\omega_{1}$-purely presentable exact category (in the sense of \cite{kelly2024flat}*{Definition 3.11}) with exact filtered colimits. Moreover, with our smallness assumptions on the generating Banach sets, the category \(\mathsf{Nuc}^{\infty}(\C)\) has a generator. By \cite{kelly2024flat}*{ Corollary 3.18} we get the following.

   \begin{corollary}
       The category \(\mathsf{Ch}(\mathsf{Nuc}^{\infty}(\C))\) may be equipped with the injective model structure. In particular it presents an $\infty$-category.
   \end{corollary}

   Thus we get the derived \(\infty\)-category \[\mathbf{D}(\mathsf{Nuc}^\infty(\C)) = N(\mathsf{Ch}(\mathsf{Nuc}^\infty(\C)))[W^{-1}],\] where \(W\) denotes the class of weak equivalences for the injective model structure on \(\mathsf{Nuc}^\infty(\C)\) (i.e. the quasi-isomorphisms). Note that this in fact coincides with the \textit{flat model structure}, as all nuclear objects are flat. Note furthermore that $\mathsf{Nuc}^{\infty}(\C)\rightarrow\mathsf{Ind(Ban_{\mathbb{C}})}$ is exact. Moreover, since $\mathsf{Nuc}^{\infty}(\C)$ is presentable and the inclusion commutes with colimits, this functor has a right adjoint $(-)^{rig,IB}$. Then, equipping $\mathsf{Ind(Ban_{\mathbb{C}})}$ also with the injective model structure, we get a Quillen adjunction
   $$\adj{i}{\mathsf{Ch(\mathsf{Nuc}^{\infty}(\C))}}{\mathsf{Ch(\mathsf{Ind(Ban_{\mathbb{C}})})}}{(-)^{rig,IB}},$$
   and hence an adjunction of $\infty$-categories
     $$\adj{\mathbf{i}}{\mathbf{D}(\mathsf{Nuc}^{\infty}(\C))}{\mathbf{D}(\mathsf{Ind(Ban_{\mathbb{C}})})}{(-)^{\mathbf{rig,IB}}}.$$

Now it is not clear that \(\mathsf{Nuc}_{b}^{\infty}(\C)\) is accessible. Moreover, filtered colimits in this category are not necessarily exact. Nonetheless, by localising at quasi-isomorphisms, we still get an $\infty$-category,
\[\mathbf{D}(\mathsf{Nuc}_{b}^\infty(\C)) = N(\mathsf{Ch}(\mathsf{Nuc}_{b}^\infty(\C)))[W^{-1}].\] 

We have a natural inclusion $\mathsf{diss}:\mathsf{Ch}(\mathsf{Nuc}_{b}^\infty(\C))\rightarrow\mathsf{Ch}(\mathsf{Nuc}^\infty(\C))$, which is just the restriction of the functor 
$$\mathrm{diss}:\mathsf{Ch(CBorn_{\mathbb{C}})}\rightarrow\mathsf{Ch(Ind(Ban_{\mathbb{C}}))}.$$



\begin{proposition}
    Any object $X$ of $\mathsf{Nuc}_{\omega_{0}}^\infty(\C)$ admits an admissible epimorphism $Y\rightarrow X$, where $Y\in\mathsf{Nuc}_{b,\omega_{0}}^\infty(\C)$.
\end{proposition}

\begin{proof}
    Write 
    $X\cong\colim_{n}X^{n}$ with each $f^{n}:X^{n}\rightarrow X^{n+1}$ being factorisably trace class, and each $X^{n}$ being  an $l^{1}$. We will inductively define a sequence $\tilde{X}^{n}$, with injective factorisably trace class maps $\tilde{f}^{n}:\tilde{X}^{n}\rightarrow\tilde{X}^{n+1}$, and strict epimorphisms $h^{n}:\tilde{X}^{n}\rightarrow X^{n}$, compatible with the maps $f^{n}$ and $\tilde{f}^{n}$. Put $\tilde{X}^{0}=X^{0}$ with $h^{0}$ being the identity. Suppose that for some $m$ we have constructed $\tilde{X}^{n}$ and $h^{n}$ for all $n\le m$. Consider the factorisably trace class map $f^{m}\circ h^{m}:\tilde{X}^{m}\rightarrow X^{m}$. As in the proof of \cite{clausenscholze3}*{Proposition 8.9}, we pick an arbitrary factorisably trace class map $g^{m}:\tilde{X}^{m}\rightarrow Y^{m+1}$. Again as in the proof of \cite{clausenscholze3}*{Proposition 8.9}, this could be given by term-wise multiplication by a sequence which converges to zero sufficiently rapidly. Now define $\tilde{X}^{m+1}= X^{m+1}\oplus Y^{m+1}$. Define $\tilde{f}^{m}:\tilde{X}^{m+1}= X^{m+1}\oplus Y^{m+1}$ to be $f^{m}\circ h^{m}$ in the first component, and $g^{m}$ in the second. Define $h^{m+1}:\tilde{X}^{m+1}\rightarrow X^{m+1}$ to be the projection onto the first factor. Finally, define $\tilde{X}=\colim_{n}\tilde{X}^{n}$, and define $h:\tilde{X}\rightarrow X$ to be the map induced by the maps $h^{n}$. The map $h$ is then a strict epimorphism.
\end{proof}

Consequently, any object $X_{\bullet}$ of $\mathsf{Ch}(\mathsf{Nuc}^\infty(\C))$ may be resolved by an object of $\mathsf{Ch}(\mathsf{Nuc}_{b}^\infty(\C))$. Indeed, $\mathsf{Nuc}_{b}^\infty(\C)$ is closed under sums in $\mathsf{Nuc}^\infty(\C)$, and any object $\colim_{\mathcal{I}}X_{i}$, with $X_{i}$ basic $\mathbb{Q}$-nuclear, admits an admissible epimorphism
$$\bigoplus_{i\in\mathcal{I}}X_{i}\rightarrow\colim_{\mathcal{I}}X_{i}.$$

The existence of the resolution now follows from \cite{kelly2016homotopy}*{Corollary 2.67}. General nonsense concerning Dwyer-Kan localisations of relative categories now implies the following.

\begin{lemma}
    The functor $\mathsf{diss}:\mathsf{Ch}(\mathsf{Nuc}_{b}^\infty(\C))\rightarrow\mathsf{Ch}(\mathsf{Nuc}^\infty(\C))$
    induces an equivalence of $\infty$-categories $\mathbf{diss}:\mathbf{D}(\mathsf{Nuc}_{b}^\infty(\C))\rightarrow\mathbf{D}(\mathsf{Nuc}^\infty(\C)).$
\end{lemma}
 
We in fact get the following commutative diagram of $\infty$-categories

\begin{displaymath}
    \xymatrix{
    \mathbf{D}(\mathsf{Nuc}_{b}^\infty(\C))\ar[d]\ar[r] & \mathbf{D}(\mathsf{CBorn}_{\mathbb{C}})\ar[d]\\
    \mathbf{D}(\mathsf{Nuc}^\infty(\C))\ar[r] & \mathbf{D}(\mathsf{Ind(Ban_{\mathbb{C}})}),
    }
\end{displaymath}

\noindent where the vertical maps are equivalences. Next we prove that the bottom (and hence the top) map is fully faithful. 

\comment{\begin{corollary}\label{cor:dagger-stein-nuclear}
The underlying bornological \(\C\)-module of an open polydisc algebra \(\mathcal{O}_{< r}(D^n)\) and dagger Stein algebra is sequentially \(\Q\)-nuclear.
\end{corollary}

\begin{proof}
By Example \ref{ex:disc-algebras}, \(\mathcal{O}_{< r}(D^1)\) for a radius \(r>0\) is \(\Q\)-nuclear, for a polyradius \(r = (r_1,\dotsc, r_n)\), we have \[\mathcal{O}_{< r}(D^n) \cong \mathcal{O}_{< r_1}(D^1) \haotimes \cdots \haotimes \mathcal{O}_{< r_n}(D^1),\] which is \(\Q\)-nuclear by Proposition \ref{prop:properties-infinitely-nuclear}. Similarly, dagger algebras are  \(\Q\)-nuclear by Example \ref{ex:disc-algebras}. Applying  Proposition \ref{prop:properties-infinitely-nuclear} to a countable inverse limit representation of a dagger Stein algebra yields the desired result. 
\end{proof}}

\begin{proposition}\label{thm:nuc-exact}
The functor \((-)^\rig \colon \mathsf{CBorn}_\C \to \mathsf{CBorn}_\C\) is exact. 
\end{proposition}

\begin{proof}
Since \((-)^\rig\) is a right adjoint functor, it commutes with kernels, so it only remains to check whether it preserves cokernels. Let \(E \onto Q\) be a cokernel and \(S\) a bounded subset in \(Q^\rig\). Then there is an increasing sequence of bounded discs \((B_n)\) with \(S = B_0\), such that the induced inclusions \(Q_{B_n} \to Q_{B_{n+1}}\) are trace-class of order zero. Each such trace-class map is given by a standard representation as in \eqref{eqref:standard-rep}. Since \(E \onto Q\) is a bornological quotient map, we may lift the sequence \((\lambda_n)\) of rapid decay, the bounded functionals \((x_n) \in Q_{B_n}'\) and the bounded sequence \((y_n) \in \Q_{B_{n+1}}\) to sequences of the same type to get lifts \(E_{A_n} \to E_{A_{n+1}}\), where the \(A_n\)'s are bounded discs that lift \(B_n\). Summarily, the sequence \((B_n)\) of bounded discs lifts to a sequence of bounded discs \(A_n\) such that the induced maps \(E_{A_n} \to E_{A_{n+1}}\) are trace-class of order zero, from which we deduce that the induced map \(E^\rig \to Q^\rig\) is a bornological quotient map as required. 
\end{proof}

\begin{corollary}
    The functor 
    $$\mathbf{i}:\mathbf{D}(\mathsf{Nuc}^{\infty}(\C))\rightarrow\mathbf{D}(\mathsf{Ind(Ban_{\mathbb{C}})})$$
    is fully faithful.
\end{corollary}

\begin{proof}
    Since $\mathbf{i}$ is the left adjoint of an adjunction, it suffices to prove that it is fully faithful at the level of homotopy categories. Thus, we may restrict to the functor
    $$i:\mathrm{D}(\mathsf{Nuc}_{b}^{\infty}(\C))\rightarrow\mathrm{D}(\mathrm{CBorn}_{\mathbb{C}}),$$
    where we are working here with the one-categorical derived categories. But now $i:\mathsf{Nuc}_{b}^{\infty}(\C)\rightarrow\mathrm{CBorn}_{\mathbb{C}}$ is part of an adjunction
    $$\adj{i}{\mathsf{Nuc}_{b}^{\infty}(\C)}{\mathrm{CBorn}_{\mathbb{C}}}{(-)^{\rig}}.$$
    Both of the functors are exact, so they give rise to an adjunction of derived categories
        $$\adj{i}{\mathrm{D}(\mathsf{Nuc}_{b}^{\infty}(\C))}{\mathrm{D}(\mathrm{CBorn}_{\mathbb{C}})}{(-)^{\rig}}.$$
        In fact, there is no `deriving' happening here, as both functors are exact. Thus, since $i$ is fully faithful at the level of one-categories, the unit
        $$X\rightarrow\mathbb{R}(-)^{\rig}\circ\mathbb{L}i(X)\cong(-)^{\rig}\circ i(X)$$
        is an equivalence for all $X$. It follows that $i$ is fully faithful.
\end{proof}

Since \((-)^\rig \colon \mathsf{CBorn}_\C \to \mathsf{Nuc}_b^\infty(\C)\) is exact by Proposition \ref{thm:nuc-exact}, it descends to a functor \[
    (-)^\rig \colon \mathbf{D}(\C) \to \mathbf{D}(\mathsf{Nuc}^\infty(\C)) \subseteq \mathbf{D}(\mathsf{Nuc}^\infty(\C)).
\]
which is just $\mathbb{R}(-)^{\rig,IB}=(-)^{\mathbf{rig,IB}}$.

\begin{proposition}
Let $X_{\bullet}$ be a complex such that each $X_{i}$ is $\mathbb{Q}$-basic nuclear in $\mathrm{Ind(Ban}_{\mathbb{C}}\mathrm{)}$. Then $X_{\bullet}$ is basic $\mathbb{Q}$-nuclear in $\mathbf{D}(\mathrm{Ind(Ban}_{\mathbb{C}}\mathrm{)})$.
\end{proposition}

\begin{proof}
    First we prove the claim for bounded complexes. Let $X_{\bullet}$ be concentrated in a single degree By shifting we may assume that it is in degree $0$, and in this case the claim is obvious. Any bounded complex is a finite colimit of such objects, and hence is basic $\mathbb{Q}$-nuclear. Then, any bounded below complex is a countable colimit of bounded complexes, and hence is basic $\mathbb{Q}$-nuclear. 

    For the unbounded case, we may write $X_{\bullet}\cong\colim_{n\le 0}\tau_{\ge n}X_{\bullet}.$ By the previous part each $X_{\bullet}$ is basic $\mathbb{Q}$-nuclear, and $X_{\bullet}$ is a countable colimit of these. 
\end{proof}

We can now prove the main result of this section.

\begin{corollary}\label{cor:rigid-nuc}
We have a fully faithful embedding
\[\bD(\mathsf{Nuc}^\infty(\C)) \to \mathbf{Nuc}^\infty(\C) \simeq \bD(\C)^\rig.\] 
\end{corollary}

\begin{proof}
The inclusion $\mathbf{D}(\mathsf{Nuc}^{\infty}(\mathbb{C}))\rightarrow\mathbf{D}(\mathsf{Ind(Ban}_{\mathbb{C}}))$ is fully faithful, and its image is closed under colimits. In fact it is generated under colimits by the category \(\mathbf{Nuc}^\infty(\C)^{\omega_1}\) of complexes of \(\Q\)-basic nuclear objects. By Remark \ref{rem:basic-nuclear-explicit}, these are contained inside the \(\omega_1\)-compact objects \(\mathbf{Nuc}^\infty(\C)^{\omega_1}\) of the rigidification \(\mathbf{Nuc}^\infty(\C)\).
\end{proof}

\begin{remark}
    In the previous preprint version of this paper, we claim that the functor \[\bD(\mathsf{Nuc}^\infty(\C)) \to \mathbf{Nuc}^\infty(\C)\] is an equivalence. However, this is not clear. We would need to represent a \(\Q\)-indexed inductive system with trace-class transition maps by an actual chain complex of compact projectives with termwise factorisably trace-class chain maps.
\end{remark}

\subsubsection{The non-Archimedean case}\label{subsubsec:nonarc-nuc}

In this subsection, let \(R\) be a Noetherian non-Archimedean Banach ring, which we assume is the ring of integers of a non-Archimedean valued field \(K\) of characteristic zero. In this case, we first have the following interesting observation:

\begin{proposition}\label{prop:trace-class-factor-nonarc}
Let \(R\) be non-Archimedean Banach ring as above. Then every trace-class map between non-Archimedean Banach \(R\)-modules factorises as a composition of two trace-class maps. Consequently, we have an equivalence of categories \[\mathsf{Nuc}(R) = \mathsf{Nuc}^\infty(R).\]  
\end{proposition}

\begin{proof}
We first consider the case of the fraction field \(K\). By a standard argument, we may always write a trace-class map \(f\) as a compostion \(f = c \circ p\), where \(c\) is a trace-class map, and \(p\) is a bounded \(K\)-linear map. To prove that when \(K\) is non-Archimedean, the map \(p\) is also trace-class, we spell out the decomposition. By definition of a trace-class map, there is a bounded sequence \((f_n) \in \mathsf{Hom}_K(V, K)\), a bounded sequence \((w_n) \in W\), and a summable sequence \((\lambda_n) \in l^1(\N, K)\) such that \[f(v) = \sum_{n \in \N} \lambda_n f_n(v) w_n.\] We may then find a null sequence \((\mu_n) \in c_0(\N, K)\) such that \(\lambda_n/\mu_n\) is summable and decompose \(f\) as the composition of \[p \colon V \to c_0(\N,K), \quad v \mapsto p(v)_n = \mu_n f_n(v)\] and the map \(c \colon c_0(\N,K) \to W\), \(c(\varphi) = \sum_{n \in \N} \frac{\lambda_n}{\mu_n} \varphi(n) w_n\). Since the map \(c\) may be written as \(c(\varphi) = \sum_{n \in \N} \frac{\lambda_n}{\mu_n} \delta_n(\varphi) w_n\) for the delta sequence \((\delta_n) \in l^\infty(\N, K)\), it is manifestly trace-class. Now since \(K\) is non-Archimedean, the null sequence \((\mu_n)\) is summable. Now let \(\delta_n \in C_0(\N, F)\) be the function satisfying \(\delta_n(m) = 1\) if and only if \(n = m\), and zero otherwise. Then \(p(v) = \sum_{n=0}^\infty \mu_n f_n(v)\delta_n\), so that \(p\) is trace-class. 

We now consider the case where \(f \colon V \to W\) is a trace-class operator between Banach modules over the ring \(R\) of integers. Then we may again write \(f(v) = \sum_{n \in \N} \alpha \dvgen^{\floor{\frac{n}{2}} + \ceil{\frac{n}{2}}} f_n(v) w_n\) for some \(\alpha\) not divisible by \(\dvgen\) and \((f_n) \in V^\vee\) and \((w_n) \in W\) bounded sequences. Then choosing \(c(\varphi) = \sum_{n \in \N} \dvgen^{\floor{\frac{n}{2}}} \varphi(n) w_n\) and \(p(v)(n) = \dvgen^{\alpha\ceil{\frac{n}{2}}} f_n(v)\) yields trace-class maps that factorise \(f\). 

Finally, let \(V\) be a nuclear bornological \(R\)-module, so that it is a filtered colimit of trace-class operators, and since any such trace-class operator is factorisable by two trace-class operators, the result follows.  \qedhere 
\end{proof}

\begin{corollary}\label{cor:rigid-nuc-nonarc}
We have equivalences of \(\infty\)-categories \[\mathbf{Nuc}(R) \simeq \mathbf{D}(\mathsf{Nuc}(R)) \simeq \bD(R)^\rig.\] In particular, for a non-Archimedean Banach ring \(R\), \(\mathbf{Nuc}(R)\) is rigid. 
\end{corollary}

\begin{proof}
The claim follows from Proposition \ref{prop:trace-class-factor-nonarc}, which identifies the categories of nuclear and \(\Q\)-nuclear bornological modules.
\end{proof}

As in the complex case, we may force nuclearity of complete bornological \(R\)-modules. Following \cite{Meyer-Mukherjee:HL}, we call an inclusion \(V \to W\) of Banach \(R\)-modules \emph{compactoid} if the image of \(V\) in \(W/\dvgen^n W\) is finitely generated as an \(R/\dvgen^n R\)-module for each \(n\). By \cite{Meyer-Mukherjee:HL}*{Proposition 3.5}, this is equivalent to the map being trace-class. As a consequence, a complete bornological \(R\)-module is nuclear if and only if it is a colimit of compactoid inclusions of Banach \(R\)-modules.

\begin{remark}
Note that there is a slight difference between the setup in \cite{Meyer-Mukherjee:HL} and this article, arising from the definition of a Banach \(R\)-module. Indeed, in \cite{Meyer-Mukherjee:HL}, the authors only consider \(\dvgen\)-torsion free Banach modules. However, this distinction plays no role in the characterisation of compactoid submodules as trace-class maps, as a map \(V \to W\) of \(\dvgen\)-torsion Banach modules is then compactoid if and only if it is represented by a sequence with only finitely many nonzero terms which is clearly trace-class. 
\end{remark} 

Let \(\mathsf{CBorn}_R^{\tf}\) denote the category \(\mathsf{Ind}^s(\mathsf{Mod}_R^{\tf, \hat{\dvgen}})\) of complete, torsion free bornological \(R\)-modules. By \cite{Cortinas-Cuntz-Meyer-Tamme:Nonarchimedean}*{Proposition 2.10}, this is equivalent to the category of strict inductive systems of \(\dvgen\)-adically complete, torsion free \(R\)-modules. We have a functor \[(-)^\rig \colon \mathsf{CBorn}_R^{\tf} \to \mathsf{Nuc}^{\tf}(R)\] from the category of complete bornological \(R\)-modules to nuclear, torsion free \(R\)-modules, that equips a bornological \(R\)-module with the compactoid bornology. By \cite{Meyer-Mukherjee:HL}*{Lemma 4.3}, this functor is right adjoint to the inclusion \(\mathsf{Nuc}^{\tf}(R) \subset \mathsf{CBorn}_R^{\tf}\).

\begin{lemma}\label{lem:compactoid-properties}
    The compactoid bornology functor preserves colimits, tensor products and is fully faithful on the full subcategory \(\mathsf{Mod}_R^{\tf, \hat{\dvgen}}\) of \(\dvgen\)-adically complete, torsion free \(R\)-modules. 
\end{lemma}

\begin{proof}
    By \cite{Meyer-Mukherjee:HL}*{Proposition 4.5, 4.6}, the compactoid bornology functor commutes with tensor products and cokernels that exist in \(\mathsf{CBorn}_R^\tf\). To see that it commutes with direct sums, let \(((X_i, \bdd_i))_{i \in I}\) be a family of complete, torsion free bornological \(R\)-modules. Then by definition of the direct sum bornology, a subset \(S \subset \bigoplus_{i \in I} X_i\) is compactoid if and only if there is a bounded subset \(T \in \bdd_i\) such that \(S \subset T\) is compactoid. This implies that \(S \in (X, \bdd_i^\rig)\). Conversely, if \(S\) is bounded in the direct sum bornology on the family \((X_i, \bdd_i^\rig)_{i \in I}\), then there is an \(i \in I\) such that \(S \in \bdd_i^\rig\), so that there is a \(T_i \in \bdd_i\) such that \(S \subseteq T_i\) is compactoid. Since inclusions \(X_i \to \bigoplus_i X_i\) are bounded, \(T \in \bdd_{\bigoplus_i X_i}\), from which it follows that \(S\) is bounded for the compactoid bornology on \(\bigoplus_i X_i\).  The last statement follows from the fact that compactoid subsets are classified by null sequences.  
\end{proof}

In what follows, we study the homological algebraic properties of the category \(\mathsf{CBorn}_R^\tf\). Note that \(\mathsf{CBorn}_R^\tf\) is not quasi-abelian as cokernels create torsion, and therefore the category is not closed under cokernels.

\begin{proposition}
    The quasi-abelian structure on \(\mathsf{CBorn}_R\) restricts to an exact category structure on \(\mathsf{CBorn}_R^\tf\). Furthermore, the projective model structure on \(\mathsf{Ch}(\mathsf{Ind}(\mathsf{Mod}_R^{\tf, \hat{\dvgen}}))\) exists, and is the restriction of the projective model structure on \(\mathsf{Ch}(\mathsf{Ind}(\mathsf{Ban}_R))\).
\end{proposition}

\begin{proof}
    For the first claim, we need to show that the collection of all kernel-cokernel pairs is an exact structure on \(\mathsf{CBorn}_R^\tf\). By \cite{Meyer-Mukherjee:Bornological_tf}*{Lemma 4.2}, submodules of torsion free \(R\)-modules are torsion free, so that the pullback of an admissible extension is admissible. And by \cite{Meyer-Mukherjee:Bornological_tf}*{Theorem 4.5}, the pushout of an admissible extension is admissible. This shows that \(\mathsf{CBorn}_R^\tf\) is a fully exact subcategory of \(\mathsf{CBorn}_R\), so that they have the same exact chain complexes. Similarly, \(\mathsf{Ind}(\mathsf{Mod}_R^{\tf, \hat{\dvgen}})\) is a fully exact subcategory of \(\mathsf{Ind}(\mathsf{Ban}_R)\), so that the exact complexes in the former category are precisely the restriction of the exact complexes in the latter category to \(\mathsf{Ind}(\mathsf{Mod}_R^{\tf, \hat{\dvgen}})\). Finally, as \(\mathsf{Ind}(\mathsf{Mod}_R^{\tf, \hat{\dvgen}})\) as enough projectives, which coincide with the projectives in \(\mathsf{Ind}(\mathsf{Ban}_R)\), the projective model structure exists and is precisely the restriction of the projective model structure on \(\mathsf{Ch}(\mathsf{Ind}(\mathsf{Ban}_R))\).    
\end{proof}

In fact the left heart $\mathsf{LH}(\mathsf{CBorn}^{tf}_{R})$ has a very straightforward description. Note that $\mathsf{LH}(\mathsf{CBorn}_{R})$ is the free sifted cocompletion on the full subcategory of $\mathrm{CBorn}_{R}$ consisting of compact projectives, $\mathsf{LH}(\mathsf{CBorn}_{R})=\mathcal{P}_{\Sigma}(\{l^{1}(V):V\textrm{ is a normed set}\}$. Then $\mathsf{LH}(\mathsf{CBorn}^{tf}_{R})=\mathcal{P}_{\Sigma}(\{l^{1}(V):V\textrm{ is a discrete normed set}\}$. Thus $\mathsf{CBorn}^{tf}_{R}$ may be thought of as the convex subcategory of $\mathsf{CBorn}_{R}$. This description manifestly gives rise to a fully faithful inclusion
$$\mathbf{D}(\mathsf{CBorn}^{tf}_{R})\rightarrow\mathbf{D}(\mathsf{CBorn}_{R})$$
which in fact has both left and right adjoints.

  As \((-)^\rig\) is additive, it extends to a symmetric monoidal functor between the homotopy categories of chain complexes \[(-)^\rig \colon \mathsf{Ch}(\mathsf{CBorn}_R^\tf) \to \mathsf{Ch}(\mathsf{CBorn}^{\tf}(R))\] of complete, torsion free \(R\)-modules, preserving homotopy colimits by Lemma \ref{lem:compactoid-properties}. By \cite{Meyer-Mukherjee:HL}*{Proposition 4.6}, this functor is exact, so that it preserves weak equivalences, thereby descending to a functor \begin{equation}\label{eqref:nonarc-rig}
  \bD(\mathsf{CBorn}_R^\tf) \to \bD(\mathsf{CBorn}_R^\tf),
  \end{equation} landing in the full subcategory \(\mathbf{Nuc}(R)\) of chain complexes of nuclear \(R\)-modules. By Lemma \ref{lem:torsion free}, \(\bA_R^{form} \subset \bD(\mathsf{CBorn}_R^\tf)\), so that the compactoid bornology functor above restricts to a symmetric monoidal functor \[(-)^\rig \colon \mathbf{A}_R^{form} \to \mathbf{Nuc}(R),\] and preserves colimits.  

We end this section with the promised proof of descent for nuclear modules over derived adic rings. Let us write $\mathbf{D}(R)^{tf}=\mathbf{D}(\mathsf{CBorn}_{R})^{tf}$. This is itself a derived algebraic context. Equip $\mathbf{Aff}_{\mathbf{D}(R)^{tf}}$ with the descendable pre-topology. 

\comment{
This first requires a nuclear refinement of the adic \'etale topology previously considered.

\begin{definition}
    The \emph{nuclear adic \'etale topology} on \(\mathbf{Aff}_R\) consists of covers \[\{\mathsf{Spec}(B) \to \mathsf{Spec}(A)\}\] such that 

\begin{enumerate}
    \item \(A \to B\) is descendable;
    \item \(B\) is adically finitely presented over \(A\) in the sense that \[B \cong A \hat{\otimes}^\mathbb{L} R\gen{x_1,\dotsc, x_n}^{rig}//(f_1, \dotsc, f_n)\] for \(f_1, \dotsc, f_n \in \pi_0(A) \hat{\otimes} R\gen{x_1, \dotsc, x_n}\);
    \item for any map $A\rightarrow C$ with $C$ a derived topologically finitely presented affinoid,
$$\{\mathrm{Spec}((R\big\slash\pi^{n}R\otimes^{\mathbb{L}}(C\hat{\otimes}_{A}^{\mathbb{L}}B))\rightarrow\mathrm{Spec}((R\big\slash\pi^{n}R\otimes^{\mathbb{L}}C)\}$$
is \'etale.
\end{enumerate}

Denote this topology by \(\tau_{ad}^{et,b}\). 
\end{definition}
}

\begin{theorem}\label{thm:rigidification-formal-desc}
    The functor \[(\mathbf{Aff}_{\mathbf{D}(R)^{tf}})^{\op} \to \mathbf{CAlg}(\mathbf{Pr}_{st}^L), \quad A \mapsto \mathbf{Nuc}(A) \defeq \mathbf{Mod}_{A^\rig}(\mathbf{Nuc}(R))\] is a sheaf of \(\infty\)-categories for the descendable pre-topology.
\end{theorem}

\begin{proof}
    As the compactoid bornology functor \((-)^\rig \colon\mathbf{D}(R)^{tf} \to \mathbf{Nuc}(R)\) is symmetric monoidal, by \cite{mathew2016galois}*{Corollary 3.21}, it preserves descendable maps in \(\mathbf{Aff}_{\mathbf{D}(R)^{tf}}\). The result now follows from Lemma \ref{lem:maxime1}. 
    \comment{
    Furthermore, by the discussion preceding the statement of the Theorem, \((-)^\rig\) preserves colimits. Now for an adically finitely presented map \(A \to B\), we have by definition \(B \cong A \hat{\otimes}^\mathbb{L} R\gen{x_1, \dotsc, x_n}^\rig//(f_1, \dotsc, f_n)\). Using that \((-)^\rig\) preserves colimits and tensor products, we get a morphism of sites \((\mathbf{A}_R^{form}, \tau_{ad}^{et}) \to (\mathbf{A}_R^{form}, \tau_{ad}^{et,b})\). Since the maps in \(\tau_{ad}^{et,b}\) are descendable, 
    }
\end{proof}

Since $\mathbf{dAfnd}_{R}\subset\mathbf{Aff}_{\mathbf{D}(R)^{tf}}$ and, the adically faithfully flat topology on $\mathbf{dAfnd}_{R}$  consists of descendable maps, we get the following. 

\begin{corollary}
      The functor \[(\mathbf{dAfnd}_{R})^{\op} \to \mathbf{CAlg}(\mathbf{Pr}_{st}^L), \quad A \mapsto \mathbf{Nuc}(A) \defeq \mathbf{Mod}_{A^\rig}(\mathbf{Nuc}(R))\] is a sheaf of \(\infty\)-categories for the adically faithfully flat topology.
\end{corollary}

Using that descendable topologies are subcanonical, for a space \(X \in \mathbf{dAn}_R\) we have \(\mathbf{Nuc}(X) \simeq \varprojlim_{\mathsf{Spec}(A) \to X}\mathbf{Mod}_{A^\rig}(\mathbf{Nuc}(R))\). Notice also that the descent result of Theorem \ref{thm:rigidification-formal-desc} follows from that of the functor \(A \mapsto \mathbf{Mod}_{A^\rig}(\bD(R))\) on derived affinoids. We denote by \(\bD^{\mathrm{nuc}}(X)\) the extension of this functor to derived analytic spaces. Finally, for a morphism \(f \colon X \to Y\) in between derived analytic spaces, we have an induced pullback functor \(f^* \colon \mathbf{Nuc}(Y) \to \mathbf{Nuc}(X)\), which is symmetric monoidal and colimit-preserving. By the presentability of the categories involved, this functor admit a right adjoint \(f_* \colon \mathbf{Nuc}(X) \to \mathbf{Nuc}(Y)\) which (uniquely) extends the adjunction \(\mathbf{Mod}_{A^{\rig}}(\bD(R)) \leftrightarrows \mathbf{Mod}_{B^\rig}(\bD(R))\) at the level of derived affinoids. 

\comment{As in the complex case, we have similar permeance properties for nuclear modules over a non-Archimedean Banach ring. 

\begin{proposition}\label{prop:properties-infinitely-nuclear2}
The full subcategory \(\mathsf{Nuc}_b(R)\) of \(\mathsf{CBorn}_R\) is closed under 
\begin{enumerate}
\item Closed subspaces;
\item Countable products;
\item Tensor products;
\item Countable colimits. 
\end{enumerate} 
\end{proposition}

\begin{proof}
Essentially the same proof as \ref{prop:properties-infinitely-nuclear} works with minor modifications. For (1) we see that the maps \(E \cap F_\alpha \to E \cap F_\beta\) are compactoid by \cite{perez2010locally}*{Theorem 8.1.3 (ii) and (iii)}, which is equivalent to the maps being trace-class by \cite{BaBK}*{Proposition 3.37}. (2) follows from \cite{BaBK}*{Proposition 3.51}.  (3) remains the same as the complex case. For (4), we use \cite{BaBK}*{Proposition 3.51} to show that the structure maps \(\frac{F_\alpha}{E \cap F_\alpha} \to \frac{F_\beta}{E \cap F_\beta}\) are trace-class. 
\end{proof}

Again, as consequence we get that the full subcategory \(\mathsf{Nuc}_b(R)\) of \(\mathsf{CBorn}_R\) is quasi-abelian.  }

\comment{\begin{theorem}\label{thm:nuc-rigid}
    Let \(R\) be the ring of integers of a nontrivially valued non-Archimedean Banach field. Then for a derived formal \(R\)-algebra \(A \in \mathbf{A}_R^{form}\), we have equivalences  \[\mathbf{Nuc}(A^\mathrm{nuc}) \simeq \varprojlim^{dual} \mathbf{Mod}_{A/\dvgen^n} \simeq \mathbf{Nuc}^\infty(A)\] of symmetric monoidal \(\infty\)-categories.
\end{theorem}

\begin{proof}
    Since the discretisation functor is strongly monoidal, by Theorem \ref{thm:tilde-nuc}, we have an equivalence \[\mathbf{D}(A) \simeq \varprojlim_n^{dual} \bD(A/p^n),\] so that the second equivalence follows from taking rigidification and Theorem \ref{thm:rigidification=nuc}. Moreover, since \(A^{nuc}\) is a colimit of \(\dvgen\)-adically complete \(R\)-modules with norm-decreasing inclusions, it again lies in \(\bD(\mathsf{Ban}_R^{\leq1})^\circ\). Consequently, we again have \(\bD(A) \simeq \varprojlim_n \bD(\mathbf{Mod}_{A/\dvgen^n})\) by Theorem \ref{thm:rigidification=nuc}. Furthemore, it is easy to see that \(\bD(A)^\rig \simeq \mathbf{Nuc}(A)^\rig\)\edit{write lemma}. Putting these things together, we get equivalences \[\mathbf{Nuc}^\infty(A) \simeq \varprojlim^{dual} \mathbf{Mod}_{A/\dvgen^n} \simeq \mathbf{Nuc}^\infty(A^{nuc}) \simeq \mathbf{Nuc}(A^{nuc})^\rig.\] It remains to show that \(\mathbf{Nuc}(A^{nuc}) \simeq \mathbf{Mod}_{A^{rig}}(\mathbf{Nuc}(R))\) is rigid. Since trace-class morphisms between Banach \(R\)-modules are factorisably trace-class, the underlying chain complex of \(A^{nuc}\) is actually \(\Q\)-nuclear. Finally, since \(\mathbf{Nuc}(R)\) is rigid, the result follows. 
\end{proof}}

\comment{\subsubsection{The internal model}

We now have an interesting subcategory of nuclear objects in \(\mathsf{CBorn}_R\). In this subsection, we show that this subcategory describes the category \(\mathbf{Nuc}^\infty(R)\) when \(R = \C\) and \(\mathbf{Nuc}(R)\) when \(R\) is a non-Archimedean Banach ring.

\begin{proposition}\label{thm:nuc-exact}
Let \(R\) be either the complex numbers or a non-Archimedean Banach ring. Then the functor \((-)^\rig \colon \mathsf{CBorn}_R \to \mathsf{Nuc}_b^\Q(R)\) is exact. 
\end{proposition}

\begin{proof}
Since \((-)^\rig\) is a right adjoint functor, it commutes with kernels, so it only remains to check whether it preserves cokernels. We first take the archimedean case. Let \(E \onto Q\) be a cokernel and \(S\) a bounded subset in \(Q^\rig\). Then there is an increasing sequence of bounded discs \((B_n)\) with \(S = B_0\), such that the induced inclusions \(Q_{B_n} \to Q_{B_{n+1}}\) are trace-class of order zero. Each such trace-class map is given by a standard representation as in \eqref{eqref:standard-rep}. Since \(E \onto Q\) is a bornological quotient map, we may lift the sequence \((\lambda_n)\) of rapid decay, the bounded functionals \((x_n) \in Q_{B_n}'\) and the bounded sequence \((y_n) \in \Q_{B_{n+1}}\) to sequences of the same type to get lifts \(E_{A_n} \to E_{A_{n+1}}\), where the \(A_n\)'s are bounded discs that lift \(B_n\). Summarily, the sequence \((B_n)\) of bounded discs lifts to a sequence of bounded discs \(A_n\) such that the induced maps \(E_{A_n} \to E_{A_{n+1}}\) are trace-class of order zero, from which we deduce that the induced map \(E^\rig \to Q^\rig\) is a bornological quotient map as required. The proof in the non-Archimedean case is similar and actually simpler. \qedhere
\end{proof}

As a consequence of Corollary \ref{cor:infinitely-nuc-quasi-ab}, we have a projective model structure on the category of unbounded complexes of \(\Q\)-nuclear bornological \(R\)-modules \(\mathsf{Ch}(\mathsf{Nuc}_b^\Q(R))\), using which we define the derived \(\infty\)-category \[\mathbf{D}(\mathsf{Nuc}^\Q(R)) = N(\mathsf{Ch}(\mathsf{Nuc}^\Q(R)))[W^{-1}],\] where \(W\) denotes the class of weak equivalences for the projective model structure on \(\mathsf{Nuc}_b^\Q(R)\). \footnote{Note here that by Proposition \ref{prop:trace-class-factor-nonarc}, for non-Archimedean base rings, the superscript \(\Q\) in \(\mathsf{Nuc}_b^\Q(R)\) is a redundancy, but we keep it for uniformity of presentation.} Moreover, since \((-)^\rig \colon \mathsf{CBorn}_R \to \mathsf{Nuc}_b^\Q(R)\) is exact by Proposition \ref{thm:nuc-exact}, it descends to a functor \begin{equation}\label{eqref:rigidification}
    (-)^\rig \colon \mathbf{D}(R) \to \mathbf{D}(\mathsf{Nuc}^\Q(R)) \subseteq \mathbf{D}(\mathsf{Nuc}^\infty(R)),
\end{equation} where the category \(\bD(\mathsf{Nuc}_b^\infty(R))\simeq \bD(\mathsf{Nuc}^\infty(R))\) is the closure in \(\bD(R)\) of the countably cocomplete category \(\bD(\mathsf{Nuc}^\Q(R))\) under colimits. }

\subsubsection{Rigidification of affinoids}

Let $K$ be a non-Archimedean non-trivially Banach field with ring of integers $R=\mathcal{O}_{K}$. Consider the left heart $\mathsf{LH}(\mathsf{CBorn}_{K})$. It is the  free sifted cocompletion of the full subcategory of $l^{1}$-spaces on normed sets
$$\mathcal{P}_{\Sigma}(\{l^{1}(V):V\textrm{ is a normed set}\}).$$
Consider the functor
$$K\otimes_{R}(-):\mathsf{LH}(\mathsf{CBorn}^{tf}_{R})\rightarrow \mathsf{LH}(\mathsf{CBorn}_{K}).$$
This functor has a right adjoint $G$. Abstractly, it is defined by $$E\mapsto(l^{1}(V)\mapsto\mathsf{Hom}(K\otimes_{R}l^{1}(V),E)).$$
However, we can describe it more concretely, at least for Banach $K$-vector spaces $E$. Indeed let $\overline{B}_{E}(0,1)$ denote the closed unit ball of $E$. Then $G(E)\cong\colim_{n}\overline{B}_{E}(0,1)$ where the transition maps are multiplication by $\pi$. Note further that $G$ commutes with sifted colimits. Moreover
$$K\otimes_{R}G(E)\rightarrow E$$
is an isomorphism, which simply follows from the fact that all transition morphisms in the diagram
$$\colim_{n}K\otimes_{R}\overline{B}_{E}(0,1)$$
are isomorphisms. It follows that $G$ is fully faithful, so that we may regard $\mathsf{LH}(\mathsf{CBorn}_{K})$ as a reflective subcategory of $\mathsf{LH}(\mathsf{CBorn}^{tf}_{R})$. Equip both the categories $\mathsf{Ch}(\mathsf{LH}(\mathsf{CBorn}_{K}))$ and $\mathsf{Ch}(\mathsf{LH}(\mathsf{CBorn}_{R})^{tf})$ with the projective model structures. Then the adjunction
$$\adj{K\otimes_{R}(-)}{\mathsf{Ch}(\mathsf{LH}(\mathsf{CBorn}^{tf}_{R}))}{\mathsf{LH}(\mathsf{CBorn}_{K})}{G(-)}$$
is in fact Quillen with strongly monoidal left adjoint. $G(-)$ is in particular exact. Hence the adjunction is a reflective Quillen adjunction, so that we get a reflective adjunction of $\infty$-categories with strongly monoidal left adjoint,
$$\adj{K\otimes_{R}^{\mathbb{L}}(-)}{\mathbf{D}(R)^{tf}}{\mathbf{D}(K)}{G(-)}.$$
In fact this means that $\mathbf{D}(K)$ is equivalent to the category of algebras over the monad $G(K\otimes^{\mathbb{L}}(-))$. But 
$$G(K\otimes^{\mathbb{L}}(E))\cong\colim_{n}E\cong(\colim_{n}V)\otimes^{\mathbb{L}} E\cong G(K)\otimes^{\mathbb{L}} E.$$
Hence $\mathbf{D}(K)$ is equivalent to modules in $\mathbf{D}(R)^{tf}$ over the algebra $G(K)$. Note also this implies that $R\rightarrow G(K)$ is a homotopy epimorphism. 

The following is obvious. 

\begin{lemma}
    $G(K)$ is nuclear as an $R$-module.
\end{lemma}

In particular $(-)^{rig}$ restricts to a strongly monoidal functor
$$\mathbf{D}(K)\rightarrow\mathbf{D}(K).$$

\comment{
Let $E$ be a Banach $K$-vector space. A subset $B$ of $E$ is said to be \textit{compactoid} if for any open neighbourhood $U$ of the origin, there is a finite set $F$ such that $E\subset U+\Gamma(F)$ where $\Gamma(F)$ denotes the absolutely convex hull. This defines a bornology on $E$. We denote the resulting bornological space by $\mathrm{Cpt}(E)$. 

In fact this gives rise to a functor
$$\mathrm{Cpt}:\mathrm{Ban}_{K}\rightarrow\mathrm{CBorn}_{K}$$
which equips a Banach $K$-vector space $E$ with the compactoid bornology. 

First let us work at the level of semi-normed modules.
Let $E$ be a $\pi$-adic, $\pi$-torsion free Banach $R$-module, which we regard as a semi-normed $R$-module, and let $S\subset E$ be compactoid. We consider the semi-norm $\rho_{S}$ on $E$ defined by 
$$\rho_{S}(e)=\mathrm{inf}\{|\pi|^{r}:e\in\pi^{r}\mathrm{Span}_{R}(S)\}$$.
This is clearly a $\pi$-adic semi-norm, in that $\rho_{S}(\pi e)=|\pi|\rho_{S}(e)$ for all $e\in E$. It is also $\pi$-torsion free. We denote by $E_{S}$ the underlying space of $E$ equipped with the semi-norm $\rho_{S}$, and by $\hat{E}_{S}$ its completion. Then we have 
$$E^{rig}\cong\colim_{S\subset E\textrm{ compactoid} }\hat{E}_{S}.$$
We then have 
$$K\hat{\otimes}_{R}^{\mathbb{L}}E^{rig}\cong \colim_{S\subset E\textrm{ compactoid} }K\hat{\otimes}_{R}\hat{E}_{S}.$$
The closed ball of radius $1$ in $K\hat{\otimes}_{R}\hat{E}_{S}$ is just $\mathrm{Span}_{R}(S)$. This coincides with the absolutely convex hull of $S$, regarded as a subset of $K\otimes_{R}E$. One then easily checks that, as semi-normed spaces,
$$(K\otimes_{R}E)_{S}\cong K\otimes_{R}(E_{S}).$$

Now let $B\subset K\otimes_{R}E$ be any compactoid. Then there is a finite subset $\{x_{1},\ldots,x_{n}\}$ of $K\otimes_{R}E$ such that $B\subset \pi E+\Gamma(x_{1},\ldots,x_{n})$. In particular, there are $\overline{x}_{i}$ of norm at most $1$ such that $B$ is absorbed by
$$S=\pi E+\mathrm{Span}_{R}(x_{1},\ldots,x_{n}).$$
In particular $B$ is absorbed by $S$. Finally, we claim that $S$ is a compactoid in $E$. Let $n>0$. There is a finite subset $\{x_{1},\ldots,x_{n}\}$ of $K\otimes_{R}E$ such that $S\subset \pi^{n}E+\mathrm{Span}_{R}(x_{1},\ldots,x_{n})$. Then
$$S=S\cap E\subset\pi^{n}E+E\cap\mathrm{Span}_{R}(x_{1},\ldots,x_{n}).$$
Let $v\in E\cap\mathrm{Span}_{R}(x_{1},\ldots,x_{n})$.



\begin{lemma}
    There is a natural equivalence $(K)\otimes^{\mathbb{L}}_{R}(-)^{rig}\cong\mathrm{Cpt}(K\otimes^{\mathbb{L}}_{R}(-))$ of functors
    $$\mathbf{D}_{\ge0}(R)^{tf}\rightarrow\mathbf{D}_{\ge0}(K).$$
\end{lemma}

\begin{proof}
    Since all functors are defined by extending sifted colimits, we may restrict to $\pi$-adic $\pi$-torsion free Banach $R$-modules $E$. 
    

\end{proof}

Let $A$ be a derived affinoid, and let $\tilde{A}$ be a formal model for $A$. Then $\mathrm{Cpt}(A)\cong K\otimes_{R}^{\mathbb{L}}(\tilde{A})^{rig}$. In particular $\mathrm{Cpt}(A)$ is nuclear as a complete bornological $K$-vector space. Now let $A\rightarrow B$ be a descendable \'{e}tale map of derived affinoids. Let $B\rightarrow C$ be a descendable \'{e}tale map admitting a descendable (faithfully flat) formal model $\tilde{B}\rightarrow\tilde{C}$, such that $A\rightarrow C$ admits a descendable (faithfully flat) formal model $\tilde{A}\rightarrow\tilde{C}$. Then $K\otimes_{R}^{\mathbb{L}}(\tilde{A})^{rig}\rightarrow K\otimes_{R}^{\mathbb{L}}(\tilde{C})^{rig}$ and $K\otimes_{R}^{\mathbb{L}}(\tilde{B})^{rig}\rightarrow K\otimes_{R}^{\mathbb{L}}(\tilde{C})^{rig}$ are descendable. Hence $\mathrm{Cpt}(A)\rightarrow\mathrm{Cpt}(C)$and $\mathrm{Cpt}(A)\rightarrow\mathrm{Cpt}(B)$ are descendable. Thus $\mathrm{Cpt}(A)\rightarrow\mathrm{Cpt}(B)$ satisfies descent for $\mathbf{QCoh}$. We have proven the following.

}

\begin{theorem}\label{thm:rigidification-rigid-desc}
    The functor \[(\mathbf{dAfnd}_{K})^{\op} \to \mathbf{CAlg}(\mathbf{Pr}_{st}^L), \quad A \mapsto \mathbf{Nuc}(A) \defeq \mathbf{Mod}_{A^{rig}}(\mathbf{Nuc}(K))\]
is a sheaf of \(\infty\)-categories for the \'{e}tale topology.
\end{theorem}

\subsection{Nuclear contexts}
At this point it is convenient to introduce \textit{nuclear contexts}. Here we still specialise to working with a context of the form
$$(\mathbf{Aff}_{\mathbf{C}},\tau,\mathbf{P},\mathbf{A})$$
where $\tau$ is some topology contained in the desecendable topology.

\begin{definition}
    We will call the context $(\mathbf{Aff}_{\mathbf{C}},\tau,\mathbf{P},\mathbf{A})$ \textit{nuclear} if for any object $\mathrm{Spec}(A)$ of $\mathbf{A}$, the underlying object of $\mathbf{A}$ is in $\mathbf{C}^{\rig}$. 
\end{definition}

For example, the usual algebraic context, and the dagger analytic context, are both nuclear. The rigid analytic context is not nuclear. There is a nuclear version of the rigid context. This can also be done for formal schemes:

\subsubsection{The nuclear rigid analytic context}\label{subsubsec:nucgeom}

Fix a non-trivially valued Banach field $K$. We define a generating class of algebras by $\mathrm{T}^{nuc}(\lambda_{1},\ldots,\lambda_{n})\defeq(K\gen{\frac{x_{1}}{\lambda_{1}},\ldots,\frac{x_{n}}{\lambda_{n}}})^{rig}$. This is of homotopy polynomial type. Indeed the polynomial algebra is nuclear, and $(-)^{\rig}$ is a monoidal functor. Let $\mathbf{dAfnd}_K^{nuc}$ denote the full subcategory of $\mathbf{Aff}_{\mathsf{CBorn}_{R}}$ consisting of those objects of the form $\mathrm{Spec}(A^{\rig})$, where $A$ is a derived affinoid. Let
$$(\mathbf{Aff}_{\mathsf{CBorn}_{K}},\tau^{nuc-rat},\mathbf{P}^{o-sm},\mathbf{dAfnd}_K^{nuc})$$
be the geometry tuple obtained as the strengthening of the tuple defined using the $\mathrm{T}^{nuc}$-rational topology and $\mathrm{T}^{nuc}$-(open) smooth morphisms, but forcing strength with respect to $\mathbf{dAfnd}_K^{nuc}$ as in Remark \ref{rem:forcestrong}. Rigidifications of rational covers of affinoids in the usual sense are in $\tau_K^{nuc-rat}$, and similarly for smooth morphisms. We in fact get induced functors on the level of geometric stacks and schemes
$$(-)^{rig}:\mathbf{Stk}_{geom}(\bA_K,\tau_{K}^{rat}, \mathbf{P}^{o-sm})\rightarrow\mathbf{Stk}_{geom}(\mathbf{dAfnd}_K^{nuc}\tau_K^{nuc-rat},\mathbf{P}^{o-sm})$$
$$(-)^{rig}:\mathbf{Sch}(\bA_K,\tau_{K}^{rat}, \mathbf{P}^{o-sm})\rightarrow\mathbf{Stk}_{geom}(\mathbf{dAfnd}_K^{nuc},\tau_K^{nuc-rat},\mathbf{P}^{nuc-sm}).$$
This also works for the \'{e}tale topology and \'{e}tale-local smooth maps.

\section{Continuous K-theory for dualisable categories}

In this section, we describe localising invariants in the context of analytic geometry. Recall that in the purely algebraic setting, invariants of rings such as \(K\)-theory, topological cyclic homology, cyclic homology, homotopy \(K\)-theory,  extend to those of small, idempotent complete stable \(\infty\)-categories by taking perfect complexes. That is, we have a diagram

\[
\begin{tikzcd}
\mathsf{Rings} \arrow{r}{F} \arrow[d, swap, "\mathrm{Perf}(-)"] & \mathsf{Sp} \\
\mathsf{Cat}_{\infty}^{\mathrm{perf}} \arrow[ur, swap, "\tilde{F}"] & 
\end{tikzcd}
\] in which the extension \(\tilde{F}\) at the level of categories recovers the definition at the level of rings for such invariants. Now taking the ind-completion of such a category, we obtain an equivalence

\[ \mathbf{Cat}_{\infty}^{\mathrm{perf}} \to \mathbf{Pr}_{st}^{L, cg}, \qquad \mathbf{C} \mapsto \mathbf{Ind}(\mathbf{C})\] of \(\infty\)-categories with the category of compactly generated stable \(\infty\)-categories. The inverse is given by the functor \[\mathbf{Pr}_{st}^{L, cg} \to  \mathbf{Cat}_{\infty}^{\mathrm{perf}}, \qquad \mathbf{C} \mapsto \mathbf{C}^\omega\] that assigns to a compactly generated stable \(\infty\)-category its full subcategory of compact objects.

We call a diagram \(\bA \overset{F}\to \mathbf{B} \overset{G}\to \bC\) in \(\mathsf{Pr}_{st}^L\) a \emph{localisation sequence} if \(G \circ F \simeq 0\), \(F\) is fully faithful, and the induced functor \(\mathbf{B}/\bA \to \bC\) on the Verdier quotient is an equivalence. Here by Verdier quotient, we mean the pushout of the functor \(F \colon \bA \to \mathbf{B}\) and the unique functor \(\bA \to \{*\}\) to the terminal object in \(\mathsf{Pr}_{st}^L\). In other words, a localisation sequence is a fibre-cofibre sequence in \(\mathsf{Pr}_{st}\). Note that if one works in the subcategory \( \mathsf{Cat}_{\infty}^{\mathrm{perf}}\), then we need to take the idempotent completion \(\mathsf{Idem}(\mathbf{B}/\bA)\),  and require that the induced map \(\mathsf{Idem}(\mathbf{B}/\bA) \to \bC\) is an equivalence. 

\begin{definition}\label{def:localising-invariant}
Let \(\mathcal{T}\) be a stable \(\infty\)-category. We call a functor \(F \colon  \mathsf{Cat}_{\infty}^{\mathrm{perf}} \to \mathcal{T}\) a \emph{localising invariant} if it preserves the final object, and maps a localisation sequence to a fibre-cofibre sequence.
\end{definition}

As we have seen in Section \ref{section:derived-borno-geom}, for a complete bornological \(R\)-algebra \(A\) over a Banach ring, the category  \(\mathbf{D}(A)\) is a compactly generated, presentable stable \(\infty\)-category. At first pass, we may try to proceed in exactly the same way as in the algebraic case by defining localising invariants on the full subcategory \(\mathbf{Perf}(R)\) of compact objects in \(\bD(R)\). Unfortunately, this does not work due to Eilenberg swindle: indeed, for any Banach ring \(R\), we have a split exact sequence of Banach spaces \[0 \to l^1(\N) \to l^1(\N) \to R \to 0,\] where the quotient is given by projection onto the first summand, and the kernel is the shift \((x_1 , x_2, \dotsc, ) \mapsto (0, x_1, x_2, \dotsc,)\). Then the functor \(F \colon \mathsf{Ban}_R \to \mathsf{Ban}_R\) defined by \(F(A) = l^1(A)\) satisfies \(F(A) \oplus A \cong F(A)\). 

Denote by \(\mathsf{Cat}_\infty^{\mathrm{dual}}\) the \(\infty\)-category of dualisable \(\infty\)-categories with morphisms given by left adjoint functors whose right adjoints are further left adjoints. Such functors are called \emph{strongly continuous}. Note that \(\mathsf{Cat}_\infty^{\mathrm{dual}}\) is a non full subcategory of \(\mathsf{Pr}_{st}^L\). Filtered cocompletion defines a fully faithful functor \[\mathsf{Ind} \colon \mathsf{Cat}_\infty^{\mathrm{perf}} \to \mathsf{Cat}_\infty^{\mathrm{dual}}, \bC \mapsto \mathsf{Ind}(\bC),\] whose essential image as pointed out previously is the \(\infty\)-category of compactly generated, presentable stable \(\infty\)-categories. In this sense, dualisable categories are ``large" stable \(\infty\)-categories. We now recall the following important result due to Efimov:

\begin{theorem}[Efimov]\label{thm:Efimov}
Any localising invariant \(E \colon \mathsf{Cat}_\infty^{\mathrm{perf}} \to \mathcal{T}\) into a stable \(\infty\)-category extends uniquely to a localising invariant \(E^\cont \colon \mathsf{Cat}_\infty^{\mathrm{dual}} \to \mathcal{T}\):
\[
\begin{tikzcd}
\mathsf{Cat}_\infty^{\mathrm{perf}} \arrow{r}{E} \arrow[d, swap, "\mathsf{Ind}"] & \mathcal{T} \\
\mathsf{Cat}_\infty^{\mathrm{dual}} \arrow[ur, swap, "E^\cont"] &. 
\end{tikzcd}
\]
\end{theorem}

By \cite{blumberg2013universal}, there is a  presentable, stable \(\infty\)-category \(\mathbf{NCMot}\) called \emph{noncommutative motives}, and a functor \(\mathcal{U} \colon \mathsf{Cat}_\infty^{\mathrm{perf}} \to \mathbf{NCMot}\) satisfying the following universal property: for any presentable, stable \(\infty\)-category \(\bD\), there is an induced equivalence 
\[
\mathbf{Fun}^L(\mathbf{NCMot}, \bD) \overset{\simeq}\to \mathbf{Fun}(\mathbf{Cat}_{\infty}^{\mathrm{perf}}, \bD), \qquad F \mapsto \mathcal{U} \circ F 
\] of \(\infty\)-categories, where the left hand side denotes colimit preserving functors, and the right hand side denotes localising invariants. The universal functor \(\cU \colon \mathbf{Cat}_\infty^{\mathrm{dual}} \to \mathbf{NCMot}\) represents algebraic \(K\)-theory in the sense that 

\[\mathsf{Hom}_{\mathbf{NCMot}}(\mathcal{U}(\mathsf{Sp}), \mathcal{U}(\mathbf{C})) \simeq K(\mathbf{C}),\] so we may view algebraic \(K\)-theory as the universal localising invariant.

By Theorem \ref{thm:Efimov}, the canonical map \[\cU \colon \mathsf{Cat}_{\infty}^{\mathrm{perf}} \to \mathbf{NCMot}\] extends uniquely to a localising invariant 

\[
\cU^\cont \colon \mathbf{Cat}_\infty^{\mathrm{dual}} \to \mathbf{NCMot}, 
\] which is given explicitly by the continuous Calkin construction \[\cU^\cont(\mathbf{C}) = \Omega \cU(\mathbf{Calk}^\cont(\mathbf{C}^\omega))\] (see \cite{nkp}). By uniqueness, \(\cU^\cont\) is the universal localising invariant of dualisable presentable, stable \(\infty\)-categories. Furthermore, the category \(\mathbf{NCMot}\) represents the unique extension of algebraic \(K\)-theory 
\[
\mathsf{Hom}_{\mathbf{NCMot}}(\cU^{\cont}(\mathsf{Sp}), \cU^\cont(\bC)) \overset{\simeq}\to K^\cont(\bC),
\] called \textit{continuous \(K\)-theory} due to Efimov.

\begin{definition}\label{def:bornological-localising}
Let \((\mathbf{Aff}_\bC, \tau, \mathbf{P}, \bA)\) be a nuclear context. For a localising invariant \(E \colon \mathbf{Cat}^{\mathrm{perf}} \to \bD\), we define its associated $\mathbf{A}$-\emph{localising invariant} as the continuous extension \[E^\an \colon \bA_R^\op \to \bD, \quad E^{\an}(\mathsf{Spec}(A)) \defeq E^\cont(\mathbf{Mod}_{A}(\mathbf{Nuc}^\infty(\bC)))\] of \(E\) on the category of infinitely nuclear \(A\)-modules. 
\end{definition}

In particular, the definition above specialises to the dagger, nuclear formal and rigid analytic contexts. In what follows, we discuss descent properties of bornological \(K\)-theory. We first see that the definition of a bornological localising invariant globalises whenever \(\mathbf{Nuc}^\infty\) satisfies descent, at least to qcqs schemes. In particular, we will show that for $X$ qcqs, the category $\mathbf{Nuc}^{\infty}(X)$ is rigid. In particular it is dualisable, and we are perfectly entitled to define 
\[E^{geom}(X) \defeq E^\cont(\mathbf{Nuc}^\infty(X)).\]

We will be particularly interested in the various analytic settings. In this situation we will write $E^{an}(X)$ instead of $E^{geom}(X)$.


\subsection{The Nisnevich topology}

As is well known from the purely algebraic setting, \(K\)-theory does not in general satisfy \'etale descent. It does however satisfy descent for a coarser topology called the \textit{Nisnevich topology}. To define this topology, we introduce some terminology originally due to Morel and Voevodsky. 

\begin{definition}\label{def:cd-structure}
Let \(\fC\) be a small category with an initial object. A \emph{cd-structure} on \(\fC\) is a collection \(\mathcal{P}\) of commuting squares \[
Q = \begin{tikzcd}
B \arrow{r}{} \arrow{d}{} & Y \arrow{d}{p}\\
A \arrow{r}{e} & X,
\end{tikzcd}
\]
 such that if \(Q \in \mathcal{P}\) and \(Q'\) is isomorphic to \(Q\), then \(Q' \in \mathcal{P}\). 
\end{definition}

Given a commuting square \(Q \in \mathcal{P}\), let \((p,e)\) denote the sieve generated by the family \(\{p \colon Y \to X, e \colon A \to X\}\). We then denote by \(\tau_P\) the smallest Grothendieck topology such that \((p,e)\) is a covering sieve, and the empty sieve is a covering sieve of the initial object. We call a covering \(\{X_i \to X\}\) \emph{simple} if it factors through some \((p,e)\).

We call a cd structure on \(\fC\) \emph{complete} if any covering sieve of an object that is not isomorphic to the initial object, contains a sieve generated by a simple sieve. This is equivalent to the requirement that any morphism whose codomain is an initial object is an isomorphism, and that for any \(Q \in \mathcal{P}\) and any morphism \(X' \to X\) in \(\fC\), the pullback \(Q' = Q \times_X X'\) exists and belongs to \(\mathcal{P}\). Finally, we call a cd-structure \textit{regular} if the commuting squares in \(\mathcal{P}\) are pullback squares, the morphism \(e\) is a monomorphism, and for each \(Q \in \mathcal{P}\), the induced square

\[
\Delta(Q) = \begin{tikzcd}
B \arrow{r}{} \arrow{d}{} & Y \arrow{d}{}\\
B \times_A B \arrow{r}{} & Y \times_X Y,
\end{tikzcd}
\] belongs to \(\mathcal{P}\).

\begin{definition}\label{def:Nisnevich-infty}
A cd-structure on a small \(\infty\)-category \(\bC\) is a cd-structure on its homotopy category. 
\end{definition}

\begin{lemma}\cite{voevodsky2008homotopy}*{Corollary 5.10}\label{lem:Nisnevich-sheaves}
Let \(\mathcal{P}\) be a complete, regular cd-structure on an \(\infty\)-category \(\bC\) with an initial object \(\emptyset\). Then \(\mathcal{F} \colon \bC^\op \to \mathbf{An}\) is a \(\tau_{\mathcal{P}}\)-sheaf if and only if \(\mathcal{F}(\emptyset)\) is contractible, and for every \(Q \in \mathcal{P}\), \(\mathcal{F}(Q)\) is a pullback square. 
\end{lemma}

We now apply Morel and Voevodsky's general machinery to our setting where \(\bC\) is the \(\infty\)-category of derived \(R\)-analytic spaces. 
\subsubsection{Generating CD-structures}
Let $\mathrm{C}$ be a category with an initial object $\emptyset$ such that any map $C\rightarrow\emptyset$ is an isomorphism. 
Let $\chi$ be a class of pullback squares
    \begin{displaymath}
\xymatrix{
\mathcal{U}\times_{\mathcal{X}}\mathcal{Y}\ar[r]\ar[d]  & \mathcal{Y}\ar[d]^{p}\\
\mathcal{U}\ar[r]^{e} & \mathcal{X}
}
\end{displaymath}
where $e$ is a monomorphism. We denote by $R(\chi)$ the class of squares in $\chi$ such that pullback squares

   \begin{displaymath}
\xymatrix{
\mathcal{U}\times_{\mathcal{X}}\mathcal{Y}\times_{\mathcal{X}}\mathcal{W}\ar[r]\ar[d]  & \mathcal{Y}\times_{\mathcal{X}}\mathcal{W}\ar[d]^{p'}\\
\mathcal{U}\times_{\mathcal{X}}\mathcal{W}\ar[r]^{e'} & \mathcal{W}
}
\end{displaymath}
along morphisms $\mathcal{W}\rightarrow\mathcal{X}$ are also in $\chi$. This is a complete cd-structure.

Let $\Delta(\chi)$ denote the class of squares $S$ in $\chi$ such that the `diagonal square', $\Delta(S)$
\begin{displaymath}
    \xymatrix{
\mathcal{U}\times_{\mathcal{X}}\mathcal{Y}\ar[d]\ar[r] & \mathcal{Y}\ar[d]\\
\mathcal{U}\times_{\mathcal{Z}}\mathcal{Y}\times_{\mathcal{X}}\mathcal{Y}\ar[r] & \mathcal{Y}\times_{\mathcal{X}}\mathcal{Y}
    }
\end{displaymath}
is also in $\chi$. Inductively define $$\Delta^{n+1}(\chi)=\Delta(\Delta^{n}(\chi))$$ 
and 
$$\Delta^{\infty}(\chi)=\bigcap_{n\ge0}\Delta^{n}(\chi)$$
This is a regular cd-structure. 

\begin{lemma}
    We have 
    $$R(\Delta(\chi)))=\Delta(R(\chi)).$$
    In particular, this is a complete and regular cd-structure.
\end{lemma}

\begin{proof}
First we show that $\Delta(R(\chi))\subset R(\Delta(\chi))$. Let $S$ be a square in $\Delta(R(\chi))\subset\Delta(\chi)$. Then it is in particular in $\Delta(\chi)$. We need to show that any pullback of $S$, $\mathcal{W}\times_{\mathcal{X}}S$, is also in $\Delta(\chi)$. But we have $\Delta(\mathcal{W}\times_{\mathcal{X}}S)\cong\mathcal{W}\times_{\mathcal{X}}\Delta(S)$ is also in $\Delta(R(\chi))$, and in particular is in $\Delta(\chi)$. Hence $S\in R(\Delta(\chi))$. Conversely suppose $S\in R(\Delta(\chi))$. Then in particular it is in $\Delta(\chi)$. We claim that it is in $\Delta(R(\chi))$. That is, we need to show that $\Delta^{n}(\mathcal{W}\times_{\mathcal{X}}S)\in\chi$ for all $n$ and all $\mathcal{W}\rightarrow\mathcal{X}$. But we have $\Delta^{n}(\mathcal{W}\times_{\mathcal{X}}S)\cong\mathcal{W}\times_{\mathcal{X}}\Delta^{n}(S)$. But $\Delta^{n}(S)\in\chi$ for all $n$, so $\mathcal{W}\times_{\mathcal{X}}\Delta^{n}(S)\in R(\chi)$ as well. This completes the proof.
\end{proof}

\begin{definition}
    Let $\mathbf{C}$ be a derived algebraic context and $\mathrm{T}$ a generating class of algebras of homotopy polynomial type. We define the $\mathrm{T}$-\textit{Nisnevich cd-structure} on $\mathbf{Sch}(\mathbf{Aff}_{\mathbf{C}},\tau_{\'{e}t},\mathbf{P}^{sm})$ to be $\chi_{N}\defeq R(\Delta(\chi^{pre}_{N}))$, where $\chi^{pre}_{N}$ consists of squares of the form
    \begin{displaymath}
        \xymatrix{
        U\times_{X}V\ar[d]\ar[r] & V\ar[d]^{p}\\
        U\ar[r]^{e} & X
        }
    \end{displaymath}
    where 
    \begin{enumerate}
        \item 
        $e:U\rightarrow X$ is an open immersion;
        \item 
        $p:V\rightarrow X$ is \'{e}tale;
        \item 
        $$U\coprod V\rightarrow X$$ is an effective epimorphism.
    \end{enumerate}
\end{definition}

Let us populate this class of squares in certain examples.

\begin{example}
    Any cover in the finite rational topology can written as a square in $\chi_{N}$. 
\end{example}

\subsubsection{Spatial squares}

Denote by \(\mathcal{P}_{sp-N}\) commuting squares of the form 

\begin{equation}\label{eq:preNis}
\begin{tikzcd}
U \times_X V \arrow{r}{} \arrow{d}{} & V \arrow{d}{p}\\
U \arrow{r}{e} & X,
\end{tikzcd}
\end{equation} where 
\begin{enumerate}
    \item 
    \(e\) is an open immersion,
    \item
    \(p\) is \'etale,  and spatial on the complement of $U$,
    \item
    \(|V|_{Ber} \times_{|X|_{Ber}} (|X|_{Ber} \setminus |U|_{Ber}) \to |X|_{Ber} \setminus |U|_{Ber}\) is a homeomorphism.
\end{enumerate} Note that in the rigid case Lemma \ref{lem:etale-pi_0}, \ref{cor:standard-etale-characterisation} and Corollary \ref{cor:diagonal-loc} say that being \'etale is equivalent to the map being derived strong and \(\pi_0\)-\'etale. We call such commuting diagrams \emph{elementary spatial Nisnevich squares}. Note that by definition, the properties of being \'{e}tale and an open immersion are local on the target.

\begin{example}
    Let
    \begin{equation}
\begin{tikzcd}
U \times_X V \arrow{r}{} \arrow{d}{} & V \arrow{d}{p}\\
U \arrow{r}{e} & X,
\end{tikzcd}
\end{equation}
be an elementary Nisnevich square of rigid analytic spaces in the sense of Andreychev (\cite{andreychev2023k}) (or, more precisely, specialising Andreychev's definition for more general analytic adic spaces). We use the characterisation of Andreychev's elementary Nisnevich squares from \cite{TangLSeminar}. This precisely means that $U\rightarrow X$ is an open immersion, that $|V|\times_{|X|}(|X|\setminus |U|)\cong |X|\setminus |U|$ is a homeomorphism of degree $1$, and that for any $v\in |V\setminus U|_{Ber}$, $v$ and $p(v)$ have isomorphic residue fields. In particular for any $v\in |V|_{Ber}\times_{|X|_{Ber}}(|X|_{Ber}\setminus |U|_{Ber})$, we have an isomorphism of residue fields $F_{v}\cong F_{f(v)}$. Now suppose $v\in |V|_{Ber}$. Then by \cite{de1996etale}, we can find an open $U$ containing $f(v)$ and an open $W_{v}$ containing $v$ such that the restriction of $f$ to $V$ is an open immersion landing in $W$. Thus we have that 
$$V\cong V\times_{X}U\cup\bigcup_{v\in |V|_{Ber}\setminus |V\times_{X}U|_{Ber}}W_{v},$$
and hence $p$ is spatial on the complement of $U$ by Lemma \ref{lem:andrspat}. The other conditions are straightforward, so that this gives an example of an elementary Nisnevich square. One can also take $U,V,X$ to be derived rigid analytic spaces, and assume that it forms an Andreychev-Nisnevich square on $0$-truncations.
\end{example}

\begin{corollary}\label{cor:equivalenttruncation}
    Let $f:X\rightarrow Y$ be an \'{e}tale morphism of derived rigid analytic spaces. Then the diagonal morphism $X\rightarrow X\times_{Y}X$ is an open immersion.
\end{corollary}

\begin{proof}
 It is certainly locally an open immersion. It remains to prove that it is injective. But we have $|X\times_{Y}X|\cong |X|\times_{|Y|}|X|$. The morphism of spaces $|X|\rightarrow |X|\times_{|Y|}|X|$ is evidently injective. 
\end{proof}

\begin{lemma}
   Let \begin{equation}
\begin{tikzcd}
U \times_X V \arrow{r}{} \arrow{d}{} & V \arrow{d}{p}\\
U \arrow{r}{e} & X,
\end{tikzcd}
\end{equation} be an elementary Nisnevich square of derived rigid analytic spaces. Then both $|U|\coprod |V|\rightarrow |X|$ and $|U|_{Ber}\coprod |V|_{Ber}\rightarrow |X|_{Ber}$ are surjective. Consequently, the map \(U \coprod V \to X\) is an effective epimorphism of geometric stacks.
\end{lemma}

\begin{proof}
We have $|V|\times_{|X|}(|X|\setminus |U|)\cong |X|\setminus |U|$ by assumption. Thus, everything not in the image of $|U|\rightarrow |X|$ is hit by something in $|V|$.

    Since $U\rightarrow X$ is an open immersion, we have $|V|_{Ber}\times_{|X|_{Ber}}|U|_{Ber}\cong |V\times_{X}U|_{Ber}.$ Thus $|V|_{Ber}\times_{|X|_{Ber}}(|X|_{Ber}\setminus |U|_{Ber})\cong |X|_{Ber}\setminus |U|_{Ber}.$ Now the same proof as for the Huber spectrum works. 
\end{proof}

\begin{theorem}\label{thm:Voevodsky-analytic}
The collection of commuting diagrams \(\mathcal{P}_{sp-N}\) in is a complete, regular cd-structure.  In particular in the rigid analytic context, such squares are contained in $\chi_{N}$.
\end{theorem}

\begin{proof}
Note that open immersions and \'etale maps are closed under base change by arbitrary morphisms. Furthermore, by construction, spatiality on the complement of $U\rightarrow X$ will be preserved by base-change. Next if $V\rightarrow V\times_{X}V$ is spatial , and $Z\rightarrow X$ any map, then $(V\times_{X}Z)\times_{Z}(V\times_{X}Z)\cong (V\times_{X}V)\times_{X}Z$, and $(V\times_{X}Z)\times_{Z}(V\times_{X}Z)$ is the pullback of $V\rightarrow V\times_{X}V$ along $V\times_{X}Z\rightarrow V$. Hence it is also spatial. Consequently, the base change of a Nisnevich square is a Nisnevich square, so that the cd-structure \(\mathcal{P}^N\) is complete.

To see that it is reduced, it remains to show that for an elementary Nisnevich square \(Q\), the induced square \(\Delta(Q)\) is an elementary Nisnevich square. Again as open immersions are closed under base change, the map \(U \times_X V \to V\) is an open immersion. Since by hypothesis \(e \colon U \to X\) is also an open immersion, the induced map \((U \times_X V) \times_U (U \times_X V) \to V \times_X V\) is am  open immersion. In fact it is the pullback of $U$ along the map $V\times_{X}V\rightarrow X$. Now $V\rightarrow V\times_{X}V$ is an open immersion. Hence it is spatial, so it is in particular spatial on the complement of $U\times_{X}(V\times_{X}V)$. 

Finally we need to check the condition on complements. Now applying spatiality on the complement of $U$ to the pullback along $V\rightarrow X$ gives 
\begin{multline*}
|V\times_{X}V|_{Ber}\setminus |V\times_{X}V\times_{X}U|_{Ber}  \\ \cong|V|_{Ber}\times_{|X|_{Ber}}|V|_{Ber}\times_{|X|_{Ber}}|X|_{Ber}\setminus|V|_{Ber}\times_{|X|_{Ber}}|V|_{Ber}\times_{|X|_{Ber}} |U|_{Ber}\\
\cong|X|_{Ber}\setminus |U|_{Ber}.
\end{multline*}
But we also have 
\begin{multline*}
|V|_{Ber}\times_{|V\times_{X}V|_{Ber}}(|V\times_{X}V|_{Ber}\setminus |V\times_{X}V\times_{X}U|_{Ber}) \\ \cong |V|_{Ber}\setminus(|V|_{Ber}\times_{|V\times_{X}V|_{Ber}}|V\times_{X}V|_{Ber}\times_{|X|_{Ber}}|U|_{Ber})\\
\cong|V|_{Ber}\setminus(|V|_{Ber}\times_{|X|_{Ber}}|U|_{Ber})
\cong |X|_{Ber}\setminus |U|_{Ber}.
\end{multline*}
This completes the proof. 
\end{proof}

\subsection{K\"{u}nneth formulae for $\mathbf{QCoh}$}

In order to prove Nisnevich descent for $K$-theory and the Grothendieck-Riemann-Roch Theorem, we will need some facts about rigidity of functors between categories of the form $\mathbf{Nuc}^{\infty}(X)$. Here we still specialise to working with a context of the form
$$(\mathbf{Aff}_{\mathbf{C}},\tau,\mathbf{P},\mathbf{A})$$
where $\tau$ is some descendable topology. By `stack', we shall mean a stack on the site $(\mathbf{A},\tau|_{\mathbf{A}})$.

\begin{definition}
\begin{enumerate}
    \item 
      A stack $\mathcal{X}$ is said to be \textit{formally affine}, if the global sections functor (i.e., the pushforward to the point)
    $$\Gamma(\mathcal{X},-):\mathbf{QCoh}(\mathcal{X})\rightarrow\mathbf{C}$$
    is conservative.
    \item
    A morphism $f:\mathcal{X}\rightarrow\mathcal{Y}$ of stacks is said to be \textit{formally affine} if for every affine $\mathrm{Spec}(A)\rightarrow\mathcal{Y}$, with $\mathrm{Spec}(A)\in\mathbf{A}$, the stack $\mathrm{Spec}(A)\times_{\mathcal{Y}}\mathcal{X}$ is formally affine.
     \item 
      If the context is nuclear, a stack $\mathcal{X}$ is said to be \textit{formally nuclear} if the global sections functor (i.e., the push-forward to the point)
    $$\Gamma(\mathcal{X},-):\mathbf{QCoh}(\mathcal{X})\rightarrow\mathbf{C}$$
    sends $\mathbf{Nuc}^{\infty}(X)$ to $\mathbf{C}^{rig}$.
    \item
   If the context is nuclear, a morphism $f:\mathcal{X}\rightarrow\mathcal{Y}$ of stacks is said to be \textit{formally nuclear} if for every affine $\mathrm{Spec}(A)\rightarrow\mathcal{Y}$, with $\mathrm{Spec}(A)\in\mathbf{A}$,  the stack $\mathrm{Spec}(A)\times_{\mathcal{Y}}\mathcal{X}$ is formally nuclear.
\end{enumerate}

\end{definition}

In particular, affines are both formally affine and, in the nuclear setting, formally nuclear.

Note that if the context is nuclear, and $f:\mathcal{X}\rightarrow\mathrm{Spec}(A)$ is formally nuclear, then $f_{*}$ sends $\mathbf{Nuc}^{\infty}(\mathcal{X})$ to $\mathbf{Nuc}^{\infty}(A)$.

\begin{proposition}\label{prop:cocontqcqs}
    Let $f:\mathcal{X}\rightarrow \mathcal{Y}$ be a qcqs morphism of stacks. That is, for any morphism $\mathrm{Spec}(A)\rightarrow\mathcal{Y}$, $\mathcal{X}\times_{\mathrm{Y}}\mathrm{Spec}(A)$ is a scheme, and $\mathcal{X}\times_{\mathrm{Y}}\mathrm{Spec}(A)\rightarrow\mathrm{Spec}(A)$ is a qcqs morphism. Then $f$ satisfies base-change. Moreover, $f_{*}$ commutes with colimits. 
\end{proposition}

\begin{proof}
    Base change can be proven by pulling back to affines. Thus we may assume from the outset that $f$ is a qcqs morphism between schemes. Now an identical proof to \cite{soor2024derived}*{Lemma 3.2} works. The fact that $f_{*}$ commutes with colimits now follows from base-change, and again \cite{soor2024derived}*{Lemma 3.2}.
\end{proof}

\begin{lemma}\label{lem:qaffformrig}
    Let $f:U\rightarrow \mathrm{Spec}(A)$ be an open immersion with $U$ quasi-compact. Then $i$ is base-change, formally affine, and, in the nuclear case, is formally nuclear. In fact in this case, 
    $$f_{*}:\mathbf{QCoh}(U)\rightarrow\mathbf{QCoh}(\mathrm{Spec}(A))$$
    and, in the nuclear setting, 
       $$f_{*}:\mathbf{Nuc}^{\infty}(U)\rightarrow\mathbf{Nuc}^{\infty}(\mathrm{Spec}(A))$$
       are fully faithful. 
\end{lemma}

\begin{proof}
Since $i$ is qcqs, it satisfies base change.

    First we prove that $f_{*}$ is fully faithful. Write $U=\bigcup_{i=1}^{n}U_{i}$ with each $U_{i}$ affine. Note that $U_{i_{1}}\times_{U}\ldots\times_{U}U_{i_{n}}\cong U_{i_{1}}\times_{\mathrm{Spec}(A)}\ldots\times_{\mathrm{Spec}(A)}U_{i_{n}}$ is affine, and all of the maps $f_{(i_{1},\ldots,i_{n})}:U_{i_{1}}\times_{U}\ldots\times_{U}U_{i_{n}}\rightarrow \mathrm{Spec}(A)$ and $g_{(i_{1},\ldots,i_{n})}:U_{i_{1}}\times_{U}\ldots\times_{U}U_{i_{n}}\rightarrow \mathrm{Spec}(A)$ are open immersions. Let $\mathcal{F}\in\mathbf{QCoh}(U)$. Then we may write
    $$\mathcal{F}\cong\mathbf{lim}_{(i_{1},\ldots,i_{n})\in\mathcal{I}^{n},n\in\mathbb{N}}(g_{(i_{1},\ldots,i_{n})})_{*}g_{(i_{1},\ldots,i_{n})}^{*}\mathcal{F},$$
    where the limit is finite, since each $U_{i}\rightarrow U$ is an open immersion. 
    Then 
    $$f_{*}\mathcal{F}\cong\mathbf{lim}_{(i_{1},\ldots,i_{n})}f_{*}(g_{(i_{1},\ldots,i_{n})})_{*}g_{(i_{1},\ldots,i_{n})}^{*}\mathcal{F}\cong\mathbf{lim}_{(i_{1},\ldots,i_{n})}(f_{(i_{1},\ldots,i_{n})})_{*}g_{(i_{1},\ldots,i_{n})}^{*}\mathcal{F}.$$
    Since the limit is finite, we have 
\begin{multline*}
f^{*}\mathbf{lim}_{(i_{1},\ldots,i_{n})}(f_{(i_{1},\ldots,i_{n})})_{*}g_{(i_{1},\ldots,i_{n})}^{*}\mathcal{F}
\cong\mathbf{lim}_{(i_{1},\ldots,i_{n})}f^{*}(f_{(i_{1},\ldots,i_{n})})_{*}g_{(i_{1},\ldots,i_{n})}^{*}\mathcal{F}\\
\cong\mathbf{lim}_{(i_{1},\ldots,i_{n})}(g_{(i_{1},\ldots,i_{n})})_{*}g_{(i_{1},\ldots,i_{n})}^{*}\mathcal{F}\cong\mathcal{F},
\end{multline*}
where in the second equivalence we used that $f_{(i_{1},\ldots,i_{n})}$ is a morphism of affines so that it satisfies base-change, and that $U\times_{\mathrm{Spec}(A)}(U_{i_{1}}\times_{U}\ldots\times_{U}U_{i_{n}})\cong(U_{i_{1}}\times_{U}\ldots\times_{U}U_{i_{n}})$. 

Finally, suppose the context is nuclear, and let $\mathcal{F}$ be in $\mathbf{Nuc}^{\infty}(U)$. Again we return to the identity
 $$f_{*}\mathcal{F}\cong\mathbf{lim}_{(i_{1},\ldots,i_{n})}(f_{(i_{1},\ldots,i_{n})})_{*}g_{(i_{1},\ldots,i_{n})}^{*}\mathcal{F}.$$
Each $(f_{(i_{1},\ldots,i_{n})})_{*}g_{(i_{1},\ldots,i_{n})}^{*}\mathcal{F}$ is in $\mathbf{Nuc}^{\infty}(A)$, since $f_{(i_{1},\ldots,i_{n})}$ is a morphism of affines and pullback always preserves infinitely nuclear objects. Then $f_{*}\mathcal{F}$ is a \textit{finite limit} of infinitely nuclear objects, and therefore is infinitely nuclear. 
\end{proof}

In particular, say that a scheme $X$ is \textit{quasi-affine} if it quasi-compact and there exists an open immersion $i:X\rightarrow \mathrm{Spec}(A)$. Then such stacks are formally affine and, in the nuclear case, formally nuclear.

\begin{corollary}\label{cor:fformrigqcqs}
   Suppose the context is nuclear. Let $f:X\rightarrow Y$ be a qcqs morphism of schemes. Then $f$ is formally nuclear.
\end{corollary}

\begin{proof}
 By base-change may we may assume that $Y$ is an affine scheme, and $X$ is a qcqs scheme. Pick a finite open affine cover $X=\bigcup_{i=1}^{n}U_{i}$. Again we write $g_{(i_{1},\ldots,i_{n})}:U_{i_{1}}\times_{X}\ldots\times_{X}U_{i_{n}}\rightarrow X$, and $f_{(i_{1},\ldots,i_{n})}:U_{i_{1}}\times_{X}\ldots\times_{X}U_{i_{n}}\rightarrow Y$ for the obvious maps. Note that all of the $U_{i_{1}}\times_{X}\ldots\times_{X}U_{i_{n}}$ are quasi-affine. Using the result of Lemma \ref{lem:qaffformrig}, the proof now proceeds in a formally identical way to the proof of Lemma \ref{lem:qaffformrig}.
\end{proof}

\begin{corollary}\label{cor:openimmersembedding}
    Let $i:X\rightarrow Y$ be qcqs monomorphism of stacks. That is, the map 
    $$X\rightarrow X\times_{Y}X$$
    is an equivalence. Then $i_{*}:\mathbf{QCoh}(X)\rightarrow\mathbf{QCoh}(Y)$ is fully faithful. In the nuclear context, $i_{*}:\mathbf{Nuc}^{\infty}(X)\rightarrow\mathbf{Nuc}^{\infty}(Y)$ is also fully faithful. 
\end{corollary}

\begin{proof}
Let $p:\mathrm{Spec}(B)\rightarrow Y$ be an atlas, and consider the fibre product diagram
    \begin{displaymath}
        \xymatrix{
        \mathrm{Spec}(B)\times_{X}Y\ar[d]^{i'}\ar[rr]^{p'} && X\ar[d]^{i}\\
        \mathrm{Spec}(B)\ar[rr]^{p}& & Y.
        }
    \end{displaymath}
    Now, $p'$ is an epimorphism of stacks. In particular $(p')^{*}$ is conservative. Thus it suffices to show that $p'^{*}i^{*}i_{*}\cong p'^{*}$. Now
    $$p'^{*}i^{*}i_{*}\cong (i')^{*}p^{*}i_{*}\cong(i')^{*}(i')_{*}(p')^{*}.$$
    Thus it suffices to prove that $(i')_{*}$ is fully faithful. But now we are in the case of an open immersion from a quasi-compact scheme to an affine, which was proven in Lemma \ref{lem:qaffformrig}.
\end{proof}


\begin{lemma}\label{lem:qnucpush}
    Suppose the context is nuclear. Let $f:\mathcal{X}\rightarrow\mathcal{Y}$ be a morphism of stacks which satisfies the base-change property, and is formally nuclear. Then $f_{*}$ restricts to a functor $\mathbf{Nuc}^{\infty}(\mathcal{X})\rightarrow\mathbf{Nuc}^{\infty}(\mathcal{Y})$. If $f_{*}:\mathbf{QCoh}(\mathcal{X})\rightarrow\mathbf{QCoh}(\mathcal{Y})$ commutes with colimits, then so does $f_{*}:\mathbf{Nuc}^{\infty}(\mathcal{X})\rightarrow\mathbf{Nuc}^{\infty}(\mathcal{Y})$. 
\end{lemma}

\begin{proof}
    Let $\mathcal{M}\in\mathbf{Nuc}^{\infty}(\mathcal{X})$. We need to show that $f_{*}\mathcal{M}$ is in $\mathbf{Nuc}^{\infty}(\mathcal{Y})$. Thus for any $g:U\rightarrow\mathcal{Y}$ with $U$ affine, we need to show that $g^{*}f_{*}\mathcal{M}$ is in $\mathbf{Nuc}^{\infty}(\mathcal{Y})$. By base-change, this is $f'_{*}\circ (g')^{*}\mathcal{M}$, where $f':\mathcal{X}\times_{\mathcal{Y}}U\rightarrow U$ and $g':\mathcal{X}\times_{\mathcal{Y}}U\rightarrow\mathcal{X}$ are the pullback maps. Thus the claim reduces to the fact that $(f')_{*}$, by assumption, sends infinitely nuclear objects to infinitely nuclear objects. The second  claim is clear.
\end{proof}

\begin{lemma}\label{lemma:qcohfibre}
    Let $f:\mathcal{Y}_{1}\rightarrow\mathcal{Y}_{2}$ be a formally affine, qcqs map of stacks. We may regard $f_{*}(\mathcal{O}_{\mathcal{Y}_{1}})$ as an associative algebra object in $\mathbf{QCoh}(\mathcal{Y}_{2})$. The natural functor
    $$\mathbf{QCoh}(\mathcal{Y}_{1})\rightarrow\mathbf{Mod}_{f_{*}(\mathcal{O}_{\mathcal{Y}_{1}})}(\mathbf{QCoh}(\mathcal{Y}_{2}))$$
    is an equivalence.
\end{lemma}

\begin{proof}
    We generalise the proof of \cite{MR3381473}*{Lemma B.1.4}. As in loc. cit. we use the Barr-Beck Lurie Theorem. The qcqs assumption means that the map $f$ satisfies base change, and that $f_{*}$ commutes with colimits. As in \cite{MR3381473}*{Lemma B.1.4}, base-change together with the formally affine assumption implies that $f_{*}$ is conservative. Hence $\mathbf{QCoh}(\mathcal{Y}_{1})$ is monadic over $\mathbf{QCoh}(\mathcal{Y}_{2})$, with monad given by $f_{*}f^{*}$. Yet again, as in \cite{MR3381473}*{Lemma B.1.4}, the projection formula (also yet again \cite{soor2024derived}*{Lemma 3.2}), implies that this monad is given by tensoring with $f_{*}(\mathcal{O}_{\mathcal{Y}_{1}})$. 
\end{proof}

\begin{corollary}\label{cor:nucfibre}
       Suppose the context is nuclear. Let $f:\mathcal{Y}_{1}\rightarrow\mathcal{Y}_{2}$ be a formally affine, formally nuclear map of stacks. We may regard $f_{*}(\mathcal{O}_{\mathcal{Y}_{1}})$ as an associative algebra object in $\mathbf{QCoh}(\mathcal{Y}_{2})$. There is an equivalence of categories
    $$\mathbf{Nuc}^{\infty}(\mathcal{Y}_{1})\rightarrow\mathbf{Mod}_{f_{*}(\mathcal{O}_{\mathcal{Y}_{1}})}(\mathbf{Nuc}^{\infty}(\mathcal{Y}_{2})).$$
\end{corollary}

\begin{proof}
    It suffices to prove that $f_{*}(\mathcal{O}_{Y_{1}})$ is infinitely nuclear. Then the monad from the proof of Lemma \ref{lemma:qcohfibre} restricts to $\mathbf{Nuc}^{\infty}$. By base-change we may assume that $\mathcal{Y}_{2}\cong\mathrm{Spec}(A)$ is affine, and $\mathbf{QCoh}(\mathcal{Y}_{2})\rightarrow\mathbf{C}$ sends infinitely nuclear objects to infinitely nuclear objects. But then by Proposition \ref{prop:nuclear-relative}, and using the fact that $A$ is in $\mathbf{C}$, we find that $f_{*}$ sends infinitely nuclear objects to infinitely nuclear objects. 
\end{proof}

Next we will prove K\"{u}nneth formulae for $\mathbf{QCoh}$ and $\mathbf{Nuc}^{\infty}$.

\begin{lemma}\label{lem:tensordiagram}
    Let $f:\mathcal{Y}_{1}\rightarrow\mathcal{Y}_{2}$ be formally affine and qcqs, and let
    \begin{displaymath}
        \xymatrix{
        \mathcal{Y}_{1}'\ar[d]^{f'}\ar[r]^{g_{1}} &\mathcal{Y}_{1}\ar[d]^{f}\\
        \mathcal{Y}_{2}'\ar[r]^{g_{2}} &\mathcal{Y}_{2}.
        }
    \end{displaymath}
    be a Cartesian diagram of stacks.  The natural symmetric monoidal functor
    $$\mathbf{QCoh}(\mathcal{Y}'_{2})\otimes_{\mathbf{QCoh}(\mathcal{Y}_{2})}\mathbf{QCoh}(\mathcal{Y}_{1})\rightarrow\mathbf{QCoh}(\mathcal{Y}_{1}')$$
    is an equivalence. 

    If in addition the context is nuclear and $f$ is formally nuclear, we have an equivalence
     $$\mathbf{Nuc}^{\infty}(\mathcal{Y}'_{2})\otimes_{\mathbf{Nuc}^{\infty}(\mathcal{Y}_{2})}\mathbf{Nuc}^{\infty}(\mathcal{Y}_{1})\rightarrow\mathbf{Nuc}^{\infty}(\mathcal{Y}_{1}').$$
\end{lemma}

\begin{proof}
    Once more, this works identically to \cite{MR3381473}*{Proposition B.1.3}, using our Lemma \ref{lemma:qcohfibre} and Corollary \ref{cor:nucfibre} as inputs. 
\end{proof}

\begin{corollary}\label{cor:f*rig}
    Let $f:\mathcal{X}\rightarrow\mathcal{Y}$ be formally affine and qcqs. Then $f^{*}:\mathbf{QCoh}(\mathcal{Y})\rightarrow\mathbf{QCoh}(\mathcal{X})$ is a rigid functor. 
If in addition the context is nuclear and $f$ is formally nuclear, then this is also true for $f^{*}:\mathbf{Nuc}^{\infty}(\mathcal{Y})\rightarrow\mathbf{Nuc}^{\infty}(\mathcal{X})$.
\end{corollary}

\begin{proof}
The proof of this works as in \cite{hoyois2021categorified}*{Proposition 2.35}.
    The $\mathbf{QCoh}(\mathcal{Y})$-linearity of $f_{*}$ follows from the projection formula. Now we have 
$$\mathbf{QCoh}(\mathcal{X})\otimes_{\mathbf{QCoh}(\mathcal{Y})}\mathbf{QCoh}(\mathcal{X})\cong\mathbf{QCoh}(\mathcal{X}\times_{\mathcal{Y}}\mathcal{X}).$$
Under this identification, the functor
$$\Delta:\mathbf{QCoh}(\mathcal{X})\otimes_{\mathbf{QCoh}(\mathcal{Y})}\mathbf{QCoh}(\mathcal{X})\rightarrow\mathbf{QCoh}(\mathcal{X})$$
is given by $\Delta_{f}^{*}$, where $\Delta_{f}:\mathcal{X}\rightarrow\mathcal{X}\times_{\mathcal{Y}}\mathcal{X}$ is the diagonal. Since $f$ is qcqs, the diagonal morphism is quasi-compact. Thus the right adjoint $(\Delta_{f})_{*}$ satisfies base-change and therefore the projection formula.
\end{proof}



 
\begin{lemma}\label{lem:maxime2}
    Suppose the context is nuclear. Let $X$ be a qcqs scheme. Then $\mathbf{Nuc}^{\infty}(X)$ is rigid. Moreover, if $U\rightarrow X$ is an affine cover, then we have equivalences 
    \[\mathbf{Nuc}^{\infty}(X) \cong\varprojlim^{dual}_{[n] \in \Delta} \mathbf{Nuc}^{\infty}(U^{\times_{X}n})\cong \varprojlim_{[n] \in \Delta} \mathbf{Nuc}^{\infty}(U^{\times_{X}n}).\]
\end{lemma}

\begin{proof}
First assume that $X$ is quasi-affine. 
    Let $i:X\rightarrow\mathrm{Spec}(A)$ be a open immersion. We then have an adjunction
    $$\adj{i^{*}}{\mathbf{Nuc}^{\infty}(A)}{\mathbf{Nuc}^{\infty}(X)}{i_{*}}$$
    with $i_{*}$ being fully faithful. Note moreover that $i_{*}$ commutes with colimits. Thus $\mathbf{Nuc}^{\infty}$ is a retract, in $\mathbf{Pr}^{L}_{st}$, of the rigid category $\mathbf{Nuc}^{\infty}(A)$. Moreover the localisation $i^{*}$ is symmetric monoidal. Thus by \cite{ramzi2024dualizable}*{Example 4.9}, $\mathbf{Nuc}^{\infty}(X)$ is rigid.

    Now let $X$ be a general qcqs scheme, and write $X=\bigcup_{i=1}^{n}U_{i}$ with each $U_{i}$ affine open. Note that we have a fibre-product diagram
    \begin{displaymath}
        \xymatrix{
        \mathbf{Nuc}^{\infty}(U_{1}\cup U_{2})\ar[d]\ar[r] & \mathbf{Nuc}^{\infty}(U_{1})\ar[d]\\
        \mathbf{Nuc}^{\infty}(U_{2})\ar[r] & \mathbf{Nuc}^{\infty}(U_{1}\times_{X}U_{2}).
        }
    \end{displaymath}
    The right and bottom legs are localisations by internal left adjoints, thanks to the projection formula. Now $U_{1}\times_{X}U_{2}$ is quasi-affine so the top-right, bottom-left, and bottom-right terms are rigid. Thus, by \cite{ramzi2024dualizable}*{Corollary 4.8},  this is in fact a pullback in the category of dualisable categories, and hence the top-left is rigid. Now we may proceed by induction to show that $\mathbf{Nuc}^{\infty}(X)$ is rigid. 

    For the final claim, by descent for $\mathbf{Nuc}^{\infty}$, we have
   $$ \mathbf{Nuc}^{\infty}(X) \cong\varprojlim_{[n] \in \Delta} \mathbf{Nuc}^{\infty}(U^{\times_{X}n}).$$
   Note that we already know that $ \mathbf{Nuc}^{\infty}(X)$ is rigid, and hence dualisable. Moreover all the functors involved are internal left adjoints by base change. Finally the limit is essentially a finite one, since the $U_{i}\rightarrow X$ are open immersions. By \cite{ramzi2024dualizable}*{Corollary 4.5}, this is a limit in the category  of dualisable categories.
\end{proof}

\comment{
\begin{lemma}
Suppose the context is nuclear. Let $X$ be a qcqs scheme. Then $\mathbf{Nuc}^{\infty}(X)$ is rigid.
\end{lemma}

\begin{proof}
We first prove the claim for quasi-affine schemes. Thus there is an open immersion $X\rightarrow \mathrm{Spec}(A)$, and $X=\bigcup_{i=1}^{m}U_{i}$, with each $U_{i}$ being affine. Observe that $U_{i_{1}}\times_{U}\ldots\times_{U}U_{i_{k}}\cong U_{i_{1}}\times_{\mathrm{Spec}(A)}\ldots\times_{\mathrm{Spec}(A)}U_{i_{k}}$ is affine. Put $U=\sqcup_{i=1}^{m}U_{i}$, which is an affine scheme. Then each $U^{\times_{X}n}$ is affine.  Now by descent for $\mathbf{Nuc}^{\infty}(-)$, we also have
\[\mathbf{Nuc}^{\infty}(X) \cong \varprojlim_{[n] \in \Delta} \mathbf{Nuc}^{\infty}(U^{\times_{X}n}).\]
On the other hand, we have 
\[\mathbf{QCoh}^{\rig}(X) \cong \varprojlim_{[n] \in \Delta}^\mathrm{dual} \mathbf{QCoh}(U^{\times_{X}n})^\rig\cong \varprojlim_{[n] \in \Delta}^\mathrm{dual} \mathbf{Nuc}^{\infty}(U^{\times_{X}n}),\]
We invoke the criteria in \cite{ramzi2024dualizable}*{Corollary 4.5}. Note this is essentially a finite limit, since $U=\bigcup_{i=1}^{k}U_{i}$ with each $U_{i}\rightarrow X$ being an open immersion. $\mathbf{Nuc}^{\infty}(X)$ is dualisable by Lemma \ref{lem:qcqsdualisable}. 

Finally, we need to verify that the projections $\mathbf{Nuc}^{\infty}(X)\rightarrow\mathbf{Nuc}^{\infty}(U^{\times_{X}n})$ are internal left adjoints. But this follows immediately from Corollary \ref{cor:f*rig}.

Now we return to the qcqs case. We pick an atlas $U\rightarrow X$, where $U$ is affine, and each $U^{\times_{X}n}$ is \textit{quasi-affine}. The proof now proceeds identically.
\end{proof}
}

An interesting question is whether $\mathbf{Nuc}^{\infty}(X)$ is the rigidification of $\mathbf{QCoh}(X)$. 

\begin{definition}
    Say that a scheme $X$ is \textit{infinitely nuclear} if we have an equivalence $\mathbf{Nuc}^{\infty}(X)\cong\mathbf{QCoh}(X)^{rig}$.
\end{definition}

\begin{example}
    Let $X$ be a qcqs discrete dagger analytic space. We claim that $X$ is infinitely nuclear. First we prove the claim for quasi-affines. Let $X\rightarrow\mathrm{Spec}(A)$ be an open immersion, with $A$ being a discrete dagger affinoid algebra. Write $X=\bigcup_{i=1}^{n}U_{i}$ with each $U_{i}=\mathrm{Spec}(A_{i})$ being discrete dagger affinoid. Note then that each $U_{i}\times_{X}U_{j}\cong U_{i}\times_{\mathrm{Spec}(A)}U_{j}\cong\mathrm{Spec}(A_{i}\hat{\otimes}^{\mathbb{L}}_{A}A_{j})$ is a discrete dagger affinoid, and hence is infinitely nuclear. We have 
    $$\mathbf{QCoh}(X)\cong \varprojlim_n \mathbf{QCoh}(U^{\times_{X}n})$$
    where $U\defeq\sqcup_{i=1}^{n}U_{i}$. By the above each $U^{\times_{X}n}$ is representable be a discrete dagger affinoid. Applying rigidification, we get 
     $$\mathbf{QCoh}(X)^{rig}\cong \varprojlim^{dual}_n \mathbf{Nuc}^{\infty}(U^{\times_{X}n})$$
    where the limit is taken in the category of dualisable categories. On the other hand, we also know that
     $$\mathbf{Nuc}^{\infty}(X)\cong \varprojlim^{dual}_n \mathbf{Nuc}^{\infty}(U^{\times_{X}n}).$$
     This gives the equivalence. 

     Proving the claim for general qcqs schemes works in the same way, just that all terms in the limit now will be of the form $\mathbf{Nuc}^{\infty}(U^{\times_{X}n})$ for $U^{\times_{X}n}$ quasi-affine.
     
\end{example}


We can straightforwardly adapt the definition of passable maps of stacks from \cite{hoyois2021categorified}*{Definition 2.33}.

\begin{definition}
 A map $f:\mathcal{X}\rightarrow\mathcal{Y}$ of stacks is
    \textit{passable} if
    \begin{enumerate}
        \item 
        the diagonal $\mathcal{X}\rightarrow\mathcal{X}\times_{\mathcal{Y}}\mathcal{X}$ is quasi-affine;
        \item 
        the pullback $f^{*}:\mathbf{Nuc}^{\infty}(Y)\rightarrow\mathbf{Nuc}^{\infty}(\mathcal{X})$ admits a $\mathbf{Nuc}^{\infty}(\mathcal{Y})$-linear right adjoint, $f_{*}$;
        \item $\mathbf{Nuc}^{\infty}(\mathcal{X})$ is dualisable as a $\mathbf{Nuc}^{\infty}(\mathcal{Y})$-module.
    \end{enumerate}    
\end{definition}

\begin{example}
   A qcqs morphism between qcqs schemes is passable. The first condition is clearly satisfied. Now we have that $\mathbf{Nuc}^{\infty}(\mathcal{X})$ and $\mathbf{Nuc}^{\infty}(\mathcal{Y})$ are rigid over $\mathbf{Sp}$ by Lemma \ref{lem:maxime2}. Thus they are in particular locally rigid, and by \cite{nkp} Proposition 4.3.10, $\mathbf{Nuc}^{\infty}(\mathcal{X})$ is locally rigid over $\mathbf{Nuc}^{\infty}(\mathcal{Y})$. In particular it is dualisable.  Finally by Corollary \ref{cor:fformrigqcqs} and base change for qcqs morphisms, $f^{*}$ is a $\mathbf{Nuc}^{\infty}(\mathcal{Y})$-linear right adjoint. 
\end{example}

\begin{lemma}\label{lem:tensorequiv}
    Suppose that $f:\mathcal{X}\rightarrow\mathcal{Y}$ is such that $\mathbf{Nuc}^{\infty}(\mathcal{X})$ is dualisable as a $\mathbf{Nuc}^{\infty}(\mathcal{Y})$-module. Let $\mathcal{Z}\rightarrow\mathcal{Y}$ be a map with quasi-affine diagonal. Then the natural functor
    $$\mathbf{Nuc}^{\infty}(\mathcal{Z})\otimes_{\mathbf{Nuc}^{\infty}(\mathcal{Y})}\mathbf{Nuc}^{\infty}(\mathcal{X})\rightarrow\mathbf{Nuc}^{\infty}(\mathcal{Z}\times_{\mathcal{Y}}\mathcal{X})$$
    is an equivalence.
\end{lemma}

\begin{proof}
    The proof works as in the proof of \cite{hoyois2021categorified}*{Proposition 2.34}. We write $\mathbf{Nuc}^{\infty}(\mathcal{Z})\cong\lim_{\mathrm{Spec}(A)\rightarrow \mathcal{Z}}\mathbf{Nuc}^{\infty}(\mathrm{Spec}(A))$. Since tensoring with dualisable modules commutes with limits, we get 
    
    $$\mathbf{Nuc}^{\infty}(\mathcal{Z})\otimes_{\mathbf{Nuc}^{\infty}(\mathcal{Y})}\mathbf{Nuc}^{\infty}(\mathcal{X})\cong\lim_{\mathrm{Spec}(A)\rightarrow\mathcal{Z}}(\mathbf{Nuc}^{\infty}(\mathrm{Spec}(A))\otimes_{\mathbf{Nuc}^{\infty}(\mathcal{Y})}\mathbf{Nuc}^{\infty}(\mathcal{X})).$$
    Now since $\mathcal{Z}\rightarrow\mathcal{Y}$ has quasi-affine diagonal, the morphism $\mathrm{Spec}(A)\rightarrow\mathcal{Z}\rightarrow\mathcal{Y}$ is quasi-affine. By Lemma \ref{lem:tensordiagram}, we get that 
    $$\mathbf{Nuc}^{\infty}(\mathrm{Spec}(A))\otimes_{\mathbf{Nuc}^{\infty}(\mathcal{Y})}\mathbf{Nuc}^{\infty}(\mathcal{X})\cong\mathbf{Nuc}^{\infty}(\mathrm{Spec}(A)\times_{\mathcal{Y}}\mathcal{X}),$$
    and
     $$\mathbf{Nuc}^{\infty}(\mathcal{Z})\otimes_{\mathbf{Nuc}^{\infty}(\mathcal{Y})}\mathbf{Nuc}^{\infty}(\mathcal{X})\cong\lim_{\mathrm{Spec}(A)\rightarrow\mathcal{Z}}\mathbf{Nuc}^{\infty}(\mathrm{Spec}(A)\times_{\mathcal{Y}}\mathcal{X})\cong\mathbf{Nuc}^{\infty}(\mathcal{Z}\times_{\mathcal{Y}}\mathcal{X}).$$
\end{proof}

\begin{remark}
    We had expected that a result akin to Lemma \ref{lem:tensorequiv} would not hold for $\mathbf{QCoh}$ except in the case of, for example, formally affine morphisms. Our reasoning was that $\mathbf{QCoh}(\mathcal{X})$ would rarely be dualisable as a $\mathbf{QCoh}(\mathcal{Y})$-module, and so in the above computation the tensor product would not commute with the limit. However in \cite{soor2025privatecommunicationhesis}/ \cite{soor2025personalcommunication}, Soor has previously proved the following: if $\mathcal{Y}$ and $\mathcal{Z}$ have affine diagonal, and $\mathcal{Z}\rightarrow\mathcal{Y}$ is a morphism such that $\mathcal{Z}$ admits a morphism $\mathrm{Spec}(A)\rightarrow\mathcal{Z}$ of universal $!$-descent, then:
\begin{enumerate}
    \item 
    for any morphism $\mathcal{X}\rightarrow\mathcal{Y}$, we have
    $$\mathbf{QCoh}(\mathcal{X})\otimes_{\mathbf{QCoh}(\mathcal{Y})}\mathbf{QCoh}(\mathcal{Z})\rightarrow\mathbf{QCoh}(\mathcal{X}\times_{\mathcal{Y}}\mathcal{Z})$$
    is an equivalence;
    \item 
    $\mathbf{QCoh}(\mathcal{Z})$ is dualisable as a $\mathbf{QCoh}(\mathcal{Y})$-module.
\end{enumerate}
The proof of the K\"{u}nneth formula works because universal $!$-descent allows one to `turn around' the limit into a colimit, and the tensor product of categories does commute with colimits.
The class of such $\mathcal{Z}$ includes \textit{all quasi-compact and separated schemes} by \cite{soor2024derived}*{Corollary 3.18}. Soor has used this to work out the theory of Fourier-Mukai transforms in this general setting in \cite{soor2025privatecommunicationhesis}/ \cite{soor2025personalcommunication}. More generally, Soor defines so-called \textit{transformable} morphisms of stacks, to be those maps which have affine diagonal, and such that the fibre product with any affine admits a universal $!$-descent morphism from an affine. He then shows that such maps satisfy the Fourier-Mukai equivalence. Transformable maps are shown to include quasi-compact and separated morphisms of schemes. 
\end{remark}

We can now prove an analogue of \cite{hoyois2021categorified}*{Proposition 2.34}.

\begin{lemma}
    Suppose that $f:\mathcal{X}\rightarrow\mathcal{Y}$ is passable. Then $f^{*}:\mathbf{Nuc}^{\infty}(\mathcal{Y})\rightarrow\mathbf{Nuc}^{\infty}(\mathcal{X})$ is rigid.
\end{lemma}

\begin{proof}
    By Lemma \ref{lem:tensorequiv} we have
    $$\mathbf{Nuc^{\infty}}(\mathcal{X})\otimes_{\mathbf{Nuc}^{\infty}(\mathcal{Y})}\mathbf{Nuc}^{\infty}(\mathcal{X})\cong\mathbf{Nuc}^{\infty}(\mathcal{X}\times_{\mathcal{Y}}\mathcal{X}).$$
    Since the diagonal is representable by quasi-affines, by Proposition \ref{prop:cocontqcqs} and Lemma \ref{lem:qnucpush} the functor
    $$\Delta_{*}:\mathbf{Nuc}^{\infty}(X)\rightarrow\mathbf{Nuc^{\infty}}(\mathcal{X})\otimes_{\mathbf{Nuc}^{\infty}(\mathcal{Y})}\mathbf{Nuc}^{\infty}(\mathcal{X})\cong\mathbf{Nuc}^{\infty}(\mathcal{X}\times_{\mathcal{Y}}\mathcal{X}).$$
    satisfies base-change and commutes with colimits. This completes the proof.
\end{proof}

We set up some notation. Let \(X\) be a quasi-compact, quasi-separated scheme. For a closed subspace \(Y \subseteq X\) with quasi-compact open complement \(i:U\rightarrow X\), we define the \(\infty\)-category \(\mathbf{Nuc}^\infty(X: Y)\) of \textit{nuclear modules with support in \(Y\)} as the homotopy fibre of the restriction \(\mathbf{Nuc}^\infty(X) \to \mathbf{Nuc}^\infty(U)\). 
By Corollary \ref{cor:openimmersembedding}, the functor $i_{*}:\mathbf{Nuc}^{\infty}(U)\rightarrow\mathbf{Nuc}^{\infty}(X)$ is fully faithful. Moreover this is an internal left adjoint by Proposition \ref{prop:cocontqcqs}. In particular, \(\mathbf{Nuc}^\infty(X:Y)\) is dualisable. In a very precise sense here we are really specifying $U$. Then $\mathbf{Nuc}^\infty(X:Y)$ is just notation.

The following is clear.

\begin{proposition}
    Let $U\subset V\subset X$ be open immersions with everything qcqs, and with closed complements $Y$ and $Z$ respectively. Then all squares in the following rectangle are pullbacks. 

    \begin{displaymath}
    \xymatrix{
    \mathbf{QCoh}(X:Y)\ar[d]\ar[r] & \mathbf{QCoh}(V:V\cap Y)\ar[d]\ar[r] & 0\ar[d]\\
    \mathbf{QCoh}(X)\ar[r] & \mathbf{QCoh}(V)\ar[r] & \mathbf{QCoh}(U).
    }
\end{displaymath}
In particular, we have $\mathbf{QCoh}(X:Z)\cong\mathrm{Fib}( \mathbf{QCoh}(X:Y)\rightarrow\mathbf{QCoh}(V:V\cap Y))$. In the nuclear setting, all squares in the following rectangle are pullbacks. 
   \begin{displaymath}
    \xymatrix{
    \mathbf{Nuc}^{\infty}(X:Y)\ar[d]\ar[r] & \mathbf{Nuc}^{\infty}(V:V\cap Y)\ar[d]\ar[r] & 0\ar[d]\\
    \mathbf{Nuc}^{\infty}(X)\ar[r] & \mathbf{Nuc}^{\infty}(V)\ar[r] & \mathbf{Nuc}^{\infty}(U).
    }
\end{displaymath}
In particular, we have $\mathbf{Nuc}^{\infty}(X:Z)\cong\mathrm{Fib}( \mathbf{Nuc}^{\infty}(X:Y)\rightarrow\mathbf{Nuc}^{\infty}(V:V\cap Y))$.
\end{proposition}

\begin{proposition}
    Let $U\subset X$ and $V\subset X$ be open immersions with closed complements $Y$ and $Z$ respectively. Then 
     we have a fibre product diagram
\begin{displaymath}
    \xymatrix{
    \mathbf{QCoh}(U\cup V)\ar[d]\ar[r] & \mathbf{QCoh}(U)\ar[d]\\
    \mathbf{QCoh}(V)\ar[r] & \mathbf{QCoh}(U\cap V),
    }
\end{displaymath}
in which all maps are localisations. In particular, we have
$$\mathbf{QCoh}(U\cup V:V\cap Y)=\mathbf{QCoh}(U\cup V:(U\cup V)\cap Y)\cong\mathbf{QCoh}(V:Y\cap V)$$. In the nuclear case, 
     we have a fibre product diagram
\begin{displaymath}
    \xymatrix{
    \mathbf{Nuc}^{\infty}(U\cup V)\ar[d]\ar[r] & \mathbf{Nuc}^{\infty}(U)\ar[d]\\
    \mathbf{Nuc}^{\infty}(V)\ar[r] & \mathbf{Nuc}^{\infty}(U\cap V).
    }
\end{displaymath}
in which all maps are localisations. In particular, we have
$$\mathbf{Nuc}^{\infty}(U\cup V:V\cap Y)=\mathbf{QCoh}(U\cup V:(U\cup V)\cap Y)\cong\mathbf{Nuc}^{\infty}(V:Y\cap V)$$. 
\end{proposition}

\begin{proof}
    The fibre product diagram comes from descent. The last claim follows from the fact that the horizontal maps have isomorphic fibres, since the diagram is a pullback.
\end{proof}

The following is a version of \cite{andreychev2023k}*{Lemma 5.11}, and immediately follows from the propositions.

\begin{lemma}\label{lem:nuc-Verdier}
Let \(X\) be a quasi-compact, quasi-separated scheme, and \(U \to X\) and \(V \to X\) be quasi-compact open immersions with closed complements $Y$ and $Z$ respectively. Then we have a commutative diagram of Verdier localisations sequences

\begin{displaymath}
    \xymatrix{
    \mathbf{QCoh}(X:Y) \ar[r] & \mathbf{QCoh}(X) \ar[r] & \mathbf{QCoh}(U) \\
\mathbf{QCoh}(X: Y \cap Z) \ar[u]\ar[r] & \mathbf{QCoh}(X:Z) \ar[u]\ar[r] & \mathbf{QCoh}(U: U \cap Z)\ar[u].
    }
\end{displaymath}

\noindent
In the nuclear setting we have a commutative diagram of Verdier localisation sequences:

\begin{displaymath}
    \xymatrix{
    \mathbf{Nuc}^\infty(X:Y) \ar[r] & \mathbf{Nuc}^\infty(X) \ar[r] & \mathbf{Nuc}^\infty(U) \\
\mathbf{Nuc}^\infty(X: Y \cap Z) \ar[u]\ar[r] & \mathbf{Nuc}^\infty(X:Z) \ar[u]\ar[r] & \mathbf{Nuc}^\infty(U: U \cap Z)\ar[u].
    }
\end{displaymath}

\end{lemma}
\comment{
\begin{proof}
We prove the $\mathbf{QCoh}$-case. Everything passes to the nuclear case. 

First observe that we have a fibre product 
\begin{displaymath}
    \xymatrix{
    \mathbf{QCoh}(U\cup V)\ar[d]\ar[r] & \mathbf{QCoh}(U)\ar[d]\\
    \mathbf{QCoh}(V)\ar[r] & \mathbf{QCoh}(U\cap V).
    }
\end{displaymath}
in which all maps are localisations. Therefore the top and bottom maps have isomorphic kernels. This gives $\mathbf{QCoh}(U\cup V:V\cap Y)=\mathbf{QCoh}(U\cup V:(U\cup V)\cap Y)\cong\mathbf{QCoh}(V:Y\cap V)$. Notice, moreover, that the following diagram has all squares and the entire rectangle being pullbacks:

\begin{displaymath}
    \xymatrix{
    \mathbf{QCoh}(X:Y)\ar[d]\ar[r] & \mathbf{QCoh}(V:V\cap Y)\ar[d]\ar[r] & )\ar[d]\\
    \mathbf{QCoh}(X)\ar[r] & \mathbf{QCoh}(V)\ar[r] & \mathbf{QCoh}(U).
    }
\end{displaymath}

Now consider the Verdier sequence
\begin{displaymath}
    \xymatrix{
    \mathbf{QCoh}(X:Y\cap Z)\ar[r] &\mathbf{QCoh}(X)\ar[r] &\mathbf{QCoh}(U\cup V)
    }
\end{displaymath}

Therefore we have a commutative diagram with Verdier sequences in the rows and columns:
\begin{displaymath}
    \xymatrix{
    \mathbf{QCoh}(U\cup V:Y\cap V)\ar[r] & \mathbf{QCoh}(V)\ar[r] &\mathbf{QCoh}(V\cap U)\\
    \mathbf{QCoh}(X:Y)\ar[r]\ar[u] & \mathbf{QCoh}(X)\ar[u]\ar[r] & \mathbf{QCoh}(U)\ar[u]\\
    K\ar[u]\ar[r] & \mathbf{QCoh}(X:Z)\ar[r]\ar[u] & \mathbf{QCoh}(U:U\cap Z)\ar[u].
    }
\end{displaymath}

Setting \(e \colon U \to X\) as the open immersion \(U \subset X\), as in Equation \ref{eq:preNis}, we have a commuting diagram with Verdier sequences in the rows 

\begin{equation}\label{eq:Verdier-QCoh}
\begin{tikzcd}
\mathbf{Nuc}^{\infty}(X:Y) \arrow{r}{} & \mathbf{Nuc}^{\infty}(X) \arrow{r}{e^*} & \mathbf{Nuc}^{\infty}(U) \\
\mathbf{Nuc}^{\infty}(X: Y \cap Z) \arrow{r}{} \arrow{u}{} & \mathbf{Nuc}^{\infty}(X:Z) \arrow{r}{e_{\vert}^*} \arrow{u}{} & \mathbf{Nuc}^{\infty}(U: U \cap Z) \arrow{u}{}
\end{tikzcd}
\end{equation} where the vertical arrows are inclusions. Here we have used that \(e^*\) has a fully faithful, colimit-preserving right adjoint \(e_* \colon \mathbf{Nuc}^{\infty}(U) \to \mathbf{Nuc}^{\infty}(X)\), so that we indeed have Verdier sequences in the rows. 

\comment{
In the rigid affinoid case, to see that \(e\) induces an adjoint pair \((e^*,e_*)\) with \(e_*\) fully faithful, since the properties of being quasi-compact and an open immersion are local on the target, we may reduce to the case where \(e\) is a rational localisation \(A \to B\)  rational localisation of affinoids (see \cite{soor2024derived}*{Lemma 3.3} for a more elaborate and general argument). Then \(e\) is in particular a homotopy epimorphism. Since \((-)^\rig\) is strongly monoidal, \(A^\rig \to B^\rig\) is a homotopy epimorphism, so that the forgetful functor \(e_* \colon \mathbf{Mod}_{B^\rig}(\bD(K)) \to \mathbf{Mod}_{A^\rig}(\bD(K))\) is fully faithful, so that we get a Verdier sequence as in \ref{eq:Verdier-QCoh}, but with \(\bD^{\mathrm{nuc}}(-)\) in place of \(\bD(-)\). Taking rigidifications completes the proof.  
}
\end{proof}
}

\subsection{Nisnevich descent for $K$-Theory}

Let us now specialise to working in a nuclear context $(\mathbf{A},\tau|_{\mathbf{A}}),$ where $\tau$ is contained in the (descendable) \'{e}tale topology. For example, this could be the algebraic context, the dagger analytic context, or the nuclear rigid context. As a consequence of Lemma \ref{lem:Nisnevich-sheaves} and Theorem \ref{thm:Voevodsky-analytic}, a presheaf on \(\mathbf{Sch}\) is a Nisnevich sheaf if and only if it maps a square in $\mathcal{X}_{N}$ to a pullback square in anima. By Lemma \ref{lem:maxime1}, $\mathbf{Nuc}^{\infty}(X)$ is such a sheaf on $(\mathbf{A},\tau|_{\mathbf{A}})$.

The following is a generalisation of Thomas-Trobaugh's result to the analytic setting:

\begin{theorem}\label{thm:K-theory-descent}
In a nuclear context, any $\mathbf{A}$-invariant satisfies Nisnevich descent on the full subcategory \(\mathbf{Sch}^{qcqs}(\mathbf{A},\tau|_{\mathbf{A}})\) of $\mathbf{Sch}$ consisting of qcqs schemes. 
\end{theorem}

\begin{proof}
Consider an elementary Nisnevich square
\[
\begin{tikzcd}
U \times_X V \arrow{r}{} \arrow{d}{} & V \arrow{d}{f} \\
U \arrow{r}{e} & X,
\end{tikzcd}
\] where \(e \colon U \to X\) is an open immersion and \(f \colon V \to X\) is an etale map. Suppose further that all schemes are qcqs. Since by Corollary \ref{cor:nuc-analytic-et-descent} (resp. Theorem \ref{thm:rigidification-formal-desc}), \(\mathbf{Nuc}^\infty(-)\) (resp. \(\mathbf{Nuc}(-)\)) satisfies \'etale descent, it is in particular a sheaf for the Nisnevich topology. Therefore, by Lemma \ref{lem:Nisnevich-sheaves} and Theorem \ref{thm:Voevodsky-analytic}, we get an induced pullback 
\[
\begin{tikzcd}
\mathbf{Nuc}^\infty(X) \arrow{r}{e^*} \arrow{d}{f^*} & \mathbf{Nuc}^\infty(U) \arrow{d}\\
\mathbf{Nuc}^\infty(V) \arrow{r}{} & \mathbf{Nuc}^\infty(U \times_X V),
\end{tikzcd}
\] of \(\infty\)-categories. Note that all terms are rigid categories, the limit is finite, and by base change, all functors are internal left adjoints. Thus this is also a limit in the category of dualisable categories. Recall further that \(e^* \colon \mathbf{Nuc}^\infty(X) \to \mathbf{Nuc}^\infty(U)\) is a localisation with kernel $\mathbf{Nuc}(X:Y)$, where $Y$ is the complement of $U$ in $X$. Applying Lemma \ref{lem:nuc-Verdier}, we get a diagram 

\[
\begin{tikzcd}
\mathbf{Nuc}^\infty(X:Y) \arrow{r}{} \arrow{d}{\simeq} & \mathbf{Nuc}^\infty(X) \arrow{r}{e^*} \arrow{d}{f^*} & \mathbf{Nuc}^\infty(U) \arrow{d}\\
\mathbf{Nuc}^\infty(V: V \times_X Y) \arrow{r}{} & \mathbf{Nuc}^\infty(V) \arrow{r}{} & \mathbf{Nuc}^\infty(U \times_X V),
\end{tikzcd}
\] of \(\infty\)-categories, where the rows are Verdier sequences of dualisable categories. Applying \(E^\cont\) yields the desired result.\qedhere
\end{proof}

\section{Localising invariants of non-Archimedean Banach rings}


Let $R$ be a Banach ring which has a pseudo-uniformiser $\pi$. We assume that the norm is $\pi$-adic, and we put $R_{n}=R\big\slash (\pi^{n}).$
In this section, we describe localising invariants of the dualisable category of nuclear \(R\)-modules from the bornological perspective. Here we will concern ourselves with truncating localising invariants, but in future work, we plan to compute the analytic K-theory of derived formal schemes, analogous to Efimov's result in the condensed setting \cite{efimov2025localizing}.

Our arguments in this section closely follow the results and notation in \cite{fedeli2023topological}. We call a localising invariant \emph{continuous} if \[E^\an(R) \defeq E^\cont(\mathbf{Nuc}(R)) \simeq \varprojlim_n E(R_n),\] where \(E(R_n) = E(\mathbf{Perf}(R_n))\). 

To set things up, consider the full subcategory \(\mathbf{Vec}(R_n) \subset \mathbf{Perf}(R_n)\) of retracts of finitely generated free \(R_n\)-modules. We then get an induced functor from the product \(\prod_{n \in \N} \mathbf{Vec}(R_n) \to \prod_{n \in \N}\mathbf{Perf}(R_n)\), which induces by the stability of the right hand side a functor \[\mathbf{Stab}(\prod_{n \in \N} \mathbf{Vec}(R_n)) \to \prod_{n \in \N}\mathbf{Perf}(R_n)\] in \(\mathbf{Cat}_\infty\). There is also a canonical functor from the lax limit \[\varprojlim_n \mathbf{Perf}(R_n) \to \prod_{n \in \N} \mathbf{Perf}(R_n)\] to the product. Denote by \(\mathbf{laxPerf}_R^b\) the fibre product of these two functors; this is a stable \(\infty\)-category by \cite{fedeli2023topological}*{Lemma 1.40}. Actually, the category \(\mathbf{laxPerf}_R\) is the stabilisation of the additive  \(\infty\)-category \(\mathbf{laxVec}_R^s\) generated by objects in \(\mathbf{laxPerf}_R^b\) that are degreewise connective, have Tor-amplitude \(\leq 0\) and \(\pi_0\)-surjective transition maps, equipped with the split exact structure. Consequently, any additive functor mapping out of \(\mathbf{laxPerf}_R^b\) with stable target is exact.

Recall the internal rigidification functor (\ref{eqref:nonarc-rig}) \(\mathbf{D}(\mathsf{CBorn_R^\tf}) \to \mathbf{Nuc}^{\tf}(R)\) from Section \ref{subsubsec:nonarc-nuc}, associating to a complete, torsion free bornological \(R\)-module the compactoid bornology. Here \(\mathbf{Nuc}_R^\tf\) denotes the smallest stable \(\infty\)-subcategory of \(\mathbf{Nuc}_R\) generated by complete, torsion free nuclear \(R\)-modules. Left Kan-extending along the inclusion \(\bD(\mathsf{CBorn_R^\tf}) \to \mathbf{D}(R)\), we get an induced functor \[(-)^\rig \colon \mathbf{D}(R) \to \mathbf{Nuc}(R)\] that is right adjoint to the inclusion.  

Consider the functor \begin{equation}\label{eqref:R}
\mathbf{R} \colon \mathbf{Ind}(\mathbf{laxPerf}_R^b) \to \mathbf{D}(R) \overset{(-)^\rig}\to \mathbf{Nuc}(R)
\end{equation} that post-composes the internal rigidification functor with the left Kan extension of the inverse limit functor \(\varprojlim \colon \mathbf{laxPerf}_R^b \to \mathbf{D}(R)\) to \(\mathbf{Ind}(\mathbf{laxPerf}_R^b)\).

\begin{lemma}\label{lem:basic-nuc-perfect}
Let \(N\) be a basic nuclear object in \(\bD(R)\). Then there is a basic nuclear object \(N'\) in \(\mathbf{Ind}(\mathbf{laxPerf}_R^b)\) and an equivalence \(N \overset{\simeq}\to \mathbf{R}(N')\). Moreover, the morphism \[\mathbf{Hom}_{\mathbf{Ind}(\mathbf{laxPerf}_R^b)}(N', -) \overset{\mathbf{R}}\to \mathbf{Hom}_{\mathbf{Nuc}(R)}(N, \mathbf{R}(-))\] induced by the equivalence \(N \simeq \mathbf{R}(N')\) is an equivalence of functors.
\end{lemma}

\begin{proof}
The same proof as in \cite{fedeli2023topological}*{1.62} works. 
\end{proof}

Combining Lemma \ref{lem:basic-nuc-perfect} and applying \cite{fedeli2023topological}*{Lemma 1.63} to the functor \(\mathbf{R}\), one obtains formally a fully faithful left adjoint \(\mathbf{L} \colon \mathbf{Nuc}(R) \to \mathbf{Ind}(\mathbf{laxPerf}_R^b)\). By another purely formal argument as in \cite{fedeli2023topological}*{Proposition 1.72}, the essential image of \(\mathbf{L}\) is given by objects in  \(\mathbf{Ind}(\mathbf{laxPerf}_R^b)\) that are nuclear in the sense of \cite{fedeli2023topological}*{Definition 1.47}. Now consider the stable \(\infty\)-category \(\mathbf{Cof}_R^b\) defined in \cite{fedeli2023topological}*{Definition 1.71}; one defines a functor \(F \colon \mathbf{Ind}(\mathbf{laxPerf}_R^b) \to \mathbf{Ind}(\mathbf{Cof}_R^b)\) by extending the functor \(\mathbf{laxPerf}_R^b \to \mathbf{Fun}(\N^\op, \mathbf{laxPerf}_R^b)\) that sends a lax perfect complex \(P\) to the cofibre  \(\mathsf{cofib}(P_j \otimes_{R_j} R_i \to P_i)\) with varying \(i\) and \(j\). 

Summarily, we get a sequence of stable \(\infty\)-categories \begin{equation}\label{eqref:almost-Verdier}\mathbf{Nuc}(R) \overset{L}\to \mathbf{Ind}(\mathbf{laxPerf}_R^b) \overset{F}\to  \mathbf{Ind}(\mathbf{Cof}_R^b) 
\end{equation} 

which need \emph{not} be a Verdier localisation sequence. One however has the following:

\begin{lemma}\cite{fedeli2023topological}*{Corollary 3.16.1}\label{lem:localising-main}
Let \(E\) be a localising invariant that commutes with products of small stable \(\infty\)-categories. Then \(E\) is continuous if and only if it maps the sequence in \eqref{eqref:almost-Verdier} to a fibre-cofibre sequence in spectra. 
\end{lemma}

\begin{proof}
Consider the sequence \(\mathbf{Nuc}(R) \overset{L}\to \mathbf{Ind}(\mathbf{laxPerf}_R^b) \overset{F}\to  \mathbf{Ind}(\mathbf{Cof}_R^b)\). Then applying the universal localising invariant \(\mathcal{U}^{\mathrm{cont}}\), we get the sequence \[\mathcal{U}^{\mathrm{cont}}(\mathbf{Nuc}(R)) \to \mathcal{U}^{\mathrm{cont}}(\prod_{n \in \N} \mathsf{Perf}(R_n)) \overset{1 - \mathrm{pr}}\to \mathcal{U}^{\mathrm{cont}}(\prod_{n \in \N} \mathsf{Perf}(R_n)),\] using \cite{fedeli2023topological}*{Lemma 3.15, Lemma 3.16}. Then \(\mathcal{U}^{\mathrm{cont}}(\mathbf{Nuc}(R))\) is the inverse limit \(\varprojlim_{n}\mathcal{U}^{\mathrm{cont}}(\mathbf{Perf}(R_n))\) if and only if it is the fibre of \(\mathcal{U}^{\mathrm{cont}}(\prod_{n \in \N} \mathbf{Perf}(R_n)) \overset{1 - \mathrm{pr}}\to \mathcal{U}^{\mathrm{cont}}(\prod_{n \in \N} \mathbf{Perf}(R_n))\), as required. 
\end{proof}

We shall now use Lemma \ref{lem:localising-main} to compute several localising invariants. Recall that a localising invariant \(E \colon \mathsf{Pr}_{st}^{L, \mathrm{dual}}  \to \mathbf{D}\) is said to be \emph{truncating} if \(E(R) \simeq E(\pi_0(R))\) for any connective \(\mathbb{E}_1\)-ring spectrum. 

\begin{remark}
Note that being truncating in this sense does not imply that the associated analytic invariants \(E^\an(R) = E^\cont(\mathbf{Nuc}(R))\) are truncating, as the embedding from ring spectra to dualisable categories does not necessarily remember the topology.  In usual analytic cyclic homology theories (\cite{Cortinas-Meyer-Mukherjee:NAHA}), one defines extensions of bornological \(\mathcal{O}_K\)-algebras \(A \to B\) whose kernels are \emph{analytically nilpotent}. And for such extensions, the authors prove that the induced morphism \[\mathbb{HA}(A) \overset{\simeq}\to \mathbb{HA}(B)\] is a weak equivalence in \(\bD(K)\). We have not investigated whether an analytically nilpotent extension \(A \to B\) of bornological \(\mathcal{O}_K\)-algebras induces an equivalence \(E(A) \to E(B)\), though we believe this is plausible for homotopy invariant theories. 
\end{remark}

In our setting, it suffices to work with analytically nilpotent extensions of categories:

\begin{definition}\label{def:analytic-nilpotent-cat}
    An additive functor \(F \colon A \to B\) between additive \(\infty\)-categories is said to be \emph{nilpotent} if (1) \(F\) is essentially surjective; (2) for all objects \(V\), \(W \in A\), the map \(\mathbf{Hom}_A(V,W) \to \mathbf{Hom}_B(F(V), F(W))\) is \(\pi_0\)-surjective; (3) there is an \(n\) such that for any collection \(f_1, \dotsc, f_n\) of \(n\)-composable morphisms in \(A\), if \(F(f_i)\) are the zero morphism in \(B\), then the composition \(f_1 \circ \cdots \circ f_n\) vanishes in \(A\). 
\end{definition}

Let \(\mathbf{Mod}_A = \mathsf{Ind}(\mathsf{Stab}(A))\) so that \(E(A) \defeq E(\mathbf{Mod}(A)^\omega) = E(\mathsf{Stab}(A))\).

\begin{lemma}\cite{elmanto2022nilpotent}*{4.2.1}\label{lem:elmanto}
    Let \(A \to B\) be a nilpotent extension of additive \(\infty\)-categories. Suppose \(E\) is a truncating localising invariant. Then the induced map \(E(A) \to E(B)\) is an equivalence. 
\end{lemma}

This is used to prove the following:

\begin{theorem}\label{thm:truncating}
Any truncating localising invariant that commutes with products of small stable \(\infty\)-categories is continuous. 
\end{theorem}

\begin{proof}
    By Lemma \ref{lem:localising-main}, it suffices to show that \[E(\mathbf{Nuc}(R)) \to E(\mathbf{laxPerf}_R^b) \to E(\mathbf{Cof}_R^b)\] is a fibre-cofibre sequence in the target of \(E\). In order to prove this, we set \(A = \mathbf{laxPerf}_R^b\), \(B = \mathbf{Cof}_R^b\) and \(F \colon A \to B\) the morphism previously constructed. This is not yet a nilpotent extension, but taking the quotient \(\mathbf{Mod}_A/ \mathbf{Nuc}(R)\) by the essential image under the fully faithful embedding \(L \colon \mathbf{Nuc}(R) \to \mathbf{Mod}_A\), and observing that \(F\) vanishes on \(\mathbf{Nuc}(R)\), we get an induced additive functor \[\mathbf{Mod}_A/ \mathbf{Nuc}(R) \to \mathbf{Mod}_B.\] By a proof identical to \cite{fedeli2023topological}*{Lemma 3.20}, the quotient on the left hand side above is generated by an additive category \(\bC\). By using the appropriate compact projective generators of \(\bD(R)\), one observes that the proof of \cite{fedeli2023topological}*{Lemma 3.22} goes through, which in conjunction with \cite{fedeli2023topological}*{Corollary 3.19.1} implies that the induces additive functor \(\bC \to B\) is nilpotent, so that by Lemma \ref{lem:elmanto}, we have an equivalence \(E(\bC) \to E(B)\), or equivalently, the induced diagram \[E(\mathbf{Nuc}(R)) \to E(A) \to E(B)\] is a fibre-cofibre sequence.
\end{proof}

The main example of a localising invariant commuting with products of small stable \(\infty\)-categories is the universal localising invariant and \(K\)-theory. To get a product-preserving \emph{truncating} localising invariant, we first recall\footnote{The construction of this invariant is due to Alexander Efimov, and we have learned of its existence through Maxime Ramzi} that there is a localising invariant \(TC^{\mathrm{mot}}\) that preserves products of small stable \(\infty\)-categories and factorises the cyclotomic trace \(K \to TC^{\mathrm{mot}} \to TC\) such that the map \(TC^{\mathrm{mot}} \to TC\) is an equivalence on connective ring spectra.

\begin{corollary}
    The fibre of the map \(K \to TC^{\mathrm{mot}}\) is continuous.
\end{corollary}

\begin{proof}
    Since \(K\) and \(TC^{\mathrm{cont}}\) preserve products of small stable \(\infty\)-categories, so does the fibre. To see that it is truncating, let \(A\) be a connective ring spectrum. Then using that \(TC^{\mathrm{mot}}\) and \(TC\) agree on \(A\), we see that the fibre is the infinitesimal \(K\)-theory of \(A\). Since \(K^{\mathrm{inf}}\) is truncating by \cite{land2019k}, the result follows.  
\end{proof}


\begin{remark}
    Note that in this section, we have not considered the relative version of the computation of localising invariants of formal schemes. More precisely, if \(B\) is a derived admissible affinoid \(R\)-algebra  (ie, an object of \(\bA_R^{form}\)), then we may no longer view \(B\) as a Banach \(R\)-module, but rather compactify its bornology to make it nuclear. In this case, the relevant category whose localising invariants one ought to consider is \(\mathbf{Nuc}_{B^\rig}(\bD(R)) = \mathbf{Nuc}(\mathbf{Mod}_{B^\rig}(\bD(R)))\). This category should be motivically equivalent to the category of nuclear \(B\)-modules taken internal to the derived category of solid \(R\)-modules; the \(K\)-theory of the latter category has recently been computed by Efimov \cite{efimov2025localizing}. We will take up the comparison between the two nuclear categories in a forthcoming article.  
\end{remark}

\section{A Grothendieck-Riemann-Roch Theorem}

Let \(R\) be a Banach ring and \(f\colon X \to Y\) be a quasi-separated morphism of quasi-compact, quasi-separated derived dagger \(R\)-analytic spaces. The goal of this section is to prove the following version of the Grothendieck-Riemann-Roch Theorem. 

\begin{theorem}\label{thm:GRR}
We have the following commuting diagram of spectra
\[
\begin{tikzcd}
K^\an(X) \arrow{r}{\mathrm{ch}} \arrow{d}{f_{!}} & TC^{- \text{ }\an}(X) \arrow{d}{f_*} \\
K^\an(Y) \arrow{r}{\mathrm{ch}} & TC^{- \text{ }\an}(Y),
\end{tikzcd}
\] where \(\mathrm{ch}\) is the Chern character and \(HC^{- \text{ }\an}\) denotes the continuous extension of negative cyclic homology. 
\end{theorem}

More generally, this works for any nuclear context $(\mathbf{A},\tau,\mathbf{P})$, with $\tau$ contained in the descendable topology.

To prove Theorem \ref{thm:GRR}, we use the general machinery developed in \cite{hoyois2021categorified}. Let \(\mathcal{X}\) be an ambient symmetric monoidal \((\infty,2)\)-category, and \(\bC\) a commutative algebra object in \(\mathcal{X}\). We may then construct the symmetric monoidal \((\infty,2)\)-category \(\mathbf{Mod}_{\bC}(\mathcal{X})\) of \(\bC\)-module objects in \(\mathcal{X}\). We say that an object \(X \in \mathbf{Mod}_{\bC}(\mathcal{X})\) is \textit{dualisable} if there is an object \(X^\vee\), an evaluation map \(\ev \colon X \otimes X^\vee \to 1\) and a co-evaluation map \(\mathrm{coev} \colon 1 \to X^\vee \otimes X\) such that the compositions 
\begin{align*}
X \overset{\mathrm{ev} \otimes 1}\to X \otimes X^\vee \otimes X \overset{1 \otimes \mathrm{coev}}\to X \\
X^\vee \overset{1 \otimes \mathrm{ev}}\to X^\vee \otimes X \otimes X^\vee \overset{\mathrm{coev} \otimes 1}\to X^\vee
\end{align*} are equivalence to the identity via an invertible \(2\)-morphism. If an object \(X \in \mathbf{Mod}_{\bC}(\mathcal{X})\) is dualisable, its dual is necessarily given by \(\Hom(X, 1)\). Let \(\mathbf{Mod}_{\bC}^{\mathrm{dual}}(\mathcal{X})\) be the \((\infty,1)\)-category of dualisable objects in \(\mathbf{Mod}_{\bC}(\mathcal{X})\) with left adjoint functors as morphisms. 

\comment{\begin{remark}
Note that it is not a priori true that \(\mathbf{Mod}_{\bC}^{\mathrm{dual}}(\mathcal{X})\) is equivalent to the \((\infty,2)\)-category \(\mathbf{Mod}_{\bC}(\mathcal{X}^{\mathrm{dual}})\) of \(\bC\)-module objects that are dualisable in the underlying ambient \((\infty,2)\)-category \(\mathcal{X}\). \edit{Clarify this, possibly with a counter-example.} This will however be the case when the category \(\mathbf{C}\) is rigid. 
\end{remark}}

Now let \(X \in \mathbf{Mod}_{\bC}^{\mathrm{dual}}(\mathcal{X})\) and \(f \colon X \to X\) an endomorphism. We define the \emph{trace} \[\mathrm{Tr}(X,-) \colon \mathbf{End}_{\mathbf{Mod}_{\bC}^{\mathrm{dual}}(\mathcal{X})}(X) \to \mathbf{End}_{\mathbf{Mod}_{\bC}(\mathcal{X})}(\mathbf{1})\] of \(f\) as the composite \[ \ev \circ (f \otimes \mathbf{1}) \circ \mathrm{coev} \colon \mathbf{1} \to X \otimes X^\vee \to X^\vee \otimes X \to \mathbf{1}.\]  In particular, if \(f = 1_X\) we have by definition the \emph{dimension} \(\mathrm{dim}(X) = \mathrm{Tr}(X,1_X)\). 

Now suppose \(\bC\) is dualisable in \(\mathcal{X}\), so that \(\bC \in \mathbf{Mod}_{\bC}^{\mathrm{dual}}(\mathcal{X})\). For such a \(\bC\), consider the loop space \[\mathcal{L}(\bC) \defeq \bC \otimes_{\bC \otimes \bC} \bC \simeq S^1 \otimes \bC = \mathrm{colim}_{S^1} \bC\] which is again a commutative algebra object in \(\mathcal{X}\) with a canonical \(S^1\)-action. Using the last identification, we obtain by adjunction the symmetric monoidal functors \[S^1 \to \mathbf{Fun}^{\otimes}(\bC, \mathcal{L}(\bC)), \qquad \bC \to \mathbf{Fun}(S^1, \mathcal{L}(\bC)).\] Choosing a base point in \(p \colon \{*\} \to S^1\), the first functor is equivalent to a functor \(p_{\bC} \colon \bC \to \mathcal{L}(S^1)\) by tensoring with \(p\), and a natural equivalence \(p_{\bC} \to p_{\bC}\) called the \textit{monodromy automorphism}. Applying the functor \(\mathbf{Mod}\), we get a symmetric monoidal functor \[S^1 \to \mathbf{Fun}^{\otimes}(\mathbf{Mod}_{\bC}(\mathcal{X}), \mathbf{Mod}_{\mathcal{L}(\bC)}(\mathcal{X}),\] which restricts to dualisable objects \(\mathbf{Mod}^{\mathrm{dual}}\) in \(\mathcal{X}\). By adjunction, we obtain an induced symmetric monoidal functor \[\mathbf{Mod}_{\mathbf{C}}^{\mathrm{dual}}(\mathcal{X}) \to \mathbf{Fun}(S^1, \mathbf{Mod}_{\mathcal{L}(\mathbf{C})}^{\mathrm{dual}}(\mathcal{X})).\] The \textit{categorified Chern character} is defined as the composite 

\begin{multline*}
\mathrm{Ch} \colon \mathbf{Mod}_{\bC}^{\mathrm{dual}}(\mathcal{X}) \to \mathbf{Fun}(S^1, \mathbf{Mod}_{\mathcal{L}(\bC)}^{\mathrm{dual}}(\mathcal{X})) \\ \simeq \mathbf{Aut}_{\mathbf{Mod}_{\mathcal{L}(\bC)}^{\mathrm{dual}}(\mathcal{X})}(\mathcal{L}(\bC)) \overset{\mathrm{Tr}(\mathcal{L}(\bC),-)}\to \mathbf{Aut}_{\mathbf{Mod}_{\mathcal{L}(\bC)}(\mathcal{X})}(\mathbf{1}) \simeq \mathcal{L}(\bC).
\end{multline*}

Equipping the source with the trivial \(S^1\)-action, the domain of the trace map with the diagonal action and the target with the trivial \(S^1\)-action, by \cite{hoyois2021categorified}*{Proposition 3.4}, this Chern character is an \(S^1\)-equivariant symmetric monoidal functor, so it factorises through the \(S^1\)-fixed points on \(\mathbf{End}_{\mathbf{Mod}_{\mathcal{L}(\bC)}(\mathcal{X})}(\mathbf{1})\)  to yield a functor \[\mathrm{Ch} \colon  \mathbf{Mod}_{\bC}^{\mathrm{dual}}(\mathcal{X}) \to \mathcal{L}(\bC)^{S^1}\] called the \emph{\(S^1\)-equivariant Chern character}. 

Finally, consider a symmetric monoidal functor \(f \colon \bD \to \bC\) between dualisable symmetric monoidal \(\infty\)-categories. We call such a functor \emph{rigid} if it is right adjointable, and if the multiplication map \(\bD \otimes_{\bC} \bD \to \bD\) has a right adjoint. The main result of \cite{hoyois2021categorified} then states the following:

\begin{theorem}\label{thm:GRR-general}\cite{hoyois2021categorified}*{Theorem 4.3}
Let \(f \colon \bD \to \bC\) be a rigid symmetric monoidal functor between dualisable \((\infty,1)\)-categories relative to an \((\infty, 2)\)-category \(\mathcal{X}\). We then have the following diagram:
\[
\begin{tikzcd}
\mathbf{Mod}_{\bC}^{\mathrm{dual}}(\mathcal{X}) \arrow{r}{\mathrm{Ch}} \arrow{d}{f_*} & \mathcal{L}(\bC)^{S^1} \arrow{d}{\mathcal{L}(f)^R}\\
\mathbf{Mod}_{\bD}^{\mathrm{dual}}(\mathcal{X}) \arrow{r}{\mathrm{Ch}}  & \mathcal{L}(\bD)^{S^1}
\end{tikzcd}
\] of \((\infty,1)\)-categories that commutes up to an invertible \(2\)-morphism. 
\end{theorem}

In the situation of Theorem \ref{thm:GRR-general}, let \(\mathcal{T} \in \mathbf{Mod}_{\bC}^{\mathrm{dual}}(\mathcal{X})\) and \(\mathcal{T}' \in \mathbf{Mod}_{\bD}^{\mathrm{dual}}(\mathcal{X})\). Then by the rigidity hypothesis, \(f \colon \bD \to \bC\) admits a right adjoint functor \(f_* \colon \bC \to \bD\). Let \(g \colon f_*(\mathcal{T}) \to \mathcal{T}'\) be a morphism in \(\mathbf{Mod}_{\bD}^{\mathrm{dual}}(\mathcal{X})\). We may then define morphisms of anima 

\[\mathbf{Mod}_{\bC}^{\mathrm{dual}}(\mathcal{X})(\bC, \mathcal{T}) \to \mathbf{Mod}_{\bD}^{\mathrm{dual}}(\mathcal{X})(\bC, \mathcal{T}'), \varphi \mapsto \bD \to f_*\bC \overset{f_*(\varphi)}\to f_* \mathcal{T} \overset{g}\to \mathcal{T}',\] and likewise \(\mathcal{L}(\bC)^{S^1}(\mathbf{1}, \mathrm{Ch}(\mathcal{T})) \to \mathcal{L}(\bC)^{S^1}(\mathbf{1}, \mathrm{Ch}(\mathcal{T}'))\).

\begin{corollary}\label{thm:GRR-spatial}
For \(\mathcal{T} \in \mathbf{Mod}_{\bC}^{\mathrm{dual}}(\mathcal{X})\) and \(\mathcal{T}' \in \mathbf{Mod}_{\bD}^{\mathrm{dual}}(\mathcal{X})\), we have the following commuting diagram of anima 

\[
\begin{tikzcd}
\mathbf{Mod}_{\bC}^{\mathrm{dual}}(\mathcal{X})(\bC, \mathcal{T}) \arrow{r}{\mathrm{Ch}} \arrow{d}{} & \mathcal{L}(\bC)^{S^1}(\mathbf{1}, \mathrm{Ch}(\mathcal{T})) \arrow{d}{}\\
\mathbf{Mod}_{\bD}^{\mathrm{dual}}(\mathcal{X})(\bD, \mathcal{T}') \arrow{r}{\mathrm{Ch}}  & \mathcal{L}(\bD)^{S^1}(\mathbf{1}, \mathrm{Ch}(\mathcal{T}')).
\end{tikzcd}
\]

\end{corollary}



\subsection{Grothendieck-Riemann-Roch for schemes}

We now specialise Theorem \ref{thm:GRR-general} and Corollary \ref{thm:GRR-spatial} to our situation. In the setting above, let \(\mathcal{X} = \mathbf{Pr}_{st}^{L}\). Let \(f \colon X\to Y\) be a passable morphism of stacks. For example, $f$ could be a qcqs morphism of qcqs schemes. In particular, the functor $f^{*}:\mathbf{Nuc}^{\infty}(Y)\rightarrow\mathbf{Nuc}^{\infty}(X)$ is rigid.

We then obtain an induced \((\infty,2)\)-functor \[f^* \colon \mathbf{Mod}_{\mathbf{Nuc}^\infty(Y)}(\mathbf{Pr}_{st}^{L}) \to \mathbf{Mod}_{\mathbf{Nuc}^\infty(X)}(\mathbf{Pr}_{st}^{L}), \quad M \mapsto \mathbf{Nuc}^\infty(X) \otimes_{\mathbf{Nuc}^\infty(Y)} M\] of symmetric monoidal \((\infty,2)\)-categories, which is left adjoint to the forgetful functor \(f_* \colon \mathbf{Mod}_{\mathbf{Nuc}^\infty(X)}(\mathbf{Pr}_{st}^{L}) \to \mathbf{Mod}_{\mathbf{Nuc}^\infty(Y)}(\mathbf{Pr}_{st}^{L})\). This is a symmetric monoidal ambidexterous adjunction by \cite{hoyois2021categorified}*{Proposition 2.21}, so that it restricts to an adjunction \[f^* \colon \mathbf{Mod}_{\mathbf{Nuc}^\infty(Y)}^{\mathrm{dual}}(\mathbf{Pr}_{st}^{L}) \leftrightarrows \mathbf{Mod}_{\mathbf{Nuc}^\infty(X)}^{\mathrm{dual}}(\mathbf{Pr}_{st}^{L}) \colon f_*\] between the \((\infty,1)\)-subcategories of dualisable objects and left adjoint functors.  Theorem \ref{thm:GRR-general} now specialises to the following:



\comment{\begin{theorem}\label{thm:quasi-coherent-rigid}
Let \(X\) be a \(k\)-analytic space. Then \(\mathbf{D}(X)\) and \(\mathbf{Nuc}(X)\) are rigid categories. 
\end{theorem}

\begin{proof}
\textbf{Sketch.} 
We first note that the category \(\mathbf{D}(X)\) is compactly generated and presentable. Furthermore, the unit object in \(\mathbf{D}(X)\) is compact. Finally, the compact objects in \(\mathbf{D}(X)\) admit duals. The result now follows from \cite{hoyois2021categorified}*{Proposition 2.20}. 
\end{proof}
}


\begin{corollary}\label{thm:categorified-GRR}
The following diagram 
\[
\begin{tikzcd}
\mathbf{Mod}_{\mathbf{Nuc}^\infty(X)}^{\mathrm{dual}}(\mathbf{Pr}_{st}^L) \arrow{r}{\mathrm{Ch}} \arrow{d}{f_*} & \mathcal{L}(\mathbf{Nuc}^\infty(X))^{S^1} \arrow{d}{\mathcal{L}(f)^R}\\
\mathbf{Mod}_{\mathbf{Nuc}^\infty(Y)}^{\mathrm{dual}}(\mathbf{Pr}_{st}^L) \arrow{r}{\mathrm{Ch}}  & \mathcal{L}(\mathbf{Nuc}^\infty(Y))^{S^1}
\end{tikzcd}
\] of \((\infty,1)\)-categories commutes.
\end{corollary}

\subsection{Noncommutative linear motives}

To use Corollary \ref{thm:categorified-GRR} in the context of \(K\)-theory, we first generalise the notion of noncommutative \(\bC\)-linear motives from compactly generated, presentable, stable \(\infty\)-categories to dualisable \(\infty\)-categories. To this end, let \(\bC\) be a fixed rigid category.  We call a diagram \(A \to B \to C\) in \(\mathbf{Mod}_{\bC}^{\mathrm{dual}}(\mathbf{Pr}^{L})\) a \emph{\(\bC\)-Verdier sequence} if its image under the forgetful functor \(\mathbf{Mod}_{\bC}^{\mathrm{dual}}(\mathbf{Pr}^{L}) \to \mathbf{Pr}^{L, \mathrm{dual}}\) is a Verdier sequence. A \emph{\(\bC\)-localising invariant} valued in a stable \(\infty\)-category \(\bD\) is a functor \(F \colon \mathbf{Mod}_{\bC}^{\mathrm{dual}}(\mathbf{Pr}^{L}) \to \bD\) preserving final objects and \(\bC\)-Verdier sequences. Finally, a \emph{finitary \(\bC\)-localising invariant} is a \(\bC\)-localising invariant that commutes with filtered colimits.

\begin{remark}
Notice that if \(F \colon \mathbf{Mod}_{\bC}^{\mathrm{dual}}(\mathbf{Pr}_{st}^{L}) \to \bD\) is a \(\bC\)-localising invariant for a presentably symmetric monoidal category \(\bC\), then \(F( - \otimes \bC) \colon \mathbf{Pr}^{L, \mathrm{dual}} \to \bD\) is a localising invariant. This is because tensoring by \(\bC\) preserves Verdier sequences by \cite{efimov2024k}*{Theorem 2.2}.
\end{remark}

For a presentable \(\infty\)-category \(\bC\), we define \(\mathbf{Ind}(\bC)\) as the full subcategory of \(\mathbf{Fun}(\bC^\op, \mathsf{An})\) generated by representable functors under small filtered colimits in \(\bC\). Let \(j \colon \bC \to \mathbf{Ind}(\bC) \subseteq \mathbf{Fun}(\bC^\op, \mathbf{An})\) denote the Yoneda embedding. By \cite{ramzi2024dualizable}, the \((\infty,1)\)-category \(\mathbf{Mod}_\bC^{\mathrm{dual}}(\mathbf{Pr}_{st}^{L})\) is presentable, so we may consider the following composition:

\begin{multline*}
\psi \colon \mathbf{Mod}_\bC^{\mathrm{dual}}(\mathbf{Pr}_{st}^{L}) \overset{j}\to \mathbf{Fun}((\mathbf{Mod}_\bC^{\mathrm{dual}}(\mathbf{Pr}_{st}^{L}))^\op, \mathsf{An}) \\ \overset{\mathsf{Sp}}\to \mathbf{Fun}((\mathbf{Mod}_\bC^{\mathrm{dual}}(\mathbf{Pr}_{st}^{L}))^\op, \mathsf{Sp}).
\end{multline*}

Now let \(S_{\mathrm{loc}}\) be the class of morphisms in \(\mathbf{Fun}((\mathbf{Mod}_\bC(\mathbf{Pr}_{st}^{L, \mathrm{dual}}))^\op, \mathsf{Sp})\) that are of the form \[0 \to \Sigma^n \psi(0), \quad \Sigma^n(\mathsf{cone}(\psi(A) \to \psi(B))) \to \Sigma^n(\psi(C)),\] where \(A \to B \to C\) is a \(\bC\)-Verdier sequence, \(n \leq 0\), and \(colim_{i \in I} \psi(C_i) \to \psi (colim C_i)\) for filtered systems \((C_i)_{i \in I}\). 

\begin{definition}\label{def:NCMot}
For a rigid category \(\bC\), we define by \(\mathbf{NCMot}(\bC)\) the full subcategory of \(\mathbf{Fun}((\mathbf{Mod}_\bC^\mathrm{dual}(\mathbf{Pr}^{L}))^\op, \mathsf{Sp})\) generated by \(S_{\mathrm{loc}}\)-local objects. We call this is the \(\infty\)-category of \emph{noncommutative \(\bC\)-linear motives}.
\end{definition}

\begin{proposition}\label{prop:linear-localising}
Let \(\bC\) be a rigid category. Then \(\mathbf{NCMot}(\bC)\) is a stable, presentable \(\infty\)-category. 
\end{proposition}
\begin{proof}
By the definition of \(S_{\mathrm{loc}}\), the category \(\mathbf{NCMot}(\bC)\) is closed under suspensions. Furthermore, by construction it is closed under all small colimits. Finally, accessibility follows from the general fact that filtered colimit preserving functors between presentable \(\infty\)-categories is accessible. 
\end{proof}

Finally, by the presentability of \(\mathbf{NCMot}(\bC)\), the inclusion \(\mathbf{NCMot}(\bC) \to \mathbf{Fun}((\mathbf{Mod}_\bC^{\mathrm{dual}}(\mathbf{Pr}^{L}))^\op, \mathsf{Sp})\) has a left adjoint \[L \colon \mathbf{Fun}((\mathbf{Mod}_\bC^\mathrm{dual}(\mathbf{Pr}^{L}))^\op, \mathsf{Sp}) \to \mathbf{NCMot}(\bC),\] whose restriction to \(\mathbf{Mod}_\bC^\mathrm{dual}(\mathbf{Pr}^{L})\) via the Yoneda embedding yields a functor \[\mathcal{U} \colon \mathbf{Mod}_\bC^\mathrm{dual}(\mathbf{Pr}^{L}) \to \mathbf{NCMot}(\bC).\]

\begin{theorem}[Corepresentability of \(K\)-theory]\label{thm:universal-loc}
The functor \(\mathcal{U} \colon \mathbf{Mod}_\bC^{\mathrm{dual}}(\mathbf{Pr}^{L}) \to \mathbf{NCMot}(\bC)\) defined above is the universal, finitary \(\bC\)-localising invariant. Furthermore, we have \[\mathbf{NCMot}(\bC)(\mathcal{U}(\bC), \mathcal{U}(\mathbf{A})) \simeq K^\cont(\mathbf{A}).\]
\end{theorem}

\begin{proof}
By \cite{ramzi2024locally}*{Corollary 4.46}, for a rigid category \(\bC\), the forgetful functor  \[\mathbf{Mod}_{\bC}^{\mathrm{dual}} \to \mathbf{Pr}^{L, \mathrm{dual}}\] is right adjoint to extension of scalars \( - \otimes \bC\). Using the rigidity of \(\bC\), one observes that the extension of scalars functor composed with the universal localising invariant \(\mathbf{Pr}^{L, \mathrm{dual}} \to \mathbf{Mod}_{\bC}^{\mathrm{dual}} \to \mathbf{NCMot}(\bC)\) is a localising invariant, and therefore factors uniquely through \(\mathcal{U} \colon \mathbf{Pr}^{L, \mathrm{dual}} \to \mathbf{NCMot}\), yielding a symmetric monoidal, colimit-preserving functor \(\mathbf{NCMot} \to \mathbf{NCMot}(\bC)\). By the presentability of noncommutative motives, we thus have an adjunction \[\mathbf{NCMot}(\bC) \rightleftarrows \mathbf{NCMot},\] so that for a \(\bC\)-linear dualisable category \(\bD\), we have the following equivalences 
\[
\mathbf{NCMot}(\bC)(\mathcal{U}(\bC), \mathcal{U}(\bA)) \simeq \mathbf{NCMot}(\mathsf{Sp}, \mathcal{U}(\bA)) \simeq K^\mathrm{cont}(\bA),
\] where we have used that \(\mathcal{U}\) is symmetric monoidal in the first equivalence, and  \cite{blumberg2013universal}*{Theorem 9.8} for the second equivalence.
\end{proof}

\subsection{The geometric Grothendieck-Riemann-Roch Theorem}

We now specialise the results of the previous section to the rigid category \(\bC = \mathbf{Nuc}^\infty(X)\) for $X$ a qcqs scheme. To simplify notation, we will suppress the ambient symmetric monoidal \(\infty\)-category \(\mathbf{Pr}_{st}^{L}\) of presentable stable \(\infty\)-categories and simply write \(\mathbf{Mod}_{\bC}^{\mathrm{dual}}\) for \(\bC\)-linear dualisable categories.

\begin{theorem}\label{main:GRR}
Let \(\bC = \mathbf{Nuc}^\infty(X)\) for a qcqs \(X \in \mathbf{dAn}_R\). The \(S^1\)-equivariant Chern character \[\mathsf{Ch} \colon \mathbf{Mod}_{\mathbf{C}}^{\mathrm{dual}} \to \mathcal{L}(\bC)^{S^1}\] is a \(\mathbf{C}\)-localising invariant, so it factors through \(\mathbf{C}\)-linear noncommutative motives \(\mathbf{NCMot}(\mathbf{C})\).
\end{theorem}

\begin{proof}
Let \(A \to B \to C\) be a \(\bC\)-Verdier sequence, so that it is a Verdier sequence of dualisable categories. By the same proof as in \cite{hoyois2017higher}*{Theorem 6.9}, the functors \(\mathbf{Mod}_{\mathbf{C}}^{\mathrm{dual}} \to \mathbf{Aut}(\mathcal{L}(\bC))\) and the categorified trace \(\mathbf{Aut}(\mathcal{L}(\bC)) \to \mathcal{L}(\bC)^{S^1}\) preserve zero objects and \(\bC\)-Verdier sequences. Consequently, the composition \(\mathsf{Ch}\) is a \(\bC\)-localising invariant.
\end{proof}

As a consequence, we have the following:

\begin{corollary}\label{cor:NCMot-GRR}
We have the following commuting diagram

\[
\begin{tikzcd}
\mathbf{NCMot}(\mathbf{Nuc}^\infty(X)) \arrow{r}{\mathrm{Ch}} \arrow{d}{f_*} &  \mathcal{L}(\mathbf{Nuc}^\infty(X))^{S^1} \arrow{d}{}  \\
\mathbf{NCMot}(\mathbf{Nuc}^\infty(Y)) \arrow{r}{\mathrm{Ch}} &  \mathcal{L}(\mathbf{Nuc}^\infty(Y))^{S^1}
\end{tikzcd} 
\] of \(\infty\)-categories. 
\end{corollary}

Note that the endomorphisms on the unit object on \(\mathcal{L}(\mathbf{Nuc}^\infty(X))\) is by definition the continuous topological Hochschild homology \(HH^\an(X) \defeq HH^{\mathrm{cont}}(\mathbf{Nuc}^\infty(X))\) of the dualisable category \(\mathbf{Nuc}^\infty(X)\) of infinitely nuclear sheaves. Taking mapping anima in Corollary \ref{cor:NCMot-GRR}, we obtain the following:

\begin{corollary}\label{cor:mapping-anima-mot}
For \(\mathcal{T} \in \mathbf{Mod}_{\mathbf{Nuc}^\infty(X)}^{\mathrm{dual}}(\mathcal{X})\) and \(\mathcal{T}' \in \mathbf{Mod}_{\mathbf{Nuc}^\infty(X)}^{\mathrm{dual}}(\mathcal{X})\), we have the following commuting diagram of anima 

\[
\begin{tikzcd}
\mathbf{NCMot}(\mathbf{Nuc}^\infty(X))(\mathcal{U}(\mathbf{Nuc}^\infty(X)), \mathcal{U}(\mathcal{T})) \arrow{r}{\mathrm{Ch}} \arrow{d}{} & \mathcal{L}(\mathbf{Nuc}^\infty(X))^{S^1}(\mathbf{1}, \mathrm{Ch}(\mathcal{T})) \arrow{d}{}\\
\mathbf{NCMot}(\mathbf{Nuc}^\infty(Y))(\mathcal{U}(\mathbf{Nuc}^\infty(Y)), \mathcal{U}(\mathcal{T}')) \arrow{r}{\mathrm{Ch}}  & \mathcal{L}(\mathbf{Nuc}^\infty(Y))^{S^1}(\mathbf{1}, \mathrm{Ch}(\mathcal{T}')).
\end{tikzcd}
\]
\end{corollary}

\begin{corollary}\label{thm:analytic-GRR}
Let \(f\colon X \to Y\) be a quasi-separated morphism of quasi-compact, quasi-separated  dagger \(R\)-analytic spaces. Then the following diagram of spectra 

\[
\begin{tikzcd}
K^\an(X) \arrow{r}{\mathrm{ch}} \arrow{d}{f_*} & TC^{-,\an}(X) \arrow{d}{} \\
K^\an(Y) \arrow{r}{\mathrm{ch}} & TC^{-,\an}(Y)
\end{tikzcd}
\] commutes.
\end{corollary}

\begin{proof}
Set \(\mathcal{T} = \mathbf{Nuc}^\infty(X)\) and \(\mathcal{T}' = \mathbf{Nuc}^\infty(Y)\) in Corollary \ref{cor:NCMot-GRR}, and use Theorem \ref{thm:universal-loc}. This yields the analytic \(K\)-theory spectra on the left vertical side of the diagram in Corollary \ref{cor:NCMot-GRR}. The right vertical side is \(\Omega \mathcal{L}(\mathbf{Nuc}^\infty(X))^{hS^1}\) and \(\Omega \mathcal{L}(\mathbf{Nuc}^\infty(Y))^{hS^1}\), which are \(HH^\an(X)^{hS^1} = TC^{-,\an}(X)\) and \(HH^\an(Y)^{hS^1} = TC^{-,\an}(Y)\), respectively. 
\end{proof}

This completes the proof of Theorem \ref{thm:GRR}.

\comment{
\subsection{$\mathcal{F}$-descent topologies}

Let $\mathbf{D},\mathbf{E}$ be presentable $(\infty,1)$-categories with $\mathbf{D}$ being small, $\tau$ be a topology on $\mathbf{D}$, and let $\mathcal{F}\in\mathbf{PreShv}(\mathbf{D};\mathbf{E})$ be a $\mathbf{E}$-valued presheaf. When $\mathbf{E}=\mathbf{An}$ we just write $\mathbf{PreShv}(\mathbf{D})$. Note that $\mathcal{F}:\mathbf{D}^{op}\rightarrow\mathbf{E}$ extends by colimits to a functor
$$\mathbf{PreShv}(\mathbf{D})\rightarrow\mathbf{PreShv}(\mathbf{D};\mathbf{E})$$

Let $f:\mathcal{X}\rightarrow\mathcal{Z}$ be a map in $\mathbf{PreShv}(\mathbf{D})$. We say that $f$ \textit{satisfies} $\mathcal{F}$-\textit{descent} if 
the map
$$\mathcal{F}(\mathcal{Z})\rightarrow\lim_{n}\mathcal{F}(\mathcal{X}^{\times_{\mathcal{Z}n}})$$
is an isomorphism.

We say that $f$ \textit{satisfies universal} $\mathcal{F}$-\textit{descent} if for any affine $U\rightarrow\mathcal{Z}$, the map
$$f':U\times_{\mathcal{Z}}\mathcal{X}\rightarrow U$$
satisfies $\mathcal{F}$-descent.

\begin{proposition}
    If $f:\mathcal{X}\rightarrow\mathcal{Z}$ satisfies universal $\mathcal{F}$-descent then for any map $\mathcal{Y}\rightarrow\mathcal{Z}$ the map 
    $$\mathcal{X}\times_{\mathcal{Z}}\mathcal{Y}\rightarrow\mathcal{Y}$$
    satisfies $\mathcal{F}$-descent. In particular $f$ satisfies $\mathcal{F}$-descent.
\end{proposition}

\begin{proof}
    Write $$\mathcal{Y}\cong\mathrm{colim}_{\mathbf{D}_{\big\slash\mathcal{Y}}}U$$
    By universality of colimits,
    $$\mathcal{X}\times_{\mathcal{Z}}\mathcal{Y}\cong\mathrm{colim}_{\mathbf{D}_{\big\slash\mathcal{Y}}}U\times_{\mathcal{Z}}\mathcal{X}$$
    Now we have
    \begin{align*}
        \mathcal{F}(U)&\cong\lim _{n}\mathcal{F}((U\times_{\mathcal{{Z}}}\mathcal{X})^{\times_{U}n})\\
        &\cong\lim _{n}\mathcal{F}(U\times_{\mathcal{{Z}}}\mathcal{X}^{\times_{Z}n})
    \end{align*}
    So
    \begin{align*}
        \mathcal{F}(\mathcal{Y})&\cong\lim_{\mathbf{D}_{\big\slash\mathcal{Y}}}\mathcal{F}(U)\\
&\cong\lim_{\mathbf{D}_{\big\slash\mathcal{Y}}}\lim _{n}\mathcal{F}(U\times_{\mathcal{{Z}}}\mathcal{X}^{\times_{Z}n})\\
        &\cong\lim _{n}\lim_{\mathbf{D}_{\big\slash\mathcal{Y}}}\mathcal{F}(U\times_{\mathcal{{Z}}}\mathcal{X}^{\times_{Z}n})\\
&\cong\lim_{n}\mathcal{F}(\mathcal{Y}\times_{\mathcal{Z}}(\mathcal{X}^{\times_{\mathcal{Z}n}}))\\
&\cong\lim_{n}\mathcal{F}((\mathcal{Y}\times_{\mathcal{Z}}\mathcal{X})^{\times_{\mathcal{Y}n}})
    \end{align*}
    as required.
\end{proof}
}
\comment{
\subsection{Descent for maps of stacks}

\begin{lemma}
Let $\mathbf{D}$ be equipped with a topology $\tau$.
    Let $\mathbf{E}=\mathbf{Cat}^{L}$ be the category of $(\infty,1)$-categories with maps being left adjoint functors, and let $\mathbf{Q}\in\mathbf{PreShv}(\mathbf{D};\mathbf{Cat}^{L})$ satisfy descent for $\tau$. Let $f:\mathcal{X}\rightarrow\mathcal{Y}$ and $g:\mathcal{Y}\rightarrow\mathcal{Z}$ be maps in $\mathbf{Shv}(\mathbf{D})$ with $f$ being an effective epimorphism of stacks. Then $g\circ f$ satisfies $\mathbf{Q}$-descent if and only if $g$ satisfies $\mathbf{Q}$-descent.
\end{lemma}

\begin{proof}
    Since $f:\mathcal{X}\rightarrow\mathcal{Y}$ is an effective epimorphism in the topos of sheaves, we have 
    $$\overline{\mathrm{im}}(g\circ f)\cong\overline{\mathrm{im}}(g)$$
    for any $g$, where $\overline{\mathrm{im}}(h)$ denotes the image of $h$ in the topos of sheaves. Since $\mathbf{Q}$ satisfies descent for $\tau$ we have
    $$\mathbf{Q}(\overline{\mathrm{im}}(h))\cong\mathbf{Q}(\mathrm{im}(h))$$
    Thus to verify that a map of sheaves 
    $$h:\mathcal{V}\rightarrow\mathcal{W}$$ 
    satisfies descent it suffices to prove that the map
    $$\mathbf{Q}(\mathcal{W})\rightarrow\mathbf{Q}(\overline{\mathrm{im}}(h))$$
    is an equivalence. Now by the above we have 
      $$\mathbf{Q}(\overline{\mathrm{im}}(g\circ f))\cong\mathbf{Q}(\overline{\mathrm{im}}(g))$$
      so clearly 
       $$\mathbf{Q}(\mathcal{Z})\rightarrow\mathbf{Q}(\overline{\mathrm{im}}(g))$$
       is an equivalence if and only if 
          $$\mathbf{Q}(\mathcal{Z})\rightarrow\mathbf{Q}(\overline{\mathrm{im}}(g\circ f))$$
          is.
\end{proof}

\begin{remark}
The main purpose of this result is the following situation. We know that some $\mathbf{Q}$ satisfies descent for class of covers of affines by affines, and we wish to `globalise' this. That is we consider representable maps of stacks $\mathcal{X}\rightarrow\mathcal{Y}$ such that for any map $U\rightarrow\mathcal{Y}$ with $U$ affine, $\mathcal{X}\times_{\mathcal{Y}}U$ admits an atlas $\{U_{i}\rightarrow\mathcal{X}\times_{\mathcal{Y}}U\}$ such that the composite cover $\{U_{i}\rightarrow U\}$ is in a class for which we know $\mathbf{Q}$ satisfies descent. Now $\coprod_{i}U_{i}\rightarrow\mathcal{X}\times_{\mathcal{Y}}U$ is an effective epimorphism of stacks. It then follows that $\mathbf{Q}$ satisfies descent for $\mathcal{X}\times_{\mathcal{Y}}U\rightarrow U$. Thus $\mathcal{X}\rightarrow\mathcal{Y}$ satisfies universal $\mathbf{Q}$-descent. 
\end{remark}
}
\comment{\subsubsection{Decent squares}

We say that a square 
\begin{displaymath}
\xymatrix{
\mathcal{X}\times_{\mathcal{Z}}\mathcal{Y}\ar[r]\ar[d]  & \mathcal{Y}\ar[d]^{p}\\
\mathcal{X}\ar[r]^{e} & \mathcal{Z}
}
\end{displaymath}

is a \textit{(universal)} $\mathcal{F}$-\textit{descent square} if the map
$$\mathcal{X}\coprod\mathcal{Y}\rightarrow\mathcal{Z}$$
is a (universal) $\mathcal{F}$-descent morphism.

Denote by $\chi^{\mathcal{F}-desc}$ the class of $\mathcal{F}$-descent squares. Then $R(\chi^{\mathcal{F}-desc})$ consists of universal $\mathcal{F}$-descent squares such that $e$ is a monomorphism. 

Let $\tau^{\mathcal{F}-cd}$ denote the topology generated by the cd-structure $\Delta(R(\chi^{\mathcal{F}-desc}))$.

\textcolor{red}{Complete?}

\subsection{Descent and rigidification}

\textcolor{red}{To be completed}
Let $\mathbf{D}$ be equipped with a topology $\tau$.
    Let $\mathbf{E}=\mathbf{Cat}^{L}$ be the category of $(\infty,1)$-categories with maps being left adjoint functors, and let $\mathbf{Q}\in\mathbf{PreShv}(\mathbf{D};\mathbf{Cat}^{L})$ satisfy descent for $\tau$. Suppose further that each $\mathbf{Q}(U)$ is a dualisable category. Finally, we let $\mathbf{E}'$ denote the full subcategory $(\mathbf{Pr}^{L})^{dlb}$ of dualisable categories.}
\comment{
\section{Projective dimension in derived algebraic contexts}

Fix a derived algebraic context $\underline{\mathbf{C}}=(\mathbf{C},\mathbf{C}_{\ge0},\mathbf{C}_{\le0},\mathbf{C}^{0})$.

\begin{definition}
    Let $\alpha$ be a cardinal. $\underline{\mathbf{C}}$ is said to be $\kappa$-\textit{small} if for all objects $P,Q\in\mathbf{C}^{0}$, $\mathrm{Hom}_{\mathbf{C}^{\heart}}(P,Q)$ is $\alpha$-small.
\end{definition}

In the language of \cite{transfiniteadamsrepresentability}*{Definition 2.11}, $\mathbf{C}^{0}$ has \textit{cardinality} $\alpha$.

Let $\mathbf{C}^{\alpha}$ denote the full subcategory of $\mathbf{C}^{\heart}$ consisting of $\alpha$-presentable projectives. Following \cite{transfiniteadamsrepresentability}, let $\mathrm{Mod}_{\alpha}(\mathbf{C}^{\alpha})$ denote the category of functors $(\mathbf{C}^{\alpha})^{op}\rightarrow\mathbf{Ab}$ which commute with $\alpha$-small products. 

\begin{lemma}
    The restricted Yoneda embedding
    $$\mathbf{C}^{\heart}\rightarrow\mathrm{Mod}_{\alpha}(\mathbf{C}^{\alpha})$$
    is an exact equivalence of categories.
\end{lemma}

\begin{proof}
 \textcolor{red}{Jack: should follow from page 8 of \cite{transfiniteadamsrepresentability}}.
 By \cite{transfiniteadamsrepresentability} Page 8, there is an equivalence of categories of projective objects on both sides. 
\end{proof}

Recall \cite{transfiniteadamsrepresentability} that an object $F$ of $\mathbf{C}^{\heart}$ is said to be $\alpha$-\textit{flat}
 if it is an $\alpha$-filtered colimit of $\alpha$-presentable projective objects. In particular it is flat. 
 
\begin{proposition}[\cite{transfiniteadamsrepresentability}*{Proposition 2.13}]
        Let $\underline{\mathbf{C}}$ be $\aleph_{k}$ small for some $k$. Then any $\aleph_{k}$-flat object gas projective dimension $n+1$. 
\end{proposition}

\begin{corollary}
    Let $K$ be a Banach ring of cardinality at most $\aleph_{1}$. Then any nuclear Fr\'{e}chet $K$-module has projective dimension $2$. 
\end{corollary}

}

\end{document}